\documentclass[12pt]{report}
\usepackage{latexsym,amsfonts,amssymb,amsthm,amsmath}
\usepackage{amscd}
\usepackage{mathrsfs,multicol}
\usepackage[all]{xy}  
\usepackage{enumerate}
\usepackage{stmaryrd}
\usepackage{color}
\usepackage{imakeidx}
\usepackage{tikz}
\usepackage{tikz-cd}
\usepackage{pdflscape}
\usepackage{hyperref, cleveref}
\usepackage[vcentermath,enableskew]{youngtab}
\usepackage{centernot}
\usepackage{multirow}
\usepackage{caption}

\makeindex

\newtheorem{theorem}{Theorem}[chapter]
\newtheorem{lemma}[theorem]{Lemma}
\newtheorem{proposition}[theorem]{Proposition}
\newtheorem{corollary}[theorem]{Corollary}
\newtheorem{definition}[theorem]{Definition}

\newtheorem{remark}[theorem]{Remark}
\numberwithin{equation}{section}
\numberwithin{theorem}{chapter}

\theoremstyle{definition}

\newtheorem*{theorem*}{Theorem}

\newtheorem*{question*}{Question}
\newtheorem{example}[theorem]{Example}
\newtheorem*{basic properties}{Basic Properties}

\newtheorem*{Important Remark}{Important Remark}
\theoremstyle{remark}
\newtheorem*{claim*}{Claim}

\newcommand{\dblsp}{\renewcommand{\baselinestretch}{2}\tiny\normalsize}
\newcommand{\sglsp}{\renewcommand{\baselinestretch}{1}\tiny\normalsize}

\newtheoremstyle{named}{}{}{\itshape}{}{\bfseries}{.}{.5em}{\thmnote{#3 }#1}
\theoremstyle{named}
\newtheorem*{namedconjecture}{Conjecture}

\newcommand{\n}{\mathfrak{n}}

\newcommand{\h}{\operatorname{h}}

\newcommand{\RR}{\mathbb{R}}
\newcommand{\NN}{\mathbb{N}}
\newcommand{\ZZ}{\mathbb{Z}}
\newcommand{\QQ}{\mathbb{Q}}
\newcommand{\FF}{\mathbb{F}}
\newcommand{\CC}{\mathbb{C}}
\newcommand{\HH}{\mathbb{H}}

\newcommand{\ord}{\operatorname{ord}}

\newcommand{\ind}{\operatorname{ind}}

\newcommand{\lp}{\left(}
\newcommand{\rp}{\right)}
\newcommand{\Log}{\textnormal{Log}}
\newcommand{\legendre}[2]{\genfrac{(}{)}{}{}{#1}{#2}}

\makeatletter
\renewcommand*\env@matrix[1][\arraystretch]{%
  \edef\arraystretch{#1}%
  \hskip -\arraycolsep
  \let\@ifnextchar\new@ifnextchar
  \array{*\c@MaxMatrixCols c}}
\makeatother

\begin{document}

\thispagestyle{empty}

\begin{center}

\vspace*{1in}

{\large \textsc{On the Coefficients of $q$-series and Modular Forms}}
\\[12mm]
William Craig
\\
Charlottesville, Virginia
\\[12mm]
B. Math. Virginia Polytechnic Institute and State University, 2019 \\[16mm]

A Dissertation presented to the Graduate Faculty
\\
of the University of Virginia in Candidacy for the Degree of
\\
Doctor of Philosophy
\\[11mm]
Department of Mathematics
\\[7mm]
University of Virginia
\\

\end{center}

\newpage

\dblsp

\newpage

\tableofcontents \addtocontents{toc}
{\protect\thispagestyle{myheadings}\markright{}}

\newpage

\chapter*{Acknowledgments} \label{acknowledgements}

\thispagestyle{myheadings}

\dblsp
\vspace*{-.65cm}

I am grateful to my advisor, Ken Ono, for wonderful support and instruction, and the rest of my committee, Peter Humphries, Evangelia Gazaki, and Ross Cameron. I would like to thank the University of Virginia and the NSF for financial support throughout my time in graduate school. The work in this thesis could not have been done without the fantastic environment in which I have found myself here. I would, in particular, like to thank my fellow graduate students for stimulating conversations and friendships, especially Alejandro De Las Pe\~nas Castano, Eleanor McSpirit, Badri Pandey, and Hasan Saad. I am indebted to many collaborators with whom I have worked on many chapters of this thesis. In particular, Chapters 5 and 7 are joint work with Anna Pun, Chapter 6 is joint work with Kathrin Bringmann, Joshua Males, and Ken Ono, and Chapter 8 is joint work with Jennifer Balakrishnan, Ken Ono, and Wei-Lun Tsai. Finally, I am indebted to my parents William and Elizabeth, my brothers John and David, and my wife Kara for their emotional support and encouragement. I am especially grateful to Kara for her patience and love throughout the many hours spent on the results presented here.

\sglsp

\chapter*{Abstract} \label{abstract}

\thispagestyle{myheadings}

\dblsp
\vspace*{-.65cm}

This thesis is on partitions and analytic number theory. In particular, I prove results about statistical properties of partitions, partition inequalities, and facts about special values of coefficients of modular forms. The central methods of this paper are the theory of integer weight modular forms and the circle method.

It is natural to study statistical questions about the parts of partitions. Recently, Beckwith and Mertens proved that the parts of partitions are asymptotically equidistributed among residue classes modulo $t$, but that there is a bias towards the residue classes inhabited by lower positive integers. In this thesis, I prove that the same phenomenon holds for partitions into distinct parts, and I prove that the biases between residue classes holds for $n > 8$. In order to prove these results, I derive explicit error terms for asymptotic estimates involving Euler--Maclaurin summation and utilize Wright's circle method to prove asymptotic formulas approximating the relevant counting functions.

Motivated by work of Dergachev and Kirillov, new work by Coll, Mayers and Mayers explores new connections between partitions and Lie theory via the index of seaweed algebras. This index may be viewed as a statistic on pairs of partitions, and in this light Coll, Mayers, and Mayers conjectured that a peculiar kind of generating function identity related to this new index statistic. Seo and Yee made a significant step towards proving this conjecture by reducing the problem to demonstrating the non-negativity of the coefficients of a certain $q$-series. In this thesis, I complete the proof of this conjecture using Wright's circle method and effective Euler--Maclaurin summation.

Hook numbers of partitions arise naturally from the connection between partitions and the irreducible representations of the symmetric group. I prove results concerning the number of $t$-hooks that appear within partitions. In joint work with Pun, I prove formulas that give the number of partitions of $n$ which have an even or odd number of $t$-hooks, and as a consequence we prove that these counting functions obey a strange distributions law. We prove these results using the Rademacher circle method.

In joint work with Bringmann, Males and Ono, I prove further asymptotic formulas about the distributions of $t$-hooks and Betti numbers in residue classes. We prove that the Betti numbers associated to Hilbert schemes on $n$ points, which naturally add up to the number of partitions of $n$, are equidistributed among residue classes modulo $b$, while equidistribution fails when partitions are divided up based on the residue class of the number of $t$-hooks. These results are proved using both Rademacher-style and Wright-style circle methods. We also use facts about 2-core and 3-core generating functions to prove that certain coefficients vanish in the cases of 2-hooks and 3-hooks.

Since DeSalvo and Pak proved that the partition function is log-concave, the Tur\'an inequalities have been a popular topic within partition theory. These inequalities govern whether certain polynomials constructed from a given sequence of numbers are hyperbolic. In joint work with Pun, I prove that the $k$-regular partition functions satisfy all the Tur\'an inequalities. We prove this using Hagis' formula for the $k$-regular partition functions and a very general criterion for proving Tur\'an inequalities proven by Griffin, Ono, Rolen, and Zagier.

The Atkin-Lehner newforms are extremely important examples of modular forms. Their coefficients are multiplicative, and the values at prime powers are dictated by two-term linear recurrence relations coming from Hecke operators. In joint work with Balakrishnan, Ono, and Tsai, I prove a methodology for identifying which coefficients of certain integer weight newform $f(z)$ are allowed to take on a given odd value. In particular, our method proves that under suitable assumptions, $f(z)$ has only finitely many Fourier coefficients equal to a given odd prime, and we give an algorithm which determines the possible locations of these prime values by computing integer points on algebraic curves with large genus.

\sglsp

\chapter{Introduction} \label{C1}
\thispagestyle{myheadings}

\dblsp
\vspace*{-.65cm}

\section{Partitions and modular forms} \label{S1.1}

In this thesis, I present original results pertaining to the coefficients of infinite series connected to partitions and modular forms. Both of these objects are ubiquitous in modern mathematics, with applications to fields as diverse as combinatorics, mathematical physics, number theory, representation theory, and topology. As my results are mainly combinatorial and number theoretic in nature, we shall introduce these objects from this perspective.

A {\it partition} is a non-increasing sequence of positive integers, which we denote by
\begin{align*}
	\lambda = (\lambda_1, \lambda_2, \cdots \lambda_\ell), \ \ \ \lambda_1 \geq \lambda_2 \geq \cdots \geq \lambda_\ell > 0.
\end{align*}
It is standard to denote by $\mathcal P$ the set of all partitions. For any $\lambda \in \mathcal P$, let us call denote by $|\lambda|$ the {\it size} of $\lambda$, which is defined by $|\lambda| = \lambda_1 + \lambda_2 + \cdots + \lambda_\ell$. If $|\lambda| = n$, then we say $\lambda$ is a partition of $n$ and we write $\lambda \vdash n$. We will also call each $\lambda_i$ a {\it part} of the partition.

The oldest question about partitions we know goes back to a letter from Leibniz to J. Bernoulli \cite{Lei}, in which Leibniz asks Bernoulli about the number of ``divulsions" of integers, which we now call partitions. In modern notation, Leibniz's question concerns the {\it partition function}
\begin{align*}
	p(n) := \#\{ \lambda \in \mathcal P : |\lambda| = n \},
\end{align*}
in particular how to evaluate the function. Many interesting questions about partitions have arisen as the theory developed. Among these are the possibility of multiplicative structure, asymptotic growth rates, formulas for partition functions, and rapid methods for computing values of partition functions.

The first published work on partitions goes back to Euler in 1741 \cite{Eul}, in which he answers questions of Naud\'{e} on a variation of Leibniz's question on the calculation of partition functions. Euler answers these questions brilliantly using the (very new at the time) mechanism of {\it generating functions}, which led him to develop many beautiful identities connecting infinite sums and infinite products. The fundamental example is the generating function for $p(n)$, which Euler proves to be 
\begin{align*}
	P(q) := \sum_{n=0}^\infty p(n) q^n = \prod_{n=1}^\infty \dfrac{1}{1 - q^n} =: \lp q; q \rp_\infty^{-1},
\end{align*}
where we have used the standard $q$-Pochhammer notation
\begin{align} \label{q-Pochhammer Symbol}
	\lp a; q \rp_n := \prod_{k=0}^{n-1} \lp 1 - a q^k \rp, \hspace{0.3in} \lp a; q \rp_\infty := \lim_{n \to \infty} \lp a; q \rp_n.
\end{align}
Euler's pioneering work deals with manipulation of formal power series which he connects to partitions, and produces many recurrence relations for partition functions via sum-product identities such as Euler's ``Pentagonal Number Theorem" \cite{And83}. Essentially, Euler answers questions about the calculation of partition functions by using generating function identities to derive recurrence relations for various partition functions. Such recurrence relations then allow for computations far more quickly than explicit enumeration. Through such fundamental results, Euler established the theory of partitions. His methods and results are still absolutely central in modern research.

Another pioneer in partition theory was the great Indian mathematician Srinivasa Ramanujan. Ramanujan's work contains in particular two revolutionary theorems on partitions. One of these are his congruences \cite{Ram19,Ram20,Ram21}, the most basic of which are
\begin{align*}
	p(5n+4) \equiv 0 \pmod{5}, \hspace{0.2in} p(7n+5) \equiv 0 \pmod{7}, \hspace{0.2in} p(11n+6) \equiv 0 \pmod{11},
\end{align*}
which he proved using techniques of manipulating $q$-series manipulation which in principle would have been accessible to Euler or Jacobi. Through these results, Ramanujan initiated the study of divisibility properties of partition numbers, which is a fundamental area of research today and has led to such developments as Dyson's rank function \cite{Dys44} and the Andrews-Garvan crank function \cite{AG88, Gar88} that give combinatorial explanations for why Ramanujan's congruences are true.

Another of Ramanujan's groundbreaking results on partitions came in his famous joint paper with Hardy \cite{HR18} in which they prove an asymptotic formula for $p(n)$ as $n \to \infty$ using the {\it circle method} (see \eqref{Hardy-Ramanujan Asymptotic}). Rademacher then improved their work in 1937 to obtain an exact formula for $p(n)$ \cite{Rad37}. These are fundamental results answering questions of the type Leibniz asked about partitions. The circle method was a major revolution which has seen important applications not just to partition functions, but to many other famous asymptotic problems in number theory. This is because the circle method can be interpreted very generally as a method for computing asymptotic formulae for any sequence of integers based on the asymptotic properties of its generating function. In particular, the method is useful for problems involving the number of ways to represent integers by elements of a given set, like Waring's problem or Goldbach-type problems.

Hardy and Ramanujan's implementation of the circle method is centrally based on relating the generating function of $p(n)$ to {\it modular forms}. Roughly speaking, modular forms are analytic functions $f : \HH \to \CC$, where $\HH = \{ \tau \in \CC: \mathrm{Im}(\tau) > 0 \}$, that transform nicely under the action of M\"{o}bius transformations on $\HH$. That is, for integers $a,b,c,d$ such that $ad - bc = 1$, a modular form should have the property that $f\lp \frac{a\tau+b}{c\tau+d}\rp$ is nicely related to $f(\tau)$ for all $\tau \in \HH$.

The connection between modular forms and partitions comes through the marriage of Euler's generating functions with complex analysis as developed in the mid-nineteenth century. Because modular forms are naturally periodic, under $\tau \mapsto \tau + 1$, any modular form $f(\tau)$ is going to come with a Fourier expansion $\sum_{n \in \ZZ} a_f(n) e^{2\pi i n \tau}$. If we let $q := e^{2\pi i \tau}$, then it turns out that the formal generating function $P(q)$ for partitions is closely connected with one of the fundamental examples of a modular form called {\it Dedekind's eta function}, denoted $\eta(\tau)$. This function satisfies the relation $P(q) = q^{\frac{1}{24}} \eta(\tau)^{-1}$ and has the modular transformation law 
\begin{align} \label{Dedekind eta transformation law}
	\eta\lp - \dfrac{1}{\tau} \rp = \sqrt{-i\tau} \cdot \eta(\tau).
\end{align}
Transformation laws of this shape play a central role in the execution of the circle method. Modular forms come in many different shapes, and have deep connections to the functional equations of $L$-functions, algebraic geometry, and many other areas of mathematics. For example, the modular transformation law of a certain theta function underpins Riemann's proof of the analytic continuation and functional equation of his zeta function \cite{Rie59}, and by extension modular transformation laws are used to derive functional equations for modular $L$-functions. For more details on the theory of modular forms, see Chapter \ref{C2} or standard texts on modular forms such as \cite{Apo90,CS17,DS05b,Ser73}. Andrews' book \cite{And84} is an excellent source for the theory of partitions.

The connections between modular forms and partitions, as well as various generalizations and refinements of the Hardy--Ramanujan circle method, play a central role in this thesis. Section \ref{S1.2} describes my results on generating function identities and partition inequalities proven using the circle method, and Section \ref{S1.3} describes joint papers with Pun and with Bringmann, Males and Ono that use the circle method to study arithmetic statistics of certain invariants attached to partitions. Section \ref{S1.4} describes an application of the circle method to Tur\'{a}n inequalities, and Section \ref{S1.5} describes joint work with Balakrishnan, Ono, and Tsai on the coefficients of certain integral weight modular forms. Chapters \ref{C3} through \ref{C8} then prove the results introduced in Chapter \ref{C1}.

\section{Partition identities and inequalities} \label{S1.2}

\subsection{Parts of partitions into distinct parts}

In their famous paper {\it Asymptotic formulae in combinatory analysis}, Hardy and Ramanujan (among other results) proved the asymptotic formula
\begin{align} \label{Hardy-Ramanujan Asymptotic}
	p(n) \sim \dfrac{1}{4n\sqrt{3}} e^{\pi \sqrt{\frac{2n}{3}}}
\end{align}
as $n \to \infty$. In fact, they are able to prove a complete divergent asymptotic expansion for $p(n)$ \cite{HR18}. As mentioned in Section \ref{S1.1}, one of the key tools in their method is the modular transformation law for Dedekind's eta function given in \eqref{Dedekind eta transformation law}. The main thrust of the proof is that the modular transformation law for $\eta(\tau)$ yields a similar transformation law for $P(q)$, which then gives good asymptotic estimates for the size of $P(q)$ near complex roots of unity. Through a remarkable series of calculations, Hardy and Ramanujan are able to translate this asymptotic information about $P(q)$ into asymptotic information about $p(n)$. There are many other important works, including Meinardus \cite{Mei54} and Wright \cite{Wri71}, which demonstrate a variety of methods of computing asymptotic expansions for partition functions. In particular, the method of Wright will be central to the Sections \ref{S1.2} and \ref{S1.3}.

In Chapter \ref{C3}, I present results about the total number of parts among partitions into distinct parts residing in given congruences classes. As is standard, we let $\ell(\lambda)$ be the number of parts possessed by the partition $\lambda$. The number of parts contained in partitions is one of the most well-studied combinatorial aspects of these objects. For example, famous work of Erd\H{o}s and Lehner \cite{EL41} shows that for large $n$, almost all partitions of $n$ contain
\begin{align*}
	\lp 1 + o(1) \rp \dfrac{\sqrt{6n}}{2\pi} \log(n)
\end{align*}
parts. Such results have been extended in various directions. One such instance is a recent result of Griffin, Ono, Rolen and Tsai \cite{GORT22} that counts expected number of parts that are multiples of a given integer.

Dartyge and Sarkozy have studied a related problem in \cite{DS05a}, in which they prove a result which indicates that the parts of partitions might favor certain congruence classes. More specifically, for positive integers $0 < r \leq t$, define the function
\begin{align*}
	T_{r,t}(\lambda) := \# \{ \lambda_i \in \lambda : \lambda_i \equiv r \pmod{t} \}.
\end{align*}
Dartyge and Sarkozy \cite[Theorem 1.1]{DS05a} prove that for $n \gg 0$ and $0 < r < s \leq t$, a positive proportion of partitions satisfy the inequality
\begin{align*}
	T_{r,t}(\lambda) - T_{s,t}(\lambda) > \dfrac{(r+s) \sqrt{n}}{50rs}.
\end{align*}
Philosophically, such a result makes sense; smaller positive integers may be repeated more times within partitions of a fixed size. We may note however that the expected number of parts of a random partition is on the order $\sqrt{n} \log(n)$, which outstrips the Dartyge--Sarkozy lower limit. At least at face value, this suggests that it could still be true that parts of partitions are equidistributed among all residue classes.

Beckwith and Mertens \cite{BM15, BM17} answer these questions. Letting $0 < r \leq t$ and $n \geq 0$ be integers, Beckwith and Mertens define\footnote{Beckwith and Mertens use the notation $\widehat T_{r,t}(n)$ for this function.}
\begin{align*}
	T_{r,t}(n) := \sum_{\lambda \vdash n} T_{r,t}(\lambda),
\end{align*}
which counts the total number of parts congruent to $r$ modulo $t$ among all partitions of $n$. In their second paper studying this function, Beckwith and Mertens prove the following theorem.

\begin{theorem}[{\cite[Theorem 1.2]{BM17}}] \label{Beckwith--Mertens}
	Let $0 < r < s \leq t$ and $n \geq 0$ be integers. Then as $n \to \infty$, we have
	\begin{align*}
		T_{r,t}(n) = e^{\pi \sqrt{\frac{2n}{3}}} \left[ \log(n) - \log\lp \dfrac{\pi^2}{6} \rp - 2 \lp \psi\lp \dfrac{r}{t} \rp + \log(t) \rp + O\lp n^{- \frac 12} \log(n) \rp \right].
	\end{align*}
	In particular, we have $T_{r,t}(n) \sim T_{s,t}(n)$ and $T_{r,t}(n) \geq T_{s,t}(n)$ as $n \to \infty$.
\end{theorem}

The asymptotic above agrees with the heuristics suggested by comparing the results of Erd\"{o}s--Lehner with those of Dartyge--Sarkozy, that the parts should be both equidistributed asymptotically and exhibit a strict inequality for $n \gg 0$.

There are two natural follow-up questions concerning this result -- does this phenomenon hold for other families of partitions, and how large must $n$ be before $T_{r,t}(n) \geq T_{s,t}(n)$ begins to hold? In Chapter \ref{C3}, we shall address both of these questions in the context of partitions into distinct parts. We say a partition $\lambda \in \mathcal P$ has {\it distinct parts} if no two $\lambda_i \in \lambda$ are equal, and we let $\mathcal D$ be the set of partitions into distinct parts. In analogy with Beckwith and Mertens, we define for integers $0 < r \leq t$ and $n \geq 0$ the function
\begin{align*}
	D_{r,t}(n) := \sum_{\lambda \in \mathcal D} D_{r,t}(\lambda) := \sum_{\lambda \in \mathcal D} \# \{ \lambda_i \in \lambda : \lambda_i \equiv r \pmod{t} \}.
\end{align*}
As in the case of $T_{r,t}(n)$, we prove an asymptotic formula for $D_{r,t}(n)$.

\begin{theorem} \label{C3 Ineffective Asymptotic}
	As $n \to \infty$, we have
	\begin{align*}
		D_{r,t}(n) = \dfrac{3^{\frac 14} e^{\pi \sqrt{\frac{n}{3}}}}{2\pi t n^{\frac 14}} \left( \log(2) + \lp \dfrac{\sqrt{3} \log(2)}{8\pi} - \dfrac{\pi}{4\sqrt{3}} \lp r - \dfrac{t}{2} \rp \rp n^{- \frac 12} + O\lp n^{-1} \rp \right).
	\end{align*}
\end{theorem}

\begin{example}
	We consider the case $t = 3$ to illustrate the accuracy of the approximation of $D_{r,3}(n)$ in Theorem \ref{C3 Ineffective Asymptotic}. Let $\widehat{D}_{r,t}(n)$ denote the main term of $D_{r,t}(n)$ from Theorem \ref{C3 Ineffective Asymptotic}. Additionally, let $Q_r(n) := \frac{D_{r,3}(n)}{ \widehat{D}_{r,3}(n)}$. The following table illustrates the convergence of $Q_r(n)$ to 1 as $n \to \infty$.
	
	\begingroup
	\renewcommand{\arraystretch}{0.5}
	\begin{center}
		\begin{tabular}{|c|c|c|c|c|c|} \hline
			$n$ & 10 & 100 & 1000 & 10000 \\ \hline
			$Q_1(n)$ & 1.159706 & 1.002613 & 1.001068 & 1.000365 \\ \hline
			$Q_2(n)$ & 0.904238 & 1.003913 & 1.001204 & 1.000378 \\ \hline
			$Q_3(n)$ & 1.167157 & 1.008440 & 1.001641 & 1.000422 \\ \hline
		\end{tabular}
	
		\vspace{0.05in}
		
		{\small Table 1: Numerics for Theorem \ref{C3 Ineffective Asymptotic}.}
	\end{center}
	\endgroup
\end{example}

Theorem \ref{C3 Ineffective Asymptotic} immediately implies $D_{r,t}(n) \sim D_{s,t}(n)$ and $D_{r,t}(n) \geq D_{s,t}(n)$ as $n \to \infty$; this is because the main term of $D_{r,t}(n)$ does not depend on $r$ and the secondary term depends monotonically on $r$. To make the inequality explicit, we improve Theorem \ref{C3 Ineffective Asymptotic} by making the error terms completely explicit. The following results contain our explicit asymptotics and the explicit bias which follows from it.

\begin{theorem} \label{C3 Effective Asymptotic}
	For any integer $t \geq 2$ and all integers $n > \frac{400 t^2}{3}$, we have
	\begin{align*}
		\bigg| D_{r,t}(n) - \dfrac{\log(2)}{t} V_0(n) + \dfrac{1}{2} B_1\lp \dfrac rt \rp V_1(n) - \frac{t}{8} B_2\lp \dfrac rt \rp V_2(n) &+ \dfrac{t^3}{192} B_4\lp \dfrac rt \rp V_4(n) \bigg| \\ &\leq \mathrm{Err}_t(n),
	\end{align*}
	where $B_n(x)$ are the Bernoulli polynomials defined in \eqref{Bernoulli Polynomial Definition}, $\mathrm{Err}_t(n)$ is defined in \eqref{E_t Definition}, and $V_s(n)$ is defined in \eqref{V Definition}.
\end{theorem}

\begin{corollary} \label{C3 Effective Inequality}
	For positive integers $1 \leq r < s \leq t$ we have $D_{r,t}(n) \geq D_{s,t}(n)$ for sufficiently large $n$. In particular, for $2 \leq t \leq 10$ this inequality holds for all $n > 8$.
\end{corollary}

\begin{remark}
	We make the following remarks regarding Theorem \ref{C3 Ineffective Asymptotic} and Corollary \ref{C3 Effective Inequality}.
	\begin{enumerate}
		\item Numerics suggest that the only tuples $(r,s,n)$ which can furnish counterexamples to $D_{r,t}(n) \geq D_{s,t}(n)$ are $(1,2,2), (2,3,4), (2,4,4), (3,4,7)$, and $(4,5,8)$. Each of these holds for sufficiently large $t$. For instance, the partitions of $8$ into distinct parts are
		\begin{align*}
			8, 7+1, 6+2, 5+3, 5+2+1, 4+3+1.
		\end{align*}
		Observe that 5 appears as a part twice and 4 only appears as a part once; this implies that $D_{5,t}(8) > D_{4,t}(n)$ for all $t \geq 5$. The other counterexamples listed above are similar in nature.
		\item Similar results are possible to derive for other restricted partition functions. In particular, Jackson and Otgonbayar \cite{JOa,JOb} have studied the analogous results for $k$-regular partitions and $k$-indivisible partitions. They prove that $k$-regular partitions have an exactly analogous bias phenomenon, whereas $k$-indivisible partitions have a more complicated bias which is not in general monotonic in $r$.
	\end{enumerate}
\end{remark}

The proofs of these results occurs in four steps. We first produce generating functions for $D_{r,t}(n)$ using standard techniques which we review in Chapter \ref{C2}. We then use a technique derived from Euler--Maclaurin summation to estimate this generating function near $q=1$. We then use a variation of the circle method due to Wright to translate these estimates into estimates for the coefficients $D_{r,t}(n)$, which we finally translate into effective inequalities through elementary computations and computer calculations.

\subsection{Seaweed algebras and the index statistic}

Partition theory arises in many surprising ways throughout mathematics. One of the most surprising might be the connections with Lie theory. For example, Macdonald \cite{Mac72} unified many disparate theorems about power of Dedekind's eta function under a Lie theoretic framework. Other applications in Lie theory have arisen through the work of Dergachev and Kirillov \cite{DK00} on calculating the index of parabolic subalgebras of $\mathrm{GL}(n)$. In Chapter \ref{C4}, we will answer a conjecture of Coll, Mayers, and Mayers connected to the work of Dergachev and Kirillov.

We first describe the construction of seaweed algebras by Dergachev and Kirillov. Let $\{ e_j \}_{1 \leq j \leq n}$ be the standard basis of $k^n$ for some field $k$. Given two partitions $\{ a_j \}_{1 \leq j \leq m}$, $\{ b_j \}_{1 \leq j \leq \ell}$ of $n$, Dergachev and Kirillov \cite{DK00} defined seaweed algebras as Lie subalgebras of $\text{Mat}(n)$ which preserve the vector spaces $\text{span}\lp e_1, e_2, \dots, e_{a_1 + \dots + a_j} \rp$ for $1 \leq j \leq m$ and $\text{span}\lp e_{b_1 + \dots + b_j + 1}, \dots, e_n \rp$ for $1 \leq j \leq \ell$.

\begin{example}[Partitions of 8]
	Let $\lambda = \lp 3, 3, 2 \rp$ and $\mu = \lp 4, 3, 1 \rp$. The seaweed algebra associated to the pair $\lp \lambda, \mu \rp$ is the set of all $8 \times 8$ matrices $X$ of the form below:
	\begin{align*}
		X = \begin{pmatrix} * & * & * & 0 & 0 & 0 & 0 & 0 \\ * & *& * & 0 & 0 & 0 & 0 & 0 \\ * & * & * & 0 & 0 & 0 & 0 & 0 \\ * & * & * & * & * & * & 0 & 0 \\ 0 & 0 & 0 & 0 & * & * & * & 0 \\ 0 & 0 & 0 & 0 & * & * & 0 & 0 \\ 0 & 0 & 0 & 0 & * & * & * & * \\ 0 & 0 & 0 & 0 & 0 & 0 & 0 & * \\ \end{pmatrix}
	\end{align*}
	Each part $\lambda_i$ of $\lambda$ is used to construct a $\lambda_i \times \lambda_i$ triangle of $*$'s in upper triangular section of the matrix, and similarly for $\mu$ in the lower triangular section.
\end{example}

In \cite[Theorem 5.1]{DK00}, Dergachev and Kirillov obtain an exact formula for the index of seaweed algebras which is calculated from a certain graph constructed from $\lambda, \mu$. We denote by $\ind_\mu(\lambda)$ the index of the seaweed algebra constructed from the pair $\lp \lambda, \mu \rp$. Coll, Mayers and Mayers in \cite{CMM20} initiate the study of $\ind$ as a partition-theoretic object, proving for example a connection between the special case of $\ind_\mu(\lambda)$ with $\mu = (1, 1, \cdots, 1)$ to the well-studied 2-colored partition function \cite[Theorem 11]{CMM20}.

Coll, Mayers and Mayers also studied the $q$-series
\begin{align*}
	G(q) := \prod_{n=1}^\infty \dfrac{1}{1 + \lp -1 \rp^n q^{2n-1}} =: \lp q, -q^3; q^4 \rp_\infty^{-1},
\end{align*}
using the standard abbreviation $\lp a, b; q \rp_\infty := \lp a; q \rp_\infty \cdot \lp b; q \rp_\infty$. Note that because of the factor $\lp -q^3; q^4 \rp_\infty^{-1}$, it is not clear whether $G(q)$ has non-negative coefficients. We consider the restricted index statistic $\ind_{(n)}(\lambda)$, which we henceforth denote by $\ind(\lambda)$. Coll, Mayers, and Mayers define $e(n)$ (resp. $o(n)$) as the number of partitions of $n$ into odd parts whose index is even (resp. odd). In this setting, they make the following interesting conjecture \cite[Conjecture 20]{CMM20} connecting the index statistic to $G(q)$.

\begin{namedconjecture}[Coll--Mayers--Mayers] \label{C4 CMM Conjecture}
	The following are true:
	
	\noindent \textnormal{(1)} All the coefficients of $G(q)$ are non-negative.
	
	\noindent \textnormal{(2)} We have $G(q) = \sum\limits_{n \geq 0} \left| e(n) - o(n) \right| q^n$.
\end{namedconjecture}

Previous papers by Seo, Yee, and Chern have made progress towards the conjecture, but a complete proof was not known. Seo and Yee \cite[Theorem 1]{SY20} made a significant conceptual step, proving using generating function methods that it would be enough to prove the first part of the conjecture; that is, if we define
\begin{align*}
	G(q) =: \sum_{n=0}^\infty a(n) q^n,
\end{align*}
then (2) would follow from (1) in the Coll--Mayers--Mayers Conjecture. Chern \cite{Che19} used a version of the circle method to prove an upper limit on the last counterexample to the conjecture, but the constants involved were too large to be calculated on a personal computer, thus the conjecture remained open. We complete the proof of the conjecture, using a different version of the circle method to prove explicit asymptotic formulas for $a(n)$. Our results are as follows:

\begin{theorem} \label{C4 a(n) Asymptotics}
	As $n \to \infty$, we have
	\begin{align*}
		a(n) \sim \dfrac{\Gamma\lp \frac 14 \rp \pi^{\frac 14}}{2^{\frac 94} 3^{\frac 38} n^{\frac 38}} I_{-\frac 34} \lp \dfrac{\pi}{2} \sqrt{\dfrac{n}{3}} \rp + (-1)^n \dfrac{\Gamma\lp \frac 34 \rp \pi^{\frac 34}}{2^{\frac{11}{4}} 3^{\frac 58} n^{\frac 58}} I_{-\frac 54}\lp \dfrac{\pi}{2} \sqrt{\dfrac{n}{3}} \rp,
	\end{align*}
	and for $n > 4800$ the difference between these has absolute value at most $E(n)$ as defined in \eqref{Error Definition}.
\end{theorem}

\begin{theorem} \label{C4 Conjecutre Proof}
	Conjecture \ref{C4 CMM Conjecture} is true. That is, we have
	\begin{align*}
		G(q) = \lp q, -q^3; q^4 \rp_\infty^{-1} = \sum_{n \geq 0} \left|e(n) - o(n)\right| q^n.
	\end{align*}
\end{theorem}

\begin{remark}
	We make several remarks about Theorems \ref{C4 a(n) Asymptotics} and \ref{C4 Conjecutre Proof}.
	\begin{enumerate}
		\item Theorem \ref{C4 a(n) Asymptotics} implies Chern's result (i.e. Theorem 1.2 of \cite{Che19}).
		\item Chern proves $a(n) \geq 0$ for $n > 2.4 \times 10^{14}$ using his explicit error terms. Our explicit error terms prove $a(n) \geq 0$ for $n \geq 3.5 \times 10^5$, which reduces the problem to a feasible computation on the author's personal computer.
		\item In combination with \cite[Theorem 1]{SY20} of Seo--Yee, our result also proves that the sign of $e(n) - o(n)$ is periodic.
	\end{enumerate}
\end{remark}

The proof of Theorems \ref{C4 a(n) Asymptotics} and \ref{C4 Conjecutre Proof} rely on an explicit application of Wright's circle method. As $G(q)$ is not any kind of modular object, we will require the explicit Euler--Maclaurin asymptotic techniques that are developed in Chapter \ref{C3}. Because the two factors $\lp q; q^4 \rp_\infty^{-1}$ and $\lp -q^3; q^4 \rp_\infty^{-1}$ have poles which nearly cancel each other, we have to add an additional layer to the calculations. In particular, we must include in the so-called ``major arc" not just behavior as $q \to 1$ but also $q \to -1$. Although this does not rely on a traditional usage of Wright's circle method, it remains in the same spirit.

\section{Arithmetic statistics of partitions} \label{S1.3}

\subsection{Distribution of $t$-hooks modulo 2}

In Sections \ref{S1.2} and \ref{S1.3}, we have discussed results derived from Wright's circle method, which in a sense is tailed to generating functions which are not suitably modular. If the generating functions are modular, then by the work of Rademacher on $p(n)$ \cite{Rad37} we can improve on these results and use the circle method\footnote{These exact formulas can also be derived using the method of Poincar\'{e} series, see for example \cite{CS17}.} to derive exact formulas. In particular, Rademacher proved that for $n \geq 1$, we have
\begin{align} \label{Rademacher Exact}
	p(n) = \dfrac{2\pi}{\lp 24n - 1 \rp^{\frac 34}} \sum_{k=1}^\infty \dfrac{K_k(n)}{k} I_{\frac 32}\lp \dfrac{\pi \sqrt{24n - 1}}{6k} \rp,
\end{align}
where $I_{\frac 32}$ is the classical $I$-Bessel function of index $\frac 32$ and $K_k(n)$ is a certain ``Kloosterman sum" defined by
\begin{align} \label{p(n) Kloosterman Sum}
	K_k(n) := \sum_{\substack{0 \leq h < k \\ (h,k) = 1}} e^{\pi i s(h,k) - 2\pi i n \frac{h}{k}}, \hspace{0.2in} s(h,k) := \sum_{r=1}^{k-1} \dfrac{r}{k} \lp \dfrac{hr}{k} - \left\lfloor \dfrac{hr}{k} \right\rfloor - \dfrac 12 \rp.
\end{align}
These results are extended in a very general setting by Zuckerman \cite{Zuc39}.

In Chapter \ref{C5}, we prove an analogous exact formula connected to hook numbers of partitions. To define hook numbers, it is most natural to refer to the {\it Young diagram} of a partition, which for $\lambda = \lp \lambda_1, \dots, \lambda_\ell \rp$ is a diagram of left-justified cells with $\lambda_i$ cells in row $i$. In these diagrams, we fill each cell $(i,j)$ with a {\it hook number} $h_{i,j}(\lambda)$, which is defined as the number of cells lying below or to the right of $(i,j)$ in the Young diagram of $\lambda$. We let $\mathcal H(\lambda)$ denote the multiset of hook numbers of $\lambda$, and $\mathcal H_t(\lambda)$ the multiset of hook numbers of $\lambda$ that are multiples of $t$, which we call {\it $t$-hooks}.
\begin{example}
	Consider the partition $\lambda = (5, 4, 1)$, with hook diagram
	$$\young(75431,5321,1).$$
	Then $\mathcal H(\lambda) = \{ 1, 1, 1, 2, 3, 3, 4, 5, 5, 7 \}$, $\mathcal{H}_2(\lambda) = \{ 2, 4 \}$ and $\mathcal{H}_5(\lambda) = \{ 5, 5 \}$.
\end{example}

Hook numbers play a central role in the representation theory of the symmetric group. It is well known that the partitions of $n$ index the irreducible representations of $S_n$ \cite{JK84}. This is not merely a bijection, but these representations can be constructed from the partitions, and properties of the hook numbers in the corresponding Young diagrams encode properties of the representations. For example, the famous Frame-Thrall-Robinson formula says that if $\rho_\lambda$ is the $S_n$-representation associated to $\lambda$ we have $\dim \rho_\lambda = \frac{n!}{\prod_{h \in \mathcal H(\lambda)} h}$.
In combinatorics, hook numbers show up in the Nekrasov-Okounkov hook length formula \cite{NO06}, which says that for any complex number $z$, we have
\begin{align} \label{Nekrasov-Okounkov}
	\sum_{\lambda \in \mathcal P} x^{|\lambda|} \prod_{h \in \mathcal H(\lambda)} \lp 1 - \dfrac{z}{h^2} \rp = \prod_{n=1}^\infty \lp 1 - q^n \rp^{z-1}.
\end{align}
This formula connects the study of hook numbers to modular forms via Dedekind's eta function, as \eqref{Nekrasov-Okounkov} connects hook numbers to powers of the eta function.

For integers $t \geq 2$ and any partition $\lambda$, we wish to study the size of the $t$-hook multisets $\mathcal H_t(\lambda)$, in particular their parity. We define
\begin{align*}
	p_t^e(n) &:= \# \{ \lambda \vdash n : \# \mathcal H_t(\lambda) \equiv 0 \pmod{2} \}, \\ p_t^o(n) &:= \# \{ \lambda \vdash n : \# \mathcal H_t(\lambda) \equiv 1 \pmod{2} \}.
\end{align*}
We wish to study the distribution of the parity of $\# \mathcal H_t(\lambda)$. Since $p_t^e(n) + p_t^o(n) = p(n)$, we wish to study $\delta_t^{e/o}(n) = \frac{p_t^{e/o}(n)}{p(n)}$. Consider the following tables which give values of these functions.

\begingroup
\renewcommand{\arraystretch}{0.75}

\begin{table}
	\begin{center}
	\begin{tabular}{|c|c|c|c|c|c|} \hline
		$t$ & $\delta_t^e(100)$ & $\delta_t^e(1000)$ & $\delta_t^e(10000)$ & $\cdots$ & $\infty$ \\ \hline
		2 & 0.56611246  &0.50027931  & 0.50000000  &$\cdots$ & $\frac{1}{2}$ \\ \hline
		4 & 0.47067843  &0.50002869  & 0.50000000  &$\cdots$ & $\frac{1}{2}$ \\ \hline
		6 & 0.52465920  &0.50007471  & 0.50000000  &$\cdots$ & $\frac{1}{2}$ \\ \hline
		8 & 0.49484348  &0.49999135  & 0.50000000  &$\cdots$ & $\frac{1}{2}$ \\ \hline
	\end{tabular}
	\caption{Data for $\delta_t^e(n)$, even $t$}
	\end{center}

	\begin{center}
	\begin{tabular}{|c|c|c|c|c|c|c|} \hline
		$t$ & $\delta^e_t(100)$ & $\delta^e_t(500)$ & $\delta^e_t(1000)$ & $\delta_t^e(1500)$  &$\cdots$ & $\infty$ \\ \hline
		3 &0.7137967695 & 0.7502983017 &0.7499480195 & 0.7500039425 & $\cdots$ & $\frac{3}{4}$ \\ \hline
		5 &0.6374948698 & 0.6252149479 &0.6250102246 & 0.6250009877 &  $\cdots$ & $\frac{5}{8}$ \\ \hline
		7 &0.5468769228 & 0.5624965413 &0.5625165550 & 0.5624989487 &  $\cdots$ & $\frac{9}{16}$ \\ \hline
		9 &0.5375271584 & 0.5313027269 &0.5312496766 & 0.5312499631 &  $\cdots$ & $\frac{17}{32}$ \\ \hline
	\end{tabular}
	\caption{Data for $\delta_t^e(n)$, $t$ odd and $n$ even.}
	\end{center}

	\begin{center}
	\begin{tabular}{|c|c|c|c|c|c|c|} \hline
		$t$ & $\delta^e_t(101)$ & $\delta^e_t(501)$ & $\delta^e_t(1001)$ & $\delta_t^e(1501)$ &  $\cdots$ & $\infty$ \\ \hline
		3 &0.2376157284 & 0.2494431573 & 0.2499820335 & 0.2500060167 &  $\cdots$ & $\frac{1}{4}$ \\ \hline
		5 &0.3755477486 & 0.3750000806 & 0.3750000001 & 0.3750000000 &  $\cdots$ & $\frac{3}{8}$ \\ \hline
		7 &0.4396942088 & 0.4374987794 & 0.4374959329 & 0.4375000006 &  $\cdots$ & $\frac{7}{16}$ \\ \hline
		9 &0.4787668076 & 0.4688094755 & 0.4687535414 & 0.4687510507 &  $\cdots$ & $\frac{15}{32}$ \\ \hline
	\end{tabular}
	\caption{Data for $\delta_t^e(n)$, $t$ odd and $n$ odd.}
	\end{center}
\end{table}

\endgroup

Numerically, this initial speculation receives support for small values of $t$ like $t = 2, 4, 6$, and $8$. However, numerical evidence below for the cases $t = 3 ,5 ,7$ and $9$ appears to refute this naive guess. In fact, these tables suggest the existence of multiple limiting values.

In Chapter \ref{C5}, we prove the following theorems that explain this data. In particular, we see what the correct limiting values of $\delta^{e/o}_t(n)$ are.

\begin{theorem} \label{C5 Even and Odd T Behavior}
	Assuming the notation above, the following are true.
	
	1) If $t>1$ is an even integer, then $$\lim_{n \to \infty} \delta_t^e(n) = \lim_{n \to \infty} \delta_t^o(n) = \dfrac{1}{2}.$$
	
	2) If $t>1$ is an odd integer, then we have $$\lim_{n \to \infty} \delta_t^e(n) = \begin{cases} \dfrac{1}{2} + \dfrac{1}{2^{(t+1)/2}} & \text{if } 2 \mid n, \\ \dfrac{1}{2} - \dfrac{1}{2^{(t+1)/2}} & \text{if } 2 \nmid n, \end{cases} \hspace{0.2in} \text{and} \hspace{0.2in} \lim_{n \to \infty} \delta_t^o(n) = \begin{cases}  \dfrac{1}{2} - \dfrac{1}{2^{(t+1)/2}} & \text{if } 2 \mid n, \\ \dfrac{1}{2} + \dfrac{1}{2^{(t+1)/2}} & \text{if } 2 \nmid n. \end{cases}$$
\end{theorem}

We also study the sign pattern of $p^e_t(n) - p^o_t(n)$, for $n \rightarrow \infty$, which determines when $p_t^e(n) > p_t^o(n)$ and $p_t^o(n) > p_t^e(n)$.

\begin{theorem}\label{C5 Distribution property}
	For $t>1$ a fixed positive integer, write $t = 2^s\ell$ for integers $s, \ell \geq 0$ such that $\ell$ odd. Then for sufficiently large $n$, the sign of $p^e_t(n) - p^o_t(n)$ is periodic with period $2^{s+1}$. In particular, when $t$ is odd the sign of $p^e_t(n) -p^o_t(n)$ is alternating for sufficiently large $n$.
\end{theorem}

These results are proven using the Rademacher circle method. In particular, we use \eqref{Nekrasov-Okounkov} to show that the generating function for $A_t(n) := p_t^e(n) - e_t^o(n)$ is a modular form. We then follow the arguments of Rademacher to prove an exact formula for $A_t(n)$, and we then study this exact formula to determine the main terms that yield Theorem \ref{C5 Even and Odd T Behavior}, and the Kloosterman sums involved yield the sign patterns in Theorem \ref{C5 Distribution property}. 

\subsection{Distributions of $t$-hooks and Betti numbers}

In analogy with the previous section, we might consider functions of the form
\begin{align*}
	p_t(a,b;n) := \# \{ \lambda \vdash n : \#\mathcal H_t(\lambda) \equiv a \pmod{b} \},
\end{align*}
which specialize to the functions $p_t^e(n)$ and $p_t^o(n)$ when $b = 2$. Although the arguments are much more involved, we can still derive an exact formula for these coefficients. In Chapter \ref{C6}, we use the circle method to produce the following asymptotics for $p_t(a,b;n)$.

\begin{theorem}\label{C6 t-hook Asymptotic}
	If $t>1$, $b$ is an odd prime, and $0\leq a<b ,$  then as $n\rightarrow \infty$ we have
	$$
	p_t(a,b;n)\sim \frac{c_t(a,b;n)}{4\sqrt{3}n}\cdot  e^{\pi \sqrt{\frac{2n}{3}}},
	$$
	where $c_t(a,b;n)$ are certain rational numbers defined in Chapter \ref{C6}, \eqref{c_t def}.
\end{theorem}

As a corollary, we obtain the following limiting distributions.

\begin{corollary}\label{C6 t-hook Distribution}
	Assuming the hypotheses in Theorem \ref{C6 t-hook Asymptotic}, if $0\leq a_1<b$ and $0\leq a_2 <b,$  then 
	$$
	\lim_{n\rightarrow \infty}\frac{p_t(a_1, b; b n+a_2)}{p(b n+a_2)}=c_t(a_1,b;a_2).
	$$
\end{corollary}
In particular, if $b | t$ we have $p_t(a_1, b; n) \sim p_t(a_2, b; n)$ as $n \to \infty$ for any $a_1, a_2$. If $b \centernot | t$, then equidistribution fails. Examples of the results are given in Chapter \ref{C6}.

The cases where $t\in \{2, 3\}$ are particularly striking. In addition to many instances of non-uniform distribution, there are situations where certain counts are actually identically zero.

\begin{theorem}\label{C6 Vanishing}
	The following are true.
	
	\begin{enumerate}
		\item[\normalfont(1)] If $\ell$ is an odd prime and $0\leq a_1, a_2<\ell$ satisfy
		$(\frac{-16a_1+8a_2+1}{\ell})=-1,$ then for every non-negative integer $n$ we have
		$$
		p_2(a_1,\ell;\ell n+a_2)=0.
		$$
		
		\item[{\normalfont(2)}] If $\ell\equiv 2\pmod 3$ is prime and $0\leq a_1, a_2<\ell^2$ have the property that
		$\ord_{\ell}(-9a_1+3a_2+1)=1$, then
		for every non-negative integer $n$ we have
		$$
		p_3\left(a_1,\ell^2;\ell^2 n+a_2\right)=0.
		$$
	\end{enumerate}
\end{theorem}

For example, Theorem~\ref{C6 Vanishing} (1) implies that
\begin{align*}
	p_2(0,3; 3n+2)=p_2(1, 3; 3n+1)=p_2(2,3;3n)=0
\end{align*}
and Theorem~\ref{C6 Vanishing} (2) implies that
\begin{align*}
	p_3(0,4;4n+3)=p_3(1,4;4n+2)=p_3(2,4;4n+1)=p_3(3,4;4n)=0.
\end{align*}

This result is proved not with the circle method, but with $q$-series identities related to the paucity of 2-cores and 3-core partitions of $n$, which is discussed in Chapter \ref{C2}. By the work of Granville and Ono \cite{GO96}, there are $t$-core partitions of $n$ for every $t \geq 4, n \geq 1$, and this explains why Theorem \ref{C6 Vanishing} only applies to the cases $t=2$ and $t=3$.

In Chapter \ref{C6}, we prove results on the Betti numbers of Hilbert schemes in algebraic geometry. We denote by $b_j(X)$ the $j$th {\it Betti numbers} of the scheme $X$, which is the dimension of its $j$th homology group, i.e. $b_j(X) = \dim \lp H_j(X,\QQ) \rp$. These numbers are generated by the usual {\it Poincar\'{e} polynomial} $P(X;T) := \sum_j b_j(X) T^j$. Work of G\"{o}ttsche \cite{Got94, Got02} and of Buryak and Feigin \cite{BF13, BFN15} establishes generating functions for these Poincar\'{e} polynomials for certain Hilbert schemes $\lp \CC^2 \rp^{[n]}$ and $\lp \lp \CC^2 \rp^{[n]} \rp^{T_{\alpha,\beta}}$, whose definitions we defer until Chapter \ref{C6}. The relevant generating functions are expressible as products of $q$-Pochhammer symbols in the relevant variables, which are closely related to (but not equal to) modular forms. Because of the infinite product representations of these generating functions, the Euler--Maclaurin asymptotic method can be used to give asymptotic estimates for the generating functions near roots of unity, which again allows applications of the circle method.

The application we consider involve the modular sums of Betti numbers
\begin{equation*}
	B\lp a,b; X \rp := \sum_{j\equiv a\pmod b} b_j\lp X \rp =
	\sum_{j\equiv a\pmod b} \dim \lp H_j\lp X,\QQ \rp \rp
\end{equation*}
where $X$ represents the Hilbert schemes we consider. Now, equidistribution in the most literal sense fails, since the odd index Betti numbers for these schemes identically vanish. However, we can prove that this is the only obstruction for equidistribution modulo $b$ for these modular Betti sums. In particular, we define the constant
\begin{equation} \label{C6 d(a,b) Definition}
	d(a,b):=\begin{cases} \frac{1}{b} \ \ \ \ \ &{\text {\rm if $b$ is odd,}}\\
		\frac{2}{b} \ \ \ \ \ &{\text {\rm if $a$ and $b$ are even,}}\\
		0 \ \ \ \ \ &{\text {\rm if $a$ is odd and $b$ is even.}}
	\end{cases}
\end{equation}

\begin{theorem}\label{C6 Betti Asymptotic}
	Assuming the notation above, the following are true.
	\begin{enumerate}[leftmargin=*]
		\item[\rm (1)] As $n\rightarrow \infty$, we have
		$$
		B\left(a,b; \left(\CC^2\right)^{[n]}\right)\sim \frac{d(a,b)}{4\sqrt{3}n}\cdot  e^{\pi \sqrt{\frac{2n}{3}}}.
		$$
		
		\item[\rm (2)] If $\alpha, \beta\in \NN$ are relatively prime, then as $n\rightarrow \infty$ we have 
		$$
		B\left(a,b; \left(\left(\CC^2\right)^{[n]}\right)^{T_{\alpha,\beta}}\right)\sim \frac{d(a,b)}{4\sqrt{3}n} \cdot e^{\pi \sqrt{\frac{2n}{3}}}.
		$$
	\end{enumerate}
\end{theorem}

Since the sum over all Betti numbers of these schemes is equal to $p(n)$, to study the distribution modulo $b$ of the modular Betti sums, one considers the ratios
\begin{equation*}
	\delta(a,b;n):=
	\frac{B\left(a,b; \left(\CC^2\right)^{[n]}\right)}{p(n)}\  \ \ \ {\text {\rm and}}\ \ \ \
	\delta_{\alpha,\beta}(a,b;n):=\frac{B\left(a,b; \left(\left(\CC^2\right)^{[n]}\right)^{T_{\alpha,\beta}}\right)}{p(n)}.
\end{equation*}
As a consequence of Theorem~\ref{C6 Betti Asymptotic}, we obtain distributions for these proportions.

\begin{corollary}\label{C6 Betti Distribution}
	If $0\leq a<b$, then the following are true. 
	\begin{enumerate}
		\item[\rm (1)] We have that
		$$
		\lim_{n\rightarrow \infty} \delta(a,b;n)=d(a,b).
		$$
		
		\item[\rm (2)] If $\alpha, \beta\in \NN$ are relatively prime, then we have
		$$
		\lim_{n\rightarrow \infty} 
		\delta_{\alpha,\beta}(a,b;n)=d(a,b).
		$$
	\end{enumerate}
\end{corollary}

\section{Applications to Tur\'{a}n inequalities} \label{S1.4}

The study of the Tur\'{a}n inequalities begins first with the study of hyperbolic polynomials. Recall that a real polynomial is called {\it hyperbolic} if all of its roots are real. For the simplest nontrivial case, i.e. quadratic polynomials $ax^2 + bx + c$, the hyperbolicity is determined by the discriminant inequality $b^2 - 4ac \geq 0$. This simple observation leads naturally to the question of how to determine the hyperbolicity of higher degree polynomials on the basis of their coefficients. This is the purpose of the higher-order Tur\'{a}n inequalities, whose precise definition we defer until Chapter \ref{C7}.

Recently, there has been great interest in proving Tur\'{a}n inequalities for polynomials of number-theoretic interest. Given a sequence of real numbers $\{ \alpha(n) \}_{n \geq 0}$, the {\it Jensen polynomial of degree $d$ and shift $n$} associated to the sequence is the polynomial
\begin{align*}
	J_\alpha^{d,n}(X) := \sum_{k=0}^d \binom{d}{k} \alpha(n+k) X^k.
\end{align*}
Jensen polynomials have a close relationship to the Riemann hypothesis, as Poly\'{a} \cite{Pol27} has shown that the Riemann hypothesis is equivalent to the hyperbolicity of all the Jensen polynomials associated to the Taylor coefficients of the Riemann xi-function. This approach to the Riemann hypothesis has recently been taken up in \cite{GORZ19}.

We now consider the Tur\'{a}n inequalities for other number-theoretic sequences. We call the sequence $\{ \alpha(n) \}_{n \geq 0}$ {\it log-concave} if $\alpha(n)^2 - \alpha(n-1) \alpha(n+1) \geq 0$ for all $n \geq 1$, or that $\alpha$ is log-concave for $n \in \NN$ if this inequality is satisfied for that particular value of $n$. Nicolas \cite{Nic78} and DeSalvo and Pak \cite{DP15} have shown that the partition function $p(n)$ is log-concave for $p(n) \geq 25$, which in turn proves that the Jensen polynomials $J_p^{2,n}(X)$ are hyperbolic for $n \geq 25$. In analogy with this case, the sequence $\{ \alpha(n) \}_{n \geq 0}$ satisfies the Tur\'{a}n inequalities of order $d$ if and only if $J_\alpha^{d,n}(X)$ is hyperbolic for all $n \geq 1$. The case of $d = 3$ for the partition function was proven by Chen, Jia and Wang in \cite{CJW19}, and they further conjectured that there were constants $N(d)$ such that the Jensen polynomials $J_p^{d,n}(X)$ would be hyperbolic for all $n \geq N(d)$. 

This conjecture was proven in a very general setting by Griffin, Ono, Rolen, and Zagier \cite{GORZ19}. They proved that if a sequence $\alpha(n)$ satisfies certain very general asymptotic properties, then certain renormalizations of the Jensen polynomials $J_\alpha^{d,n}(X)$ converge uniformly to certain {\it Hermite polynomials} $H_d(X)$ for fixed $d$ as $n \to \infty$. Since this family of polynomials is known to have only simple real roots, this proves the hyperbolicity for sufficiently large $n$, provided certain asymptotic formulas hold for the sequence $\alpha(n)$. The Hardy-Ramanujan asymptotic formula turns out to be sufficient for $p(n)$, which completes the proof.

In Chapter \ref{C7}, we investigate this in the case of the so-called $k$-regular partitions $p_k(n)$, which counts the number of partitions none of whose parts are divisible by $k$ (or none of whose parts occur with $k$ or more multiplicities). Hagis \cite{Hag71} has derived an exact formula for this function using the circle method, and using this formula we prove the following:

\begin{theorem} \label{C7 Main Theorem}
	If $k \geq 2$ and $d \geq 1$, then $$\lim\limits_{n \to \infty} \widehat{J}^{d,n}_{p_k}(X) = H_d(X),$$ uniformly for $X$ on compact subsets of $\mathbb{R}$, where $\widehat{J}^{d,n}_{p_k}(X)$ are renormalized Jensen polynomials for $p_k(n)$ as defined in (\ref{J-Hat Def}).
\end{theorem}

\begin{corollary} \label{C7 Corollary}
	For $k \geq 2$, $d \geq 1$, then $J^{d,n}_{p_k}(X)$ is hyperbolic for $n \gg 0$.
\end{corollary}

\begin{remark}
	By Corollary \ref{C7 Corollary}, there exists a minimal natural number $N_k(d)$ such that $J^{d,n}_{p_k}(X)$ is hyperbolic for all $n \geq N_k(d)$. These numbers are not the focus of this paper, however a brief discussion is worthwhile, as these numbers dictate the effectiveness of the main theorem. The following table provides conjectural values of $N_k(d)$ for small $k$ and $d$.
	
	\begin{center}
		\begin{tabular}{|c|c|c|c|c|} \hline
			$d$ & $N_2(d)$ & $N_3(d)$ & $N_4(d)$ & $N_5(d)$ \\ \hline
			$2$ & $32$ & $57$ & $16$ & $41$ \\ \hline
			$3$ & $120$ & $184$ & $63$ & $136$ \\ \hline
			$4$ & $266$ & $390$ & $137$ & $294$ \\ \hline
		\end{tabular}
	\end{center}
\end{remark}

These results will be proved in Chapter \ref{C7}. The basic idea of the proof is that any sequence with a suitably modular generating function necessarily has an asymptotic formula of a shape similar to that of the Hardy-Ramanujan formula for $p(n)$. It happens that $p_k(n)$ has such a formula, as its generating function is a modular form, and we may conclude on this basis after verifying the conditions of \cite[Theorem 3]{GORZ19} that the Jensen polynomials $J_{p_k}^{d,n}(X)$ are eventually hyperbolic for any fixed $d$ as $n \to \infty$.

\section{Variants of Lehmer's conjecture} \label{S1.5}

One of the most important examples of a modular form is furnished by {\it Ramanujan's Delta function} $\Delta(z)$, which is defined by
\begin{align*}
	\Delta(z) := \eta^{24}(z) = q \prod_{n=1}^\infty \lp 1 - q^n \rp^{24} =: \sum_{n=1}^\infty \tau(n) q^n,
\end{align*}
where we now write $q = e^{2\pi i z}$ to avoid abusing notation. The coefficients $\tau(n)$ are referred to as {\it Ramanujan's $\tau$-function}, so-called because of Ramanujan's study of this function in ``On certain arithmetical functions" \cite{Ram16}. $\Delta(z)$ stands out, for example, as the unique normalized cusp form of level 1 and weight 12 (see Chapter \ref{C2} for definitions). Ramanujan's study of this function has led to many important developments, both in his proven results and in his conjectures. Ramanujan was able to prove many congruences for $\tau(n)$ \cite{Ram16}, which Serre later viewed as evidence of a much larger theory of Galois representations \cite{Ser68}. Ramanujan also conjectured the multiplicativity of this function and that its values at prime powers form a recursive sequence; this was proven first by Mordell \cite{Mor17} and foreshadowed the theory of Hecke operators. Ramanujan's conjectured bounds for $\tau$ at prime values were a corollary of Deligne's celebrated proof of the Weil Conjectures \cite{Del74, Del80}.

Much is known about $\tau(n)$, and yet some basic questions remain unanswered. For example, Lehmer's conjecture\footnote{The author is not aware of any written record of Lehmer conjecturing an answer to this problem, but we will follow convention and refer to this problem as Lehmer's conjecture.} asks whether there are any positive integers $n$ such that $\tau(n) = 0$. Lehmer himself proved \cite{Leh47} that if $\tau(n) = 0$ for any positive integers $n$, then there must be some prime $p$ such that $\tau(p) = 0$. Serre was able to show using the Chebotarev Density Theorem that the set of such primes, if there are any, has density zero within the primes \cite{Ser81}. This result was improved upon several times; it is now known due to work of Thorner and Zaman \cite{TZ18} that
\begin{align*}
	\#\{ p \leq X \ \text{prime} : \tau(p) = 0 \} \ll \pi(X) \cdot \dfrac{\lp \log\log(X) \rp^2}{\log(X)}.
\end{align*}
One can observe by multiplicativity that if $\tau(p) = 0$ for even one prime p, then $\tau(n) = 0$ for a positive proportion of integers $n$, and it is now known due to Hu, Iyer, and Shashkov \cite{HIS} that the density of of $n$ for which $\tau(n) = 0$ is at most $1.15 \times 10^{-12}$. In yet another direction, Calegari and Sardani \cite{CS21} have shown that at most finitely many non-CM newforms with fixed tame $p$ level $N$ have vanishing $p$th Fourier coefficient.

We consider a generalization of this question, asking for all solutions to the equation $\tau(n) = \alpha$ for any odd $\alpha$. Murty, Murty and Shorey \cite{MMS87} proved that $\tau(n) = \alpha$ for at most finitely many values of $\alpha$; however, their method involves enormous bounds coming from Baker's theory of linear forms in logarithms, and so in practice it is not very useful for explicitly solving the equation. In fact, this approach has only been used to show that the only solution to $\tau(n) = \pm 1$ is $\tau(1) = 1$. For $\alpha=\pm \ell$, where $\ell$ is almost any odd prime,
it is widely believed that there are  no solutions.  However, there are counterexamples, such as
Lehmer's prime value example \cite{Leh65}
\begin{equation}\label{LehmerPrime}
	\tau(251^2)=80561663527802406257321747.
\end{equation}
\noindent 
Lygeros and Rozier \cite{LR13} have subsequently discovered further prime values.

We study the same problem with a different method that not only proves that $\tau(n) = \alpha$ has finitely many solutions, but also theoretically locates where those solutions are allowed to occur. This method works not only for Ramanujan's tau-function, but also for any Atkin-Lehner newform \cite{AL70} with ``trivial mod 2 Galois representation." We will not directly use Galois representations, but the idea of having a trivial mod 2 Galois representation is exemplified by the congruence
\begin{align*}
	\Delta(z) = q \prod_{n=1}^\infty \lp 1 - q^n \rp^{24} \equiv q \prod_{n=1}^\infty \lp 1 - q^{8n} \rp^3 \equiv \sum_{n=1}^\infty q^{(2n+1)^2} \pmod{2},
\end{align*}
proven using the Jacobi triple product identity. This shows that the odd values of $\tau(n)$ are supported on odd squares, which is what the reader should have in mind when thinking of ``trivial mod 2 Galois representations." In this thesis, we prove a variety of theorems which are aimed at resolving equations of the form $\tau(n) = \alpha$ for $\alpha \in \ZZ$ odd, as well as generalizations of this question to other {\it newforms}, of which $\Delta(z)$ is the first example.

As the full results of this work are quite technical and have a large number of distinct cases, we outline here the results we obtain for just the function $\tau(n)$, but the proofs of Chapter \ref{C8} will be framed in terms of the fully general case.

Let $\ell$ be an odd prime and $m \geq 1$. We wish to resolve equations of the form $\tau(n) = \pm \ell^m$ for $m \geq 1$. Our first major result, which proves that this equation has only finitely many solutions, comes in two steps. The first uses the theory of Lucas sequences to force such solutions to occur for $n = p^{d-1}$ for $p$ an odd prime and only finitely many possible values of $d$. The second stage uses the recurrence relations for Lucas sequences to use any solution $\tau(p^{d-1}) = \pm \ell^m$ to construct an integer point on an algebraic curve of large genus. Such curves only have finitely many integer points by Siegel's theorem, which will complete the proof of finiteness. Our first result for $\tau(n)$ may be stated as follows.

\begin{theorem}
	Let $\ell \in \ZZ^+$ be an odd prime, and let $n > 1$ be such that $\left| \tau(n) \right| = \ell^m$ for an integer $m \geq 1$. Then $n = p^{d-1}$ for some odd prime $p$ and positive integer $d$ satisfying $d | \ell\lp \ell^2 - 1 \rp$. Furthermore, there are at most finitely many solutions $(n,m)$ such that $\left| \tau(n) \right| = \ell^m$.
\end{theorem}

A more general phenomenon is that $\tau(n)$ accumulates prime divisors as $n$ accumulates not merely distinct prime divisors because of its multiplicativity, but also generally will accumulate more prime divisors as $n$ accumulates more of some fixed prime divisor. This result is framed in terms of the classical functions $\Omega(n)$ and $\omega(n)$, which count the prime divisors of $n$ with and without multiplicity, respectively. Our second major result (which does not rely on any algebraic geometry) is as follows.

\begin{theorem}
	Let $n > 1$ be an integer. Then we have
	\begin{align*}
		\Omega\lp \tau(n) \rp \geq \sum_{p | n} \lp \sigma_0\lp \textnormal{ord}_p(n) + 1 \rp - 1 \rp \geq \omega(n),
	\end{align*}
	where $\sigma_0(n)$ counts the number of positive divisors of $n$.
\end{theorem}

Proceeding in another direction, we know from previous discussion that a solution to $\tau(n) = \pm \ell^m$ can be used to produce an integral point on one of finitely many explicitly determined algebraic curves. Using this procedure, we can find all integral points on these curves and then determine by reversing this process whether that point induces a solution $\tau(n) = \pm \ell^m$. Using various techniques in effective algebraic geometry, we obtain the following theorem.

\begin{theorem}
	We have for all $n > 1$ that
	\begin{align*}
		\tau(n) \not \in \{ \pm 3, \pm 5, \pm 7, \pm 13, \pm 17, -19, \pm 23, \pm 37, \pm 691 \}.
	\end{align*}
\end{theorem}

\begin{remark}
	We add the following comments to this theorem.
	\begin{enumerate}
		\item If we assume the generalized Riemann hypothesis, additional values can be ruled out.
		\item The case $\pm 691$ requires additional input from the special congruence satisfied by $\tau(n)$ modulo 691.
	\end{enumerate}
\end{remark}

The last major result cannot be phrased in terms of $\tau(n)$ alone, because it requires that we allow the weight of the newform to vary. Previous results, when framed in their fully general context, show that for a given newform $f$ of weight $k$ with integer Fourier coefficients $a_f(n)$, there are only finitely many solutions to equations of the form $a_f(n) = \pm \ell^m$ for odd primes $\ell$ and $m \geq 1$. This last theorem moves the weight $k$ instead of the coefficient $n$.

\begin{theorem}
	Let $\ell$ be an odd prime and $m \geq 1$ an integer. Then there exists an effectively computable constants $M^{\pm}(\ell,m) = O_\ell(m)$ such that $\pm \ell^m$ is not a coefficient of any such newform $f$ of weight $2k > M^{\pm}(\ell, m)$ with integer coefficients, trivial mod 2 Galois representation, and even level coprime to $\ell$.
\end{theorem}

Many results have followed this work, particularly discussing generalizations of the first three results. For examples, see \cite{AH,AH22,BOT22,BGPS22,DJ21,HM21,LL21}. Most notably, in \cite{BGPS22} it is shown that $\tau(n) \not = \pm \ell^m$ for any odd primes $3 \leq \ell < 100$ and any positive integer $m$.

These results are proved using the theory of newforms, the theory of Lucas sequences, and effective algebraic geometry. In particular, the theory of newforms gives a connection between $\tau(n)$ and Lucas sequences, where a theorem of Bilu, Hanrot, and Voutier \cite{BHV01} on primitive prime divisors gives a method of determining the exact locations where solutions $\tau(p^m) = \ell^m$ are allowed to occur. Once these locations are determined, the recurrence relations for Lucas sequences are used to show that any solution $\tau(n) = \ell^m$ corresponds to some point on a finite family of algebraic curves of large genus. We then compute all points on these curves using a variety of methods.

\sglsp

\chapter{Background} \label{C2}
\thispagestyle{myheadings}

\dblsp
\vspace*{-.65cm}

\section{Partitions: Combinatorial Aspects} \label{S2.1}

\subsection{Single-variable generating functions} \label{S2.1.1}

As mentioned in Chapter \ref{C1}, the first fundamental contribution to the theory of partitions is undoubtedly due to Euler, who introduced the tool of generating functions (as defined by Abraham de Moivre) to the theory. Given a sequence $\{ a_n \}_{n \geq 0}$, the {\it generating function} associated to that sequence is the formal power series
\begin{align*}
	A(x) := \sum_{n \geq 0} a_n x^n.
\end{align*}
In partition theory, it is customary to use $q$ as the formal variable, although many older works use $x$. We also make regular use of the $q$-Pochhammer symbol, as defined in \eqref{q-Pochhammer Symbol}, by
\begin{align*}
	\lp a; q \rp_\infty := \prod_{n=0}^\infty \lp 1 - a q^n \rp.
\end{align*}
We will generally use $q$ for this variable. To see how the method works, we consider the most important theorem of this type coming from Euler's method. We note here that, as usual in partition theory, we let $p(0) = 1$ (i.e. the empty set denotes the only partition of 0).

\begin{theorem}
	Let $p(n)$ be the partition function. Then we have as formal power series the identity
	\begin{align*}
		P(q) := \sum_{n \geq 0} p(n) q^n = \prod_{n=1}^\infty \dfrac{1}{\lp 1 - q^n \rp} = \lp q;q \rp_\infty^{-1}.
	\end{align*}
	Furthermore, the function $P(q)$ is analytic for values $q \in \CC$ such that $|q| < 1$.
\end{theorem}

\begin{proof}
	The main tool is the well-known geometric series identity
	\begin{align*}
		\dfrac{1}{1 - q^n} = 1 + q^n + q^{2n} + q^{3n} + \cdots = \sum_{k \geq 0} q^{kn}.
	\end{align*}
	If this identity is taken as one of formal power series it is simply true; if we interpret each side as functions of a complex variable, we have to assume $|q| < 1$. Now, we may expand
	\begin{align*}
		\prod_{n=1}^\infty \dfrac{1}{1 - q^n} = \prod_{n=1}^\infty \sum_{k = 0}^\infty q^{kn} = \sum_{k_1, k_2, \dots \geq 0} q^{k_1 + 2 k_2 + 3 k_3 + \cdots}.
	\end{align*}
	The $k_i$ should be understood as corresponding to the $k$ in the middle expression, and the constants $1, 2, 3, \dots$ correspond to the range of values of $n$. Now, given a partition $\lambda \in \mathcal P$, we may write
	$|\lambda| = m_1(\lambda) + 2 m_2(\lambda) + 3 m_3(\lambda) + \dots$, where $m_i(\lambda)$ denotes the number of times $i$ is repeated in $\lambda$. Since the family of values $\{ m_i(\lambda) \}$ uniquely determines the underlying partition, we have
	\begin{align*}
		\sum_{k_1, k_2, \dots \geq 0} q^{k_1 + 2 k_2 + 3 k_3 + \cdots} = \sum_{\lambda \in \mathcal P} q^{|\lambda|} = \sum_{n \geq 0} p(n) q^n.
	\end{align*}
	This completes the proof.
\end{proof}

Once this technique is understood, the proofs can be made much shorter. Really, the essence of the proof is that
\begin{align*}
	\prod_{n=1}^\infty \dfrac{1}{1 - q^n} = \prod_{n=1}^\infty \lp 1 + q^n + q^{n+n} + q^{n+n+n} + \cdots \rp = \sum_{n \geq 0} p(n) q^n,
\end{align*}
where the first equality is by geometric series and the second comes from interpreting the term selected from each geometric series denote a multiplicity of a part. Euler used this kind of thinking to great effect, and ever since his time this method has been indispensable in partition theory, as we will see throughout the remainder of the chapter and the thesis.

We will use as a further example of this method a famous theorem of Euler and a generalization of it which we shall require later.

\begin{definition}
	Let $k \geq 2$ be an integers. A partition into parts which are repeated at most $k-1$ times is called a {\it $k$-distinct partitions}, and a partition in which no part is a multiple of $k$ is called a {\it $k$-regular partition}.
\end{definition}

The most basic examples of this are the cases $k=2$. We call a 2-distinct partition simply a {\it distinct partition}, as it is by definition a partition all of whose parts are distinct. We call a 2-regular partition an odd partition, as it is a partition into all of whose parts are odd. One of the foundational theorems of Euler which demonstrates the power of generating functions is the following:

\begin{theorem}[Euler, Glaischer]
	For all $n \geq 0$, the number of odd partitions of $n$ is equal to the number of distinct partitions of $n$. More generally, the number of $k$-distinct partitions of $n$ equals the number of $k$-regular partitions of $n$.
\end{theorem}

\begin{proof}
	Euler's case, i.e. $k=2$, follows from the algebraic identity
	\begin{align*}
		\prod_{n=1}^\infty \dfrac{1}{\lp 1 - q^{2n-1} \rp} = \prod_{n=1}^\infty \dfrac{\lp 1 - q^{2n} \rp}{\lp 1 - q^n \rp} = \prod_{n=1}^\infty \lp 1 + q^n \rp,
	\end{align*}
	along with the fact that by thinking with Euler's methodology for generating functions the left side counts partitions into odd parts and the right side counts partitions into distinct parts. The more general case, often referred to as Glaischer's Theorem, has a similar style of proof which flows from the equation
	\begin{align*}
		\prod_{\substack{n \geq 1 \\ k \centernot | n}} \dfrac{1}{1 - q^n} = \prod_{n=1}^\infty \dfrac{\lp 1 - q^{kn} \rp}{\lp 1 - q^n \rp} = \prod_{n=1}^\infty \lp 1 + q^n + q^{2n} + \cdots + q^{(k-1)n} \rp.
	\end{align*}
	This completes the proof.
\end{proof}

During this proof, we obtained the generating function identity
\begin{align} \label{k-regular generating function}
	\sum_{n \geq 0} p_k(n) q^n = \dfrac{\lp q^k; q^k \rp_\infty}{\lp q; q \rp_\infty},
\end{align}
where $p_k(n)$ denotes the number of $k$-regular (or $k$-distinct) partitions of $n$. Results such as these are the prime examples of Euler's methodology, but we shall see later that the method Euler developed is even more general than this.

\subsection{Partition statistics and two-variable generating functions} \label{S2.1.2}

As early as the original works of Euler on partition theory, a central thread in the theory concerns studying intrinsic combinatorial properties exhibited by partitions. The most basic of these is the {\it number of parts}. For a partition $\lambda = (\lambda_1, \lambda_2, \dots, \lambda_r)$ with each $\lambda_i \geq 1$, we may define a number-of-parts function $\ell(\lambda) = r$. We may view $\ell$ as a function $\ell : \mathcal P \to \ZZ$. The theory of {\it partition statistics} may be roughly defined as the study of functions on the set $\mathcal P$ into some natural space, like $\ZZ$, that track some combinatorial feature of interest. Some famous examples of partition statistics that may be expressed as maps $\mathcal P \to \ZZ$ include size, number of parts, the rank \cite{Dys44}, and the crank \cite{AG88}.

Many partition statistics, or more broadly maps between sets of partitions, are most naturally expressed in terms of a sort of geometric method of representing partitions. The standard way of thinking geometrically about a partition is the {\it Ferrers diagram}, in which a partition $\lambda = (\lambda_1, \lambda_2, \dots, \lambda_r)$ is represented as a left-aligned collection of boxes in which the $i$th row as $\lambda_i$ boxes. These diagrams were seen in Section \ref{S1.3}. We may also view the {\it hook numbers} of Section \ref{S1.3} as a variation on partition statistics, more specifically as a function that takes a partition $\lambda$ to a multiset of size $|\lambda|$, or alternatively to an element of $\ZZ^n$.

One of the central types of results which enter into the theory of partition statistics are so-called two-variable generating functions; these typically track the size of a partition in one variable and the partition statistic in a second variable. We shall go back to Euler's very first paper in partition theory to see how this works in the case of parts of a partition. As mentioned in Chapter \ref{C1}, this paper of Euler was dedicated to resolving several counting questions of Naud\'{e} \cite{Eul}. We shall consider one of these to demonstrate a more general formulation of the methodology of Section \ref{S2.1.1}. One of Naud\'{e}'s questions is the following:

\begin{question*}
	How many partitions of 50 are there into seven distinct parts?
\end{question*}

To frame this algebraically, let $d(m,n)$ denote the number of ways to partition $n$ into $m$ distinct parts. Naud\'{e}'s question is to evaluate $d(7,50)$. The answer, as proven by Euler, is $d(7,50) = 522$, which is too large to be reasonably calculated by enumerating all 522 examples. Euler's method relied on generating functions, but not precisely the type discussed in Section \ref{S2.1.1}. Euler instead constructs a two-variable generating function
\begin{align*}
	\sum_{m,n \geq 0} d(m,n) z^m q^n
\end{align*}
which simultaneously keeps track of both the size of partitions and the number of parts in the partitions. The basic idea here is that the exponent of $z$ should keep track of individual parts while ignoring their size, while the exponent of $q$ should play a role just like in Section \ref{S2.1.1}. Using the modern $q$-Pochhammer symbol, Euler's observation was that
\begin{align*}
	\sum_{m,n \geq 0} d(m,n) z^m q^n = \lp zq; q \rp_\infty = \prod_{n=1}^\infty \lp 1 + z q^n \rp.
\end{align*}
Euler did not expand this infinite product by hand in order to calculate $d(7,50)$. Instead, he leveraged algebra. Using the immediate observation
\begin{align*}
	\sum_{m,n \geq 0} d(m,n) z^m q^n = \lp 1 + zq \rp \prod_{n=2}^\infty \lp 1 + z q^n \rp = \lp 1 + zq \rp \prod_{n=1}^\infty \lp 1 + (zq) q^n \rp,
\end{align*}
Euler derived a functional equation from this generating function that leads to the recurrence relation $d(m,n) = d(m,n-m) + d(m-1, n-m)$. This gave him a much quicker method for calculating values of the functions $d(m,n)$, and it isn't too difficult even by hand to show that $d(7,50) = 522$ using this method.

\begin{remark}
	It should be noted that partition recurrences played a central role in partition theory, and in particular for computing large values of partition functions. In fact, one of the fastest ways to compute a table of values for $p(n)$ is to use Euler's pentagonal number theorem
	\begin{align*}
		\lp q; q \rp_\infty = 1 + \sum_{n=1}^\infty (-1)^n \lp q^{\frac{n(3n+1)}{2}} + q^{\frac{n(3n-1)}{2}} \rp
	\end{align*}
	to prove the recurrence relation
	\begin{align*}
		p(n) = \sum_{k \in \ZZ \backslash \{ 0 \}} (-1)^{k+1} p\lp n - \dfrac{k(3k+1)}{2} \rp.
	\end{align*}
	This remains to this day a very efficient method for computing tables of values of $p(n)$.
\end{remark}

We now move to a more general setting, whereby we wish to combine the area of partition statistics with the area of generating functions. To this aim, let $s : \mathcal P \to \ZZ$ be a partition statistic, and let $p_s(m,n)$ be the number of partitions $\lambda$ of $n$ such that $s(\lambda) = m$. It is generally desirable to compute generating functions of the form
\begin{align*}
	\sum_{\lambda \in \mathcal P} z^{s(\lambda)} q^{|\lambda|} = \sum_{\substack{n \geq 0 \\ m \in \ZZ}} p_s(m,n) z^m q^n.
\end{align*}
One could also replace the family $\mathcal P$ with some other family of partitions, say the collection of partitions into distinct parts.

There are many problems about partition statistics that may be addressed in a natural way from the framework of two-variable generating functions. One of the most immediate would be to calculated the limiting behavior of the average of the partition statistic $s$, or equivalently to study asymptotics for $s(n) := \sum_{\lambda \vdash n} s(\lambda)$. The framework of two-variable generating functions makes this fairly straightforward. If we let $S(q) = \sum_{n \geq 0} s(n) q^n$, then it is immediate from calculus that
\begin{align*}
	S(q) = \dfrac{\partial}{\partial z} \bigg|_{z=1} \sum_{\lambda \in \mathcal P} z^{s(\lambda)} q^{|\lambda|} = \dfrac{\partial}{\partial z} \bigg|_{z=1} \sum_{\substack{n \geq 0 \\ m \in \ZZ}} p_s(m,n) z^m q^n.
\end{align*}
This can be used, for instance, to count the number of parts that appear amongst all partitions of $n$, as the following proposition demonstrates.

\begin{proposition}
	For $\lambda \in \mathcal P$, let $\ell(\lambda)$ denote the number of parts of $\lambda$, $p(m,n)$ the number of partitions of $n$ into exactly $m$ parts, and $L(n) = \sum_{\lambda \vdash n} \ell(\lambda)$. Then we have the generating function identities
	\begin{align*}
		\sum_{m,n \geq 0} p(m,n) z^m q^n = \lp zq; q \rp_\infty^{-1}
	\end{align*}
	and
	\begin{align*}
		\sum_{n \geq 0} L(n) q^n = \lp q; q \rp_\infty^{-1} \sum_{m \geq 1} \dfrac{q^m}{1 - q^m}.
	\end{align*}
\end{proposition}

\begin{proof}
	The proof of the first generating function follows along the same lines as Euler's solution to Naud\'{e}'s problem. The second follows by taking derivatives. From definitions it is clear that $L(n) = \sum_{m \geq 0} m p(m,n)$, so we have
	\begin{align*}
		\sum_{n \geq 0} L(n) q^n = \dfrac{\partial}{\partial z} \bigg|_{z=1} \prod_{n=1}^\infty \dfrac{1}{\lp 1 - zq^n \rp} &= \sum_{m \geq 1} \dfrac{q^m}{\lp 1 - q^m \rp^2} \ \prod_{n \not = m} \dfrac{1}{\lp 1 - q^n \rp} \\ &= \lp q; q \rp_\infty^{-1} \sum_{m \geq 1} \dfrac{q^m}{1 - q^m}.
	\end{align*}
	This completes the proof.
\end{proof}

We will see generating functions much like this one in Chapter \ref{C3}, with some modifications. The proofs there will be more in the classical spirit of Euler, with no derivatives present, but the same proofs can be done with derivatives.

Another natural question which we can study using two-variable generating function is the question of the distribution of $s(\lambda)$ among residue classes. More precisely, let us define
\begin{align*}
	p_s(a,b;n) = \#\{ \lambda \vdash n : s(\lambda) \equiv a \pmod{b} \}.
\end{align*}
It is clear that
\begin{align} \label{Mod B Distribution Equation}
	p_s(0,b;n) + p_s(1,b;n) + \cdots + p_s(b-1,b;n) = p(n),
\end{align}
and it is therefore natural to consider the distribution of values of $p_s(a,b;n)$ as $n \to \infty$ as $a$ varies. Two natural questions arise in this context. One, are there arithmetic progressions $\ell n + r$ for which
\begin{align} \label{Crank-Type Equation}
	p_s(0,b; \ell n + r) = p_s(1,b; \ell n + r) = \cdots = p_s(b-1,b; \ell n + r)
\end{align}
for all $n$? If so, then by \eqref{Mod B Distribution Equation} this induces a congruence
\begin{align*}
	p(\ell n + r) \equiv 0 \pmod{b}.
\end{align*}
There is a way to establish this kind of congruence using the two-variable generating function for $p_s(m,n)$. This idea is used for example in Garvan's famous paper on vector cranks \cite[Lemma 2.2]{Gar88}, which established simultaneously all three of Ramanujan's congruences for $p(n)$ and led to the Andrews-Garvan crank \cite{AG88}. These papers have given birth to an entire field of crank statistics that give combinatorial witness to partition-theoretic congruences for a variety of congruence functions. One such paper by Bringmann, Gomez, Rolen and Tripp \cite{BGRT22}, which explores cranks for colored partition functions via the theory of theta blocks, lays out this principle in a very general form. In particular, Lemma 2.1 of \cite{BGRT22} implies that equalities like 
\eqref{Crank-Type Equation} are equivalent to the divisibility of the Laurent polynomial $\sum_{m \in \ZZ} p_s(m,n) z^m$ by cyclotomic polynomials. Equivalently, this is equivalent to such polynomials vanishing at $z = \zeta_b := e^{\frac{2\pi i}{b}}$ for suitable choices of $n$. This principle stands at the heart of much of the modern developments on crank functions, and thus on partition congruences.

The principle that lies behind this application is, however, actually much more broad than this. The principle extends far beyond partition theory to other areas of number theory. In fact, the principle at play here is really an orthogonality relation. In the setting of roots of unity, the orthogonality relation says that if $\zeta_b := e^{\frac{2\pi i}{b}}$ for $b \geq 2$, then
\begin{align} \label{Orthogonality}
	\dfrac{1}{b} \sum_{k=0}^{b-1} \zeta_b^{km} =
	\begin{cases}
		1 & m \equiv 0 \pmod{b}, \\ 0 & \text{otherwise}.
	\end{cases}
\end{align}
One of the major takeaways from the proof of famous results like Dirichlet's theorem on primes in arithmetic progressions is that orthogonality relations for $b$th roots of unity (or Dirichlet characters modulo $b$) are the correct device for creating indicator functions for residue classes modulo $b$. In a partition-theoretic context, what this means is that the generating functions for $p_s(a,b;n)$ are obtainable via orthogonality from the two-variable generating function for $p_s(m,n)$.

\begin{proposition} \label{OrthogonalityProp}
	Let $S(z,q) := \sum\limits_{\lambda \in \mathcal P} z^{s(\lambda)} q^{|\lambda|}$ and let $\zeta_b := e^{\frac{2\pi i}{b}}$ for any integer $b \geq 2$. Then we have for $0 \leq a < b$ that
	\begin{align*}
		\sum_{n \geq 0} p_s(a,b;n) q^n = \dfrac{1}{b} \sum_{k=0}^{b-1} \zeta_b^{-ak} S(\zeta_b^k; q).
	\end{align*}
\end{proposition}

\begin{proof}
	From \eqref{Orthogonality}, it may be easily deduced that
	\begin{align*}
		\dfrac{1}{b} \sum_{k=0}^{b-1} \zeta_b^{(m-a)k} =
		\begin{cases}
			1 & m \equiv a \pmod{b}, \\ 0 & \text{otherwise}.
		\end{cases}
	\end{align*}
	On this basis, we have
	\begin{align*}
		\dfrac{1}{b} \sum_{k=0}^{b-1} \zeta_b^{-ak} S(\zeta_b^k; q) = \dfrac{1}{b} \sum_{\lambda \in \mathcal P} q^{|\lambda|} \cdot \sum_{k=0}^{b-1} \zeta_b^{\lp s(\lambda) - a \rp k} = \sum_{n \geq 0} p_s(a,b;n) q^n,
	\end{align*}
	which completes the proof.
\end{proof}

This proposition plays a central role in Chapters \ref{C5} and \ref{C6}.

\section{Modular Forms}

Modular forms are holomorphic (or sometimes meromorphic) functions that satisfy certain nice transformation laws when acted on by the modular group $\mathrm{SL}_2(\ZZ)$ or one of its subgroups. Modular forms are very important objects in many areas of mathematical study. Historically, modular transformation laws played central roles in Riemann's functional equation for the zeta function \cite{Rie59} and in the Hardy-Ramanujan-Rademacher circle method \cite{HR18,Rad37}. In more modern times, modular forms are crucial in the famous proof of Fermat's last theorem and play a central role in the work on sphere packings \cite{CKMRV17, Via17} which earned a Fields medal in 2022 for Maryna Viazovska. The subject of this section is to give suitable definitions for modular forms and to give some important examples and properties exhibited by spaces of modular forms.

\subsection{$\mathrm{SL}_2(\ZZ)$ and congruence subgroups}

The modular group $\mathrm{SL}_2(\ZZ)$ is the group of all two by two integer matrices whose determinant is 1. The modular group acts in a natural way on the upper half plane $\HH$, but this action is best studied at first from a geometric perspective. In particular, this action is inherited from the larger group $\mathrm{SL}_2(\RR)$ by M\"{o}bius transformations; that is, a matrix $\gamma = \lp \begin{smallmatrix} a & b \\ c & d \end{smallmatrix} \rp \in \mathrm{SL}_2(\RR)$ acts on points $z \in \HH$ by $\gamma z = \frac{az+b}{cz+d}$. This action is isometric with respect to the hyperbolic metric $ds^2 = y^{-2} \lp dx^2 + dy^2 \rp$.

The action of $\mathrm{SL}_2(\ZZ)$ on $\HH$ by M\"{o}bius transformations is discontinuous, i.e. the orbits have no limit points. Thus, we can define a quotient space $\mathcal F := \HH/\mathrm{SL}_2(\ZZ)$, which we call the {\it fundamental domain} of this action, which is a surface containing exactly one representative from each orbit. In fact, this surface is a Riemann surface that has genus zero when compactified, and the points on this surface parameterize complex elliptic curves via the modular $j$-function.

The first examples of modular forms will be functions which have nice transformation laws when acted on by elements of $\mathrm{SL}_2(\ZZ)$. However, in many contexts we are required to restrict ourselves to certain subgroups of $\mathrm{SL}_2(\ZZ)$ called {\it congruence subgroups}. A congruence subgroup of level $N$ is a subgroup $\Gamma \leq \mathrm{SL}_2(\ZZ)$ that contains the subgroup
\begin{align*}
	\Gamma(N) := \left\{ \begin{pmatrix} a & b \\ c & d \end{pmatrix} \in \mathrm{SL}_2(\ZZ) : \begin{pmatrix} a & b \\ c & d \end{pmatrix} \equiv \begin{pmatrix} 1 & 0 \\ 0 & 1 \end{pmatrix} \pmod{N} \right\}.
\end{align*}
The subgroups $\Gamma(N)$ are called the {\it principal congruence subgroups}. These are the kernels of the reduction maps $\mathrm{SL}_2(\ZZ) \to \mathrm{SL}_2(\ZZ/N\ZZ)$, and are thus normal subgroups of finite index in $\mathrm{SL}_2(\ZZ)$. The two examples with which we will be most concerned are
\begin{align*}
	\Gamma_0(N) := \left\{ \begin{pmatrix} a & b \\ c & d \end{pmatrix} \in \mathrm{SL}_2(\ZZ) : \begin{pmatrix} a & b \\ c & d \end{pmatrix} \equiv \begin{pmatrix} * & * \\ 0 & * \end{pmatrix} \pmod{N} \right\}
\end{align*}
and
\begin{align*}
	\Gamma_1(N) := \left\{ \begin{pmatrix} a & b \\ c & d \end{pmatrix} \in \mathrm{SL}_2(\ZZ) : \begin{pmatrix} a & b \\ c & d \end{pmatrix} \equiv \begin{pmatrix} 1 & * \\ 0 & 1 \end{pmatrix} \pmod{N} \right\},
\end{align*}
where in each case the $*$ signifies that any residue class is allowed. Since these each contain $\Gamma(N)$, they also have finite index in $\mathrm{SL}_2(\ZZ)$ and their indexes can be computed without too much difficulty. As with the modular group, each congruence subgroup $\Gamma$ acts on $\HH$ by M\"{o}bius transformations and induces a quotient space which is a Riemann surface. These surfaces are compactified by adding the ``cusps". One way of writing this is that the compactification of the Riemann surface $\HH/\Gamma$ is the space $\lp \HH \cup \QQ \cup \{ i \infty \} \rp/\Gamma$, where the action of $\gamma = \lp \begin{smallmatrix} a & b \\ c & d \end{smallmatrix} \rp$ on $i\infty$ is taken to be $\frac{a}{c}$. Because these groups have finite index, the number of equivalence classes of cusps (which by abuse of notation we refer to as cusps) is finite.

\subsection{Definition and examples}

We may now define a modular form. Throughout this section, we let $N \geq 1$ be an integer, $k$ any integer, and $\chi$ a Dirichlet character modulo $N$. The most basic case of modular forms, which we consider first, arise from the case $N = 1$ and $\chi = 1$. To define a modular form, we define an operator called the {weight $k$ slash action}, which is a group action of $\mathrm{SL}_2(\RR)$ on functions $f : \HH \to \CC$ and is defined for $\gamma = \lp \begin{smallmatrix} a & b \\ c & d \end{smallmatrix} \rp$ by
\begin{align*}
	\lp f|_k \gamma \rp (z) = \lp \det \gamma \rp^{\frac k2} \lp cz + d \rp^{-k} f(\gamma z).
\end{align*}

\begin{definition}
	A holomorphic function $f : \HH \to \CC$ is called a {\it modular form} of weight $k$ on $\mathrm{SL}_2(\ZZ)$ if 
	\begin{enumerate}
		\item $f$ is invariant under the action of $|_k$; that is, if $\lp f|_k \gamma \rp = f$ for all $\gamma \in \mathrm{SL}_2(\ZZ)$.
		\item $f$ has a Fourier expansion near infinity of the form
		\begin{align*}
			f(q) = \sum_{n=0}^\infty a_n q^n
		\end{align*}
		in the variable $q = e^{2\pi i z}$.
	\end{enumerate}
\end{definition}

This definition may be immediately generalized or refined in a variety of ways. One such family are the {\it weakly holomorphic modular forms}, which satisfy (1) and have finite principal parts in their Fourier expansions; i.e. $f$ may be meromorphic at the cusps of $\HH$. A similar refinement are the {\it cusp forms}, which satisfy (1) but have $a_0 = 0$. Many modular forms, particular in the theory of Borcherds products, may be meromorphic functions on $\HH$. There is a theory of half-integral weight modular forms as well, which include functions like $\eta(\tau)$, which requires an altered version of (1).

Another kind of generalization involves generalizing to congruence subgroups and allowing a twist by Dirichlet characters. Here, we say that a holomorphic function $f$ on $\HH$ is a {\it modular form on a congruence subgroup $\Gamma$ of weight $k$ and Nebentypus $\chi$} if it has a transformation law
\begin{align*}
	f\lp \gamma z \rp = \chi(d) \lp cz + d \rp^k f(z)
\end{align*}
for any $\gamma = \lp \begin{smallmatrix} a & b \\ c & d \end{smallmatrix} \rp$ and that have Fourier expansions near every cusp of $\Gamma$ analogous to those for modular forms on $\mathrm{SL}_2(\ZZ)$, but in a different uniformizing variable $q$ that takes on the value 0 at the cusp. We denote the space holomorphic modular forms on $\Gamma$ of weight $k$ and Nebentypus $\chi$ by $M_k(\Gamma, \chi)$. If there is no Nebentypus, we write $M_k(\Gamma)$, and if we wish instead to consider a space of cusp forms we write $S_k(\Gamma, \chi)$ or $S_k(\Gamma)$, respectively.

To give some understanding of the theory, it is useful to give examples. Some we have already seen in Chapter \ref{C1} include Dedekind's eta function $\eta(\tau)$, which is a modular form of weight $\frac 12$ and ``multiplier system" on $\mathrm{SL}_2(\ZZ)$ (multiplier system is a generalization of Nebentypus which we will not need). For the sake of brevity, we shall now focus on summarizing the theory of integer weight modular forms on $\mathrm{SL}_2(\ZZ)$.

\subsection{Modular forms of integer weight on $\mathrm{SL}_2(\ZZ)$}

This section gives an overview of the classification of modular forms on all of $\mathrm{SL}_2(\ZZ)$ with integer weight. The fundamental examples in this scenario are the {\it Eisenstein series of weight $k$}, defined for $k \geq 4$ by
\begin{align*}
	E_k(z) = \dfrac{1}{2\zeta(k)} \sum_{c,d \in \ZZ \backslash \{ (0,0) \}} \dfrac{1}{\lp cz + d \rp^k}.
\end{align*}
Note that this sum vanishes identically if $k$ is odd and does not converge absolutely unless $k > 2$; we therefore restrict ourselves to even integers $k \geq 4$. Because of the properties of the weight $k$ slash action, to prove that $E_k$ is modular of weight $k$ we need only prove this for a set of generators for $\mathrm{SL}_2(\ZZ)$, which is furnished by the matrices $T = \lp \begin{smallmatrix} 1 & 1 \\ 0 & 1 \end{smallmatrix} \rp$ and $S = \lp \begin{smallmatrix} 0 & -1 \\ 1 & 0 \end{smallmatrix} \rp$. To say $E_k$ is invariant under the action of $|_k T$ is to say that $E_k(z + 1) = E_k(z)$, i.e. that $E_k(z)$ is periodic, and invariance under $|_k S$ means that $E_k\lp - \frac 1z \rp = z^k E_k(z)$. The later is proven
\begin{align*}
	E_k\lp \dfrac{-1}{z} \rp = \sum_{c,d} \dfrac{1}{\lp - \frac{c}{z} + d \rp^k} = \sum_{c,d} \dfrac{z^k}{\lp dz - c \rp^k} = z^k E_k(z),
\end{align*}
since the change $cz+d \mapsto dz-c$ merely permuted the order of summands in an absolutely convergent series. Periodicity also follows by showing $z \mapsto z+1$ merely reorders summands.

The Fourier series is calculated as follows. By using the logarithmic derivative of the product expansion of $\sin(\pi z)$ and identities for sine and cosine in terms of complex exponentials, we have
\begin{align*}
	\pi i + \dfrac{2\pi i}{e^{2\pi i z} - 1} = \pi \cot\lp \pi z \rp = \dfrac{1}{z} + \sum_{n \geq 1} \lp \dfrac{1}{z + n} + \dfrac{1}{z - n} \rp.
\end{align*}
Now, if we interpret the ratio on the right hand side in terms of geometric series and differentiate $k$ times with respect to $z$ we obtain
\begin{align*}
	\sum_{n \in \ZZ} \dfrac{1}{\lp z + n \rp^k} = \dfrac{\lp 2\pi i \rp^k}{(k-1)!} \sum_{n=1} n^{k-1} e^{2\pi i n z}.
\end{align*}
Replacing $z$ with $mz$ and $q = e^{2\pi i z}$ and summing over $m$ while ignoring the pair $(m,n) = (0,0)$, 
\begin{align*}
	\sum_{\substack{n \in \ZZ \\ m \geq 1}} \dfrac{1}{\lp mz + n \rp^k} = \dfrac{\lp 2\pi i \rp^k}{(k-1)!} \sum_{m=1}^\infty \sum_{n \in \ZZ} n^{k-1} q^{mn} = \dfrac{\lp 2\pi i \rp^k}{(k-1)!} \sum_{n \geq 1} \sigma_{k-1}(n) q^n,
\end{align*}
where $\sigma_k(n) := \sum\limits_{d|n} d^k$ is the standard divisor sum function. When the sums for $n \in \ZZ$, $m \leq -1$ and $m = 0, n \not = 0$ are accounted for, we obtain
\begin{align*}
	2\zeta(k) E_k(z) = 2 \zeta(k) + 2 \dfrac{\lp 2\pi i \rp^k}{(k-1)!} \sum_{n \geq 1} \sigma_{k-1}(n) q^n = 2\zeta(k) \lp 1 + \dfrac{2k}{B_k} \sum_{n \geq 1} \sigma_{k-1}(n) q^n \rp,
\end{align*}
and thus
\begin{align*}
	E_k(z) = 1 + \dfrac{2k}{B_k} \sum_{n \geq 1} \sigma_{k-1}(n) q^n.
\end{align*}

After defining these examples of modular forms, one would like to compute the spaces $M_k := M_k\lp \mathrm{SL}_2(\ZZ) \rp$ and $S_k := S_k\lp \mathrm{SL}_2(\ZZ) \rp$. The key principle which aids in this process is the so-called {\it valence formula}, which is a zero-counting principle for modular forms. The idea that for a modular form $f(z)$ of weight $k$, one may study its zeros and poles via the argument principle. Let 
\begin{align*}
	\mathcal F = \{ z \in \HH : - \frac 12 \leq \mathrm{Re}(z) \leq \frac 12, |z| \geq 1 \} \cup \lp \{ i \infty \} \cup \QQ \rp / \sim 
\end{align*}
be the fundamental domain of $f(z)$ with the cusp at $i\infty$ adjoined. The quotient is calculated by identifying points which are in the same $\mathrm{SL}_2(\ZZ)$-orbit, which for example identifies the lines $\mathrm{Re}(z) = \pm \frac 12$ and all the cusps $\{ i\infty \} \cup \QQ$. In its natural topology this space is compact, and so by analysis $f(z)$ has only finitely many zeros and poles in this region. By leveraging the modular transformation laws in various ways and using standard contour integration techniques, one may show that any modular form $f(z)$ of weight $k$ which is holomorphic on $\mathcal F$ satisfies the relation
\begin{align*}
	v_{i\infty}(f) + \dfrac{v_i(f)}{2} + \dfrac{v_\rho(f)}{3} + \sum_{z \not \equiv i, \rho} v_x(f) = \dfrac{k}{12},
\end{align*}
where $\rho = e^{\frac{2\pi i}{3}}$, $v_x(f)$ denotes the order of vanishing of $f$ at $x$, and the sum takes places over point in $\mathcal F$ not equivalent to $i$ or $\rho$ modulo the action of $\mathrm{SL}_2(\ZZ)$. This formula can be leveraged to calculate the dimensions of spaces of modular forms. We see how this plays out in the following theorem, which completely classifies the spaces $M_k$ and $S_k$ for integral values of $k$. For this theorem, we need the Delta function $\Delta(z)$, which is defined by
\begin{align*}
	\Delta(z) = \dfrac{E_6^2 - E_4^3}{1728}.
\end{align*}

\begin{theorem}
	Let $k \in \ZZ$. Then the following are true:
	\begin{enumerate}
		\item[(1)] If $k < 0$ or $k$ is odd, then $M_k = S_k = \emptyset$.
		\item[(2)] $M_0 = \CC$, $S_0 = \{ 0 \}$ and $M_2 = S_2 = \{ 0 \}$.
		\item[(3)] If $k = 4, 6, 8, 10$, then $M_k = \CC E_k$.
		\item[(4)] If $k \geq 4$ is even, multiplication by $\Delta$ induces an isomorphism $M_{k-12} \to S_k$.
	\end{enumerate}
\end{theorem}

\begin{proof}
	Suppose $f \in M_k$ for some $k \in \ZZ$. Then by the valence formula, there must be a solution to the system
	\begin{align*}
		a + \dfrac{b}{2} + \dfrac{c}{3} + \sum_{z \not \equiv i, \rho} d_k = \dfrac{k}{12},
	\end{align*}
	where $a = v_{i \infty}(f), b = v_i(f), c = v_\rho(f)$, and $d_x = v_x(f)$. It is straightforward to derive (1) from the fact that $a,b,c,d_x \geq 0$ must be integers.
	
	We proceed with some observations about Eisenstein series. Since $E_4 \in M_4$, the only possible solution to this system is $a = b = d_x = 0$ and $c = 1$, and therefore $E_4(\rho) = 0$ is the only zero of $E_4$ modulo the group action. Similarly, $E_6(i) = 0$ is the only zero of $E_6$ modulo the group action. From definitions we also have $\Delta(i), \Delta(\rho) \not = 0$, and so $\Delta(i\infty) = 0$ follows from $\Delta \in S_{12}$.
	
	Now, if $f \in M_{k - 12}$, then it is clear that $f \Delta \in M_k$, and in fact by checking Fourier expansions that $f \Delta \in S_k$. Furthermore, if $f,g \in M_{k - 12}$ such that $f \Delta = g \Delta$, then noting the fact just proven that $\Delta \not = 0$ on $\mathcal F \backslash \{ i \infty \}$ readily implies $f = g$, so we obtain an injective map $M_{k - 12} \to S_k$. Similarly, given $f \in S_k$ we may show that $f/\Delta \in M_{k - 12}$ using the same non-vanishing assumption, so we have an isomorphism.
	
	It remains only to compute the dimensions of the spaces $M_4, M_6, M_8$, and $M_{10}$. It is clear from computing Fourier expansions that $M_k = S_k \oplus \CC E_k$, and the valence formula implies that $\Delta$ is the lowest weight cusp form, so $S_4 = S_6 = S_8 = S_{10} = \emptyset$, and so each of $M_4, M_6, M_8, M_{10}$ is one-dimensional, which completes the proof.
\end{proof}

We observe that from this theorem, one may quickly derive that the isomorphism of graded algebras
\begin{align*}
	\bigoplus_{k \in \ZZ} M_k \cong \CC[E_4, E_6].
\end{align*}
by comparing dimensions.

\subsection{Hecke operators, eigenforms and Atkin-Lehner newforms}

We have shown in the previous section that $M_k$ and $S_k$ are finite-dimensional complex vector spaces. More generally, the Riemann-Roch theorem has been used to compute the dimensions of the spaces $M_k(\Gamma_0(N), \chi)$ and $S_k(\Gamma_0(N), \chi)$ for integers $N \geq 1$ and Dirichlet characters $\chi$ modulo $N$ \cite{CO77}. After learning how to compute the dimension of spaces of modular forms, one of the next natural directions to seek out is a canonical basis of some kind that has nice properties. By fairly straightforward arguments, the spaces $M_k(\Gamma_0(N), \chi)$ decomposes as
\begin{align*}
	M_k(\Gamma_0(N), \chi) \cong S_k(\Gamma_0(N), \chi) \oplus \mathcal E_k(\Gamma_0(N), \chi),
\end{align*}
where $\mathcal E_k(\Gamma_0(N), \chi)$ is the so-called space of Eisenstein series. As it turns out, the full space of Eisenstein series is relatively easy to explicate in terms of the cusps of the Riemann surface $\HH / \Gamma_0(N)$, and so the problem of computing a basis there is not as interesting for our purposes. We are concerned now with the problem of computing an interesting basis for $S_k(\Gamma_0(N), \chi)$.

Before discussing generalities, we will discuss some of the history which motivated the discovery of the theory of newforms. The first nontrivial space of cusp forms with level one is the space $S_{12}(\Gamma_0(1))$, which is generated by
\begin{align*}
	\Delta(z) = q \prod_{n=1}^\infty \lp 1 - q^n \rp^{24} =: \sum_{n \geq 1} \tau(n) q^n,
\end{align*}
we call the coefficients $\tau(n)$ {\it Ramanujan's $\tau$-function}. In \cite{Ram16}, Ramanujan endeavors to understand the basic properties of this function. In particular, he conjectures that $\tau(n)$ is a multiplicative function of $n$, so that $\tau(mn) = \tau(m) \tau(n)$ for $m,n$ coprime, and that for any prime $p$ and $m \geq 1$ we have a recurrence relation
\begin{align*}
	\tau(p^{m+1}) = \tau(p) \tau(p^n) - p^{11} \tau(p^{n-1}).
\end{align*}
These conjectures were proven by Mordell \cite{Mor17}, but Hecke later demonstrated that Ramanujan's observation runs much deeper. What Hecke discovered is that $\Delta(z)$ is merely the first example of an entire theory of {\it eigenforms}. Hecke's major discovery was the family of {\it Hecke operators} $T(n)$. To define these, we fix for the rest of the section a space $M_k(\Gamma_0(N), \chi)$. The $n$th Hecke operator $T(n)$ on this space acts on $f(z) = \sum_{n \geq 1} a(n) q^n \in M_k(\Gamma_0(N), \chi)$ by
\begin{align*}
	T(n) f(z) = \sum_{n \geq 1} b(n) q^n, \ \ \ b(n) = \sum_{d | \gcd(n,N)} \chi(d) d^{k-1} a\lp \dfrac{mn}{d^2} \rp.
\end{align*}
These operators are constructed by summing over cosets of the action of determinant $n$ matrices on $\Gamma_0(N)$. Hecke showed that these operators have many nice properties. For instance, we have $T(mn) = T(m) T(n) = T(n) T(m)$ for all $m,n$ coprime, and for primes $p$ and $m \geq 1$ we have a recurrence relation
\begin{align*}
	T(p^{m+1}) = T(p) T(p^m) - p^{k-1} T(p^{m-1}).
\end{align*}
These operators are also Hermitian with respect to the {\it Petersson inner product}, which we will not need here. As it natural in linear algebra, we consider eigenvectors of the Hecke operators. In particular, say $f \in M_k(\Gamma_0(N), \chi)$ is an {\it eigenform} if it is an eigenvector of every Hecke operator. If we choose a normalized eigenform $f(z) = \sum_{n \geq 1} a_f(n) q^n \in S_k(\Gamma_0(N), \chi)$, then we can show using the formulas above that $a_f(n)$ is the eigenvalue of $f$ when hit by the operator $T(n)$. In this setting, because of the relations satisfied by the Hecke operators, any eigenform automatically has multiplicative coefficients and the values $a_f(p^m)$ satisfy recurrence relations.

Hecke's theory is able to demonstrate on its own that $M_k$ always has a basis of eigenforms. Issues arise in more general cases, because only those Hecke operators $T(n)$ with $\gcd(n,N) = 1$ behave nicely at first glance; the kinds of arguments used by Hecke are able to demonstrate only that $M_k(\Gamma_0(N), \chi)$ has a basis of functions that are eigenvalues of all Hecke operators $T(n)$ of this special type.

To resolve this deficiency, Atkin and Lehner developed the theory of {\it newforms} in \cite{AL70} (see also \cite{CS17}). They first develop a theory of Hecke operators for the spaces $S_k(\Gamma_0(N), \chi)$, and show that a basis can be found for all these spaces consisting of eigenforms of all the Hecke operators. Within this framework, they isolate within $S_k(\Gamma_0(N), \chi)$ into two spaces, one of which is generated by eigenforms that arise from spaces $S_k(\Gamma_0(N/d),\chi)$, which they call the space of {\it oldforms}, and the space of {\it newforms} which is orthogonal to it with respect to the Petersson inner product, and that the spaces generated by the newforms and oldforms give $S_k(\Gamma_0(N), \chi)$ are a direct sum. Finally, the space of oldforms may be generated by elements which are newforms with respect to other levels.

The theory of newforms is rich and, as newforms form a basis of all cusp forms, is a central tool for studying the vector spaces $S_k(\Gamma_0(N), \chi)$. We will only require the theory of newforms in Chapter \ref{C8}, we will defer the statement of relevant results until that time (see Theorem \ref{Newforms}).

\section{Partitions: Analytic Aspects} \label{S2.3}

\subsection{The circle method}

In the introduction, we have mentioned results of Hardy, Ramanujan, and Rademacher about the size of $p(n)$. The monumental breakthrough of Hardy and Ramanujan in \cite{HR18} was, to repeat \eqref{Hardy-Ramanujan Asymptotic}, that
\begin{align*}
	p(n) \sim \dfrac{1}{4n\sqrt{3}} e^{\pi \sqrt{\frac{2n}{3}}}
\end{align*}
as $n \to \infty$. We now wish to give a rough outline of the style of thought which leads to this result.

We now give a rough outline of the ideas of Hardy, Ramanujan, and Rademacher; for a more detailed account, see Apostol's excellent account in \cite{Apo90}.

The starting point of the argument is Euler's generating function
\begin{align*}
	P(q) := \sum_{n \geq 0} p(n) q^n = \prod_{n=1}^\infty \dfrac{1}{\lp 1 - q^n \rp}.
\end{align*}
We now view $P(q)$ as a complex analytic function in the variable $q$. The representation of $P(q)$ as an infinite product converges absolutely for $|q| < 1$. One can also see that at each root of unity $\zeta$, an infinite number of terms in this infinite product have a pole at $\zeta$; thus $P(q)$ has essential singularities at each root of unity. Standard complex analysis therefore shows that the region $|q| < 1$ is the largest on which $P(q)$ may be considered.

By Cauchy's theorem, we may represent $p(n)$ by the contour integral
\begin{align*}
	p(n) = \dfrac{1}{2\pi i} \int_C \dfrac{P(q)}{q^{n+1}} dq,
\end{align*}
where $C$ is any circle, oriented counterclockwise, centered at $q = 0$ and having radius $0 < r < 1$. This is a ``formula" for $p(n)$, but is not of any use until one has some kind of idea how to evaluate it.

The insight of Hardy and Ramanujan, very briefly summarized, is that the size of $P(q)$ on $C$ can be deduced from the essential singularities at each $q = \zeta_k^h := e^{\frac{2\pi i h}{k}}$ by means of modular transformation laws. More specifically, we break up the circle $C$ into arcs $C_{h,k}$ centered at $\zeta_k^h$, where $\frac hk$ runs through the set of rational numbers $0 \leq \frac hk < 1$ in reduced form with denominator bounded by some integer $N \geq 1$. We then rewrite the ``Cauchy formula" for $p(n)$ as a finite sum,
\begin{align*}
	p(n) = \sum_{\substack{0 \leq h < k \leq N \\ \gcd(h,k) = 1}} \dfrac{1}{2\pi i} \int_{C_{h,k}} \dfrac{P(q)}{q^{n+1}} dq.
\end{align*}

The idea at this stage is to understand the size of $P(q)$ as $q \to \zeta_k^h$. This is achieved via the modular transformation law of Dedekind's eta function \eqref{Dedekind eta transformation law}, from which it may be deduced that if $x = \zeta_k^h \exp\lp - \frac{2\pi z}{k^2} \rp$ and $x' = \zeta_k^H \exp\lp - \frac{2\pi}{z} \rp$, where $hH \equiv -1 \pmod{k}$, then we have the identity
\begin{align*}
	P(x) = e^{\pi i s(h,k)} \sqrt{\dfrac{z}{k}} \exp\lp \dfrac{\pi}{12z} - \dfrac{\pi z}{12 k^2} \rp P(x').
\end{align*}
Here $s(h,k)$ is defined as in \eqref{p(n) Kloosterman Sum}. Now, in order to obtain an asymptotic result, we allow the radius $r$ of the circle $C$ to vary with $n$; more specifically, as $n \to \infty$ we let $r \to 1$; this way the arcs $C_{h,k}$ are approaching the essential singularity at $q = \zeta_k^h$. It is fairly straightforward to see that as $r \to 1$, we must have $z \to 0$ and so $P(x') \to 1$ very rapidly. Thus, as $q \to \zeta_k^h$ we have
\begin{align*}
	P\lp \exp\lp \dfrac{2\pi i h}{k} - \dfrac{2\pi z}{k^2} \rp \rp \sim e^{\pi i s(h,k)} \sqrt{\dfrac{z}{k}} \exp\lp \dfrac{\pi}{12z} - \dfrac{\pi z}{12k^2} \rp =: P_{h,k}(z).
\end{align*}
The idea is now that as $n \to \infty$, $P_{h,k}(z)$ is a very good approximation of $P(q)$ as $q \to \zeta_k^h$, and so (by tracking details very carefully) we must have a formula like
\begin{align*}
	p(n) \approx \sum_{h,k} \dfrac{i}{k^2} e^{- \frac{2\pi i n h}{k}} \int_{z_1(h,k)}^{z_2(h,k)} P_{h,k}(z) e^{\frac{2\pi n z}{k^2}} dz,
\end{align*}
where $z_1(h,k)$ and $z_2(h,k)$ should be viewed as the endpoints of the arcs $C_{h,k}$ under suitable changes of variables.

Hardy and Ramanujan are able to obtain from the above considerations an asymptotic series for $p(n)$ by letting $n \to \infty$ in a suitable manner and keeping track of error terms. Rademacher is able to use a much stronger approximation of error terms to force the resulting asymptotic series to actually converge; thus obtaining \eqref{Rademacher Exact}. The theme which should be kept in mind, which lies at the heart of the circle method in any formulation, is that it is the growth rate of $P(q)$ nearby its singularities at roots of unity (as determined in this case by a modular transformation law) that allow the estimation of $p(n)$.

\subsection{Wright's variation}

Before we proceed, we should mention the idea behind a variation on this line of thinking, which is due to Wright \cite{Wri71}. This variation has the downside that it is incapable of producing exact formulas, but the upside that modular transformation laws are not required.

The heart of any version of the circle method is necessarily reliant upon asymptotic information of generating functions as $|q| \to 1$. In the circle method as executed by Hardy-Ramanujan or Wright, each root of unity $\zeta_k^h$ produces one term in an asymptotic series expansion for the coefficients of the generating function (in one case this asymptotic series diverges, in the other it converges). The idea behind Wright's method is that in very general settings, the singularity associated to one particular root of unity will exhibit a growth rate which rapidly outstrips all other roots of unity. For the generating function $P(q)$, and in fact for most generating functions in partition theory, this is the case of $q \to 1$. Intuitively, this is because $q = 1$ occurs as a pole more often in the product expansion of $P(q)$ than any other pole. In scenarios like these, Wright is able to deduce that if we prove which pole is the ``dominant pole" and if we are able to compute an asymptotic expansion of $P(q)$ nearby this dominant pole, then that would suffice to recover the Hardy-Ramanujan asymptotic formula for $p(n)$.

We state here one formulation of Wright's circle method, which will be restated and proved in Chapter \ref{C6}. The reader should take time to consider how the hypotheses of this result are really statements about a ``dominant pole" nearby $q = 1$.

\begin{theorem} \label{WrightCircleMethod}
	Suppose that $F(q)$ is analytic for $q = e^{-z}$ where $z=x+iy \in \CC$ satisfies $x > 0$ and $|y| < \pi$, and suppose that $F(q)$ has an expansion $F(q) = \sum_{n=0}^\infty c(n) q^n$ near 1. Let $c,N,M>0$ be fixed constants. Consider the following hypotheses:
	
	\begin{enumerate}[leftmargin=*]
		\item[\rm(1)] As $z\to 0$ in the bounded cone $|y|\le Mx$ (major arc), we have
		\begin{align*}
			F(e^{-z}) = z^{B} e^{\frac{A}{z}} \left( \sum_{j=0}^{N-1} \alpha_j z^j + O_\delta\left(|z|^N\right) \right),
		\end{align*}
		where $\alpha_s \in \CC$, $A\in \RR^+$, and $B \in \RR$. 
		
		\item[\rm(2)] As $z\to0$ in the bounded cone $Mx\le|y| < \pi$ (minor arc), we have 
		\begin{align*}
			\lvert	F(e^{-z}) \rvert \ll_\delta e^{\frac{1}{\mathrm{Re}(z)}(A - \kappa)}.
		\end{align*}
		for some $\kappa\in \RR^+$.
	\end{enumerate}
	If  {\rm(1)} and {\rm(2)} hold, then as $n \to \infty$ we have for any $N\in \RR^+$ 
	\begin{align*}
		c(n) = n^{\frac{1}{4}(- 2B -3)}e^{2\sqrt{An}} \lp \sum\limits_{r=0}^{N-1} p_r n^{-\frac{r}{2}} + O\left(n^{-\frac N2}\right) \rp,
	\end{align*}
	where $p_r := \sum\limits_{j=0}^r \alpha_j c_{j,r-j}$ and $c_{j,r} := \dfrac{(-\frac{1}{4\sqrt{A}})^r \sqrt{A}^{j + B + \frac 12}}{2\sqrt{\pi}} \dfrac{\Gamma(j + B + \frac 32 + r)}{r! \Gamma(j + B + \frac 32 - r)}$. 
\end{theorem}

For details of how this approach works, one may consult Wright's work \cite{Wri71} or a modern formulation in \cite{NR17}. One can also see the inner workings of the proof in Chapters \ref{C3}, \ref{C4}, or \ref{C6} where we implement various versions of this method.

\sglsp

\chapter{Biases for Parts of Partitions} \label{C3}
\thispagestyle{myheadings}

\dblsp
\vspace*{-.65cm}

\section{Bernoulli and Euler Polynomials}

In this section, we recall the famous {\it Bernoulli polynomials} $B_n(x)$ and {\it Euler polynomials} $E_n(x)$ and several of their properties we will need later. The generating functions for these polynomials are given in \cite[(24.2.3)]{DLMF} by
\begin{align} \label{Bernoulli Polynomial Definition}
	\sum_{n \geq 0} B_n(x) \dfrac{t^n}{n!} := \dfrac{t e^{xt}}{e^t - 1}
\end{align}
and
\begin{align*}
	\sum_{n \geq 0} E_n(x) \dfrac{t^n}{n!} := \dfrac{2 e^{xt}}{e^t + 1}.
\end{align*}
The {\it Bernoulli numbers} $B_n$ are defined by $B_n := B_n(0)$. We require a classical bound of Lehmer \cite{Leh40} regarding the size of Bernoulli polynomials on $0 \leq x \leq 1$ (and thus also a bound for Bernoulli numbers) which says for $n \geq 2$ that
\begin{align} \label{Bernoulli Inequality}
	\left| B_n(x) \right| \leq \dfrac{2 \zeta(n) n!}{(2\pi)^{n}},
\end{align}
where $\zeta(s) := \sum_{n \geq 1} n^{-s}$ is the {\it Riemann zeta function}. We recall the fact that $B_{2n+1} = 0$ for $n > 0$ (see \cite[(24.2.2)]{DLMF}). We also require the identity
\begin{align} \label{E_n(0) Equation}
	E_n(x) = \dfrac{2}{n+1} \left[ B_{n+1}(x) - 2^{n+1} B_{n+1}\lp \frac{x}{2} \rp \right],
\end{align}
which is \cite[(24.4.22)]{DLMF}.

\section{Generating functions}

This section is dedicated to defining the generating function for $D_{r,t}(n)$ and an important factorization of this generating function. Define
\begin{align*}
	\mathcal{D}_{r,t}(q) := \sum_{n \geq 0} D_{r,t}(n) q^n.
\end{align*}
We also use the standard {\it $q$-Pochhammer symbol} $(a;q)_\infty$, which is defined by $$(a;q)_\infty := \prod_{n \geq 1} \lp 1 - a q^{n-1} \rp$$ for $|q| < 1$. Recall that $(-q;q)_\infty$ is the generating function for the number of partitions of $n$ into distinct parts, as each term $(1 + q^m)$ appearing in the product dictates whether a given partition has a part of size $m$. By a slight modification of this argument, we obtain $\mathcal{D}_{r,t}(q)$.

\begin{lemma} \label{Generating Function}
	We have the generating function identity
	$$\mathcal{D}_{r,t}(q) = (-q;q)_\infty \sum_{k \geq 0} \dfrac{q^{kt + r}}{1 + q^{kt + r}}.$$
\end{lemma}

\begin{proof}
	By modifying Euler's generating function $\lp -q;q \rp_\infty$ for partitions into distinct parts, we see that $\frac{q^m}{1 + q^m} (-q;q)_\infty$ is the generating function for partitions into distinct parts which include $m$ as a part. Furthermore, since all parts are distinct, this is also the generating function for $D_{r,t}(n)$. Therefore, summing over $m$ equivalent to $r$ modulo $t$ yields
	$$\mathcal{D}_{r,t}(q) = \sum_{\substack{m \geq 0 \\ m \equiv r \pmod{t}}} \dfrac{q^m (-q;q)_\infty}{1 + q^m} = (-q;q)_\infty \sum_{k \geq 0} \dfrac{q^{kt + r}}{1 + q^{kt + r}}.$$
	This completes the proof.
\end{proof}

Next we require a brief lemma regarding a natural decomposition of this generating function, which will be useful for computing asymptotics. Define the functions $\xi(q) := (-q;q)_\infty$ and $L_{r,t}(q) := \sum_{k \geq 0} \frac{q^{kt+r}}{1 + q^{kt+r}}$, so that $\mathcal{D}_{r,t}(q) = \xi(q) L_{r,t}(q)$. Additionally, define $B(z) := \frac{e^{-z}}{z\lp 1 - e^{-z} \rp}$ and $E(z) := \frac{e^{-z}}{1 + e^{-z}}$. This notation is assumed throughout the remainder of the paper. The importance of the functions $B(z)$ and $E(z)$ comes from the following series expansions connecting them to $\mathcal{D}_{r,t}(q)$, which we record now for convenience. Throughout the remainder of the paper, we let $\Log(z)$ denote the principal branch of the logarithm.

\begin{lemma} \label{Xi and L z-Expansions}
	Let $\xi(q)$ and $L_{r,t}(q)$ be defined as above. Then, for $q = e^{-z}$ with $\mathrm{Re}(z) > 0$, we have
	$$\Log\lp \xi\lp e^{-z} \rp \rp = z \lp \sum_{m \geq 0} B\lp \lp m + \frac 12 \rp 2z \rp - \sum_{m \geq 0} B\lp \lp m+1 \rp 2z \rp \rp$$
	and
	$$L_{r,t}\lp e^{-z} \rp = \sum_{k \geq 0} E\lp \lp k + \frac{r}{t} \rp tz \rp.$$
\end{lemma}

\begin{proof}
	Expanding $\Log \lp \xi(q) \rp$ as a Taylor series, we have
	\begin{align*}
		\Log \lp \xi(q) \rp = \sum_{n \geq 1} \Log\lp 1 + q^n \rp = - z \sum_{m \geq 1} \dfrac{(-1)^m q^m}{mz \lp 1 - q^m \rp}.
	\end{align*}
	For $q = e^{-z}$, it follows from the definition of $B(z)$ that
	\begin{align*}
		\Log \lp \xi\lp e^{-z} \rp \rp = z \lp \sum_{m \geq 0} B\lp \lp m + \frac 12 \rp 2z \rp - \sum_{m \geq 0} B\lp \lp m+1 \rp 2z \rp \rp.
	\end{align*}
	This proves the first part of the lemma. The second is an analogous calculation with $E(z)$ in place of $B(z)$, i.e.
	\begin{align*}
		L_{r,t}\lp e^{-z} \rp = \sum_{k \geq 0} \dfrac{e^{-(kt + r)z}}{1 + e^{-(kt+r)z}} = \sum_{k \geq 0} E\lp \lp k + \frac{r}{t} \rp tz \rp.
	\end{align*}
	This completes the proof.
\end{proof}

We also record the Taylor expansions of $B(z)$ and $E(z)$ for later use. From the fact that $\frac{z}{e^z \pm 1} = \frac{z e^{-z}}{1 \pm e^{-z}}$ the generating function for the Bernoulli numbers $B_n$ is given by $B(z) = \frac{1}{z^2} - \frac{1}{2z} + \sum_{n \geq 0} \frac{B_{n+2}}{(n+2)!} z^n$, and similarly $E(z) = \sum_{n \geq 0} \frac{e_n}{n!} z^n$, where $e_n := \frac{E_n(0)}{2}$. We note for later that by \eqref{E_n(0) Equation}, we have
\begin{align} \label{e_n Evaluation}
	e_n = \dfrac{1 - 2^{n+1}}{n+1} B_{n+1}.
\end{align}

\section{Euler--Maclaurin summation}

This section recalls a not too widely known but very flexible method for computing asymptotic expansions of infinite sums coming from classical Euler--Maclaurin summation. This method has seen a large increase in usage over the last several years. This thesis alone uses the method in Chapters \ref{C3}, \ref{C4}, and \ref{C6} in various forms. Outside of this thesis, good references for its usage are \cite{BM17, BJM21a, BJM21b, JOa, JOb}. This formula is particularly useful for computing the asymptotic growth of products of $q$-Pochhammer symbols that don't have nice modular transformation laws, which to a significant extent explains its newfound prominence. Zagier \cite{Zag06} gives an excellent exposition of this method. Since this work, the method has been refined and generalized in a variety of ways. Because of the existence of many variations that fall under a unified theme, we provide here a unified treatment which covers many results under one umbrella. We begin this section by recalling the classical Euler--Maclaurin summation formula and we show how this formula is used to produce asymptotic formulas. We close the section with versions of these asymptotic formulas whose error terms are computed explicitly.

\subsection{Asymptotic Euler--Maclaurin summation} \label{C3 Euler-Maclaurin Section}

Recall the classical Euler--Maclaurin summation formula, which says that for integrable functions $f(z)$ on the interval $[a,b]$, we have for any integer $N \geq 1$ the formula
\begin{align*}
	\sum_{j = a+1}^b f(j) - \int_a^b f(x) dx &= \sum_{m=1}^N \dfrac{B_m}{m!} \lp f^{(m-1)}(b) - f^{(m-1)}(a) \rp \\ &+ (-1)^{N+1} \int_a^b f^{(N)}(x) \dfrac{\widehat{B}_N(x)}{N!} dx,
\end{align*}
where the {\it modified Bernoulli polynomials} $\widehat{B}_N(x)$ are given by $\widehat{B}_N(x) := B_N\lp x - \lfloor x \rfloor \rp$, where $\lfloor x \rfloor$ denotes the greatest integer less than or equal to $x$.

A natural extension of this question concerns infinite sums of the form $\sum_{n=1}^\infty f(nz)$ for complex-valued $z$. In \cite{Zag06}, Zagier gives a wonderful exposition of various methods by which one might naively try to extract such expansions from asymptotic expansions of $f(z)$. Here and throughout this thesis we use asymptotic expansion in its strong sense; that is, we say $f(z) \sim \sum_{n=0}^\infty c_n z^n$ if $f(z) - \sum_{n=0}^{N-1} c_n z^n = O(z^N)$ for all $N \geq 1$. Note that not all asymptotic expansions converge. Zagier shows in \cite[Proposition 3]{Zag06} how to correcty derive such expansions. As we wish to provide some additional details that will help with some later proofs, we defer for now the statement of the formulas\footnote{Proposition \ref{C3 Euler-Maclaurin Rapid Decay} is in fact a modest generalization of what Zagier formally states in \cite{Zag06}.}.

We now fix notation which will be used freely for the remainder of the thesis. For $\delta > 0$, we define $D_\delta := \{ z \in \CC : \left| \mathrm{arg}(z) \right| < \frac{\pi}{2} - \delta \}$. Note that if we set $z = \eta + iy$ for $\eta > 0$, then $z \in D_\delta$ if and only if $0 < |y| < M\eta$ for some constant $M > 0$ which depends on $\delta$. The {\it modified Bernoulli polynomial} $\widehat{B}_N(x)$ is the periodic function defined by $\widehat{B}_N(x) := B_N\lp x - \lfloor x \rfloor \rp$, where $\lfloor x \rfloor$ is the greatest integer less than or equal to $x$. We also use the {\it Hurwitz zeta function} $\zeta(s,x) := \sum_{n \geq 0} \frac{1}{(n + x)^{s}}$ and the {\it Euler--Mascheroni constant} $\gamma$. We furthermore set
\begin{align*}
	I_f := \int_0^\infty f(x) dx
\end{align*}
for any function $f$ for which this integral converges. The asymptotic formulas we derive require two types of decay conditions of $f(x)$ at infinity, which we call {\it sufficient decay} and {\it rapid decay}. The first holds if $f(x) = O\lp x^{-N} \rp$ as $x \to \infty$ for some $N > 1$, and the later holds if this true for every $N > 1$. We may now state as a consequence of the classical Euler--Maclaurin formula the following lemma, which is a slightly rewritten form of identities appearing in \cite[Proposition 2.1]{BM17}, which itself is based on the aforementioned work of Zagier \cite{Zag06}.

\begin{lemma} \label{C3 Euler-Maclaurin Exact}
	Suppose that $f(z)$ is $C^\infty$ for $z$ in $D_\delta$ for some $\delta > 0$ such that $f(z)$ and all its derivatives have sufficient decay as $z \to \infty$ in $D_\delta$. Then for any real number $0 < a \leq 1$ and any positive integer $N$, we have
	\begin{align*}
		\sum_{m \geq 0} f\lp (m+a)z \rp = \dfrac{1}{z} \int_{az}^\infty f(x) dx &+ \sum_{n=0}^{N-1} \dfrac{(-1)^n B_{n+1}}{(n+1)!} f^{(n)}(az) z^n \\ &- (-z)^N \int_0^\infty f^{(N)}\lp (x+a) z \rp \dfrac{\widehat{B}_N(x)}{N!} dx,
	\end{align*}
	where $f^{(N)}\lp (x+a)z \rp$ is taken to be a derivative with respect to $x$.
\end{lemma}

\begin{proof}
	The proof of \cite[Proposition 2.1]{BM17} implies with a slight change of variable in the last term that
	\begin{align*}
		\sum_{m \geq 0} f\lp (m+a)z \rp = \dfrac{1}{z} \int_{az}^\infty f(x) dx &+ \sum_{n=0}^{N-1} \dfrac{(-1)^n B_{n+1}}{(n+1)!} f^{(n)}(az) z^n \\ &- (-1)^N \int_0^\infty \dfrac{d^N}{dx^N} \left[ f\lp (x+a)z \rp \right] \dfrac{\widehat{B}_N(x)}{N!} dx.
	\end{align*}
	This is equivalent to the stated formula, as evaluating the inner derivatives brings into view the factor $z^N$ in the last term.
\end{proof}

We now state the asymptotic formula of Bringmann, Jennings-Shaffer and Mahlburg, which is a generalization of \cite[Proposition 2.1]{BM17} and \cite[Proposition 3]{Zag06}.

\begin{proposition}[{\cite[Theorem 1.2]{BJM21a}}] \label{C3 Euler-Maclaurin Rapid Decay}
	Suppose $0 \leq \delta < \frac{\pi}{2}$ and that $f : \CC \to \CC$ is holomorphic on a domain containing $D_\delta$, in particular containing the origin. Assume that $f(z)$ and all its derivatives have sufficient decay as $z \to \infty$ in $D_\delta$. Then for $a \in \RR$ and $N \geq 1$ an integer, we have
	$$\sum_{m \geq 0} f\lp (m + a)z \rp \sim \dfrac{I_f}{z} - \sum_{n \geq 0} c_n \dfrac{B_{n+1}(a)}{n+1} z^n$$
	uniformly as $z \to 0$ in $D_\delta$.
\end{proposition}

The following proposition is a refinement of Proposition \ref{C3 Euler-Maclaurin Rapid Decay} where the function $f(z)$ is allowed to have a pole at the origin. In other words, this extends the conclusion of Proposition \ref{C3 Euler-Maclaurin Rapid Decay} to functions $f(z)$ with principal parts $P_f(z)$ with the added property that $f(z) - P_f(z)$ has sufficient decay at infinity.

\begin{proposition}[{\cite[Lemma 2.2]{BCMO22}}] \label{C3 Euler-Maclaurin Sufficient Decay}
	Let $0 < a \leq 1$ and $A \in \RR^+$, and assume that $f(z) \sim \sum_{n=n_0}^{\infty} c_n z^n$ $(n_0\in\ZZ)$ as $z \rightarrow 0$ in $D_\delta$. Furthermore, assume that $f$ and all of its derivatives are of sufficient decay in $D_\delta$. Then we have that
	\begin{align*}
		\sum_{n=0}^\infty f((n+a)z)\sim \sum_{n=n_0}^{-2} c_{n} \zeta(-n,a)z^{n}+ \frac{I_{f,A}^*}{z} &- \frac{c_{-1}}{z} \left( \Log \left(Az \right) +\psi(a)+\gamma \right) \\ &-\sum_{n=0}^\infty c_n \frac{B_{n+1}(a)}{n+1} z^n,
	\end{align*}
	as $z \rightarrow 0$ uniformly in $D_\delta$, where 
	\begin{align*}
		I_{f,A}^*:=\int_{0}^{\infty} \left(f(u)-\sum_{n=n_0}^{-2}c_{n}u^n-\frac{c_{-1}e^{-Au}}{u}\right)du.
	\end{align*}
\end{proposition}

\begin{remark}
	The proof of this lemma comes from \cite{BCMO22} and is joint work with Bringmann, Males, and Ono. The results following this proposition are due to the author and are taken from \cite{Cra22}.
\end{remark}

\begin{proof}
	Let $h$ be any holomorphic function on a domain containing $D_\theta$, so that in particular $h$ is holomorphic at the origin, such that $h$ and all of its derivatives have sufficient decay, and $h(z) \sim \sum_{n=0}^{\infty} b_n z^n$ as $z \rightarrow 0$ in $D_\theta$. Then we have for $a\in\RR$
	\begin{align}\label{Eqn: EM holomorphic}
		\sum_{n=0}^\infty h((n+a)z)\sim\frac{I_h}{z}-\sum_{n=0}^\infty b_n \frac{ B_{n+1}(a)}{n+1}z^n,
	\end{align}
	as $z \rightarrow 0$ in $D_\theta$. For the given $A$, write
	\begin{align}\label{Eqn: f as g}
		f(z) = g(z) + \frac{c_{-1}e^{-Az}}{z} + \sum_{n=n_0}^{-2} c_nz^n,
	\end{align}
	which means that
	\begin{align*}
		g(z) = f(z) - \frac{c_{-1}e^{-Az}}{z} - \sum_{n=n_0}^{-2} c_nz^n.
	\end{align*}
	The final term in \eqref{Eqn: f as g} yields the first term in the right-hand side of the lemma. Since $g$ has no pole, \eqref{Eqn: EM holomorphic} gives that
	\begin{align*}
		\sum_{n=0}^\infty g((n+a)z)\sim\frac{I_g}{z}- \sum_{n=0}^\infty c_n(g) \frac{ B_{n+1}(a)}{n+1}z^n,
	\end{align*}
	where $c_n(g)$ are the coefficients of $g$. Note that $I_g = I_{f,A}^*$. We compute that
	\begin{align*}
		- \sum_{n=0}^\infty c_n(g) \frac{ B_{n+1}(a)}{n+1}z^n = -  \sum_{n=0}^\infty \left(c_n - \frac{(-A)^{n+1} c_{-1}}{(n+1)!}\right) \frac{ B_{n+1}(a)}{n+1}z^n.
	\end{align*}
	Combining the contribution from the second term with the contribution from the second term from \eqref{Eqn: f as g}, we obtain
	\begin{align*}
		\frac{c_{-1}}{z} \left( \sum_{n =0}^{\infty} \frac{e^{-A(n+a)z}}{n+a} + \sum_{n=1}^{\infty} \frac{B_n(a)}{n \cdot n!} (-Az)^n \right).
	\end{align*}
	Using \cite[equation (5.10)]{BJM21a}, the term in the parenthesis equals $-(\Log(Az)+\psi(a)+\gamma)$. Combining the contributions yields the statement of the proposition.
\end{proof}

The results of this section are sufficient for Chapter \ref{C6}, whereas Chapters \ref{C3} and \ref{C4} require explicit versions of these results.

\subsection{Effective Euler--Maclaurin summation}

This section is dedicated to reproving the results of the previous section with explicitly computable error terms. This is achieved by simply keeping track of the higher degree terms that were dropped in the proof of Propositions \ref{C3 Euler-Maclaurin Rapid Decay} and \ref{C3 Euler-Maclaurin Sufficient Decay}. These two propositions essentially follow from ``erasing" higher-order terms in Lemma \ref{C3 Euler-Maclaurin Exact}. Therefore, making the error terms in these results effective is essentially a matter of bookkeeping. These effective error terms become the central tool for implementing an effective version of Wright's circle method, which is central to Chapters \ref{C3} and \ref{C4}.

\begin{proposition} \label{C3 Euler-Maclaurin Rapid Decay Effective}
	Let $f(z)$ be $C^\infty$ in $D_\delta$ with power series expansion $f(z) = \sum_{n \geq 0} c_n z^n$ that converges absolutely in the region $0 \leq |z| < R$ for some positive constant $R$, and let $f(z)$ and all its derivatives have sufficient decay as $z \to \infty$ in $D_\delta$. Then for any real number $0 < a \leq 1$ and any integer $N > 0$,
	\begin{align*}
		\bigg| \sum_{m \geq 0} f\lp (m+a)z \rp &- \dfrac{I_f}{z} + \sum_{n = 0}^{N-1} c_n \dfrac{B_{n+1}(a)}{n+1} z^n \bigg| \\ &\leq \dfrac{M_{N+1} J_{f,N+1}(z)}{(N+1)!} |z|^N + \sum_{k \geq N} |c_k| \lp 1 + \dfrac{k!}{10 (k-N)!} \rp |z|^k,
	\end{align*}
	where $M_{N+1} := \max\limits_{0 \leq x \leq 1} \left| B_{N+1}(x) \right|$ and
	\begin{align*}
		J_{f,N+1}(z) := \int_0^\infty \left| f^{(N+1)}\lp w \rp \right| |dw|,
	\end{align*}
	where the path of integration proceeds along the line through the origin and $z$.
\end{proposition}

\begin{proof}
	From Proposition \ref{C3 Euler-Maclaurin Rapid Decay}, we already know that
	\begin{align*}
		S_N(z) := \sum_{m \geq 0} f\lp (m+a)z \rp - \dfrac{I_f}{z} + \sum_{n=0}^{N-1} c_n \dfrac{B_{n+1}(a)}{n+1} z^n = O_N\lp z^N \rp.
	\end{align*}
	It suffices to make this upper bound effective. We use the shorthand $$J_{N+1,a}(z) := \int_{az}^\infty f^{(N+1)}\lp w \rp \dfrac{\widehat{B}_{N+1}\lp \frac wz - a \rp}{(N+1)!} dw,$$
	which is the integral from last term of Lemma \ref{C3 Euler-Maclaurin Exact} with a substitution $w = \lp x+a \rp z$. By Lemma \ref{C3 Euler-Maclaurin Exact}, we may write
	\begin{align*}
		S_N(z) = \dfrac{-1}{z} \int_0^{az} f(x) dx + \sum_{n=0}^{N} \dfrac{(-1)^n B_{n+1}}{(n+1)!} f^{(n)}(az) z^n &+ \sum_{n=0}^{N} c_n \dfrac{B_{n+1}(a)}{n+1} z^n \\ &- (-z)^N J_{N+1,a}(z).
	\end{align*}
	Because $0 < a \leq 1$ and $0 < |z| < R$, we have $|az| < R$ and so we may expand $f(x)$ and its derivatives as power series for $0 \leq x \leq |az|$. Using these power series representations and the absolute convergence of $\int_0^{az} f(x) dx$, we have
	\begin{align*}
		S_N(z) = - \sum_{k \geq 0} \dfrac{c_k}{k+1} a^{k+1} z^k &+ \sum_{n=0}^{N} \dfrac{(-1)^n B_{n+1}}{(n+1)!} \sum_{k \geq 0} \dfrac{(k+n)!}{k!} c_{n+k} a^k z^{n+k} \\ &+ \sum_{n=0}^{N} c_n \dfrac{B_{n+1}(a)}{n+1} z^n - (-z)^N J_{N+1,a}(z).
	\end{align*}
	It is already known, for instance by Proposition \ref{C3 Euler-Maclaurin Rapid Decay}, that $S_N(z) = O(z^N)$, so the lower-order terms in the above identity necessarily cancel. Thus, we have
	\begin{align*}
		S_N(z) = - \sum_{k \geq N} \dfrac{c_k}{k+1} a^{k+1} z^k + \sum_{n=0}^{N} \dfrac{(-1)^n B_{n+1}}{(n+1)!} & \sum_{k \geq N-n} c_{n+k} \dfrac{(n+k)!}{k!} a^k z^{n+k} \\ &- (-z)^N J_{N+1,a}(z).
	\end{align*}
	By taking $k \mapsto k-n$ in the second term and rearranging, we obtain
	\begin{align*}
		S_N(z) &= \sum_{k \geq N} c_k \left[- \dfrac{a^{k+1}}{k+1}  + \sum_{n=0}^N \dfrac{1}{n+1}\left[ (-1)^n B_{n+1} \binom{k}{n} a^{k-n} \right] \right] z^k - \lp -z \rp^N J_{N+1,a}(z).
	\end{align*}
	
	We now bound the remaining terms. The integral $J_{N+1,a}(z)$ is bounded trivially by
	\begin{align*}
		\left| J_{N+1,a}(z) \right| \leq \dfrac{M_{N+1}}{(N+1)!} J_{f,N+1}(z) = O_N(1)
	\end{align*}
	since $f(z)$ is bounded near zero and has sufficient decay as $z \to \infty$ in $D_\delta$.
	
	We also have, using Lehmer's bound \eqref{Bernoulli Inequality} and elementary estimates that for $k \geq N$,
	\begin{align*}
		\left| - \dfrac{a^{k+1}}{k+1} + \sum_{n=0}^N \dfrac{1}{n+1} \left[ (-1)^n B_{n+1} \binom{k}{n} a^{k-n} \right] \right| &\leq \dfrac{a^{k+1}}{k+1} + \dfrac{a^k}{2} + a^k \sum_{\substack{n=1 \\ n \text{ odd}}}^N \dfrac{2 \zeta(n+1) n!}{a^n (2\pi)^{n+1}} \binom{k}{n} \\ &< \dfrac{1}{k+1} + \dfrac{1}{2} + \dfrac{\pi}{6} \sum_{\substack{n=1 \\ n \text{ odd}}}^N \dfrac{k!}{(2 \pi)^n (k-n)!}.
	\end{align*}
	Since $1 \leq n \leq N \leq k$, $\frac{k!}{(k-n)!} < \frac{k!}{(k-N)!}$, and $\frac{\pi}{6} \sum_{n \geq 0} \frac{1}{(2\pi)^{2n+1}} < \frac{1}{10}$,
	\begin{align*}
		\left| - \dfrac{a^{k+1}}{k+1} + \sum_{n=0}^N \dfrac{1}{n+1} \left[ (-1)^n B_{n+1} \binom{k}{n} a^{k-n} \right] \right| < 1 + \dfrac{k!}{10 (k-N)!}.
	\end{align*}
	Thus,
	\begin{align*}
		\bigg| \sum_{k \geq N} c_k \bigg[ - \dfrac{a^{k+1}}{k+1} + &\sum_{n=0}^N \dfrac{1}{n+1} \left[ (-1)^n B_{n+1} \binom{k}{n} a^{k-n} \right] \bigg] z^k \bigg| \\ &\leq \sum_{k \geq N} |c_k| \lp 1 + \dfrac{k!}{10 (k-N)!} \rp |z|^k.
	\end{align*}
	Combining all bounds completes the proof.
\end{proof}

The proposition above shows how Euler--Maclaurin summation can be used to derive effective asymptotics for certain infinite series involving a function $f(z)$ with rapid decay at infinity. In analogy with Proposition \ref{C3 Euler-Maclaurin Sufficient Decay}, we now show how to derive explicit bounds for the case of sufficient decay at infinity.

\begin{proposition} \label{C3 Euler-Maclaurin Sufficient Effective}
	Let $f(z)$ be $C^\infty$ in $D_\delta$ with Laurent series $f(z) = \sum_{n = n_0}^\infty c_n z^n$ that converges absolutely in the region $0 < |z| < R$ for some positive constant $R$. Suppose $f(z)$ and all its derivatives have sufficient decay as $z \to \infty$ in $D_\delta$. Then for any real numbers $0 < a \leq 1$, $A > 0$ and any integer $N > 0$, we have
	\begin{align*}
		\bigg| \sum_{m \geq 0} &f\lp (m + a)z \rp - \sum_{n = n_0}^{-2} c_n \zeta(-n,a) z^n - \dfrac{I_{f,A}^*}{z} + \dfrac{c_{-1}}{z} \lp \Log\lp Az \rp + \gamma + \psi\lp a \rp \rp \\ &+ \sum_{n \geq 0} c_n^* \dfrac{B_{n+1}(a)}{n+1} z^n \bigg| \leq \dfrac{M_{N+1} J_{g,N+1}(az)}{(N+1)!} |z|^N + \sum_{k \geq N} |b_k| \lp 1 + \dfrac{k!}{10 (k-N)!} \rp |z|^k,
	\end{align*}
	where $g(z) := f(z) - \frac{c_{-1} e^{-Az}}{z} - \sum_{n = n_0}^{-2} c_n z^n$, $b_n := c_n - \frac{(-A)^{n+1} c_{-1}}{(n+1)!}$, $M_N$ and $J_{g,N}$ are defined as in Proposition \ref{C3 Euler-Maclaurin Rapid Decay Effective}, and
	\begin{align*}
		c_n^* := \begin{cases} c_n & \text{if } n \leq N-1, \\ \dfrac{(-A)^{n+1} c_{-1}}{(n+1)!} & \text{if } n \geq N. \end{cases}
	\end{align*}
\end{proposition}

\begin{proof}
	Since
	$$f(z) = g(z) + \dfrac{c_{-1} e^{-Az}}{z} + \sum_{n = n_0}^{-2} c_n z^n,$$
	then $g(z)$ is holomorphic at $z=0$ and has sufficient decay at infinity. Because $f(z)$ has a Laurent series converging for $0 < |z| < R$, it follows that $g(z)$ has a Taylor series $g(z) = \sum_{n = 0}^\infty b_n z^n$ which converges for $|z| < R$. Also note that $I_g = I^*_{f,A}$ by definition. Therefore, Proposition \ref{C3 Euler-Maclaurin Rapid Decay Effective} implies for $N > 0$ that
	\begin{align*}
		\bigg| \sum_{m \geq 0} g\lp (m+a)z \rp - \dfrac{I_{f,A}^*}{z} &+ \sum_{n = 0}^{N-1} b_n \dfrac{B_{n+1}(a)}{n+1} z^n \bigg| \\ &\leq \dfrac{M_{N+1} J_{g,N+1}(z)}{(N+1)!} |z|^N + \sum_{k \geq N} |b_k| \lp 1 + \dfrac{k!}{10 (k-N)!} \rp |z|^k
	\end{align*}
	for $z \in D_\delta$ with $0 < |z| < R$. From the definition of $g(z)$ this becomes
	\begin{align*}
		\Bigg| \sum_{m \geq 0} \bigg[ f\lp (m+a)z \rp &- \dfrac{c_{-1} e^{-A(m+a)z}}{(m+a)z} \bigg] - \sum_{n = n_0}^{-2} c_n \zeta(-n,a) z^n - \dfrac{I_{f,A}^*}{z} + \sum_{n = 0}^{N-1} b_n \dfrac{B_{n+1}(a)}{n+1} z^n \Bigg| \\ &\leq\dfrac{M_{N+1} J_{g,N+1}(z)}{(N+1)!} |z|^N + \sum_{k \geq N} |b_k| \lp 1 + \dfrac{k!}{10 (k-N)!} \rp |z|^k.
	\end{align*}
	By the definition of $b_n$ we have
	\begin{align*}
		\sum_{n = 0}^{N-1} b_n \dfrac{B_{n+1}(a)}{n+1} z^n = \sum_{n = 0}^{N-1} c_n \dfrac{B_{n+1}(a)}{n+1} z^n - \sum_{n = 0}^{N-1} \dfrac{(-A)^{n+1} c_{-1}}{(n+1)!} \dfrac{B_{n+1}(a)}{n+1} z^n,
	\end{align*}
	and if we adopt the notation
	\begin{align*}
		\dfrac{c_{-1}}{z} H_{a,N}(z) := \dfrac{c_{-1}}{z} \lp \sum_{m \geq 0} \dfrac{e^{-A(m+a)z}}{m+a} + \sum_{n=0}^{N-1} \dfrac{B_{n+1}(a)}{(n+1) (n+1)!} (-Az)^{n+1} \rp,
	\end{align*}
	it follows that
	\begin{align*}
		\bigg| \sum_{m \geq 0} f\lp (m+a)z \rp &- \sum_{n = n_0}^{-2} c_n \zeta(-n,a) z^n - \dfrac{c_{-1}}{z} H_{a,N}(Az) - \dfrac{I_{f,A}^*}{z} + \sum_{n = 0}^{N-1} c_n \dfrac{B_{n+1}(a)}{n+1} z^n \bigg| \\ &\leq \dfrac{M_{N+1} J_{g,N+1}(z)}{(N+1)!} |z|^N + \sum_{k \geq N} |b_k| \lp 1 + \dfrac{k!}{10 (k-N)!} \rp |z|^k.
	\end{align*}
	By \cite[Equation 5.10]{BJM21a}, it is known that
	\begin{align*}
		H_a(z) := \sum_{n \geq 0} \dfrac{e^{-(m+a)z}}{m+a} + \sum_{n \geq 0} \dfrac{B_{n+1}(a)}{(n+1) (n+1)!} (-z)^{n+1}
	\end{align*}
	satisfies $H_a(Az) = - \Log(Az) - \gamma - \psi(a)$ for any $A > 0$. Since $$H_{a,N}(Az) = H_a(Az) - \sum_{n \geq N} \frac{B_{n+1}(a)}{(n+1) (n+1)!} (-Az)^{n+1},$$
	this completes the proof.
\end{proof}

\section{Statement of Wright's Circle Method}

In this section, we recall a result of Bringmann, Ono, Males, and the author from \cite{BCMO22}, which is a variation of the circle method going back to Wright \cite{Wri71}. Wright's circle method gives asymptotics for the coefficients of $q$-series $F(q)$ having a nice factorization and suitable analytic properties. Given a circle $\mathcal{C}$ centered at the origin with radius less than 1, we define its {\it major arc} as that region of $\mathcal C$ where $F(q)$ is largest. In our applications, this is given by $\mathcal{C}_1 := \mathcal{C} \cap D_\delta$ for $\delta > 0$. The {\it minor arc} of $\mathcal{C}$ is then defined by $\mathcal{C}_2 := \mathcal{C} \backslash \mathcal{C}_1$. In the circle method, the integral taken over $\mathcal{C}_1$ gives the main term for the coefficients of $F(q)$ and the integral over $\mathcal{C}_2$ is merely an error term.

Here, we recall the version of Wright's circle method which we will use in the proof of Theorem \ref{C3 Ineffective Asymptotic}.

\begin{proposition}[{\cite[Proposition 4.4]{BCMO22}}] \label{Wright Circle Method}
	Suppose that $F(q)$ is analytic for $q = e^{-z}$ where $z=x+iy \in \CC$ satisfies $x > 0$ and $|y| < \pi$, and suppose that $F(q)$ has an expansion $F(q) = \sum_{n=0}^\infty c(n) q^n$ near 1. Let $N,M>0$ be fixed constants. Consider the following hypotheses:
	
	\begin{enumerate}
		\item[\rm(1)] As $z\to 0$ in the bounded cone $|y|\le Mx$ (major arc), we have
		\begin{align*}
			F(e^{-z}) = C z^{B} e^{\frac{A}{z}} \left( \sum_{j=0}^{N-1} \alpha_j z^j + O_\delta\left(|z|^N\right) \right),
		\end{align*}
		where $\alpha_s \in \CC$, $A,C \in \RR^+$, and $B \in \RR$. 
		
		\item[\rm(2)] As $z\to0$ in the bounded cone $Mx\le|y| < \pi$ (minor arc), we have 
		\begin{align*}
			\lvert	F(e^{-z}) \rvert \ll_\delta e^{\frac{1}{\mathrm{Re}(z)}(A - \kappa)},
		\end{align*}
		for some $\kappa\in \RR^+$.
	\end{enumerate}
	If  {\rm(1)} and {\rm(2)} hold, then as $n \to \infty$ we have for any $N\in \RR^+$ 
	\begin{align*}
		c(n) = C n^{\frac{1}{4}(- 2B -3)}e^{2\sqrt{An}} \lp \sum\limits_{r=0}^{N-1} p_r n^{-\frac{r}{2}} + O\left(n^{-\frac N2}\right) \rp,
	\end{align*}
	where $p_r := \sum\limits_{j=0}^r \alpha_j c_{j,r-j}$ and $c_{j,r} := \dfrac{(-\frac{1}{4\sqrt{A}})^r \sqrt{A}^{j + B + \frac 12}}{2\sqrt{\pi}} \dfrac{\Gamma(j + B + \frac 32 + r)}{r! \Gamma(j + B + \frac 32 - r)}$. 
\end{proposition}

\begin{remark}
	The constant $C$ in Proposition \ref{Wright Circle Method} does not appear in the statement of Wright's circle method proved in Proposition \ref{WrightCircleMethod} proved in Chapter \ref{C6}, but is trivially equivalent to this result by factoring out $C$ from each $\alpha_i$.
\end{remark}

We defer the proof of this proposition until Chapter \ref{C6}. However, the line of attack is exemplified by the proof of Theorem \ref{C3 Effective Asymptotic}.

\section{Estimates with Bessel functions}

We now consider certain estimates with Bessel functions which we will require when effectively implementing Wright's circle method. Recall that the {\it modified Bessel function} $I_\nu(z)$ is defined for any $\nu \in \CC$ by
\begin{align*}
	I_\nu(x) := \lp \dfrac x2 \rp^\nu  \dfrac{1}{2\pi i} \int_{\mathcal D} t^{-\nu-1} \exp\lp \dfrac{x^2}{4t} + t \rp dt,
\end{align*}
where $\mathcal D$ is any contour running from $-\infty$ below the negative real axis, counterclockwise around 0, and back to $-\infty$ above the negative real axis. We shall choose $\mathcal D = \mathcal D_- \cup \mathcal D_0 \cup \mathcal D_+$, each of which depend on a particular choice of $z = \eta + iy$ with $\eta = \frac{\pi}{\sqrt{12n}}$ for $n > 0$. These components of $\mathcal D$ are given by
\begin{align*}
	\mathcal D_{\pm} := \{ u + iv \in \CC : u \leq \eta, v = \pm 10 \eta \}, \\
	\mathcal D_0 := \{ u + iv \in \CC : u = \eta, |v| \leq 10 \eta \}.
\end{align*}
Note that this dependence on $z$ does not change the value of the integral, since one can shift the paths of integration. We shall compare the size of $I_\nu(z)$ to its main term. In particular, define
\begin{align*}
	\widehat{I}_\nu(n) := \lp \dfrac{\pi^2}{12n} \rp^{\frac{\nu}{2}} \dfrac{1}{2\pi i} \int_{\mathcal D_0} t^{-\nu-1} \exp\lp \dfrac{\pi^2}{12t} + \lp n + \dfrac{1}{24} \rp t \rp dt.
\end{align*}
The following lemma shows how $\widehat{I}_\nu(n)$ approximates $I_\nu(z)$ for certain values of $z$.

\begin{lemma} \label{Bessel Estimates}
	Let $n \geq 1$ be an integer and $\nu \leq -1$. Then
	\begin{align*}
		\bigg| I_\nu\lp \pi \sqrt{\dfrac{1}{3}\lp n + \dfrac{1}{24} \rp} \rp &- \widehat{I}_\nu(n) \bigg| \\ &< 2 \lp \dfrac{2\pi^2}{24n+1} \rp^{\frac{\nu}{2}} \exp\lp \dfrac{3\pi}{4} \sqrt{\dfrac{n}{3}} \rp \int_0^\infty \lp 10 + u \rp^{-\nu - 1} e^{-\lp n + \frac{1}{24} \rp u} du.
	\end{align*}
\end{lemma}

\begin{proof}
	By a change of variables $t \mapsto \lp n + \frac{1}{24} \rp t$ and shifting of the path of integration back to $\mathcal D$, we see that
	\begin{align*}
		I_\nu\lp \pi \sqrt{\dfrac{1}{3}\lp n + \dfrac{1}{24} \rp} \rp = \lp \dfrac{\pi^2}{12\lp n + \frac{1}{24} \rp} \rp^{\frac{\nu}{2}} \dfrac{1}{2\pi i} \int_{\mathcal D} t^{-\nu-1} \exp\lp \dfrac{\pi^2}{12t} + \lp n + \dfrac{1}{24} \rp t \rp dt.
	\end{align*}
	Thus, we have
	\begin{align*}
		I_\nu\lp \pi \sqrt{\dfrac{1}{3}\lp n + \dfrac{1}{24} \rp} \rp &- \widehat{I}_\nu(n) \\ &= \lp \dfrac{2\pi^2}{24n+1} \rp^{\frac{\nu}{2}} \dfrac{1}{2\pi i} \int_{\mathcal D_+ \cup \mathcal D_-} t^{-\nu-1} \exp\lp \dfrac{\pi^2}{12t} + \lp n + \dfrac{1}{24} \rp t \rp dt.
	\end{align*}
	For $t \in \mathcal D_-$, we may set $t = \lp \eta - u \rp - 10 \eta i$. Since we have $\mathrm{Re}\lp \frac{\pi^2}{12t} \rp \leq \frac{\pi}{4} \sqrt{\frac{n}{3}}$ for all $u \geq 0$ and $|t| \leq |\eta - u| + |10 \eta i| < 11\eta + u = \frac{11\pi}{\sqrt{12n}} + u$, we have
	\begin{align*}
		\left| t^{-\nu - 1} \exp\lp \dfrac{\pi^2}{12t} + nt \rp \right| &\leq |t|^{-\nu - 1} \exp\lp \dfrac{\pi}{4} \sqrt{\dfrac{n}{3}} + \lp n + \dfrac{1}{24} \rp \lp \eta - u \rp \rp \\ &\leq \lp \dfrac{11 \pi}{\sqrt{12n}} + u \rp^{-\nu - 1} \exp\lp \dfrac{3\pi}{4} \sqrt{\dfrac{n}{3}} - \lp n + \dfrac{1}{24} \rp u \rp,
	\end{align*}
	where the last inequality uses $-\nu - 1 \geq 0$. The same bound holds for $\mathcal D_+$. Since $\frac{11 \pi}{\sqrt{12n}} < 10$, we conclude that
	\begin{align*}
		\bigg| I_\nu\lp \pi \sqrt{\dfrac{1}{3}\lp n + \dfrac{1}{24} \rp} \rp &- \widehat{I}_\nu(n) \bigg| \\ &< 2 \lp \dfrac{2\pi^2}{24n+1} \rp^{\frac{\nu}{2}} \exp\lp \dfrac{3\pi}{4} \sqrt{\dfrac{n}{3}} \rp \int_0^\infty \lp 10 + u \rp^{-\nu - 1} e^{-\lp n + \frac{1}{24} \rp u} du.
	\end{align*}
	This completes the proof.
\end{proof}

\section{Effective asymptotics}

In this section, we prove effective bounds for the functions $L_{r,t}(q)$ and $\xi(q)$ on both the major and minor arcs. The first subsection covers major arc bounds, and the second covers minor arc bounds.

\subsection{Major arc effective bounds}

In this subsection, we compute effective bounds on the functions $L_{r,t}(q)$ and $\xi(q)$ on the major arc. We also note that in the region $0 \leq |y| < 10\eta$, the hypothesis $\eta < \frac{\pi}{40t}$ always implies $|z| < \frac{\sqrt{101} \pi}{80} < \frac{2}{5}$.

\begin{lemma} \label{L Major Arc}
	Let $t \geq 2$ and $0 < r \leq t$ be integers and $z = \eta + iy$ a complex number satisfying $0 \leq |y| < 10\eta$ and $\eta < \frac{\pi}{40t}$. Then we have
	\begin{align*}
		\left| L_{r,t}\lp e^{-z} \rp - \dfrac{\log(2)}{tz} + \dfrac{1}{2} B_1\lp \dfrac{r}{t} \rp - \dfrac{t}{8} B_2\lp\frac{r}{t}\rp z + \dfrac{t^3}{192} B_4\lp \frac rt \rp z^3 \right| < \dfrac{1}{20} t^5 |z|^5.
	\end{align*}
\end{lemma}

\begin{proof}
	The proof relies on an application of Proposition \ref{C3 Euler-Maclaurin Rapid Decay Effective} to $E(z) = \sum_{n=0}^\infty \frac{e_n}{n!} z^n$, whose radius of convergence is $\pi$. We note $M_6 = \frac{1}{42}$. Thus, applying Proposition \ref{C3 Euler-Maclaurin Rapid Decay Effective} to $E(z) = \sum_{k \geq 0} \frac{e_k}{k!} z^k$ with $a = \frac rt$, we obtain
	\begin{align*}
		\bigg| \sum_{k \geq 0} E\lp \lp k + \frac{r}{t} \rp z \rp - \dfrac{I_E}{z} &+ \dfrac{1}{2} B_1\lp \dfrac rt \rp - \dfrac{1}{8} B_2\lp \dfrac rt \rp z + \dfrac{1}{192} B_4\lp \dfrac rt \rp z^3 \bigg| \\ &\leq \dfrac{J_{E,6}(z)}{30240} |z|^5 + \sum_{k \geq 5} |e_k| \lp 1 + \dfrac{k!}{10 (k-5)!} \rp |z|^k.
	\end{align*}
	We also have $I_E = \int_0^\infty \frac{dx}{e^x + 1} = \log(2)$ and therefore by Lemma \ref{Xi and L z-Expansions} we have
	\begin{align*}
		\bigg| L_{r,t}\lp e^{-z} \rp - \dfrac{\log(2)}{tz} &+ \dfrac{1}{2} B_1\lp \dfrac{r}{t} \rp - \dfrac{t}{8} B_2\lp\frac{r}{t}\rp z + \dfrac{t^3}{192} B_4\lp \frac rt \rp z^3 \bigg| \\ &\leq \dfrac{J_{E,6}(z)}{30240} |tz|^5 + |tz|^5 \sum_{k \geq 5} |e_k| \lp 1 + \dfrac{k!}{10 (k-5)!} \rp |tz|^{k-5},
	\end{align*}
	which is valid for all for all $|z| < \frac{\pi}{t}$, hence in particular when $\eta < \frac{\pi}{40t}$ and $0 \leq |y| < 10\eta$. We now proceed to estimate each piece on the right-hand side.
	
	Let $\alpha = \frac{\pi}{2} \frac{z}{|z|}$. Then we bound $J_{E,6}(z)$ by the decomposition
	\begin{align*}
		J_{E,6}(z) = \int_0^\alpha \left| E^{(6)}(w) \right|dw + \int_\alpha^\infty \left| E^{(6)}(w) \right|dw.
	\end{align*}
	The function $E^{(6)}(z)$ is given by
	\begin{align*}
		E^{(6)}(z) = \dfrac{e^z\lp e^z - 1 \rp \lp e^{4z} - 56 e^{3z} + 246 e^{2z} - 56 e^z + 1 \rp}{\lp e^z + 1 \rp^7}.
	\end{align*}
	By the triangle inequality, we have
	\begin{align*}
		\left| E^{(6)}(z) \right| \leq \dfrac{e^\eta \lp e^\eta + 1 \rp \lp e^{4\eta} + 56 e^{3\eta} + 246 e^{2\eta} + 56 e^\eta + 1 \rp}{\lp e^\eta - 1 \rp^7}.
	\end{align*}
	These bounds entail that for $u = \mathrm{Re}\lp w \rp$ and the major arc $0 \leq |\mathrm{Im}(w)| < 10 u$, we have
	\begin{align*}
		\int_{\alpha}^\infty \left| E^{(6)}(w) \right| dw \leq \sqrt{101} \int_{\pi/2}^\infty \dfrac{e^u \lp e^u + 1 \rp \lp e^{4u} + 56 e^{3u} + 246 e^{2u} + 56 e^u + 1 \rp}{\lp e^u - 1 \rp^7} |du| < 81.
	\end{align*}
	The power series representation of $E^{(6)}(w)$ is valid in the region from $0$ to $\alpha$. Combining the estimates $|w| < \frac{\pi}{2}$, \eqref{Bernoulli Inequality}, \eqref{e_n Evaluation}, the vanishing of $B_{2n+1}$ for $n \geq 1$, and the fact that $\zeta(n)$ is decreasing for $n > 1$, we have
	\begin{align*}
		\left| E^{(6)}(w) \right| \leq \sum_{k=6}^\infty \dfrac{2^{k+1} \left|B_{k+1}\right| \pi^{k-6}}{(k-6)! 2^{k-6}} \leq \dfrac{\zeta(8) 2^7}{\pi^7} \sum_{k \geq 3} \dfrac{(2k+2)!}{2^{2k+1} (2k-5)!} < 429.
	\end{align*}
	Therefore, we find that
	\begin{align*}
		J_{E,6}(z) < \dfrac{429\pi}{2} + 81 < 755.
	\end{align*}
	We may also show using \eqref{Bernoulli Inequality} and \eqref{e_n Evaluation} that $\left| \frac{e_n}{n!} \right| \leq \frac{\pi}{3} \cdot \lp \frac{1}{\pi} \rp^n$, and therefore since $|z| < \frac{\sqrt{101}}{40t}$ we have
	\begin{align*}
		\sum_{k \geq 5} |e_k| \lp 1 + \dfrac{k!}{10 (k-5)!} \rp |tz|^{k-5} < \dfrac{1}{3\pi^4} \sum_{k \geq 5} \lp 1 + \dfrac{k!}{10 (k-5)!} \rp \lp \dfrac{\sqrt{101}}{40} \rp^{k-5} < \dfrac{1}{4}.
	\end{align*}
	Thus,
	\begin{align*}
		\bigg| L_{r,t}\lp e^{-z} \rp - \dfrac{\log(2)}{tz} + \dfrac{1}{2} B_1\lp \dfrac{r}{t} \rp - \dfrac{t}{8} B_2\lp\frac{r}{t}\rp z &+ \dfrac{t^3}{192} B_4\lp \frac rt \rp z^3 \bigg| \\ &\leq \dfrac{755}{30240} |tz|^5 + \dfrac{|tz|^5}{4} < \dfrac{7}{25} t^5 |z|^5.
	\end{align*}
	This completes the proof.
\end{proof}

\begin{corollary} \label{L Major Arc Corollary}
	Let $0 < r \leq t$ be integers and $z = \eta + iy$ a complex number satisfying $0 \leq |y| < 10\eta$ and $\eta < \frac{\pi}{40t}$. Then
	\begin{align*}
		\left| L_{r,t}\lp e^{-z} \rp \right| < \dfrac{14}{|tz|}.
	\end{align*}
\end{corollary}

\begin{proof}
	By the triangle inequality and Lemma \ref{L Major Arc}, we have
	\begin{align*}
		\left| L_{r,t}\lp e^{-z} \rp \right| < \dfrac{\log(2)}{t|z|} + \left| \dfrac 12 B_1\lp \dfrac rt \rp \right| + \left| \dfrac t8 B_2\lp \dfrac rt \rp z \right| + \left| \dfrac{t^3}{192} B_4\lp \dfrac rt \rp z^3 \right| + \dfrac{7}{25} |tz|^5.
	\end{align*}
	The fact that $\eta < \frac{\pi}{40t}$ entails $|z| < \frac{\sqrt{101} \pi}{40t} < \frac{4}{5t}$. Using the trivial bound on $B_1\lp \frac rt \rp$, Lehmer's bound \eqref{Bernoulli Inequality} and $|tz| < \frac{4}{5}$, we obtain
	\begin{align*}
		\left| L_{r,t}\lp e^{-z} \rp \right| < \dfrac{\log(2) + \frac{1}{4}|tz| + \frac{5}{96}|tz|^2 + \frac{1}{1344} |tz|^4 + \frac{7}{25} |tz|^6}{|tz|} < \dfrac{14}{|tz|},
	\end{align*}
	which completes the proof.
\end{proof}

\begin{lemma} \label{Xi Major Arc}
	For any integer $t \geq 2$ and any complex number $z = \eta + iy$ with $0 \leq |y| < 10\eta$ and $\eta < \frac{\pi}{40t}$, we have
	\begin{align*}
		\bigg| \Log \lp \xi\lp e^{-z} \rp \rp - \dfrac{\pi^2}{12 z} + \dfrac{\log(2)}{2} - \dfrac{z}{24} \bigg| < 471 |z|^8.
	\end{align*}
\end{lemma}

\begin{proof}
	By Lemma \ref{Xi and L z-Expansions}, we have
	\begin{align*}
		\Log \lp \xi\lp e^{-z} \rp \rp = z \sum_{m \geq 0} \left[ B\lp \lp m + \frac 12 \rp 2z \rp - B\lp \lp m+1 \rp 2z \rp \right],
	\end{align*}
	where $B(z) = \frac{e^{-z}}{z\lp 1 - e^{-z} \rp}$. We apply Proposition \ref{C3 Euler-Maclaurin Sufficient Effective} with $N = 7$ and $A=1$. Noting that $M_8 = \frac{1}{30}$, $c_{-2} = 1$, and $c_{-1} = - \frac 12$, we have
	\begin{align*}
		\bigg| \sum_{m \geq 0} B\lp (m+a)z \rp - \dfrac{\zeta(2,a)}{z^2} - \dfrac{I_{B,1}^*}{z} &- \dfrac{1}{2z} \lp \Log\lp z \rp + \gamma + \psi\lp a \rp \rp - \sum_{n = 0}^\infty c_n^* \dfrac{B_{n+1}(a)}{n+1} z^n \bigg| \\ &\leq \dfrac{J_{g,8}(z)}{1209600} |z|^7 + \sum_{k \geq 7} |b_k| \lp 1 + \dfrac{k!}{10 (k-7)!} \rp |z|^k,
	\end{align*}
	where $b_k = \frac{B_{k+2}}{(k+2)!} + \frac{(-1)^{k+1}}{2(k+1)!}$ and $g(z) = \frac{e^{-z}}{z\lp 1 - e^{-z} \rp} - \frac{1}{z^2} + \frac{e^{-z}}{2z}$. Note that like $B(z)$, the power series representation of $g(z)$ has radius of convergence $2\pi$. We now reduce the bounds on the right-hand side of the above. Setting $\alpha = \frac{3\pi}{2} \frac{z}{|z|}$, we decompose $J_{g,8}(z)$ as
	\begin{align*}
		J_{g,8}(z) = \int_0^\alpha \left| g^{(8)}(w) \right| |dw| + \int_\alpha^\infty \left| g^{(8)}(w) \right| |dw|,
	\end{align*}
	where the paths proceed radially as originally defined. We first bound $g^{(6)}(w)$ on the interval near zero. Invoking \eqref{Bernoulli Inequality}, we can see that
	\begin{align*}
		\left| b_k \right| \leq \frac{1}{12} \cdot \lp \frac{1}{2\pi} \rp^{k} + \frac{1}{2(k+1)!}.
	\end{align*}
	for all $k$, so for $|w| < \frac{3\pi}{2}$ we have
	\begin{align*}
		\left| g^{(8)}(w) \right| \leq \sum_{k \geq 0} \dfrac{(k+8)!}{k!} \lp \frac{1}{12} \cdot \lp \frac{1}{2\pi} \rp^{k+8} + \frac{1}{2(k+9)!} \rp \lp \dfrac{3\pi}{2} \rp^k < 367.
	\end{align*}
	Thus, we have that
	\begin{align*}
		\int_0^\alpha \left| g^{(8)}(w) \right| |dw| < 367\dfrac{3\pi}{2} < 1730.
	\end{align*}
	Now, $g^{(8)}(w)$ may be written in the form
	\begin{align*}
		g^{(8)}(w) = \sum_{j=1}^9 \dfrac{p_j(w)}{\lp e^w - 1 \rp^{9-j} w^j}
	\end{align*}
	for certain polynomials $p_j(w)$ of degree $j-1$ with non-negative coefficients. For $w$ on the major arc, we have $u = \mathrm{Re}(w) \leq |w| \leq \sqrt{101}u$, and therefore by the triangle inequality we have
	\begin{align*}
		\left| g^{(8)}(w) \right| \leq \sum_{j=1}^9 \dfrac{p_j\lp \sqrt{101} u \rp}{\lp e^u - 1 \rp^{9-j} u^j}.
	\end{align*}
	Integrating with the aid of a computer, we have
	\begin{align*}
		\int_\alpha^\infty \left| g^{(8)}(w) \right| |dw| \leq \sqrt{101} \int_{\frac{3\pi}{2}}^\infty  \sum_{j=1}^9 \dfrac{p_j\lp \sqrt{101} u \rp}{\lp e^u - 1 \rp^{9-j} u^j} du < 2206410.
	\end{align*}
	Therefore, we find that
	\begin{align*}
		J_{g,8}(z) < 1730 + 2206410 = 2208140.
	\end{align*}
	By the previous bound on $\left| b_k \right|$ as well as the fact that $|z| < \frac{2}{5}$ on the major arc, we have that
	\begin{align*}
		\sum_{k \geq 7} \left|b_k\right| &\lp 1 + \dfrac{k!}{10 (k-7)!} \rp |z|^{k-7} \\ &< \sum_{k \geq 7} \lp \frac{1}{12} \cdot \lp \frac{1}{2\pi} \rp^k + \frac{1}{2(k+1)!} \rp \lp 1 + \dfrac{k!}{10 (k-7)!} \rp \lp \dfrac{2}{5} \rp^{k-7} < \dfrac{1}{100}.
	\end{align*}
	Therefore, by letting $z \mapsto 2z$ and applying the bounds just derived, we obtain
	\begin{align*}
		\bigg| \sum_{m \geq 0} B\lp (m+a)2z \rp - \dfrac{\zeta(2,a)}{4z^2} - \dfrac{I_{B,1}^*}{2z} - \dfrac{1}{4z} \lp \Log\lp 2z \rp + \gamma + \psi\lp a \rp \rp &- \sum_{n = 0}^\infty c_n^* \dfrac{B_{n+1}(a)}{n+1} 2^n z^n \bigg| \\ &< 235 |z|^7.
	\end{align*}
	By the expansion from Lemma \ref{Xi and L z-Expansions}, we may conclude immediately that
	\begin{align*}
		\bigg| \Log \lp \xi\lp e^{-z} \rp \rp + \dfrac{\zeta(2,1) - \zeta\lp 2, \frac 12 \rp}{4 z} + \dfrac{ \psi(1) - \psi\lp \frac 12 \rp}{4} &- \sum_{n = 0}^\infty c_n^* \dfrac{B_{n+1}(1) - B_{n+1}\lp \frac 12 \rp}{n+1} 2^n z^{n+1} \bigg| \\ &< 470 |z|^8.
	\end{align*}
	
	We now proceed to simplify terms in the bounds above. By the definition of $c_n^*$ along with $c_{-1} = -\frac 12$, we may calculate
	\begin{align*}
		\sum_{n = 0}^\infty c_n^* \dfrac{B_{n+1}(1) - B_{n+1}\lp \frac 12 \rp}{n+1} 2^n z^{n+1} = \dfrac{z}{24} - \sum_{n \geq 7} \dfrac{(-1)^{n+1}\lp B_{n+1}(1) - B_{n+1}\lp \frac 12 \rp \rp}{(n+1) (n+1)!} 2^{n-1} z^{n+1}.
	\end{align*}
	Now, because of the identity $\zeta\lp s, \frac 12 \rp = \lp 2^s - 1 \rp \zeta(2)$, we have $\zeta\lp 2, 1 \rp - \zeta\lp 2, \frac 12 \rp = - \frac{\pi^2}{3}$. Furthermore, by \cite[(5.4)]{DLMF} we have 
	$\psi(1) = -\gamma$ and $- \psi\lp \frac 12 \rp = -2 \log(2) - \gamma$. Therefore, using the triangle inequality in the form $|x| \leq |x-y| + |y|$ and $|z| < \frac{\pi}{2}$, we may obtain
	\begin{align*}
		\bigg| \Log \lp \xi\lp e^{-z} \rp \rp - \dfrac{\pi^2}{12 z} &+ \dfrac{ \log(2)}{2} - \dfrac{z}{24} \bigg| \\ &< 470 |z|^8 + \left| \sum_{n \geq 7} \dfrac{(-1)^{n+1}\lp B_{n+1}(1) - B_{n+1}\lp \frac 12 \rp \rp}{(n+1) (n+1)!} 2^{n-1} z^{n-7} \right| \cdot |z|^8.
	\end{align*}
	Lehmer's bound \eqref{Bernoulli Inequality} along with the straightforward inequality $\zeta(n+1) \leq \zeta(2) = \frac{\pi^2}{6}$ for $n \geq 1$ implies that
	\begin{align*}
		\dfrac{\left| B_{n+1}(1) - B_{n+1}\lp \frac 12 \rp \right|}{(n+1)!} \leq \dfrac{4 \zeta(n+1)}{(2\pi)^{n+1}} \leq \dfrac{\pi}{3 \lp 2\pi \rp^n}
	\end{align*}
	for $n \geq 1$. Therefore using the fact that $|z| < \frac{\pi}{2}$ on the major arc with $\eta < \frac{\pi}{40t}$, we have
	\begin{align*}
		\sum_{n \geq 7} \dfrac{\left| B_{n+1}(1) - B_{n+1}\lp \frac 12 \rp \right|}{(n+1) (n+1)!} 2^{n-1} |z|^{n-7} \leq \dfrac{1}{6\pi^6} \sum_{n \geq 7} \dfrac{1}{(n+1) 2^{n-7}} < 1.
	\end{align*}
	Putting together all evaluations, we conclude that
	\begin{align*}
		\bigg| \Log \lp \xi\lp e^{-z} \rp \rp - \dfrac{\pi^2}{12 z} + \dfrac{\log(2)}{2} - \dfrac{z}{24} \bigg| < 471 |z|^8.
	\end{align*}
	This completes the proof.
\end{proof}

\begin{corollary} \label{Xi Major Arc Corollary}
	For any integer $t \geq 2$ and any complex number $z = \eta + iy$ satisfying $0 \leq |y| < 10\eta$ and $\eta < \frac{\pi}{40t}$, we have
	\begin{align*}
		\left| \xi\lp e^{-z} \rp - \exp\lp \dfrac{\pi^2}{12z} - \dfrac{\log(2)}{2} + \dfrac{z}{24} \rp \right| < \dfrac{630 |z|^8}{\sqrt{2}} \exp\lp \dfrac{\pi^2}{12|z|} \rp.
	\end{align*}
\end{corollary}

\begin{proof}
	Suppose $f(z), g(z), e(z)$ are any three functions that satisfy
	\begin{align*}
		\left| \Log \lp f(z) \rp - \Log \lp g(z) \rp \right| \leq e(z)
	\end{align*}
	for $|z| < \frac{\pi}{40t}$. Note that we may factorize
	\begin{align*}
		\left| f(z) - g(z) \right| = \left| \exp\lp \Log \lp f(z) \rp - \Log \lp g(z) \rp \rp - 1 \right| \cdot \left| g(z) \right|.
	\end{align*}
	Applying this factorization with $f(z) := \xi\lp e^{-z} \rp$ and $g(z) := \exp\lp \frac{\pi^2}{12 z} - \frac{\log(2)}{2} + \frac{z}{24} \rp$ will give the result. Using Lemma \ref{Xi Major Arc} and Taylor series, we have
	\begin{align*}
		\left| \exp\lp \Log \lp \xi\lp e^{-z} \rp \rp - \dfrac{\pi^2}{12 z} + \dfrac{\log(2)}{2} - \dfrac{z}{24} \rp - 1 \right| < \sum_{n \geq 1} \dfrac{1}{n!} \lp 471 |z|^8 \rp^n = \exp\lp 471 |z|^8 \rp - 1.
	\end{align*}
	For $|z| < \frac{\sqrt{101} \pi}{80}$, we have $471 |z|^8 < 0.28$, and since $e^x - 1 < \frac{4}{3} x$ for $0 < x < 0.55$, we have
	\begin{align*}
		\left| \exp\lp \Log \lp \xi\lp e^{-z} \rp \rp - \dfrac{\pi^2}{12 z} + \dfrac{\log(2)}{2} - \dfrac{z}{24} \rp - 1 \right| < 628 |z|^8.
	\end{align*}
	Using $\eta \leq |z|$ and $\eta < \frac{\pi}{80}$, we may conclude that
	\begin{align*}
		\left| \exp\lp \dfrac{\pi^2}{12 z} - \dfrac{\log(2)}{2} + \dfrac{z}{24} \rp \right| \leq \exp\lp \dfrac{\pi^2 \eta}{12|z|^2} + \dfrac{\eta}{24} \rp < \dfrac{501}{500\sqrt{2}} \exp\lp \dfrac{\pi}{12|z|}\rp.
	\end{align*}
	Combining the given bounds completes the proof.
\end{proof}

\subsection{Minor arc effective bounds}

We now calculate effective bounds on both $\xi(q)$ and $L_{r,t}(q)$ for the minor arc $10\eta \leq |y| < \pi$, subject to the additional constraint $\eta < \frac{\pi}{40t} \leq \frac{\pi}{80}$.

\begin{lemma} \label{Xi Minor Arc Bound}
	Let $t \geq 2$ be an integer. Assume $z = \eta + iy$ satisfies $10 \eta \leq |y| < \pi$ and $0 < \eta < \frac{\pi}{40t}$. Then we have
	\begin{align*}
		\left| \xi\lp e^{-z} \rp \right| < \exp\lp \dfrac{41}{50\eta} \rp.
	\end{align*}
\end{lemma}

\begin{proof}
	Let $q = e^{-z}$. Recall that
	\begin{align*}
		\Log \lp \xi(q) \rp = - \sum_{m \geq 1} \dfrac{(-1)^m q^m}{m\lp 1 - q^m \rp}.
	\end{align*}
	By taking absolute values and splitting off the $m=1$ term and noting that $\log P\lp |q| \rp = \sum_{m \geq 1} \frac{|q|^m}{m \lp 1 - |q|^m \rp}$, we have
	\begin{align*}
		\left| \Log \lp \xi(q) \rp \right| \leq \Log \lp P\lp |q| \rp \rp - |q| \lp \dfrac{1}{1 - |q|} - \dfrac{1}{|1 - q|} \rp,
	\end{align*}
	where $P(q) = (q;q)_\infty^{-1}$. To bound $\Log \lp P\lp |q| \rp \rp$, we recall that $|q| = e^{-\eta}$ and use the series expansion
	\begin{align*}
		\Log \lp P\lp |q| \rp \rp = \sum_{m \geq 1} \dfrac{|q|^m}{m\lp 1 - |q|^m \rp} = \sum_{m \geq 1} \dfrac{e^{-mx}}{m\lp 1 - e^{-mx} \rp}.
	\end{align*}
	From the fact that $\frac{e^{-x}}{1 - e^{-x}} < \frac{1}{x}$ for all $x > 0$, we may therefore deduce that
	\begin{align} \label{Log-P Bound}
		\Log \lp P\lp |q| \rp \rp < \sum_{m \geq 1} \dfrac{1}{m^2 \eta} = \dfrac{\pi^2}{6\eta}.
	\end{align}
	Now, we have $\left| 1 - q \right|^2 = 1 - 2 \cos(y) e^{-\eta} + e^{-2\eta}$. In the region $10\eta \leq |y| < \pi$, we have by the fact that $\cos(x)$ is decreasing for $0 < x < \pi$ that $\left| 1 - q \right|^2 \geq 1 - 2 \cos(10\eta) e^{-\eta} + e^{-2\eta}$. It can be checked in an elementary manner that $1 - 2 \cos(10\eta) e^{-\eta} + e^{-2\eta} > 95 \eta^2$, and so we have $\left| 1 - q \right| > \sqrt{95} \eta$. By using the bound $1 - |q| = 1 - e^{-\eta} \leq \eta$, we have for $10\eta \leq |y| < \pi$ and $\eta < \frac{\pi}{80}$ that
	\begin{align*}
		\dfrac{1}{|1 - q|} - \dfrac{1}{1 - |q|} < \lp \dfrac{1}{\sqrt{95}} - 1 \rp \dfrac{1}{\eta}.
	\end{align*}
	Therefore, using $|q| \leq 1$ we have
	\begin{align*}
		\left|\Log\lp \xi\lp q \rp \rp \right| \leq \lp \dfrac{\pi^2}{6} + \dfrac{1}{\sqrt{95}} - 1 \rp \dfrac{1}{\eta} < \dfrac{3}{4\eta}.
	\end{align*}
	Exponentiating completes the proof.
\end{proof}

\begin{lemma} \label{L Minor Arc Bound}
	Let $t \geq 2$ and $0 < r \leq t$ be integers. Assume $z = \eta + iy$ is a complex number satisfying $\eta > 0$. Then we have
	\begin{align*}
		\left| L_{r,t}\lp e^{-z} \rp \right| < \dfrac{1}{\eta^2}.
	\end{align*}
\end{lemma}

\begin{proof}
	Let $q = e^{-z}$ and let $\sigma_0(n) = \sum_{d|n} 1$ be the standard divisor counting function. Then we have that
	\begin{align*}
		\left| L_{r,t}(q) \right| \leq \sum_{m \geq 1} \dfrac{e^{-m\eta}}{1 - e^{-m\eta}} = \sum_{m \geq 1} \sigma_0(m) e^{-m\eta} \leq \sum_{m \geq 1} m e^{-m\eta} = \dfrac{e^\eta}{\lp e^\eta - 1 \rp^2} < \dfrac{1}{\eta^2}.
	\end{align*}
	This completes the proof.
\end{proof}

\section{Proof of Theorem \ref{C3 Ineffective Asymptotic}}

By Lemmas \ref{L Major Arc} and \ref{Xi Major Arc} in the following section, we have the asymptotics 
\begin{align*}
	L_{r,t}\lp e^{-z} \rp = \frac{\log(2)}{tz} - \frac{1}{2}\lp \frac{r}{t} - \frac{1}{2} \rp + O(z), \ \ \ \ \xi\lp e^{-z} \rp = \frac{1}{\sqrt{2}} e^{\frac{\pi^2}{12z} + O(z)}
\end{align*}
on the major arc. These imply asymptotics for $\mathcal D_{r,t}(q) = L_{r,t}(q) \xi(q)$, which is the generating function for $D_{r,t}(n)$ by Lemma \ref{Generating Function}, that satisfies (1) in Proposition \ref{Wright Circle Method}. Lemmas \ref{Xi Minor Arc Bound} and \ref{L Minor Arc Bound} imply condition (2), and so we may apply Proposition \ref{Wright Circle Method}, which yields the claimed asymptotic formula.

\section{Proof of Theorem \ref{C3 Effective Asymptotic}}

In this section, we complete the proof of Theorem \ref{C3 Effective Asymptotic} by following the proof of \cite[Proposition 1.8]{NR17} (which is a version of Wright's circle method slightly different from Proposition \ref{Wright Circle Method}) and making the bounds in each step effective. Let $\mathcal{C}$ be the circle in the complex plane with center 0 and radius $e^{-\eta}$, where $\eta = \frac{\pi}{\sqrt{12n}}$. By Cauchy's formula and Lemma \ref{Generating Function}, we have
\begin{align*}
	D_{r,t}(n) = \dfrac{1}{2\pi i} \int_{\mathcal C} \dfrac{\mathcal{D}_{r,t}(q)}{q^{n+1}} dq = \dfrac{1}{2\pi i} \int_{\mathcal C} \dfrac{L_{r,t}(q) \xi(q)}{q^{n+1}} dq.
\end{align*}
Throughout, we fix $q = e^{-z}$ with $z = \eta + iy$, so that $|q| = e^{-\eta}$. We will estimate $D_{r,t}(n)$ by decomposing this integral into convenient pieces. Choose $\delta > 0$ so that for the major arc $\mathcal{C}_1$, $z = \eta + iy \in \mathcal{C}_1$ satisfies $0 < |y| < 10\eta$. We shall also assume that $\eta < \frac{\pi}{40t}$, which is equivalent to the bound $n > \frac{400t^2}{3}$.

Let $\mathcal{C}_2 := \mathcal{C} \backslash \mathcal{C}_1$ denote the minor arc. Define for $s \geq 0$ the integrals
\begin{align} \label{V Definition}
	V_s(n) &:= \dfrac{1}{2\pi i} \int_{\mathcal{C}_1} \dfrac{z^{s-1}}{q^{n+1}} \exp\lp \dfrac{\pi^2}{12z} - \dfrac{\log(2)}{2} + \dfrac{z}{24} \rp dq \notag \\ &= \dfrac{1}{2\pi \sqrt{2} i} \int_{D_0} z^{s-1} \exp\lp \dfrac{\pi^2}{12z} + \lp n + \dfrac{1}{24} \rp z \rp dz.
\end{align}
We use the integrals $V_s(n)$ to estimate $D_{r,t}(n)$. In particular, we have the decomposition
\begin{align*}
	D_{r,t}(n) - \alpha_0 V_0(n) - \alpha_{1,r} V_1(n) - \alpha_{2,r} V_2(n) - \alpha_{4,r} V_4(n) = E_1 + E_2 + E_3,
\end{align*}
where $\alpha_0 = \frac{\log(2)}{t}$, $\alpha_{1,r} = - \frac{1}{2} B_1\lp \frac rt \rp$, $\alpha_{2,r} = \frac{t}{8} B_2\lp \frac rt \rp$, $\alpha_{4,r} = \frac{t^3}{192} B_4\lp \frac rt \rp$, and
\begin{align*}
	E_1 &:= \dfrac{1}{2\pi i} \int_{\mathcal{C}_2} \dfrac{L_{r,t}(q) \xi(q)}{q^{n+1}} dq, \\
	E_2 &:= \dfrac{1}{2\pi i} \int_{\mathcal{C}_1} \dfrac{L_{r,t}(q) \lp \xi(q) - \exp \lp \dfrac{\pi^2}{12z} - \dfrac{\log(2)}{2} + \dfrac{z}{24} \rp \rp}{q^{n+1}} dq, \\
	E_3 &:= \dfrac{1}{2\pi i} \int_{\mathcal{C}_1} \dfrac{\lp L_{r,t}(q) - \alpha_0 z^{-1} - \alpha_1 - \alpha_2 z  - \alpha_4 z^3 \rp \exp \lp \dfrac{\pi^2}{12z} - \dfrac{\log(2)}{2} + \dfrac{z}{24} \rp}{q^{n+1}} dq.
\end{align*}
Although $\alpha_{i,r}$ for $i > 0$ depends on $r$, we suppress this dependence when $r$ is understood from context. Because $|z|^2 \geq \eta^2$ and $\eta = \frac{\pi}{\sqrt{12n}}$, we have
\begin{align*}
	\left| \exp\lp \dfrac{\pi^2}{12z} + nz \rp \right| = \exp\lp \dfrac{\pi^2 \eta}{12|z|^2} + n \eta \rp \leq \exp\lp \pi \sqrt{\dfrac{n}{3}} \rp.
\end{align*}
Furthermore, we note that
\begin{align*}
	\left| \int_{\mathcal{C}_1} q^{-1} dq \right| = \left| \left[ \Log\lp e^{-z} \rp \right]_{z = \eta - 10 \eta i}^{z = \eta + 10 \eta i} \right| \leq 20 \eta.
\end{align*}
and
\begin{align*}
	\left| \int_{\mathcal{C}_2} q^{-1} dq \right| \leq \mathrm{len}(\mathcal{C}_2) \cdot \max\limits_{z \in \mathcal{C}_2} |z| \leq 4.2 \pi^2.
\end{align*}
We also note that on the major arc $0 < |y| < 10\eta < \pi$ we have $\eta \leq |z| < \sqrt{101} \eta$. Since $\eta < \frac{\pi}{80}$, we also have $|z| < \frac{\sqrt{101}\pi}{80} < \frac{2}{5}$. These inequalities will be used freely in what follows.

To bound $E_3$, we recall that Lemma \ref{L Major Arc} says that for $\eta < \frac{\pi}{40t}$ on the major arc, we have
\begin{align*}
	\left| L_{r,t}(q) - \alpha_0 z^{-1} - \alpha_1 - \alpha_2 z - \alpha_4 z^3 \right| < \dfrac{1}{20} t^5 |z|^5,
\end{align*}
and we therefore have using these equations and numerical estimates that
\begin{align*}
	\left| E_3 \right| &\leq \dfrac{10 \eta}{\pi} \left| L_{r,t}(q) - \alpha_0 z^{-1} - \alpha_1 - \alpha_2 z - \alpha_4 z^3 \right| \left| \exp\lp \dfrac{\pi^2}{12z} - \dfrac{\log(2)}{2} + nz + \dfrac{z}{24} \rp \right| \\ &< \dfrac{14381 t^5}{n^3} \exp\lp \pi \sqrt{\dfrac{n}{3}} \rp.
\end{align*}
To bound $E_2$, we apply Corollary \ref{Xi Major Arc Corollary}, which we recall says
\begin{align*}
	\left| \xi\lp e^{-z} \rp - \exp\lp \dfrac{\pi^2}{12z} - \dfrac{\log(2)}{2} + \dfrac{z}{24} \rp \right| < \dfrac{630|z|^8}{\sqrt{2}} \exp\lp \dfrac{\pi^2}{12|z|} \rp.
\end{align*}
Therefore, we have that
\begin{align*}
	\left| E_2 \right| &\leq \dfrac{10 \eta}{\pi}  \left| L_{r,t}(q) \right| \left| \xi(q) - \exp\lp \dfrac{\pi^2}{12z} - \dfrac{\log(2)}{2} + \dfrac{z}{24} \rp \right| \left| \exp\lp nz \rp \right| \\ &< \dfrac{945285959087}{t n^4} \exp\lp \pi \sqrt{\dfrac{n}{3}} \rp.
\end{align*}
Finally, using Lemmas \ref{Xi Minor Arc Bound} and \ref{L Minor Arc Bound} we have
\begin{align*}
	\left| E_1 \right| \leq \dfrac{4.2 \pi^2}{2\pi} \left| L_{r,t}(q) \right| \left| \xi(q) \right| \left| \exp\lp nz \rp \right| < 9 n \exp\lp \lp \dfrac{3 \sqrt{3}}{2\pi} + \dfrac{\pi}{\sqrt{12}} \rp \sqrt{n} \rp.
\end{align*}
We have therefore shown that
\begin{align*}
	\left| D_{r,t}(n) - \alpha_0 V_0(n) - \alpha_1 V_1(n) - \alpha_2 V_2(n) - \alpha_4 V_4(n) \right| \leq \mathrm{Err}_t(n)
\end{align*}
where
\begin{align} \label{E_t Definition}
	\mathrm{Err}_t(n) := \dfrac{14381 t^5}{n^3} \exp\lp \pi \sqrt{\dfrac{n}{3}} \rp &+ \dfrac{945285959087}{t n^4} \exp\lp \pi \sqrt{\dfrac{n}{3}} \rp \notag \\ &+ 9 n \exp\lp \lp \dfrac{3 \sqrt{3}}{2\pi} + \dfrac{\pi}{\sqrt{12}} \rp \sqrt{n} \rp.
\end{align}
This completes the proof of Theorem \ref{C3 Effective Asymptotic}.

\section{Proof of Corollary \ref{C3 Effective Inequality}}

We now wish to resolve the inequality $D_{r,t}(n) \geq D_{s,t}(n)$ for integers $n \geq 0$ and $0 < r < s \leq t$. We define for convenience $\alpha_{j,r}^* := \alpha_{j,r} - \alpha_{j,r+1}$ and
\begin{align*}
	M_{r,t}(n) := \alpha_0 V_0(n) + \alpha_{1,r} V_1(n) + \alpha_{2,r} V_2(n) + \alpha_{4,r} V_4(n).
\end{align*}
Note that since $D_{r,t}(n) - D_{s,t}(n) = \sum_{j=r}^{s-1} D_{j,t}(n) - D_{j+1,t}(n)$, it suffices to prove $D_{r,t}(n) \geq D_{r+1,t}(n)$ for all $n > 8$ and $0 < r < t$. We therefore focus on this inequality.

By Theorem \ref{C3 Effective Asymptotic} applied to both terms in $D_{r,t}(n) - D_{r+1,t}(n)$, in order to show $D_{r,t}(n) - D_{r+1,t}(n) \geq 0$ it suffices to show
\begin{align*}
	M_{r,t}(n) - M_{r+1,t}(n) \geq 2 \mathrm{Err}_t(n).
\end{align*}
Collecting together like terms and simplifying, this is equivalent to
\begin{align*} 
	\alpha_{1,r}^* V_1(n) + \alpha_{2,r}^* V_2(n) + \alpha_{4,r}^* V_4(n) \geq 2 \mathrm{Err}_t(n).
\end{align*}
We also wish to bound the terms $\alpha_{j,r}^*$ for $j = 1,2,4$. Since $B_1(x) = x - \frac 12$, $B_2(x) = x^2 - x + \frac 16$, and $B_4(x) = x^4 - 2x^3 + x^2 - \frac{1}{30}$, and $1 \leq r < t$ (since $r+1 \leq t$), we have $\alpha_{1,r}^* = \frac{1}{2t}$, $\alpha_{2,r}^* = \frac{t-2r-1}{8t^2} \geq - \frac{3}{16}$, and $\alpha_{4,r}^* \geq - \frac{233}{48}$ for $2 \leq t \leq 10$. Therefore, it would suffice to prove that
\begin{align*}
	\dfrac{V_1(n)}{2t} \geq \dfrac{3}{16} V_2(n) + \dfrac{233}{48} V_4(n) + 2 \mathrm{Err}_t(n).
\end{align*}

Now, note that by the definition of $\widehat{I}_s(n)$ used in Lemma \ref{Bessel Estimates}, we have
\begin{align*}
	V_s(n) = \dfrac{1}{\sqrt{2}} \lp \dfrac{24n+1}{2\pi^2} \rp^{- \frac s2} \widehat{I}_{-s}(n),
\end{align*}
and therefore by Lemma \ref{Bessel Estimates} we may conclude that for $s \geq 1$,
\begin{align*}
	\bigg| V_s(n) - \dfrac{1}{\sqrt{2}} \lp \dfrac{24n+1}{2\pi^2} \rp^{- \frac s2} &I_{-s}\lp \pi \sqrt{\dfrac{1}{3}\lp n + \dfrac{1}{24} \rp} \rp \bigg| \\ &\leq \sqrt{2} \exp\lp \dfrac{3\pi}{4} \sqrt{\dfrac{n}{3}} \rp \int_0^\infty \lp 10 + u \rp^{s - 1} e^{-\lp n + \frac{1}{24} \rp u} du.
\end{align*}
Now by a substitution $u \mapsto \frac{1}{n + \frac{1}{24}}u$, we have
\begin{align*}
	\sqrt{2} \int_0^\infty \lp 10 + u \rp^{-\nu - 1} e^{-\lp n + \frac{1}{24} \rp u} du = \dfrac{24\sqrt{2}}{24n+1} \int_0^\infty \lp 10 + \dfrac{24u}{24n+1} \rp^{s - 1} e^{-u} du
\end{align*}
For $\beta_1 = 1$, $\beta_2 = 11$, and $\beta_4 = 1349$, we may conclude that each of $s = 1, 2, 4$ satisfies
\begin{align*}
	\left| V_s(n) - \dfrac{1}{\sqrt{2}} \lp \dfrac{24n+1}{2\pi^2} \rp^{- \frac s2} I_{-s}\lp \pi \sqrt{\dfrac{1}{3}\lp n + \dfrac{1}{24} \rp} \rp \right| < \dfrac{24 \beta_s \sqrt{2}}{24n+1} \exp\lp \dfrac{3\pi}{4} \sqrt{\dfrac{n}{3}} \rp.
\end{align*}
Therefore, if we set $n' := n + \frac{1}{24}$ for convenience, to prove the desired inequality it would suffice to show that
\begin{align} \label{Reduced Inequality}
	\dfrac{\pi}{4t\sqrt{6n'}} &I_{-1}\lp \pi \sqrt{\dfrac{n'}{3}} \rp \geq \dfrac{\pi^2}{64 n' \sqrt{2}} I_{-2}\lp \pi \sqrt{\dfrac{n'}{3}} \rp + \dfrac{233\pi^4}{6912 \sqrt{2} (n')^2} I_{-4}\lp \pi \sqrt{\dfrac{n'}{3}} \rp \notag \\ &+ \lp \dfrac{1}{t \sqrt{2}} + \dfrac{33 \sqrt{2}}{16} + \dfrac{314317 \sqrt{2}}{48} \rp \dfrac{1}{n'} \exp\lp \dfrac{3\pi}{4} \sqrt{\dfrac{n}{3}} \rp + 2 \mathrm{Err}_t(n).
\end{align}

In summary, we have shown that in order to show that $D_{r,t}(n) \geq D_{s,t}(n)$ for all $0 < r < s \leq t$ for a fixed value of $n$, it suffices to consider the case $s = r+1$ for each $r$, and all of these cases follow from the inequality \eqref{Reduced Inequality} is true. In the process of deriving \eqref{Reduced Inequality}, we have also assumed $n > \frac{400t^2}{3}$. Therefore, we define the integer $N_t(n)$ as the smallest positive integer satisfying $N_t(n) > \frac{400t^2}{3}$ and so that \eqref{Reduced Inequality} is true for all $n > N_t(n)$, from which it follows that $D_{r,t}(n) \geq D_{s,t}(n)$ for all $n > N_t(n)$. The table below gives values of $N_t(n)$, which are computed with the aid of a computer.

It therefore only remains to check the possible values of $D_{r,t}(n) - D_{r+1,t}(n)$ for $n \leq N_t(n)$ by computer and determine all possible counterexamples which arise from these cases. All such counterexamples satisfy $n \leq 8$ for $2 \leq t \leq 10$, which completes the proof.

\begin{center}
	\begin{tabular}{|c|c|c|c|c|c|} \hline
		$t$ & 2 & 3 & 4 & 5 & 6 \\ \hline
		$N_t(n)$ & 108077 & 112183 & 115240 & 117804 & 120247 \\ \hline
	\end{tabular}
	
	\vspace{0.1in}
	
	\begin{tabular}{|c|c|c|c|c|} \hline
		$t$ & 7 & 8 & 9 & 10 \\ \hline
		$N_t(n)$ & 122995 & 126772 & 133268 & 147752 \\ \hline
	\end{tabular}
	
	\vspace{0.05in}
	
	{\small Table 2: Numerics for Corollary \ref{C3 Effective Inequality}.}
\end{center}

\sglsp

\chapter{The Coll--Mayers--Mayers Conjecture} \label{C4}
\thispagestyle{myheadings}

\dblsp
\vspace*{-.65cm}

\section{The work of Seo and Yee}

The work in this section is not due to the author, but to Seo and Yee in \cite{SY20}. However, since this is a crucial ingredient in the proof of Conjecture \ref{C4 CMM Conjecture}, it is important to understand the ideas which went into this proof. Thus, we shall summarize the main line of argument used by Seo and Yee.

We recall the $q$-series which is the object of this chapter, which is
\begin{align*}
	G(q) := \lp q, -q^3; q^4 \rp_\infty^{-1} =: \sum_{n=0}^\infty a(n) q^n.
\end{align*}

Recall that $o(n)$, $e(n)$ denote the number of partitions into odd parts having odd/even index, respectively. Seo and Yee \cite[Theorem 1]{SY20} prove the following generating function identity.
\begin{theorem} \label{Seo-Yee}
	We have
	\begin{align*}
		G(q) = \sum_{n=0}^\infty \lp -1 \rp^{\lceil \frac n2 \rceil} \lp o(n) - e(n) \rp q^n.
	\end{align*}
\end{theorem}

\subsection{Meanders}

As the construction of the index statistic goes back to the work of Dergachev and Kirillov in \cite{DK00}, we first must state their theorem which computes this index. To do so, we must associate to each pair of partitions $\lambda, \mu \vdash n$ a certain graph, which we call $G = G(\lambda, \mu)$. The construction of this graph is as follows:

\begin{itemize}
	\item Start with an empty graph on $n$ vertices.
	\item For $\lambda = \lp \lambda_1, \dots, \lambda_r \rp$, label these vertices $v_{1, 1}, v_{1, 2}, \dots, v_{1,\lambda_1}, v_{2,1}, \dots, v_{2,\lambda_2}, \dots, v_{r, \lambda_r}$.
	\item For each $1 \leq i \leq r$ and $1 \leq j \leq \lfloor \lambda_i/2 \rfloor$, draw a {\it top edge} between $v_{i,j}$ $v_{i,\lambda_i+1-j}$.
	\item Do the same process for $\mu = \lp \mu_1, \dots, \mu_s \rp$, and call the newly constructed edges {\it bottom edges}.
\end{itemize}

The graph $G$ is called the {\it meander} associated to the pair $\lambda, \mu$. An example of this construction is given below, which is the meander associated to the pair of partitions of 8 given by $\lambda = (3,3,2)$ and $\mu = (4,3,1)$.

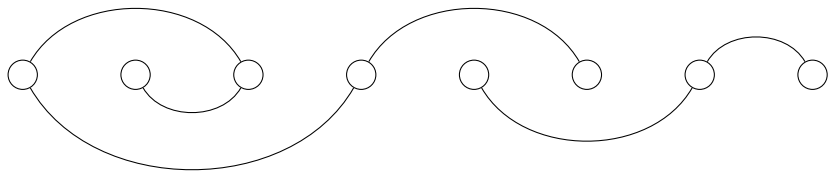
\begin{center}
\begin{tikzpicture}
	\node[shape=circle,draw=black] (1) at (0,0) {};
	\node[shape=circle,draw=black] (2) at (1.5,0) {};
	\node[shape=circle,draw=black] (3) at (3,0) {};
	\node[shape=circle,draw=black] (4) at (4.5,0) {};
	\node[shape=circle,draw=black] (5) at (6,0) {};
	\node[shape=circle,draw=black] (6) at (7.5,0) {};
	\node[shape=circle,draw=black] (7) at (9,0) {};
	\node[shape=circle,draw=black] (8) at (10.5,0) {};
	
	\path [-] (1) edge[bend left=60] (3);
	\path [-] (4) edge[bend left=60] (6);
	\path [-] (7) edge[bend left=60] (8);
	\path [-] (1) edge[bend right=60] (4);
	\path [-] (2) edge[bend right=60] (3);
	\path [-] (5) edge[bend right=60] (7);
\end{tikzpicture}
\end{center}

Dergachev and Kirillov \cite{DK00} construct certain Lie algebras $\mathfrak g$ which depend on the pair $\lambda, \mu \vdash n$ which we defined in the introduction as {\it seaweed algebras}. They prove that
\begin{align*}
	\ind(\mathfrak g) = 2C + P - 1,
\end{align*}
where $C$ and $P$ are defined as the number of cycles and paths occurring in $G$, respectively. In this context, an isolated vertex is counted as a path. For example, for the meander above, we have $C = 0$ and $P = 2$, and so the index is 1. Note that because the degree of each vertex in $G$ is at most two, every connected component of $G$ is either a path of a cycle.

In this way, Coll, Mayers and Mayers consider the index as a partition-theoretic statistic by defining the index of the pair $\lp \lambda, \mu \rp$, which we might denote by $\ind_\mu(\lambda)$, as the index of the associated seaweed algebra; this index can be computed directly from the meander without recourse to Lie theory. One of the main ideas in this proof is that the number of paths in the meander can be given combinatorial meaning, as we shall soon see.

\subsection{Proof of Theorem \ref{Seo-Yee}}

We now wish to give some indication of the method of Seo and Yee. First, we make the observation that
\begin{align*}
	\ind_\mu(\lambda) \equiv P + 1 \pmod{2}.
\end{align*}
Thus, to know the parity of the index it would suffice to count the paths in the meander $G$ associated to the pair. Because the edges of $G$ are constructed by a pairing of vertices, one can see that the number of vertices in $G$ that are not touched by any top edge are in bijective correspondence with the odd parts of $\lambda$; similarly those not touched by a bottom edge are in bijection with odd parts of $\mu$. Since each path in $G$ can be identified with its endpoints (the same point counted twice in the case of the empty path), we see that
\begin{align*}
	P = \dfrac{\text{op}(\lambda) + \text{op}(\mu)}{2}
\end{align*}
for any pair of partitions $\lambda, \mu \vdash n$, where $\text{op}(\lambda)$ denotes the number of odd parts in $\lambda$. Note that this is always an integer, since we must have $\text{op}(\lambda) \equiv \text{op}(\mu) \equiv n \pmod{2}$. This is the key observation which lies behind the method of Seo and Yee.

Letting $\mathcal O$ denote the set of partitions into odd parts. Seo and Yee then compute the difference $o(n) - e(n)$ in terms of the relevant counts
\begin{align*}
	d_j(n) := \# \{ \lambda \in \mathcal O, \lambda \vdash n : \text{op}(\lambda) \equiv j \pmod{4} \}.
\end{align*}
In particular, $o(n) = d_0(n)$ if $n$ is even and $d_2(n)$ if $n$ is odd, and $e(n) = d_3(n)$ if $n$ is even and $d_1(n)$ if $n$ is odd. They are then able to prove the result via standard generating calculations based on the Euler-style generating function
\begin{align*}
	F(z,q) := \sum_{k,n \geq 0} f(k,n) z^k q^n = \lp zq, zq^3; q^4 \rp_\infty^{-1},
\end{align*}
where
\begin{align*}
	f(k,n) := \# \{ \lambda \in \mathcal O : |\lambda| = n, \text{op}(\lambda) = k \}
\end{align*}
In the spirit of Proposition \ref{OrthogonalityProp}, Seo and Yee then calculate, noting that $f(k,n) = 0$ if $k \not \equiv n \pmod{2}$, that
\begin{align*}
	F(i, -iq) &= \sum_{k,n \geq 0} (-1)^n f(k,n) i^{k+n} q^n \\ &= \sum_{2 | k, n} f(k,n) (-1)^{\frac k2} (-1)^{\frac n2} q^n + \sum_{2 \centernot | k,n} (-1)^{\frac{k+1}{2}} (-1)^{\frac{n+1}{2}} f(k,n) q^n \\ &= \sum_{2 | k, n} \lp -1 \rp^{\frac n2} \lp o(n) - e(n) \rp q^n + \sum_{2 \centernot | k,n} (-1)^{\frac{n+1}{2}} \lp o(n) - e(n) \rp q^n \\ &= \sum_{n \geq 0} \lp -1 \rp^{\lceil \frac n2 \rceil} \lp o(n) - e(n) \rp q^n.
\end{align*}
This completes the proof of Theorem \ref{Seo-Yee}, which in turn means that Conjecture \ref{C4 CMM Conjecture} will follow if one can show $a(n) \geq 0$.

\section{Notation and an application of Euler--Maclaurin summation}

This section sets up notation which will be used for the rest of the chapter and states a result which follows from the Euler--Maclaurin asymptotic method, more specifically Proposition \ref{C3 Euler-Maclaurin Sufficient Effective}.

Define the function
\begin{align*}
	B_{r,t}(z) := \dfrac{e^{-\frac rt z}}{z\lp 1 - e^{-z} \rp} = \sum_{n \geq -2} \dfrac{B_{n+2}\lp 1 - \frac rt \rp}{(n+2)!} z^n,
\end{align*}
where $0 < r \leq t$ are integers and $B_n(x)$ are the Bernoulli polynomials defined in \eqref{Bernoulli Polynomial Definition}. Due to Lehmer's bound \eqref{Bernoulli Inequality}, this Laurent expansion is absolutely convergent in the punctured disk $0 < |z| < 2\pi$. This absolute convergence is important for producing effective estimates of certain infinite sums related to $G(q)$, which will be seen in Lemma \ref{B_rt Bound}.

Because $B_{r,t}(z)$ has sufficient decay and has a Laurent series converging in the region $0 < |z| < 2\pi$, Proposition \ref{C3 Euler-Maclaurin Sufficient Effective} can be applied to $B_{r,t}(z)$ for $0 < |z| < 2\pi$. To state this application, we first introduce convenient notation. Let
\begin{align*}
	\beta_{r,t} := \log\left(\Gamma\left(\frac rt\right) \right) - \frac 12\log(2\pi), \hspace{0.5in} g_{r,t}(z) := B_{r,t}(z) - \frac{1}{z^2} - \frac{\lp \frac 12 - \frac rt \rp e^{-\frac rt z}}{z},
\end{align*}
and introduce the functions $F_a^{r,t}(z)$, $E_a^{r,t}(z)$ defined by
\begin{align} \label{B_rt Approx Def}
	F_a^{r,t}(z) := \dfrac{\zeta(2,a)}{z^2} + \dfrac{\beta_{r,t}}{z} - \dfrac{1}{z} \lp \frac 12 - \frac rt \rp \lp \Log(z) + \gamma + \psi(a) \rp + \sum_{n = 0}^\infty c_n^* \dfrac{B_{n+1}(a)}{n+1} z^n
\end{align}
and
\begin{align} \label{B_rt Error Def}
	E^{r,t}(z) := \dfrac{J_{g_{r,t},4}(z)}{720} |z|^3 + \sum_{k \geq 3} \left| \dfrac{B_{k+2}\lp 1 - \frac rt \rp}{(k+2)!} - \dfrac{\lp -r \rp^{k+1} \lp \frac 12 - \frac rt \rp}{t^{k+1} (k+1)!} \right| \lp 1 + \dfrac{k!}{10(k-3)!} \rp |z|^k,
\end{align}
where we define the coefficients $c_n^*$ as in Proposition \ref{C3 Euler-Maclaurin Sufficient Effective} by
\begin{align*}
	c_n^* := \begin{cases} \dfrac{B_{n+1}\lp 1 - \frac rt \rp}{(n+2)!} & \text{ if } n \leq 2, \\ \dfrac{(-r)^{n+1} \lp \frac 12 - \frac rt \rp}{t^{n+1} (n+1)!} & \text{ otherwise}. \end{cases}
\end{align*}
We now state our application of Proposition \ref{C3 Euler-Maclaurin Sufficient Effective} to $B_{r,t}(z)$.

\begin{lemma} \label{B_rt Bound}
	Let $0 < r \leq t$ be integers and $\delta > 0$ a constant. Then for any real number $0 < a \leq 1$ and $z \in D_\delta$ with $0 < |z| < 2\pi$ , we have
	\begin{align*}
		\bigg| \sum_{m \geq 0} B_{r,t}\lp (m+a)z \rp - F_a^{r,t}(z) \bigg| \leq E^{r,t}(z).
	\end{align*}
\end{lemma}

\begin{proof}
	$B_{r,t}(z)$ satisfies the criteria of Proposition \ref{C3 Euler-Maclaurin Sufficient Effective}, and therefore for any $A > 0$ and $N = 3$ we have
	\begin{align*}
		\bigg| \sum_{m \geq 0} B_{r,t}\lp (m+a)z \rp - \dfrac{\zeta(2,a)}{z^2} - \dfrac{I_{B_{r,t},A}^*}{z} &+ \dfrac{c_{-1}}{z}\lp \Log(Az) + \gamma + \psi(a) \rp - \sum_{n = 0}^\infty c_n^* \dfrac{B_{n+1}(a)}{n+1} z^n \bigg| \\ &\leq \dfrac{M_4 J_{g_{r,t},4}}{24} |z|^3 + \sum_{k \geq 3} |b_k| \lp 1 + \dfrac{k!}{10(k-3)!} \rp |z|^k,
	\end{align*}
	where $b_k = \frac{B_{k+2}\lp 1 - \frac rt \rp}{(k+2)!} - \frac{\lp -r \rp^{k+1} \lp \frac 12 - \frac rt \rp}{t^{k+1} (k+1)!}$. To simplify the integral
	\begin{align*}
		I_{B_{r,t},A}^* = \int_{0}^\infty \left(\dfrac{e^{-\frac rt z}}{z\lp 1 - e^{-z} \rp} - \dfrac{1}{z^2} + \lp \dfrac{r}{t} - \dfrac{1}{2} \rp \frac{e^{-Az}}{z} \right) dz,
	\end{align*}
	we use the substitutions $z \mapsto \frac{t}{r} z$ and $A = \frac{r}{t}$, which gives
	\begin{align*}
		I_{B_{r,t},\frac{r}{t}}^* = \int_{0}^\infty \left(\dfrac{e^{-z}}{z \lp 1 - e^{- \frac{t}{r} z} \rp} - \dfrac{1}{\frac{t}{r} z^2} + \lp \dfrac{r}{t} - \dfrac{1}{2} \rp \frac{e^{-z}}{z} \right) dz.
	\end{align*}
	\cite[Lemma 2.3]{BCMO22} states that for any real number $N > 0$,
	\begin{align*}
		\int_0^\infty\bigg(\frac{e^{-x}}{x\left(1-e^{Nx}\right)} &- \frac{1}{Nx^2}+\left(\frac 1N-\frac 12\right)\frac{e^{-x}}{x} \bigg)dx
		\\ &= \log\left(\Gamma\left(\frac 1N\right) \right) +\left(\frac 12-\frac 1N\right) \log\left(\frac 1N\right)-\frac 12\log(2\pi),
	\end{align*}
	and so the case $N = \frac tr$ implies
	\begin{align*}
		I_{B_{r,t},\frac rt}^* = \log\left(\Gamma\left(\frac rt\right) \right) +\left(\frac 12-\frac rt\right) \log\left(\frac rt\right)-\frac 12\log(2\pi) = \beta_{r,t} + \left(\frac 12-\frac rt\right) \log\left(\frac rt\right).
	\end{align*}
	A short calculation therefore shows
	\begin{align*}
		\bigg| \sum_{m \geq 0} B_{r,t}\lp (m+a)z \rp &- \dfrac{\zeta(2,a)}{z^2} - \dfrac{I^*_{B_{r,t},\frac rt}}{z} - \dfrac{1}{z}\lp \dfrac{1}{2} - \dfrac{r}{t} \rp \lp \Log\lp\frac{r}{t} z\rp + \gamma + \psi(a) \rp \\ &- \sum_{n = 0}^\infty c_n^* \dfrac{B_{n+1}(a)}{n+1} z^n \bigg| \leq \dfrac{J_{g_{r,t},4}}{720} |z|^3 + \sum_{k \geq 3} |b_k| \lp 1 + \dfrac{k!}{10(k-3)!} \rp |z|^k.
	\end{align*}
	By the definitions \eqref{B_rt Approx Def} and \eqref{B_rt Error Def}, this completes the proof.
\end{proof}

\section{Asymptotic estimates}

The proof of Theorem \ref{C4 Conjecutre Proof} uses a variation of Wright's circle method. As with any variation of the circle method, there are various stages where estimates must be made. This section collects together the most important estimates, which are subdivided into three groups. The first two are dedicated to proving bounds on $G(q)$ on the {\it major arc} and {\it minor arc}, which play central roles in Wright's circle method and are defined in the first part. The last part considers elementary bounds on the functions $F_a^{r,t}(z)$ and $E_a^{r,t}(z)$ which make later computations more straightforward.

\subsection{Effective Major Arc Bounds}

Before we proceed, we define the terms {\it major arc} and {\it minor arc}. When using Wright's circle method, one must define the {\it major arc}, which is the region of some circle $C$ with fixed radius $|q|$, where $q = e^{-z}$ lies near a dominant pole of the generating function. In most examples, the dominant pole lies near $q=1$ and only one major arc is required. In our case, however, we will require two major arcs, which lie near $q = \pm 1$. The major arc near $q=1$ will consist of those $q = e^{-z}$ for which $z = x + iy$ satisfies $0 \leq |y| < 15x$, and the corresponding constraint near $q = -1$ is $\pi - 15x < |y| \leq \pi$. We will in practice use a change of coordinates $q \mapsto -q$ to translate the $q = -1$ major arc into the $q=1$ major arc of the function $G(-q)$, which gives back the restriction $0 \leq |y| < 15x$. The minor arc will consist of the complement of the two major arcs, that is, it consists of all $q = e^{-z}$ with $15x \leq |y| \leq \pi - 15x$. We begin now by deriving important bounds that hold on major arcs.

\begin{proposition} \label{Major Arc Bound}
	Let $q = e^{-z}$, $z = x + iy$ satisfy $x > 0$ and $0 \leq |y| < 15x$.
	
	\noindent \textnormal{(1)} We have for $0 < x < \frac{2}{5t}$ that
	\begin{align*}
		\left| \Log\lp\lp q^r; q^t \rp_\infty^{-1}\rp - tz F_1^{r,t}(tz) \right| \leq |tz| E^{r,t}(tz).
	\end{align*}
	
	\noindent \textnormal{(2)} We have for $0 < x < \frac{1}{5t}$ that
	\begin{align*}
		\left| \Log\lp\lp -q^r; q^t \rp_\infty^{-1}\rp - tz F_1^{r,t}(2tz) + tz F_{1/2}^{r,t}(2tz) \right| \leq 2 |tz| E^{r,t}(2tz).
	\end{align*}
\end{proposition}

\begin{proof}
	By expanding logarithms into Taylor series, we obtain
	\begin{align} \label{Log Product Expansion}
		\Log\lp\lp \varepsilon q^r; q^t \rp_\infty^{-1}\rp = - \sum_{n \geq 0} \Log\lp 1 - \varepsilon q^{tn+r} \rp = \sum_{m \geq 1} \dfrac{\varepsilon^m q^{rm}}{m\lp 1 - q^{tm} \rp}.
	\end{align}
	Setting $q = e^{-z}$ and multiplying the above expression by $\frac{tz}{tz}$, we obtain
	\begin{align*}
		\Log\lp\lp \varepsilon q^r; q^t \rp_\infty^{-1}\rp = tz \sum_{m \geq 1} \varepsilon^m \dfrac{e^{-rmz}}{tmz\lp 1 - e^{-tmz} \rp} = tz \sum_{m \geq 1} \varepsilon^m B_{r,t}(tmz).
	\end{align*}
	Now, for $0 < x < \frac{2}{5}$ we have since $y^2 < 225x^2$ that $|z| = \sqrt{x^2 + y^2} < \frac{2\sqrt{226}}{5} < 2\pi$. Therefore, the Laurent expansion for $B_{r,t}(tz)$ is convergent for $0 < x < \frac{2}{5t}$, and likewise for $B_{r,t}(2tz)$ if $0 < x < \frac{1}{5t}$. If we set $\varepsilon = 1$, (1) follows directly from Lemma \ref{B_rt Bound}. If $\varepsilon = -1$, by applying Lemma \ref{B_rt Bound} to each summand of
	\begin{align*}
		\Log\lp\lp -q^r; q^t \rp_\infty^{-1}\rp = tz \sum_{m \geq 0} B_{r,t}\lp (m+1) 2tz \rp - tz \sum_{m \geq 0} B_{r,t} \lp \lp m + \frac 12 \rp 2tz \rp,
	\end{align*}
	(2) follows as well.
\end{proof}

\subsection{Effective Minor Arc Bounds}

We now estimate $G(q)$ on the minor arc $15x \leq |y| \leq \pi - 15x$ when $x$ is small. In order to do this, we first prove several helpful results so that the proof of the main bound will be more readable.

\begin{lemma} \label{Alpha Estimations}
	Let $m \geq 1$ be an integer, and $q = e^{-z}$, $z = x + iy$ with $0 < x < \frac{\pi}{480}$ and $15x \leq |y| < \frac{\pi}{2m}$. Then there exists a constant $\alpha_m > 0$ such that
	\begin{align*}
		\dfrac{|q|^m}{m\left| 1 + \lp -1 \rp^{m+1} q^{2m} \right|} - \dfrac{|q|^m}{m\lp 1 - |q|^{2m} \rp} < \dfrac{e^{- \frac{m\pi}{480}}}{2m^2 x} \lp \dfrac{2m}{\alpha_m} - 1 \rp.
	\end{align*}
	Furthermore, in the cases $1 \leq m \leq 3$ we may choose $\alpha_1 = 29$, $\alpha_2 = 55$, and $\alpha_3 = 77$.
\end{lemma}

\begin{proof}
	Since $15x \leq |y| < \frac{\pi}{2m}$, we have $\cos\lp 2my \rp \geq - \cos\lp 30mx \rp$, and so
	\begin{align*}
		\left| 1 + \lp -1 \rp^{m+1} q^{2m} \right|^2 &= 1 - 2 \lp -1 \rp^{m+1} \cos\lp 2my \rp e^{-2mx} + e^{-4mx} \\ &\geq 1 - 2 \cos\lp 30mx \rp e^{-2mx} + e^{-4mx}.
	\end{align*}
	From the Taylor expansion
	\begin{align*}
		1 - 2 \lp -1 \rp^{m+1} \cos\lp 30mx \rp e^{-2mx} + e^{-4mx} = 904 m^2 x^2 - 1808 m^3 x^3 + \cdots,
	\end{align*}
	it is apparent that $1 - 2 \lp -1 \rp^{m+1} \cos\lp 30mx \rp e^{-2mx} + e^{-4mx} > \alpha_m^2 x^2$ for some $\alpha_m > 0$ and $0 < x < \frac{\pi}{480}$. This shows that $\left| 1 + \lp -1 \rp^{m+1} q^{2m} \right| > \alpha_m x$ for all $0 < x < \frac{\pi}{480}$, and so
	\begin{align*}
		\dfrac{|q|^m}{m\left| 1 + \lp -1 \rp^{m+1} q^{2m} \right|} - \dfrac{|q|^m}{m\lp 1 - |q|^{2m} \rp} < \dfrac{|q|^m}{m \alpha_m x} - \dfrac{|q|^m}{m\lp 1 - |q|^{2m} \rp}.
	\end{align*}
	By the inequalities $1 - |q|^{2m} = 1 - e^{-2mx} > 2mx$ and $|q|^m > e^{- \frac{m\pi}{480}}$ for $0 < x < \frac{\pi}{480}$, we arrive at the desired bound.
	
	We now evaluate $\alpha_1, \alpha_2, \alpha_3$. Let $f_m(x) := 1 - 2 \lp -1 \rp^{m+1} \cos\lp 30mx \rp e^{-2mx} + e^{-4mx}$, and consider the auxiliary function $g_m(x) := f_m(x) - \alpha_m^2 x^2$. Note that $g_m(0) = g_m^\prime(0) = 0$ since both $f_m(x)$ and $x^2$ have a double zero at $x=0$. In order to prove that $f_m(x) > \alpha_m^2 x^2$ for $0 < x < \frac{\pi}{480}$, it will therefore suffice to prove that $g_m^{\prime\prime}(0) > 0$, i.e. that $f_m^{\prime\prime}(x) > 2 \alpha_m^2$, for $0 < x < \frac{\pi}{480}$. Now,
	\begin{align*}
		f^{\prime\prime}_m(x) = 16 m^2 e^{-4 m x} \lp 1 + 112 e^{2 m x} \cos\lp 30 m x \rp - 15 e^{2 m x} \sin\lp 30 m x \rp \rp,
	\end{align*}
	and so the $\alpha_m$ we choose must satisfy
	\begin{align*}
		\alpha_m^2 < 8 m^2 e^{-4 m x} \lp 1 + 112 e^{2 m x} \cos\lp 30 m x \rp - 15 e^{2 m x} \sin\lp 30 m x \rp \rp
	\end{align*}
	for all $0 < x < \frac{\pi}{480}$. For each $1 \leq m \leq 3$, $f^{\prime\prime}_m(x)$ is decreasing on the interval $0 < x < \frac{\pi}{480}$, and so it suffices to choose $\alpha_m$ that satisfy
	\begin{align*}
		\alpha_m^2 < 8 m^2 e^{- \frac{m}{120}} \lp 1 + 112 e^{\frac{m}{240}} \cos\lp \frac{m}{16} \rp - 15 e^{\frac{m}{240}} \sin\lp \frac{m}{16} \rp \rp.
	\end{align*}
	For each of the values $1 \leq m \leq 3$, the values $\alpha_1 = 29$, $\alpha_2 = 55$, and $\alpha_3 = 77$ solve the required inequality.
\end{proof}

\begin{lemma} \label{Beta Estimates}
	Let $q = e^{-z}$, $z = x + iy$ with $\frac{3\pi}{4} \leq |y| \leq \pi - 15x$ and $0 < x < \frac{\pi}{480}$. Then we have
	\begin{align*}
		- \dfrac{e^{-2x}}{2\lp 1 - e^{-4x} \rp} + \dfrac{\cos\lp 2y \rp \lp e^{-2x} - e^{-6x} \rp}{2\left| 1 - q^4 \right|^2} < - \dfrac{1}{10x}. 
	\end{align*}
\end{lemma}

\begin{proof}
	We have $3\pi \leq |4y| \leq 4\pi - 60x$, and since $\cos\lp y \rp$ is increasing in the region $3\pi \leq y \leq 4y$ we have $\cos\lp 4y \rp \leq \cos\lp 4\pi - 60x \rp = \cos\lp 60 x \rp$. Therefore, we have
	\begin{align*}
		\left| 1 - q^4 \right|^2 = 1 - 2 \cos\lp 4y \rp e^{-4x} + e^{-8x} \geq 1 - 2 \cos\lp 60x \rp e^{-4x} + e^{-8x}
	\end{align*}
	and thus
	\begin{align*}
		- \dfrac{e^{-2x}}{2\lp 1 - e^{-4x} \rp} &+ \dfrac{\cos\lp 2y \rp \lp e^{-2x} - e^{-6x} \rp}{2\left| 1 - q^4 \right|^2} \\ &\leq - \dfrac{e^{-2x}}{2\lp 1 - e^{-4x} \rp} + \dfrac{e^{-2x} - e^{-6x}}{2\lp 1 - 2 \cos\lp 60x \rp e^{-4x} + e^{-8x} \rp} =: F(x).
	\end{align*}
	Fix any $A > 0$. The inequality $F(x) < - \frac{A}{x}$ is equivalent to
	\begin{align*}
		2 x e^{-6x}\lp 1 - \cos\lp 60x \rp \rp > 2A\lp 1 - e^{-4x} \rp \lp 1 - 2 \cos\lp 60x \rp e^{-4x} + e^{-8x} \rp.
	\end{align*}
	If we set $A = \frac{1}{10}$ and rearrange, this is equivalent to showing that
	\begin{align*}
		2x e^{-6x} + \dfrac{2}{5} e^{-4x} &\cos\lp 60x \rp + \dfrac{1}{5} e^{-4x} + e^{-12x} \\ &> \dfrac{1}{5} + 2x e^{-6x} \cos\lp 60x \rp + \dfrac{2}{5} e^{-8x} \cos\lp 60x \rp + \dfrac{1}{5} e^{-8x}.
	\end{align*}
	By a term-by-term comparison, it suffices to show that $e^{-12x} > \frac{1}{5}$ for $0 < x < \frac{\pi}{480}$, which is true.
\end{proof}

\begin{lemma} \label{PI^2/12 Lemma}
	Let $q = e^{-z}$, $z = x + iy$ with $x>0$ and $15x \leq |y| \leq \pi - 15x$. Then
	\begin{align*}
		\Log\lp \lp |q|; |q|^2 \rp_\infty^{-1} \rp < \dfrac{\pi^2}{12x}.
	\end{align*}
\end{lemma}

\begin{proof}
	We have by expanding series that
	\begin{align*}
		\Log\lp \lp |q|; |q|^2 \rp_\infty^{-1} \rp = \sum_{m \geq 1} \dfrac{e^{-mx}}{m\lp 1 - e^{-2mx} \rp}.
	\end{align*}
	We have $\frac{e^{-x}}{1 - e^{-2x}} < \frac{1}{2x}$; this inequality is equivalent to showing that $2x < e^x - e^{-x}$, which can be proven for all $x > 0$ using elementary calculus. We therefore have by substitutions that
	\begin{align*}
		\dfrac{e^{-mx}}{m\lp 1 - e^{-2mx} \rp} < \dfrac{1}{2m^2 x},
	\end{align*}
	and the result follows by summing over $m$.
\end{proof}

\begin{lemma} \label{Elementary Bounds}
	For $1 \leq m \leq 3$ and $0 < x < \frac{\pi}{480}$, we have
	\begin{align*}
		\dfrac{e^{-mx}}{m\lp 1 - e^{-2mx}\rp} > \dfrac{499}{1000 m^2 x}.
	\end{align*}
\end{lemma}

\begin{proof}
	For any $A>0$, the inequality $\frac{e^{-mx}}{1 - e^{-2mx}} > \frac{A}{mx}$ reduces to $mx > A\lp e^{mx} - e^{-mx} \rp$. The left and right-hand sides have equal values at $x=0$, and so by taking derivatives it would suffice to show that $A\lp e^{mx} + e^{-mx} \rp < 1$ for $0 < x < \frac{\pi}{480}$. The left-hand side is now an increasing function of $x$, and so it suffices to check that the inequality is true for $A = \frac{499}{1000}$, $1 \leq m \leq 3$ and $x = \frac{\pi}{480}$, which holds.
\end{proof}

We may now prove the main minor arc bound on $G(q)$.

\begin{proposition} \label{Minor Arc Bounds}
	Let $q = e^{-z}$ for $z = x + iy$ satisfying $0 < x < \frac{\pi}{480}$ and $15x \leq |y| \leq \pi - 15x$. Then we have
	\begin{align*}
		\left| G\lp q \rp \right| < \exp\lp \dfrac{1}{5 x} \rp.
	\end{align*}
\end{proposition}

\begin{proof}
	By taking exponentials, it suffices to prove that $\mathrm{Re}\lp \Log \lp G(q) \rp \rp < \frac{1}{5x}$. As in the proof of Proposition \ref{Major Arc Bound}, we may use Taylor expansions to show
	\begin{align*}
		\Log\lp G(q) \rp = \Log\lp\lp q; q^4 \rp_\infty^{-1}\rp + \Log\lp\lp -q^3; q^4 \rp_\infty^{-1}\rp = \sum_{m \geq 1} \dfrac{q^m}{m\lp 1 + \lp -1 \rp^{m+1} q^{2m} \rp}.
	\end{align*}
	By taking real parts, we have
	\begin{align} \label{Exact Real Part}
		\mathrm{Re}\lp \Log\lp G(q) \rp \rp = \sum_{m \geq 1} \dfrac{\cos\lp my \rp \lp |q|^m + \lp -1 \rp^{m+1} |q|^{3m} \rp}{m \left| 1 + \lp -1 \rp^{m+1} q^{2m} \right|^2}.
	\end{align}
	Note that since cosine is even, we may assume without loss of generality that $y > 0$. This proof uses the idea of ``splitting off terms" in this series expansion. In particular, we make use of the string of inequalities
	\begin{align} \label{Real-Aboslute Value Inequality}
		\mathrm{Re}\lp \dfrac{q^m}{m\lp 1 + \lp -1 \rp^{m+1} q^{2m} \rp} \rp \leq \dfrac{|q|^m}{m \left| 1 + \lp -1 \rp^{m+1} q^{2m} \right|} \leq \dfrac{|q|^m}{m \lp 1 - |q|^{2m} \rp},
	\end{align}
	in order to bound \eqref{Exact Real Part}. A priori, one may show immediately using Lemma \ref{PI^2/12 Lemma} and \eqref{Real-Aboslute Value Inequality} that $\mathrm{Re}\lp \Log\lp G(q) \rp \rp < \frac{\pi^2}{12x}$, which is insufficient for our purposes.  The idea of splitting terms off is to use \eqref{Real-Aboslute Value Inequality} more carefully to keep track of some of the error introduced in this process, eventually pushing the a priori bound of $\frac{\pi^2}{12x}$ below the required $\frac{1}{5x}$. More specifically, by applying \eqref{Real-Aboslute Value Inequality} we have for any integer $k \geq 0$ (with $k=0$ denoting an empty sum) a corresponding ``splitting bound"
	\begin{align*}
		\mathrm{Re}\lp \Log\lp G(q) \rp \rp &\leq \sum_{m \geq 1} \dfrac{|q|^m}{m\lp 1 - |q|^{2m} \rp} \\ &+ \sum_{m=1}^k \lp \dfrac{\cos\lp my \rp \lp |q|^m + \lp -1 \rp^{m+1} |q|^{3m} \rp}{m \left| 1 + \lp -1 \rp^{m+1} q^{2m} \right|^2} - \dfrac{|q|^m}{m\lp 1 - |q|^{2m} \rp} \rp.
	\end{align*}
	The infinite sum is a sort of main term which we must reduce below $\frac{1}{5x}$ by means of the finite sum. By Lemma \ref{PI^2/12 Lemma} we have
	\begin{align*}
		\Log\lp \lp |q|; |q|^2 \rp_\infty^{-1} \rp = \sum_{m \geq 1} \frac{|q|^m}{m\lp 1 - |q|^{2m} \rp} < \dfrac{\pi^2}{12x},
	\end{align*}
	and therefore
	\begin{align} \label{Main Splitting Bound}
		\mathrm{Re}\lp \Log\lp G(q) \rp \rp < \dfrac{\pi^2}{12x} &+ \sum_{m=1}^k \lp \dfrac{\cos\lp my \rp \lp e^{-mx} + \lp -1 \rp^{m+1} e^{-3mx} \rp}{m \left| 1 + \lp -1 \rp^{m+1} q^{2m} \right|^2} - \dfrac{e^{-mx}}{m\lp 1 - e^{-2mx} \rp} \rp.
	\end{align}
	Note that if the conditions of Lemma \ref{Alpha Estimations} are satisfied, then comparison between the first and second terms in \eqref{Real-Aboslute Value Inequality} implies that for any $k \geq \ell \geq 0$ we have
	\begin{align} \label{Secondary Splitting Bound}
		\mathrm{Re}\lp \Log\lp G(q) \rp \rp &< \dfrac{\pi^2}{12x} + \sum_{m=1}^\ell \dfrac{e^{- \frac{m\pi}{480}}}{2m^2 x} \lp \dfrac{2m}{\alpha_m} - 1 \rp \notag \\ &+ \sum_{m=\ell+1}^k \lp \dfrac{\cos\lp my \rp \lp e^{-mx} + \lp -1 \rp^{m+1} e^{-3mx} \rp}{m \left| 1 + \lp -1 \rp^{m+1} q^{2m} \right|^2} - \dfrac{e^{-mx}}{m\lp 1 - e^{-2mx} \rp} \rp.
	\end{align}
	
	Our objective now is to prove that the right-hand side of either \eqref{Main Splitting Bound} or \eqref{Secondary Splitting Bound} is bounded above by $\frac{1}{5x}$ for all $0 < x < \frac{\pi}{480}$ and all $15x \leq y \leq \pi - 15x$. This will not be done all at once, but in stages. In the first stage of the proof, we will split the interval $\frac{\pi}{2} \leq y \leq \pi - 15x$ into several subintervals. On each subinterval, some version of \eqref{Main Splitting Bound} will be sufficient to prove the desired inequality. After this is completed, we will be able to apply the $\ell = 1$ case of \eqref{Secondary Splitting Bound}. We will use this case to prove the result in the range $\frac{\pi}{4} \leq y < \frac{\pi}{2}$. We then use the case $\ell = 2$ of \eqref{Secondary Splitting Bound} to cover the range $\frac{\pi}{6} \leq y < \frac{\pi}{4}$, and finally we will use the case $\ell = 3$ of \eqref{Secondary Splitting Bound} to cover the range $15x \leq y < \frac{\pi}{6}$. All of these cases together prove the desired result in the full range $15x \leq y \leq \pi - 15x$. We begin now with the application of \eqref{Main Splitting Bound} to the interval $\frac{\pi}{2} \leq y \leq \pi - 15x$. \\
	
	Suppose $\frac{5\pi}{6} \leq y \leq \pi - 15x$. Because $\cos\lp y \rp, \cos\lp 3y \rp \leq 0$ in this range, we have using the $k = 3$ case of \eqref{Main Splitting Bound} that
	\begin{align*}
		\mathrm{Re}\lp \Log\lp G(q) \rp \rp < \dfrac{\pi^2}{12x} - \dfrac{e^{-x}}{1 - e^{-2x}} &+ \dfrac{\cos\lp 2y \rp \lp e^{-2x} - e^{-6x} \rp}{2\lp 1 - \cos\lp 4y \rp e^{-4x} + e^{-8x} \rp} \\ &- \dfrac{e^{-2x}}{2\lp 1 - e^{-4x} \rp} - \dfrac{e^{-3x}}{3\lp 1 - e^{-6x} \rp}.
	\end{align*}
	By Lemmas \ref{Beta Estimates} and \ref{Elementary Bounds}, we therefore have
	\begin{align*}
		\mathrm{Re}\lp \Log\lp G(q) \rp \rp &< \dfrac{\pi^2}{12x} - \dfrac{1}{10x} - \dfrac{e^{-x}}{1 - e^{-2x}} - \dfrac{e^{-3x}}{3\lp 1 - e^{-6x} \rp} \\ &< \lp \dfrac{\pi^2}{12} - \dfrac{1}{10} - \dfrac{499}{1000} \lp 1 + \dfrac{1}{9} \rp \rp \dfrac{1}{x},
	\end{align*}
	for all $0 < x < \frac{\pi}{480}$, which establishes $\mathrm{Re}\lp \Log\lp G(q) \rp \rp < \frac{1}{5x}$ in this region.
	
	We now consider the region $\frac{3\pi}{4} \leq y < \frac{5\pi}{6}$. In this region we have $\cos\lp y \rp \leq 0$, and so by the $k = 2$ variant of \eqref{Main Splitting Bound} we have
	\begin{align*}
		\mathrm{Re}\lp \Log\lp G(q) \rp \rp < \dfrac{\pi^2}{12x} - \dfrac{e^{-x}}{1 - e^{-2x}} + \dfrac{\cos\lp 2y \rp \lp e^{-2x} - e^{-6x} \rp}{2\lp 1 - 2 \cos\lp 4y \rp e^{-4x} + e^{-8x} \rp} - \dfrac{e^{-2x}}{2\lp 1 - e^{-4x} \rp}.
	\end{align*}
	By considering partial derivatives of the numerator and denominator separately, we can see that in the region $\frac{3\pi}{4} \leq y < \frac{5\pi}{6}$ the fraction $\frac{\cos\lp 2y \rp \lp e^{-2x} - e^{-6x} \rp}{2\lp 1 - \cos\lp 4y \rp e^{-4x} + e^{-8x} \rp}$ is an increasing function of $y$, and therefore we have in this region by applying Lemma \ref{Elementary Bounds} that
	\begin{align*}
		\mathrm{Re}\lp \Log\lp G(q) \rp \rp &< \dfrac{\pi^2}{12x} - \dfrac{e^{-x}}{1 - e^{-2x}} - \dfrac{e^{-2x}}{2\lp 1 - e^{-4x} \rp} + \dfrac{e^{-2x} - e^{-6x}}{4\lp 1 - e^{-4x} + e^{-8x} \rp} \\ &< \lp \dfrac{\pi^2}{12} - \dfrac{499}{1000} \lp 1 + \dfrac{1}{4} \rp \rp \dfrac{1}{x} + \dfrac{e^{-2x} - e^{-6x}}{4\lp 1 - e^{-4x} + e^{-8x} \rp}.
	\end{align*}
	It is clear that the term $\frac{e^{-2x} - e^{-6x}}{4\lp 1 - e^{-4x} + e^{-8x} \rp}$ is extremely small in $0 < x < \frac{\pi}{480}$. In particular, it is straightforward to show that this quantity is less than $\frac{7}{1000}$ for $0 < x < \frac{\pi}{480}$. It follows that
	\begin{align*}
		\lp \dfrac{\pi^2}{12} - \dfrac{499}{1000} \lp 1 + \dfrac{1}{4} \rp \rp \dfrac{1}{x} &+ \dfrac{e^{-2x} - e^{-6x}}{4\lp 1 - e^{-4x} + e^{-8x} \rp} \\ &< \lp \dfrac{\pi^2}{12} - \dfrac{499}{1000} \lp 1 + \dfrac{1}{4} \rp \rp \dfrac{1}{x} + \dfrac{7}{1000} < \dfrac{1}{5x}
	\end{align*}
	for $0 < x < \frac{\pi}{480}$, and therefore $\mathrm{Re}\lp \Log\lp G(q) \rp \rp < \frac{1}{5x}$ for $0 < x < \frac{\pi}{480}$ and $\frac{3\pi}{4} \leq y < \frac{5\pi}{6}$.
	
	Consider now the range $\frac{\pi}{2} \leq y < \frac{3\pi}{4}$. Here, we have $\cos\lp y \rp, \cos\lp 2y \rp \leq 0$ and therefore by the $k = 2$ case of \eqref{Main Splitting Bound} and Lemma \ref{Elementary Bounds} we obtain
	\begin{align*}
		\mathrm{Re}\lp \Log\lp G(q) \rp \rp < \dfrac{\pi^2}{12x} - \dfrac{e^{-x}}{1 - e^{-2x}} - \dfrac{e^{-2x}}{2\lp 1 - e^{-4x} \rp} < \lp \dfrac{\pi^2}{12} - \dfrac{49}{100}\lp 1 + \dfrac{1}{4} \rp \rp \dfrac{1}{x}
	\end{align*}
	As in the previous case, this establishes $\mathrm{Re}\lp \Log\lp G(q) \rp \rp < \frac{1}{5x}$ for all $0 < x < \frac{\pi}{480}$ and, by taking together all previous cases as well as this one, all $\frac{\pi}{2} \leq y < \pi$.
	
	Note that we are reduced to the region $15x \leq y < \frac{\pi}{2}$, and so we may invoke the case $\ell = 1$ of \eqref{Secondary Splitting Bound}. Consider the range $\frac{\pi}{4} \leq y < \frac{\pi}{2}$. By the $\ell=1$, $k=3$ case of \eqref{Secondary Splitting Bound} along with Lemma \ref{Elementary Bounds} and the fact that $\cos\lp 2y \rp, \cos\lp 3y \rp \leq 0$ in this region, we obtain
	\begin{align*}
		\mathrm{Re}\lp \Log\lp G(q) \rp \rp &< \dfrac{\pi^2}{12x} - \dfrac{27 e^{-\frac{\pi}{480}}}{58x} - \dfrac{e^{-2x}}{2\lp 1 - e^{-4x} \rp} - \dfrac{e^{-3x}}{3\lp 1 - e^{-6x} \rp} \\ &< \lp \dfrac{\pi^2}{12} - \dfrac{27 e^{-\frac{\pi}{480}}}{58} - \dfrac{499}{1000}\lp \dfrac{1}{4} + \dfrac 19 \rp \rp \dfrac{1}{x},
	\end{align*}
	which is less than $\frac{1}{5x}$, so the desired result is proven in the region $\frac{\pi}{4} \leq y < \frac{\pi}{2}$.
	
	We now consider the range $\frac{\pi}{6} \leq y < \frac{\pi}{4}$, within which the $\ell = 2$ case of \eqref{Secondary Splitting Bound} applies by Lemma \ref{Alpha Estimations}. By \eqref{Secondary Splitting Bound} with $\ell = 2$ and $k = 3$, we have
	\begin{align*}
		\mathrm{Re}\lp \Log\lp G(q) \rp \rp < \dfrac{\pi^2}{12x} &- \dfrac{27 e^{-\frac{\pi}{480}}}{58x} - \dfrac{51 e^{- \frac{\pi}{240}}}{440x} + \dfrac{\cos\lp 3y \rp \lp e^{-3x} + e^{-9x} \rp}{3 \left| 1 + q^6 \right|^2} - \dfrac{e^{-3x}}{3\lp 1 - e^{-6x} \rp}.
	\end{align*}
	Since in this range we have $\cos\lp 3y \rp \leq 0$, we have
	\begin{align*}
		\mathrm{Re}\lp \Log\lp G(q) \rp \rp < \dfrac{\pi^2}{12x} &- \dfrac{27 e^{-\frac{\pi}{480}}}{58x} - \dfrac{51 e^{- \frac{\pi}{240}}}{440x} - \dfrac{e^{-3x}}{3\lp 1 - e^{-6x} \rp},
	\end{align*}
	which is as in earlier cases yields the desired result for $0 < x < \frac{\pi}{480}$ by Lemma \ref{Elementary Bounds}.
	
	Finally, consider the interval $0 < 15x \leq y < \frac{\pi}{6}$. We may use case $\ell = k = 3$ of \eqref{Secondary Splitting Bound}, which implies
	\begin{align*}
		\mathrm{Re}\lp \Log\lp G(q) \rp \rp < \dfrac{\pi^2}{12x} - \dfrac{21 e^{-\frac{\pi}{480}}}{58x} - \dfrac{51 e^{-\frac{\pi}{240}}}{440x} - \dfrac{71 e^{-\frac{\pi}{160}}}{1386x}.
	\end{align*}
	for $0 < x < \frac{\pi}{480}$. The right-hand side above is always less than $\frac{1}{5x}$, and this completes the proof in the region $15x \leq y < \frac{\pi}{6}$. This completes the proof of the proposition.
\end{proof}

\subsection{Bounds on $F_a^{r,t}(z)$ and $E^{r,t}(z)$}

We will need the following effective estimates of the functions $F_a^{r,t}(z)$ and $E^{r,t}(z)$ which appear in Lemma \ref{B_rt Bound}.

\begin{lemma} \label{E-Bounds}
	Let $0 < a \leq 1$ be a real number and $z = x + iy$ any complex number satisfying $|z| < 1$ and $0 \leq |y| < 15x$. Then we have
	\begin{align*}
		E^{1,4}(z) < 28 |z|^3
	\end{align*}
	and
	\begin{align*}
		E^{3,4}(z) < 56 |z|^3.
	\end{align*}
\end{lemma}

\begin{proof}
	Recall that
	\begin{align*}
		E^{r,t}(z) =  \dfrac{J_{g_{r,t},4}(z)}{720} |z|^3 + \sum_{k \geq 3} \left| \dfrac{B_{k+2}\lp 1 - \frac rt \rp}{(k+2)!} - \dfrac{\lp -r \rp^{k+1} \lp \frac 12 - \frac rt \rp}{t^{k+1} (k+1)!} \right| \lp 1 + \dfrac{k}{10(k-3)!} \rp |z|^k,
	\end{align*}
	where
	\begin{align*}
		c_n^* = \begin{cases} \dfrac{B_{n+1}\lp 1 - \frac rt \rp}{(n+2)!} & \text{ if } n \leq 2, \\ \dfrac{(-r)^{n+1} \lp \frac 12 - \frac rt \rp}{t^{n+1} (n+1)!} & \text{ otherwise}. \end{cases}
	\end{align*}
	We first consider the two integrals $J_{g_{1,4},4}(z)$ and $J_{g_{3,4},4}(z)$, which we recall are taken over a path of integration going through the origin and $z$. We bound these integrals by splitting them each into upper and lower parts, taking advantage of the decay properties of $g_{r,t}(x)$ in the upper parts and power series expansions in the lower parts. For both cases $r = 1,3$, we have
	\begin{align*}
		g^{(4)}_{r,t}(w) = \dfrac{e^{-\frac{rw}{4}}}{1024\lp e^w - 1 \rp^5 w^6} \lp \sum_{j=0}^5 e^{jw} p_{r,t,j}(w) + \sum_{j=0}^5 \tilde{c}_j e^{\frac{4j+r}{4} w} \rp
	\end{align*}
	for certain constants $\tilde{c}_j$ and degree 5 polynomials $p_{r,t,j}(w)$. For $\alpha = \frac{3\pi}{2} \frac{z}{|z|}$, applying the triangle inequality and the major arc bounds $\mathrm{Re}(w) \leq |w| \leq \sqrt{226}\mathrm{Re}(w)$ imply upper bounds on $\left| g^{(4)}_{r,4}(w) \right|$ that depend only on $u = \mathrm{Re}(w)$. Using these upper bounds, we can conclude that
	\begin{align*}
		\int_\alpha^\infty \left| g^{(4)}_{1,4}(w) \right| |dw| < 19900, \ \ \ \text{ and } \ \ \ \int_\alpha^\infty \left| g^{(4)}_{3,4}(w) \right| |dw| < 39900.
	\end{align*}
	To bound the remainder of the integrals $J_{g_{r,4},4}(z)$, we use the power series representations of $g_{r,4}^{(4)}(w)$, namely
	\begin{align*}
		g_{1,4}^{(4)}(w) = \sum_{n \geq 0} \dfrac{(n+4)!}{n!} \lp \dfrac{B_{n+6}\lp \frac 34 \rp}{(n+6)!} + \dfrac{(-1)^n}{4^{n+6} (n+5)!} \rp w^n
	\end{align*}
	and
	\begin{align*}
		g_{3,4}^{(4)}(w) = \sum_{n \geq 0} \dfrac{(n+4)!}{n!} \lp \dfrac{B_{n+6}\lp \frac 14 \rp}{(n+6)!} + \dfrac{(-1)^{n+1} 3^{n+5}}{4^{n+6}} \rp w^n.
	\end{align*}
	By applying \eqref{Bernoulli Inequality}, $\zeta(n+6) \leq \frac{\pi^6}{945}$, $|w| < \frac{3\pi}{2}$ and other elementary estimates, we have
	\begin{align*}
		\left| g^{(4)}_{1,4}(w) \right| \leq \sum_{n \geq 0} \dfrac{(n+4)!}{n!} \lp \dfrac{2\pi^6}{945 (2\pi)^{n+6}} + \dfrac{1}{4^{n+6} (n+5)!} \rp \lp \dfrac{3\pi}{2} \rp^n < 1
	\end{align*}
	and
	\begin{align*}
		\left| g^{(4)}_{3,4}(w) \right| \leq \sum_{n \geq 0} \dfrac{(n+4)!}{n!} \lp \dfrac{2\pi^6}{945 (2\pi)^{n+6}} + \dfrac{3^{n+5}}{4^{n+6} (n+5)!} \rp \lp \dfrac{3\pi}{2} \rp^n < 2.
	\end{align*}
	Therefore, for $\alpha = \frac{3\pi}{2} \frac{z}{|z|}$, we have
	\begin{align*}
		J_{g_{1,4},4}(z) = \int_0^\alpha \left| g^{(4)}_{1,4}(w) \right| |dw| + \int_\alpha^\infty \left| g^{(4)}_{1,4}(w) \right| |dw| < 20000
	\end{align*}
	and
	\begin{align*}
		J_{g_{3,4},4}(z) = \int_0^\alpha \left| g^{(4)}_{3,4}(w) \right| |dw| + \int_\alpha^\infty \left| g^{(4)}_{3,4}(w) \right| |dw| < 40000.
	\end{align*}
	We now bound the other summand of $E^{r,4}(z)$ in the cases $r = 1,3$. Using \eqref{Bernoulli Inequality}, along with $\zeta(n) \leq \frac{\pi^2}{6}$ for $n \geq 2$ and $|z| < 1$, we have
	\begin{align*}
		\sum_{k \geq 3} \left| \dfrac{B_{k+2}\lp 1 - \frac 14 \rp}{(k+2)!} - \dfrac{(-1)^{n+1}}{4^{k+2} (k+1)!} \right| \lp 1 + \dfrac{k!}{10(k-3)!} \rp |z|^k < \dfrac{|z|^3}{1000}
	\end{align*}
	and
	\begin{align*}
		\sum_{k \geq 3} \left| \dfrac{B_{k+2}\lp 1 - \frac 34 \rp}{(k+2)!} - \dfrac{(-3)^{k+1}}{4^{k+2} (k+1)!} \right| \lp 1 + \dfrac{k!}{10(k-3)!} \rp |z|^k < \dfrac{|z|^3}{100}.
	\end{align*}
	Therefore, we find that
	\begin{align*}
		\left| E^{1,4}(z) \right| < \dfrac{20000}{720} |z|^3 + \dfrac{1}{1000} |z|^3 < 28 |z|^3
	\end{align*}
	and
	\begin{align*}
		\left| E^{3,4}(z) \right| < \dfrac{40000}{720} |z|^3 + \dfrac{1}{100} |z|^3 < 56 |z|^3,
	\end{align*}
	which completes the proof.
\end{proof}

We now estimate a certain combination of the functions $F_a^{r,t}(z)$ in a similar manner. Define the functions $G_1^*(q), G_2^*(q)$ respectively by
\begin{align*}
	G_1^*(q) := \exp\lp \dfrac{\pi^2}{48z} - \dfrac{1}{4} \Log\lp z \rp + \beta_{1,4} - \dfrac{\log(2)}{4} - \dfrac{z}{24} \rp
\end{align*}
and
\begin{align*}
	G_2^*(-q) := \exp\lp \dfrac{\pi^2}{48z} + \dfrac{1}{4} \Log\lp z \rp + \beta_{3,4} + \dfrac{\log(2)}{4} - \dfrac{z}{24} \rp.
\end{align*}
These will be useful in estimating $G(q)$ along the two major arcs in the circle method.

\begin{lemma} \label{F-Bounds}
	Let $q = e^{-z}$, $z = x + iy$ satisfy $0 < x < \frac{\pi}{480}$ and $0 \leq |y| < 15x$.
	\begin{enumerate}
		\item We have
		\begin{align*}
			\left| 4z F_1^{1,4}(4z) + 4z F_1^{3,4}(8z) - 4z F_{1/2}^{3,4}(8z) - \Log\lp G_1^*(q) \rp \right| \leq \dfrac{|z|^4}{2}.
		\end{align*}
		\item We have
		\begin{align*}
			\left| 4z F_1^{3,4}(4z) + 4z F_1^{1,4}(8z) - 4z F_{1/2}^{1,4}(8z) - \Log\lp G_2^*(-q) \rp \right| \leq \dfrac{|z|^4}{2}.
		\end{align*}
	\end{enumerate}
\end{lemma}

\begin{proof}
	Define the functions
	\begin{align*}
		F_1(z) := 4z \lp F_1^{1,4}(4z) + F_1^{3,4}(8z) - F_{1/2}^{3,4}(8z) \rp
	\end{align*}
	and
	\begin{align*}
		F_2(z) := 4z\lp F_1^{3,4}(4z) + F_1^{1,4}(8z) - F_{1/2}^{1,4}(8z) \rp.
	\end{align*}
	By expanding each of the terms $F_a^{r,t}(z)$, we have series expansions
	\begin{align*}
		F_1(z) - \Log\lp G_1^*(q) \rp = \sum_{n \geq 3} \alpha_{n+1} z^{n+1}
	\end{align*}
	and
	\begin{align*}
		F_2(z) - \Log\lp G_2^*(-q) \rp = \sum_{n \geq 3} \alpha^\prime_{n+1} z^{n+1}
	\end{align*}
	where
	\begin{align*}
		\alpha_{n+1} = (-1)^{n+1} \dfrac{B_{n+1}(1) + 3 \cdot 6^n \lp B_{n+1}\lp \frac 12 \rp - B_{n+1}(1) \rp}{4 (n+1) \cdot (n+1)!}
	\end{align*}
	and
	\begin{align*}
		\alpha_{n+1}^\prime = \lp -1 \rp^{n+1} \dfrac{- 3^{n+1} B_{n+1}(1) + 2^n\lp B_{n+1}(1) - B_{n+1}\lp \frac 12 \rp \rp}{4 (n+1) (n+1)!}.
	\end{align*}
	Now, by \eqref{Bernoulli Inequality}, we have for $n \geq 2$ that $M_n = \max\limits_{0 \leq x \leq 1} \left| B_n(x) \right| \leq \frac{2 \zeta(n) n!}{\lp 2\pi \rp^n}$,  we have for $n \geq 1$ the bounds
	\begin{align*}
		\left| \alpha_{n+1} \right| &\leq \dfrac{\left| B_{n+1}(1) \right| + 3 \cdot 6^n \lp \left| B_{n+1}(1) \right| + \left| B_{n+1}\lp \frac 12 \rp \right| \rp}{4 (n+1) \cdot (n+1)!} \\ &\leq \dfrac{(1 + 6^{n+1}) M_{n+1}}{4 (n+1) \cdot (n+1)!} < \dfrac{\pi^2}{12(n+1)} \lp \dfrac{1}{\lp 2\pi \rp^{n+1}} + \lp \dfrac{3}{\pi} \rp^{n+1} \rp,
	\end{align*}
	and likewise
	\begin{align*}
		\left| \alpha^\prime_{n+1} \right| \leq \dfrac{\lp 3^{n+1} + 2^{n+1} \rp M_{n+1}}{4 (n+1) \cdot (n+1)!} < \dfrac{\pi^2}{12(n+1)} \lp \dfrac{3}{\pi} \rp^{n+1}.
	\end{align*}
	Therefore, noting that on the major arc $0 \leq |y| < 15x$ with $0 < x < \frac{\pi}{480}$ we have $|z| < \frac{\sqrt{226} \pi}{480}$, we have
	\begin{align*}
		\left| F_1(z) - \Log\lp G_1^*(q) \rp \right| \leq \sum_{n \geq 3} \left| \alpha_{n+1} \right| |z|^{n+1} < \dfrac{|z|^4}{2} 
	\end{align*}
	and likewise
	\begin{align*}
		\left| F_2(z) - \Log\lp G_2^*(-q) \rp \right| \leq \sum_{n \geq 3} \left| \alpha^\prime_{n+1} \right| |z|^{n+1} < \dfrac{|z|^4}{2}.
	\end{align*}
	This completes the proof.
\end{proof}

We now use the bounds so far derived to give an estimate for $G(q)$ on arcs near $q = \pm 1$. For this final lemma, we require some new notation. For any complex-valued function $f(z)$ and any real-valued function $g(z)$, we shall say that $f(z) = O_{\leq}\lp g(z) \rp$ if $\left| f(z) \right| \leq g(z)$ for all $z$ in a specified region (which will always be clear from context).

\begin{lemma} \label{Major Error Bound}
	Let $q = e^{-z}$, $z = x + iy$, $0 < x < \frac{\pi}{480}$ and $0 \leq |y| < 15x$.
	\begin{enumerate}
		\item We have $\Log\lp G(q) \rp = \Log\lp G_1^*(q) \rp + E_+(q)$ where $E_+(q) = O_{\leq}\lp 4033 |z|^4 \rp$.
		\item We have $\Log\lp G(-q) \rp = \Log\lp G_2^*(-q) \rp + E_-(q)$ where $E_-(q) = O_{\leq}\lp 2689 |z|^4 \rp$.
	\end{enumerate}
\end{lemma}

\begin{proof}
	Let $F_1(z), F_2(z)$ be defined as in the proof of Lemma \ref{F-Bounds}. Then by Proposition \ref{Major Arc Bound} and Lemma \ref{E-Bounds},  we have
	\begin{align*}
		\left| \Log\lp G(q) \rp - F_1(z) \right| < 4|z| E^{1,4}(4z) + 8|z| E^{3,4}(8z) < 4032 |z|^4
	\end{align*}
	and similarly
	\begin{align*}
		\left| \Log\lp G(-q) \rp - F_2(z) \right| < 2688 |z|^4
	\end{align*}
	in the relevant domain. We therefore obtain by Lemma \ref{F-Bounds} (1) and (2) that
	\begin{align*}
		\left| \Log\lp G(q) \rp - \Log\lp G_1^*(q) \rp \right| < 4033 |z|^4
	\end{align*}
	and
	\begin{align*}
		\left| \Log\lp G(-q) \rp - \Log\lp G_2^*(-q) \rp \right| < 2689 |z|^4.
	\end{align*}
	This proves the result.
\end{proof}

\section{Proof of Theorem \ref{C4 a(n) Asymptotics}}

We now proceed to the proof of Theorem \ref{C4 Conjecutre Proof} (Theorem \ref{C4 a(n) Asymptotics} will be proven along the way), which relies on a variation of Wright's circle method. We set $q = e^{-z}$ with $\mathrm{Re}\lp z \rp = x$ and $\mathrm{Im}\lp z \rp = y$. Since $G(q)$ has no poles inside the unit disk, we have by Cauchy's theorem that
\begin{align*}
	a(n) = \dfrac{1}{2\pi i} \int_C \dfrac{G(q)}{q^{n+1}} dq,
\end{align*}
where $C$ is the circle oriented counterclockwise centered at 0 with radius $|q| = e^{-x}$. We choose $C$ so that $x = \frac{\pi}{\sqrt{48n}}$. Impose the constraint $0 < x < \frac{\pi}{480}$ throughout, which by algebraic manipulations is equivalent to the assumption that $n > 4800$. Define the three arcs 
\begin{center}
	\begin{align*}
		C_1 &:= \{ q = e^{-z} \in C : 0 \leq |y| < 15 x \}, \\ C_2 &:= \{ q = e^{-z} \in C : \pi - 15x \leq |y| < \pi \},
	\end{align*}
\end{center}
and
\begin{align*}
	\widetilde{C} := \{ q = e^{-z} \in C : 15x \leq |y| \leq \pi - 15x \}.
\end{align*}
We may decompose $a(n)$ as
\begin{align*}
	a(n) =  J^*_1(n) + J^*_2(n) + J^{\text{maj}}_1(n) + J^{\text{maj}}_2(n) + J^{\text{min}}(n),
\end{align*}
where for $k = 1, 2$ we define
define
\begin{align*}
	\widetilde{G}_1(q) &:= \exp\lp \dfrac{\pi^2}{48z} - \dfrac{1}{4} \Log\lp z \rp + \beta_{1,4} - \dfrac{\log(2)}{4} \rp, \\ \widetilde{G}_2(-q) &:= \exp\lp \dfrac{\pi^2}{48z} + \dfrac{1}{4} \Log\lp z \rp + \beta_{3,4} + \dfrac{\log(2)}{4} \rp,
\end{align*}
and
\begin{align*}
	J^*_k(n) &:= \dfrac{1}{2\pi i} \int_{C_k} \dfrac{\widetilde{G}_k(q)}{q^{n+1}} dq, \\ J^{\text{maj}}_k(n) &:= \dfrac{1}{2\pi i} \int_{C_k} \dfrac{G(q) - \widetilde{G}_k(q)}{q^{n+1}} dq, \\ J^{\text{min}}(n) &:= \dfrac{1}{2\pi i} \int_{\widetilde C} \dfrac{G(q)}{q^{n+1}} dq.
\end{align*}
$J_1^*(n)$ and $J_2^*(n)$ are the dominant terms, and so we begin with an analysis of the error terms.

\subsection{Error Bound for $J^{\text{min}}(n)$}

By Proposition \ref{Minor Arc Bounds}, we have on all $\tilde C$ that $\left| G(q) \right| < \exp\lp \frac{1}{5x} \rp$. Since the length of $\widetilde C$ is less than $2\pi$ and $\left| \int_{\widetilde{C}} q^{-1} dq \right| < 2\pi |q|^{-1} = 2 \pi e^{\frac{\pi}{480}} < \frac{21\pi}{10}$, it follows that
\begin{align*}
	\left| J^{\text{min}}(n) \right| = \left| \dfrac{1}{2\pi i} \int_{\widetilde C} \dfrac{G(q)}{q^{n+1}} dq \right| < \frac{21\pi}{10} \exp\lp nx + \dfrac{1}{5x} \rp = \frac{21\pi}{10} \exp\lp \lp \dfrac{\pi}{4\sqrt{3}} + \dfrac{4\sqrt{3}}{5\pi} \rp \sqrt{n} \rp.
\end{align*}

\subsection{Error Bounds for $J_1^{\text{maj}}(n)$ and $J_2^{\text{maj}}(n)$}

We now consider $J^{\text{maj}}_1(n)$ and $J^{\text{maj}}_2(n)$. For $J_1^{\text{maj}}(n)$, we may assume now that $0 \leq |y| < 15x$. Since we have $x \leq |z| \leq \sqrt{226} x$,
\begin{align*}
	\left| \widetilde{G}_1(q) \right| = \dfrac{\Gamma\lp \frac 14 \rp}{2^{\frac 34} \pi^{\frac 12}} |z|^{- \frac 14} \exp\lp \dfrac{\pi^2x}{48|z|^2} \rp \leq \dfrac{\Gamma\lp \frac 14 \rp}{2^{\frac 34} \pi^{\frac 12}} x^{- \frac 14} \exp\lp \dfrac{\pi^2}{48x} \rp
\end{align*}
and
\begin{align*}
	\left| \widetilde{G}_2(-q) \right| = \dfrac{\Gamma\lp \frac 34 \rp}{2^{\frac 14} \pi^{\frac 12}} |z|^{\frac 14} \exp\lp \dfrac{\pi^2x}{48|z|^2} \rp \leq \dfrac{226^{\frac 18} \Gamma\lp \frac 34 \rp}{2^{\frac 14} \pi^{\frac 12}} x^{\frac 14} \exp\lp \dfrac{\pi^2}{48x} \rp.
\end{align*}
Similarly, we have
\begin{align*}
	\left|G_1^*(q)\right| = \left|\widetilde{G}_1(q)\right| \cdot \left|\exp\lp - \dfrac{z}{24} \rp \right| = \left|\widetilde{G}_1(q)\right| \cdot \exp\lp - \dfrac{x}{24} \rp < \left| \widetilde{G}_1(q) \right|
\end{align*}
and $\left|G_2^*(-q)\right| < \left|\widetilde{G}_2(-q)\right|$. By Lemma \ref{Major Error Bound}, we have
\begin{align*}
	\Log\lp G(q) \rp - \Log\lp G_1^*(q) \rp = E_+(q) = O_{\leq}\lp 4033 |z|^4 \rp,
\end{align*}
and therefore by exponentiation $G(q) = G_1^*(q) \exp\lp E_+(q) \rp$. Now, since $|z| < \frac{\sqrt{226} \pi}{480}$ on $C_1$, we have $\left| E_+(q) \right| < 4033 |z|^4 < 0.38$. Because $\exp\lp t \rp = 1 + O_{\leq}\lp 2 t \rp$ for $0 < t < 0.76$, we have
\begin{align*}
	\left| \exp\lp E_+(q) \rp - 1 \right| < \dfrac{3}{2} \left| E_+(q) \right| < 6050 |z|^4.
\end{align*}
In particular, this implies
\begin{align*}
	\left| G(q) - G_1^*(q) \right| = \left| G_1^*(q) \right| \cdot \left| \exp\lp E_+(q) \rp - 1 \right| < \dfrac{9833929}{n^{\frac{15}{8}}} \exp\lp \dfrac{\pi}{4} \sqrt{\dfrac{n}{3}} \rp.
\end{align*}
We now make a similar estimate for $G_1^*(q) - \widetilde{G}_1(q)$. It is clear from definitions that 
\begin{align*}
	\Log\lp G_1^*(q) \rp - \Log\lp \widetilde{G}_1(q) \rp = - \frac{z}{24} = O_{\leq}\lp \frac{1}{24} |z| \rp,
\end{align*}
and so reasoning as earlier, we may write $G_1^*(q) = \widetilde{G}_1(q) \exp\lp -\frac{z}{24} \rp$. We have $\exp\lp t \rp = 1 + O_{\leq}\lp \frac{12}{11} t \rp$ for $0 < t < 0.005$, so since $\frac{|z|}{24} < 0.005$ we have $\exp\lp -\frac{z}{24} \rp = 1 + O_{\leq}\lp \frac{1}{22} |z| \rp$. Therefore, we find that
\begin{align*}
	\left|G_1^*(q) - \widetilde{G}_1(q)\right| \leq \dfrac{|z|\Gamma\lp \frac 14 \rp}{22 \cdot 2^{\frac 34} \pi^{\frac 12}} x^{- \frac 14} \exp\lp \dfrac{\pi^2}{48x} \rp < \dfrac{1}{2 n^{\frac 38}} \exp\lp \dfrac{\pi}{4} \sqrt{\dfrac{n}{3}} \rp.
\end{align*}
Thus, on $C_1$ we have
\begin{align*}
	\left|G(q) -\widetilde{G}_1(q)\right| \leq \lp \dfrac{1}{2 n^{\frac 38}} + \dfrac{9833929}{n^{\frac{15}{8}}} \rp \exp\lp \dfrac{\pi}{4} \sqrt{\dfrac{n}{3}} \rp.
\end{align*}
Now, let $D_0 = \{ z\in \CC : \mathrm{Re}\lp z \rp = x, \left| \mathrm{Im}\lp z \rp \right| \leq 15 x \}$, which is the image of $C_1$ under the change of variables $q \mapsto z$. Since $D_0$ has length $30x$, we have
\begin{align*}
	\left|J_1^{\text{maj}}(n)\right| \leq \dfrac{1}{2\pi} \int_{C_1} \dfrac{\left| G(q) - \widetilde{G}_1(q) \right|}{|q|^{n+1}} dq &\leq \dfrac{1}{2\pi} \int_{D_0} \left| G(q) -  \widetilde{G}_1(q) \right| \left| \exp\lp nz \rp \right| |dz| \\ &\leq \dfrac{30x}{2\pi} \cdot \lp \dfrac{1}{2 n^{\frac 38}} + \dfrac{9833929}{n^{\frac{15}{8}}} \rp \exp\lp \dfrac{\pi}{4} \sqrt{\dfrac{n}{3}} + nx \rp \\ &< \lp \dfrac{2}{n^{\frac{7}{8}}} + \dfrac{21291081}{n^{\frac{19}{8}}} \rp \exp\lp \dfrac{\pi}{2} \sqrt{\dfrac{n}{3}} \rp.
\end{align*}
We may similarly analyze the case of $G(q) - G_2^*(q)$. Note to begin that we may shift $C_2$ to $C_1$ by the substitution $q \mapsto -q$, and so
\begin{align*}
	\left|J^{\text{maj}}_2(n)\right| \leq \left| \dfrac{1}{2\pi i} \int_{C_1} \dfrac{G(-q) - \widetilde{G}_2(-q)}{q^{n+1}} dq \right|.
\end{align*}
We have by Lemma \ref{Major Error Bound} that $\Log\lp G(-q) \rp - \Log\lp G_2^*(-q) \rp = E_-(q) = O_{\leq}\lp 2689 |z|^4 \rp$. Thus $\left|E_-(q)\right| < \frac{3}{10}$ and as before we have $\exp\lp t \rp = 1 + O_{\leq} \lp \frac{3}{2}t \rp$. Thus, $\exp\lp E_-(q) \rp = 1 + O_{\leq}\lp 4034 |z|^4 \rp$, and by the same reasoning as in the first case we obtain
\begin{align*}
	\left| G(-q) - G_2^*(-q) \right| < \dfrac{8183085}{n^{\frac{17}{8}}} \exp\lp \dfrac{\pi}{4} \sqrt{\dfrac{n}{3}} \rp.
\end{align*}
As in the previous case, we have $G_2^*(-q) - \widetilde{G}_2(-q) = \widetilde{G}_2(-q) \times O_{\leq}\lp \frac{1}{22} |z| \rp$, and therefore
\begin{align*}
	\left|G_2^*(-q) - \widetilde{G}_2(-q)\right| \leq \dfrac{|z|}{22} \cdot \dfrac{226^{\frac 18} \Gamma\lp \frac 34 \rp}{2^{\frac 14} \pi^{\frac 12}} x^{\frac 14} \exp\lp \dfrac{\pi^2}{48x} \rp \leq \dfrac{3}{10 n^{\frac 58}} \exp\lp \dfrac{\pi}{4} \sqrt{\dfrac{n}{3}} \rp.
\end{align*}
Combining the two cases,
\begin{align*}
	\left|G(-q) - \widetilde{G}_2(-q)\right| \leq \lp \dfrac{3}{10 n^{\frac 58}} + \dfrac{8183085}{n^{\frac{17}{8}}} \rp \exp\lp \dfrac{\pi}{4} \sqrt{\dfrac{n}{3}} \rp
\end{align*}
and therefore
\begin{align*}
	\left|J^{\text{maj}}_2(n)\right| &\leq \dfrac{30x}{2\pi} \cdot \lp \dfrac{3}{10 n^{\frac 58}} + \dfrac{8183085}{n^{\frac{17}{8}}} \rp \exp\lp \dfrac{\pi}{4} \sqrt{\dfrac{n}{3}} + nx \rp \\ &< \lp \dfrac{13}{20 n^{\frac 98}} + \dfrac{17716899}{n^{\frac{21}{8}}} \rp \exp\lp \dfrac{\pi}{2} \sqrt{\dfrac{n}{3}} \rp.
\end{align*}

\subsection{Estimates for $J_1^*(n)$ and $J_2^*(n)$}

Having bounded the explicit error terms, we now estimate the integrals $J^*_1(n), J_2^*(n)$ in terms of more familiar {\it $I$-Bessel functions}. Recall that for any real $s$, the function $I_s(x)$ may be defined by
\begin{align*}
	I_s(x) := \dfrac{1}{2\pi i} \int_{\tilde D} w^{-s-1} \exp\lp \dfrac{x}{2} \lp \dfrac{1}{w} + w \rp \rp dw,
\end{align*}
where $\tilde D$ is any contour that loops from $-\infty$ below $\RR_{<0}$ around zero counterclockwise and back to $-\infty$ above $\RR_{<0}$. Let $D_0 := \{ w \in \CC: \mathrm{Re}\lp w \rp = x, \left| \mathrm{Im}\lp w \rp \right| \leq 15 x \}$ as earlier, and let
\begin{align*}
	D_{\pm} := \{ w \in \CC : \mathrm{Re}\lp w \rp \leq x, \mathrm{Im}\lp w \rp = \pm 15 x \}.
\end{align*}
Define the (counterclockwise-oriented) path $D := D_- \cup D_0 \cup D_+$. Letting $\tilde D$ be the image of $D$ under the change of variables $z = \frac{\pi}{4\sqrt{3n}} w$, we can see that
\begin{align*}
	\dfrac{1}{2\pi i} \int_D z^{-\frac 14} \exp\lp \dfrac{\pi^2}{48z} + nz \rp dz = \dfrac{\pi^{\frac 34}}{2^{\frac 32} 3^{\frac 38} n^{\frac 38}} I_{-\frac 34} \lp \dfrac{\pi}{2} \sqrt{\dfrac{n}{3}} \rp,
\end{align*}
and similarly
\begin{align*}
	\dfrac{1}{2\pi i} \int_D z^{\frac 14} \exp\lp \dfrac{\pi^2}{48z} + nz \rp dz = \dfrac{\pi^{\frac 54}}{2^{\frac 52} 3^{\frac 58} n^{\frac 58}} I_{-\frac 54}\lp \dfrac{\pi}{2} \sqrt{\dfrac{n}{3}} \rp.
\end{align*}
By changing variables $q \mapsto z$, we have
\begin{align*}
	J_1^*(n) = \dfrac{\Gamma\lp \frac 14 \rp}{2^{\frac 34} \pi^{\frac 12}} \cdot \dfrac{1}{2\pi i} \int_{D_0} z^{-\frac 14} \exp\lp \dfrac{\pi^2}{48z} + nz \rp dz,
\end{align*}
and therefore
\begin{align*}
	\dfrac{\Gamma\lp \frac 14 \rp \pi^{\frac 14}}{2^{\frac 94} 3^{\frac 38} n^{\frac 38}} I_{-\frac 34} \lp \dfrac{\pi}{2} \sqrt{\dfrac{n}{3}} \rp - J_1^*(n) = \dfrac{\Gamma\lp \frac 14 \rp}{2^{\frac 34} \pi^{\frac 12}} \cdot \dfrac{1}{2\pi i} \int_{D_+ \cup D_-} z^{-\frac 14} \exp\lp \dfrac{\pi^2}{48z} + nz \rp dz.
\end{align*}
The same procedure applied to $J_2^*(n)$ yields
\begin{align*}
	(-1)^n \dfrac{\Gamma\lp \frac 34 \rp \pi^{\frac 34}}{2^{\frac{11}{4}} 3^{\frac 58} n^{\frac 58}} I_{-\frac 54}\lp \dfrac{\pi}{2} \sqrt{\dfrac{n}{3}} \rp - J_2^*(n) = \dfrac{\Gamma\lp \frac 34 \rp}{2^{\frac 14} \pi^{\frac 12}} \cdot \dfrac{1}{2\pi i} \int_{D_+ \cup D_-} z^{\frac 14} \exp\lp \dfrac{\pi^2}{48z} + nz \rp dz.
\end{align*}
For the remainder, we define
\begin{align*}
	M_1(n) := \dfrac{\Gamma\lp \frac 14 \rp \pi^{\frac 14}}{2^{\frac 94} 3^{\frac 38} n^{\frac 38}} I_{-\frac 34} \lp \dfrac{\pi}{2} \sqrt{\dfrac{n}{3}} \rp, \hspace{0.2in} M_2(n) := (-1)^n \dfrac{\Gamma\lp \frac 34 \rp \pi^{\frac 34}}{2^{\frac{11}{4}} 3^{\frac 58} n^{\frac 58}} I_{-\frac 54}\lp \dfrac{\pi}{2} \sqrt{\dfrac{n}{3}} \rp,
\end{align*}
which are the main terms of $J_1^*(n)$, $J_2^*(n)$. For $t \in D_-$, set $t = \lp x - u \rp - 15x i$ for $u \geq 0$. Since $\mathrm{Re}\lp \frac{\pi^2}{48t} \rp \leq \frac{\pi}{4} \sqrt{\frac{n}{3}}$ and $|t| \geq 15x$, we have
\begin{align*}
	\left| t^{- \frac 14} \exp\lp \dfrac{\pi^2}{48t} + nt \rp \right| \leq |t|^{- \frac 14} \exp\lp \dfrac{\pi}{4} \sqrt{\frac{n}{3}} + n\lp x - u \rp \rp \leq \dfrac{2^{\frac 12} 3^{\frac 18} n^{\frac 18}}{15^{\frac 14} \pi^{\frac 14}} \exp\lp \dfrac{\pi}{2} \sqrt{\frac{n}{3}} - nu \rp.
\end{align*}
This bound holds not only for $t \in D_-$, but also $t \in D_+$, and therefore since $\int_0^\infty \exp\lp -nu \rp du = \frac{1}{n}$, we have
\begin{align*}
	\left| \dfrac{1}{2\pi i} \int_{D_+ \cup D_-} z^{-\frac 14} \exp\lp \dfrac{\pi^2}{48z} + nz \rp dz \right| &\leq 2 \lp \dfrac{2^{\frac 12} 3^{\frac 18} n^{\frac 18}}{15^{\frac 14} \pi^{\frac 14}} \exp\lp \dfrac{\pi}{2} \sqrt{\frac{n}{3}} \rp \rp \int_0^\infty e^{-nu} du \\ &= \dfrac{2^{\frac 32} 3^{\frac 18}}{15^{\frac 14} \pi^{\frac 14} n^{\frac 78}} \exp\lp \dfrac{\pi}{2} \sqrt{\frac{n}{3}} \rp.
\end{align*}
It therefore follows that
\begin{align*}
	\left| M_1(n) - J_1^*(n) \right| \leq \dfrac{2^{\frac 34} \Gamma\lp \frac 14 \rp}{3^{\frac 18} 5^{\frac 14} \pi^{\frac 34} n^{\frac 78}} \exp\lp \dfrac{\pi}{2} \sqrt{\frac{n}{3}} \rp < \dfrac{8}{5n^{\frac 78}} \exp\lp \dfrac{\pi}{2} \sqrt{\frac{n}{3}} \rp.
\end{align*}
Similarly, for $t \in D_\pm$, set $t = \lp x - u \rp \pm 15 \eta i$ for $u \geq 0$. Then $\left| t \right|^2 \leq 226 x^2 + u^2$. Since we have assumed $n > 4800$, it is clear that $226 x^2 + u^2 \leq 1 + u^2$, and so
\begin{align*}
	\left| t^{\frac 14} \exp\lp \dfrac{\pi^2}{48t} + nt \rp \right| \leq |t|^{\frac 14} \exp\lp \dfrac{\pi}{4} \sqrt{\frac{n}{3}} + n\lp x - u \rp \rp \leq \lp 1 + u^2 \rp^{\frac 18} \exp\lp \dfrac{\pi}{2} \sqrt{\dfrac{n}{3}} - nu \rp,
\end{align*}
and therefore
\begin{align*}
	\left| \dfrac{1}{2\pi i} \int_{D_+ \cup D_-} z^{\frac 14} \exp\lp \dfrac{\pi^2}{48z} + nz \rp dz \right| \leq 2 \exp\lp \dfrac{\pi}{2} \sqrt{\dfrac{n}{3}} \rp \int_0^\infty \lp 1 + u^2 \rp^{\frac 18} \exp\lp -nu \rp.
\end{align*}
Since $\lp 1 + u^2 \rp^{\frac 18} \leq 1 + u^{\frac 14}$ for $u > 0$, we have for $n > 1$ that
\begin{align*}
	\int_0^\infty \lp 1 + u^2 \rp^{\frac 18} \exp\lp -nu \rp \leq \int_0^\infty \exp\lp -nu \rp du + \int_0^\infty u^{1/4} \exp\lp -nu \rp du < \dfrac{2}{n}.
\end{align*}
and so
\begin{align*}
	\left| \dfrac{1}{2\pi i} \int_{D_+ \cup D_-} z^{\frac 14} \exp\lp \dfrac{\pi^2}{48z} + nz \rp dz \right| < \dfrac{4}{n} \exp\lp \dfrac{\pi}{2} \sqrt{\dfrac{n}{3}} \rp.
\end{align*}
As a consequence, we have
\begin{align*}
	\left| M_2(n) - J_2^*(n) \right| \leq \dfrac{2^{\frac{11}{4}} \Gamma\lp \frac 34 \rp}{\pi^{\frac 12} n} \exp\lp \dfrac{\pi}{2} \sqrt{\dfrac{n}{3}} \rp < \dfrac{5}{n} \exp\lp \dfrac{\pi}{2} \sqrt{\dfrac{n}{3}} \rp.
\end{align*}
Combining all the estimates made thus far, we may conclude that
\begin{align} \label{Effective Asymptotic}
	\left|a(n) - M_1(n) - M_2(n) \right| \leq E(n),
\end{align}
where
\begin{align} \label{Error Definition}
	E(n) := \dfrac{21\pi}{10} & \exp\lp \lp \dfrac{\pi}{4} + \dfrac{12}{5\pi} \rp \sqrt{\dfrac{n}{3}} \rp \notag \\ &+ \left[ \dfrac{4}{n^{\frac 78}} + \dfrac{5}{n} + \dfrac{13}{20 n^{\frac 98}} + \dfrac{21291081}{n^{\frac{19}{8}}} + \dfrac{17716899}{n^{\frac{21}{8}}} \right] \exp\lp \dfrac{\pi}{2} \sqrt{\dfrac{n}{3}} \rp.
\end{align}
Taken together, \eqref{Effective Asymptotic} and \eqref{Error Definition} now imply Theorem \ref{C4 a(n) Asymptotics}.

\section{Proof of Theorem \ref{C4 Conjecutre Proof}}

In this section, we prove that $a(n) \geq 0$ for all $n \geq 0$. Note that in order to prove $a(n) \geq 0$ for a particular value of $n$, it would suffice to show that $a(n) \geq M_1(n) + M_2(n) - E(n)$. The majority of this proof consists in simplifying this sufficient condition on $n$ until an explicitly lower bound is achieved. 

For simplicity, it is easiest to remove the $(-1)^n$ from $M_2(n)$ by leveraging $M_2(n) \leq \left|M_2(n)\right|$. Thus, to prove $a(n) \geq 0$ it would suffice to prove that $M_1(n) - \left|M_2(n)\right| - E(n) \geq 0$, that is,
\begin{align} \label{First Reduced Form}
	M_1(n) \geq \left|M_2(n)\right| + E(n).
\end{align}
Note that to prove \eqref{First Reduced Form}, it would suffice to prove $M_1(n) \geq 2 \left|M_2(n)\right|$ and $M_1(n) \geq 2 E(n)$. We now prove these inequalities one at a time. Taking the definitions of $M_1(n)$ and $M_2(n)$, the inequality $M_1(n) \geq 2\left|M_2(n)\right|$ may be rearranged to the form
\begin{align*}
	\dfrac{I_{-\frac 34}\lp \dfrac{\pi}{2} \sqrt{\dfrac{n}{3}} \rp}{I_{-\frac 54}\lp \dfrac{\pi}{2} \sqrt{\dfrac{n}{3}} \rp} > \dfrac{2 \Gamma\lp \frac 34 \rp}{\Gamma\lp \frac 14 \rp} \sqrt{\dfrac{\pi}{6n}}
\end{align*}
Now, the $I$-Bessel function has the power series expansion
\begin{align*}
	I_s(t) = \lp \dfrac{t}{2} \rp^s \sum_{k\geq 0} \dfrac{t^{2k}}{4^k k! \Gamma\lp s + k + 1 \rp},
\end{align*}
from which one may clearly see that $I_{-\frac 34}(t) > I_{-\frac 54}(t)$ for all $t > 1$. In particular, it is clear that for all $n > 4800$ that $M_1(n) > 2\left|M_2(n)\right|$. 

For a fixed $n > 4800$, in order to prove $a(n) \geq 0$ we have shown that it will suffice to prove $M_1(n) > 2 E(n)$. We prove this result by first bounding $M_1(n)$ from below. By \cite[Exercise 13.2, pg. 269]{Olv97}, we have for $t>0$ real that
\begin{align*}
	I_{-\frac 34}(t) = \dfrac{e^t}{\sqrt{2\pi t}}\lp 1 + \delta_1(t) \rp - i e^{-\frac 34 \pi i} \dfrac{e^{-t}}{\sqrt{2\pi t}} \lp 1 + \gamma_1(t) \rp,
\end{align*}
where $\delta_1(t), \gamma_1(t)$ satisfy the bounds
\begin{align*}
	\left| \gamma_1(t) \right| < \dfrac{5}{16t} \exp\lp \dfrac{5}{16t} \rp \hspace{0.1in} \text{and} \hspace{0.1in} \left| \delta_1(t) \right| < \dfrac{5\pi}{16t} \exp\lp \dfrac{5\pi}{16t} \rp.
\end{align*}
Therefore, we have
\begin{align*}
	\left| I_{-\frac 34}(t) - \dfrac{e^t}{\sqrt{2\pi t}} \right| < \dfrac{5\pi}{16 \sqrt{2\pi} t^{\frac 32}} \exp\lp t + \dfrac{5\pi}{16t} \rp + \dfrac{e^{-t}}{\sqrt{2\pi t}}\lp 1 + \dfrac{5}{16t} \exp\lp \dfrac{5}{16t} \rp \rp,
\end{align*}
from which it follows that
\begin{align*}
	I_{-\frac 34}(t) > \dfrac{e^t}{\sqrt{2\pi t}} - \left[ \dfrac{5\pi}{16 \sqrt{2\pi} t^{\frac 32}} \exp\lp t + \dfrac{5\pi}{16t} \rp + \dfrac{e^{-t}}{\sqrt{2\pi t}}\lp 1 + \dfrac{5}{16t} \exp\lp \dfrac{5}{16t} \rp \rp \right].
\end{align*}
We wish now to show $I_{-\frac 34}(t) > \frac{99 e^t}{100 \sqrt{2\pi t}}$ for suitably large $t$, for which it will suffice to consider their ratio (since both are positive). We have from the above inequality that $I_{-\frac 34}(t) \lp \frac{99e^t}{10\sqrt{2\pi t}} \rp^{-1}$ is plainly an increasing function of $t$, and so we can see that if we set $t = \frac{\pi}{2} \sqrt{\frac{n}{3}}$, the inequality holds for all $n > 4800$. Thus, to prove \eqref{First Reduced Form} for any given $n > 4800$ it will suffice to show that
\begin{align*} 
	\dfrac{99}{100} \cdot \dfrac{\Gamma\lp \frac 14 \rp}{2^{\frac 54} 3^{\frac 18} \pi^{\frac 34} n^{\frac 58}} &\exp\lp \dfrac{\pi}{2} \sqrt{\dfrac{n}{3}} \rp \geq \dfrac{21\pi}{5} \exp\lp \lp \dfrac{\pi}{4\sqrt{3}} + \dfrac{4\sqrt{3}}{5\pi} \rp \sqrt{n} \rp \\ &+ \left[ \dfrac{8}{n^{\frac 78}} + \dfrac{10}{n} + \dfrac{13}{10 n^{\frac 98}} + \dfrac{42582162}{n^{\frac{19}{8}}} + \dfrac{35433798}{n^{\frac{21}{8}}} \right] \exp\lp \dfrac{\pi}{2} \sqrt{\dfrac{n}{3}} \rp,
\end{align*}
which on dividing through by $\frac{1}{n^{5/8}} \exp\lp \frac{\pi}{2} \sqrt{\frac{n}{3}} \rp$ and making a convenient numerical estimate, it will suffice to show
\begin{align} \label{Final Reduced Form}
	\dfrac{21\pi n^{\frac 58}}{5} \exp\lp \lp \dfrac{12}{5\pi} - \dfrac{\pi}{4} \rp \sqrt{\dfrac{n}{3}} \rp + \left[ \dfrac{8}{n^{\frac 14}} + \dfrac{10}{n^{\frac 38}} + \dfrac{13}{10 n^{\frac 12}} + \dfrac{42582162}{n^{\frac{7}{4}}} + \dfrac{35433798}{n^2} \right] < \dfrac{11}{20}.
\end{align}
It is clear that for $n \geq 350000$ (in fact, much smaller $n$ will do) the left-hand side is a decreasing function of $n$. It can also be checked with a direct calculation that \eqref{Final Reduced Form} is true for $n = 350000$. Our method only assumed $n > 4800$, so we have now proven that $a(n) \geq 0$ for all $n \geq 350000$. The author has checked the values of $a(n)$ for $1 \leq n \leq 350000$ using his personal computer and found all to be non-negative. Therefore, Theorem \ref{C4 Conjecutre Proof} follows.

\sglsp

\chapter{Distribution of $t$-hook parity} \label{C5}
\thispagestyle{myheadings}

\dblsp
\vspace*{-.65cm}

The purpose of this chapter is to prove Theorems \ref{C5 Even and Odd T Behavior} and \ref{C5 Distribution property}. This is joint work with Anna Pun.

\section{The Nekrasov-Okounkov formula}

Generating functions connected to hook numbers are central in Chapters \ref{C5} and \ref{C6}. The most important formula in this direction is the Nekrasov-Okounkov formula \cite{Han10,NO06} which states that
\begin{align*}
	\sum_{\lambda \in \mathcal P} x^{|\lambda|} \prod_{h \in \mathcal H(\lambda)} \lp 1 - \dfrac{z}{h^2} \rp = \prod_{n=1}^\infty \lp 1 - x^n \rp^{z-1}.
\end{align*}
This result is fundamental in its close relationship to Dedekind's eta-function and many partition-theoretic identities. Using the famous work of Garvan, Kim, and Stanton on $t$-cores \cite{GKS90}, Han reproved and generalized the formula of Nekrasov-Okounkov in various ways which enabled the calculation of many kinds of generating functions connected to counting hook numbers in partitions. In particular, we shall be concerned with the multisets
\begin{align*}
	\mathcal H_t(\lambda) := \{ h \in \mathcal H(\lambda) : t | h \},
\end{align*}
i.e. the multiset of all hook numbers in $\lambda$ which are divisible by $t$. Of particular interest to us is the following theorem of Han.

\begin{theorem}[{\cite[Theorem 1.3]{Han10}}] \label{Han}
	Let $t$ be a positive integer. For any complex numbers $y$ and $z$ we have
	\begin{align*}
		\sum_{\lambda \in \mathcal P} x^{|\lambda|} \prod_{h \in \mathcal H_t(\lambda)} \lp y - \dfrac{tyz}{h^2} \rp = \prod_{k \geq 1} \dfrac{\lp 1 - x^{tk} \rp^t}{(1 - (yx^t)^k)^{t-z} \lp 1 - x^k \rp}.
	\end{align*}
\end{theorem}

\section{Generating functions and statement of results}

Since $p^e_t(n) + p^o_t(n) = p(n)$, $A_t(n) := p^e_t(n) - p_t^o(n)$ can serve a useful auxiliary role in our study. The utility of the function $A_t(n)$ comes from the generating function 
\begin{equation} \label{A_t(n) Generating Function}
	G_t(x): = \sum\limits_{n \geq 0} A_t(n) x^n = \prod\limits_{k \geq 1} \dfrac{(1 - x^{4tk})^t(1 - x^{tk})^{2t}}{(1 - x^{2tk})^{3t}(1-x^k)},
\end{equation}
proven in Corollary 5.2 of \cite{Han10}, which comes from specializing the values of $y, z$ in Theorem \ref{Han} by specializing to $y = -1$ and $z = 0$, along with relatively simple manipulations with infinite products. 

The driving force which brings to bear the applicationo of Rademacher's circle method is the fact that $G_t(x)$ may readily be written in terms of a modular infinite product via the Dedekind eta-function (see \eqref{G_t in eta}). Using the generating function \eqref{A_t(n) Generating Function}, we prove the following exact formula for $A_t(n)$, given as a Rademacher-type infinite series expansion.

\begin{theorem} \label{Exact Formula}
	If $n, t$ are positive integers with $t > 1$, then
	\begin{alignat*}{2}
		A_t(n) = \dfrac{2^{t/2}}{(24n-1)^{3/4}} \sum_{\substack{k \geq 1 \\ \gcd(k,2t) = 1}} \dfrac{\pi}{k} \sum_{\substack{0 \leq h < k \\ \gcd(h,k) = 1}} e^{\frac{-2\pi i n h}{k}} w(t,h,k) \sum_{m = 0}^{U_{t,k}} e^{\frac{2\pi i (4t)^* H m}{k}} c_1(t,h,k;m) &\\ \cdot \left( \dfrac{t - 24m}{t} \right)^{3/4} I_{\frac 32} \left( \dfrac{\pi}{12k} \sqrt{\dfrac{(t - 24m)(24n-1)}{t}} \right)& \\
		+ \dfrac{2^{t/2}}{(24n-1)^{3/4}} \sum_{\substack{k \geq 1 \\ 2 || k_0}} \dfrac{2\pi}{k} \sum_{\substack{0 \leq h < k \\ \gcd(h,k) = 1}} e^{\frac{-2\pi i n h}{k}} w(t,h,k) \sum_{m = 0}^{U_{t,k}} e^{\frac{2\pi i (2^\dagger t_0^*) H m}{k}} c_2(t,h,k;m) &\\ \cdot \left( \dfrac{t_0(1 + 3\gcd(k,t)^2) - 12m}{t_0} \right)^{3/4} I_{\frac 32} \left( \dfrac{\pi}{6k} \sqrt{\dfrac{t_0(1+3\gcd(k,t)^2) - 12m)(24n-1)}{t_0}} \right) &\\
		+ \dfrac{1}{(24n-1)^{3/4}} \sum_{\substack{k \geq 1 \\ 4 | k_0}} \dfrac{2\pi}{k} \sum_{\substack{0 \leq h < k \\ \gcd(h,k) = 1}} e^{\frac{-2\pi i n h}{k}} w(t,h,k) \sum_{m = 0}^{U_{t,k}} e^{\frac{2\pi i (t_0^* H) m}{k}} c_3(t,h,k;m) &\\ \cdot \left(\dfrac{t_0 - 24m}{t_0} \right)^{3/4} I_{\frac 32} \left( \dfrac{\pi}{6k} \sqrt{\dfrac{(t_0 - 24m)(24n-1)}{t_0}} \right),&
	\end{alignat*}
	where $k_0 := \dfrac{k}{\gcd(k,t)}$, $t_0 := \dfrac{t}{\gcd(k,t)}$, $H$ satisfies $hH \equiv -1 \pmod{k}$, $h^*$ (resp. $h^\dagger$) denotes the inverse of $h$ modulo $k_0$ (resp. $k_0/2$), $U_{t,k}$ is defined by
	\begin{align*}
		U_{t,k} := \begin{cases}
			\left\lfloor \dfrac{t}{24} \right\rfloor & \textnormal{if } 2 \centernot | k_0, \\[0.2in]
			\left\lfloor \dfrac{t_0(1 + 3\gcd(k,t)^2)}{12} \right\rfloor & \textnormal{if } 2 || k_0, \\[0.2in]
			\left\lfloor \dfrac{t_0}{24} \right\rfloor & \textnormal{if } 4 | k_0,
		\end{cases}
	\end{align*}
	$w(t,h,k)$ is defined by \eqref{w-def}, $c_j(t,h,k;m)$ are defined by \eqref{c-def}, and $I_{\frac 32}(z)$ is the classical modified $I$-Bessel function.
\end{theorem}

\begin{example}
	We illustrate Theorem \ref{Exact Formula} using the numbers $A_t(d;n)$,
	which denote partial sums for $A_t(n)$ over $1 \leq k \leq d$. Theorem \ref{Exact Formula} is therefore the statement that $\lim\limits_{d \rightarrow \infty} A_t(d;n) = A_t(n)$. We offer some examples in the table below.
	
	\begin{table}[h]
		\begin{center}
		\begin{tabular}{|c|c|c|c|c|c|} \hline
			$d$ & $10$  & $100$ & $1000$ & $\cdots$ & $\infty$ \\ \hline
			$50$ &  $\approx 114580.084$& $\approx 114579.996$ & $\approx 114580.000$  & $\cdots$ & $114580$ \\ \hline
			$100$  & $\approx 81486201.594$ & $\approx 81486198.001$ & $\approx 81486198.000$ & $\cdots$ & $81486198$ \\ \hline
		\end{tabular}
		\vspace{0.5ex}
		\caption{Values of $A_3(d;n)$}
		\end{center}
	\end{table}
\end{example}

\vspace{-5ex}
This exact formula gives the following corollary.

\begin{corollary}\label{Dominating term of A_t}
	For $t>1$ a fixed positive integer, write $t = 2^s\ell$ for integers $s, \ell \geq 0$ such that $\ell$ is odd. Then as $n \rightarrow \infty$ we have
	\begin{equation*} \label{Dominant Term}
		A_t(n) \sim \displaystyle\dfrac{\pi}{2^{s+\frac{t}{2}}}\bigg(\dfrac{1+ 3\cdot 4^s}{24n -1}\bigg)^\frac{3}{4}I_{\frac{3}{2}}\bigg(\dfrac{\pi\sqrt{(1 + 3\cdot 4^s)(24n -1)}}{6\cdot 2^{s+1}}\bigg)\sum_{\substack{0< h < 2^{s+1}\\h \text{ odd}}}w_2(t,h,2^{s+1})e^{-\frac{\pi i nh}{2^s}}.
	\end{equation*}
	In particular, when $t$ is odd we have 
	\begin{equation*} \label{Dominant Term when t odd}
		A_t(n) \sim \displaystyle (-1)^n \dfrac{\pi\cdot 2^{(3-t)/2}}{(24n-1)^{3/4}}I_{\frac{3}{2}}\bigg(\dfrac{\pi\sqrt{24n -1}}{6}\bigg).
	\end{equation*}
\end{corollary}

\begin{proof}
	For $z \in \mathbb{R}^+$, it is known that $I_{\frac 32}(z) \sim \dfrac{e^z}{\sqrt{2\pi z}}$. From this asymptotic relation, we can derive a condition for isolating the dominant term in Theorem \ref{Exact Formula}. In particular, let $\{ g_i(n) \}_{i=0}^\infty$ be a countable collection of functions such that $\lim\limits_{n \to \infty} \dfrac{g_0(n)}{g_i(n)} > 1$ for all $i \not = 0$ and let $\{ a_i(n) \}_{i=0}^\infty$ be complex numbers each of which grow at most polynomially in $n$. Then if the series $\sum\limits_{i=0}^\infty a_i(n) I_{3/2}(g_i(n))$ converges and $a_0(n)$ does not vanish, we have
	$$\sum\limits_{i=0}^\infty a_i(n) I_{\frac 32}(g_i(n)) \sim a_0(n) I_{3/2}(g_0(n))$$
	as $n \to \infty$. This reduces the proof to an analysis of the analogs of $g_i(n)$ and $a_0(n)$ in Theorem \ref{Exact Formula}.
	
	Each of the arguments of $I_{\frac 32}(z)$ is maximized when $m = 0$, so we are left with the task of finding the largest possible value of of coefficients on $\sqrt{24n-1}$, which are given by $\dfrac{\pi}{12k}$ if $2 \centernot | k_0$, $\dfrac{\pi}{6k} \sqrt{1 + 3\gcd(k,t)^2}$ if $2 || k_0$, and $\dfrac{\pi}{6k}$ if $4 | k_0$. Among these, it is clear that the case where $2 || k_0$ is the largest. Now, this expression can be rewritten as
	$$\dfrac{\pi}{6k} \sqrt{1 + 3\gcd(k,t)^2} = \dfrac{\pi}{6k} \sqrt{1 + \dfrac{3}{k_0^2} k^2}.$$
	When $k_0$ is held fixed, since $k_0 \geq 2$ this expression is strictly decreasing in $k$, and therefore the optimal choice of $k$ must be of the form $k = 2^s k_0$. It is also clear that $k_0 = 2$ is optimal, and therefore $k = 2^{s+1}$ has the dominant $I$-Bessel function. Since by Lemma \ref{Kloosterman sum nonzero} the associated Kloosterman sum does not vanish, this completes the proof.
\end{proof}

\section{Set-up and notation}

The approach that will be utilized in the proof of Theorem \ref{Exact Formula} is commonly referred to as the ``circle method". Initially developed by Hardy and Ramanujan and refined by Rademacher, the circle method has been employed with great success for the past century in additive number theory. The crowning achievement of the circle method lies in producing an exact formula for the partition function $p(n)$, and it has been utilized to produce similar exact formulas for variants of the partition function. A helpful and instructive sketch of the application of Rademacher's circle method to the partition function $p(n)$ is given in Chapter 5 of \cite{Apo90}. Here, we will provide a summary of the circle method, in order to clarify the key steps and the general flow of the argument.

The function $A_t(n)$ has as its generating function $G_t(x)$. Our objective is to use $G_t(x)$ to produce an exact formula for $A_t(n)$. Consider the Laurent expansion of $G_t(x)/x^{n+1}$ in the punctured unit disk. This function has a pole at $x = 0$ with residue $p(n)$ and no other poles. Therefore, by Cauchy's residue theorem we have
\begin{equation}
	A_t(n) = \dfrac{1}{2\pi i} \int_C \dfrac{G_t(x)}{x^{n+1}} dx,
\end{equation}
where $C$ is any simple closed curve in the unit disk that contains the origin in its interior. The task of the circle method is to choose a curve $C$ that allows us to evaluate this integral, and this is achieved by choosing $C$ to lie near the singularities of $G_t(x)$, which are the roots of unity. For every positive integer $N$ and every pair of coprime non-negative integers $0 \leq h < k \leq N$, we choose a special contour $C$ in the complex upper half-plane and divide this contour into arcs $C_{h,k}$ near the roots of unity $e^{2\pi i h/k}$. Integration along $C$ can then be expressed as a finite sum of integrals along the arcs $C_{h,k}$, and elementary functions $\psi_{h,k}$ are chosen with behavior similar to $G_t(x)$ near the singularity $e^{2\pi ih/k}$. The functions $\psi_{h,k}$ are found by using properties of $G_t(x)$ deduced from the functional equation of the Dedekind eta-function $\eta(\tau) := e^{\pi i \tau / 12} \prod\limits_{n \geq 1} (1 - e^{2\pi i n \tau})$ and the relation between $G_t(x)$ and $\eta(\tau)$ given by
\begin{equation}\label{G_t in eta}
	G_t(e^{2\pi i \tau}) = \dfrac{\eta(t\tau)^{2t} \eta(4t\tau)^t}{\eta(\tau)\eta(2t\tau)^{3t}}.
\end{equation}
The error created by replacing $G_t(x)$ by $\psi_{h,k}(x)$ can be estimated, and the integrals of the $\psi_{h,k}$ along $C_{h,k}$ evaluated. This procedure produces estimates that can be used to formulate a convergent series for $A_t(n)$. Our implementation of the circle method will follow along these same lines.

To preface the proof of Theorem \ref{Exact Formula}, we summarize notation which will be used prominently throughout the rest of the chapter. The values of $t, n, h$, and $k$ are always non-negative integers. Additionally, we assume $t > 1$ and that $h, k$ satisfy $0 \leq h < k$ and $\gcd(h,k) = 1$. Frequently, it is necessary to remove common factors between $k$ and $t$, and so we define $k_0 := \frac{k}{\gcd(k,t)}$ and $t_0 := \frac{t}{\gcd(k,t)}$. We will also make use of multiplicative inverses to a variety of moduli, and use distinct notations to distinguish these. We will always denote by $H$ an integer satisfying $hH \equiv -1 \pmod{k}$, and $h^*, h^\dagger$ will denote inverses of $h$ modulo $k_0$ and $k_0/2$ respectively. The complex numbers $x$ and $z$ are related by $x = \exp\left( \frac{2\pi i}{k} \left( h + iz \right) \right)$. Note that although $x$ depends on $h$ and $k$, the dependence is suppressed since these values will be clear in context. The notation $x'$ will always be used to denote a modular transformation of the variable $x$. The modular transformations also make use of the Dedekind sum $s(u,v)$, which for any integers $u,v$ is given by
\begin{align*}
	s(u,v) := \sum_{m = 1}^v \left(\left( \dfrac{m}{v} \right)\right) \left(\left( \dfrac{um}{v} \right)\right)
\end{align*}
where
\begin{align*}
	((m)) := \begin{cases}
		m - \lfloor m \rfloor - \dfrac{1}{2} & m \not \in \mathbb{Z}, \\
		0 & m \in \mathbb{Z}.
	\end{cases}
\end{align*}
These Dedekind sums will always arise in the context of certain roots of unity $e^{\pi i s(u,v)}$, and so it is convenient to adopt the notation $\omega_{u,v} := e^{\pi i s(u,v)}$.

\section{Modular transformation laws}

We first recall the transformation formula for the generating function of $p(n)$ (see, for example \cite{Hag71} or p. 96 in \cite{Apo90}).

\begin{theorem}\label{transformation formula for F}
	Let $k,t$ be positive integers with $t > 1$ and $0 \leq h < k$ an integer coprime to $k$ and $H$ an integer satisfying $hH \equiv -1 \pmod{k}$. Let $z$ be a complex number with $\textnormal{Re}(z) > 0$ and let $x, x'$ be defined by $x = \exp\bigg(\dfrac{2\pi i}{k} (h + iz) \bigg)$ and $x' = \exp\bigg(\dfrac{2\pi i}{k} \left(  H + \dfrac{i}{z} \right)\bigg)$. If $F(x)$ is defined by $F(x) := \prod\limits_{m=1}^\infty \dfrac{1}{1 - x^m}$, then \begin{equation*}
		F(x)  = \sqrt{z} \cdot \omega_{h,k} \exp \bigg(\dfrac{\pi(z^{-1} -z)}{12k}\bigg)F(x').
	\end{equation*}
\end{theorem}
The proof of this theorem comes directly from the modular transformation properties of Dedekind's eta-function.

By (\ref{G_t in eta}), $G_t(x)$ can be expressed in terms of $F(x)$:
$$G_t(x)  = \dfrac{F(x)\big[F\big(x^{2t}\big)\big]^{3t}}{\big[F\big(x^{t}\big)\big]^{2t}\big[F\big(x^{4t}\big)\big]^{t}}.$$
We can therefore apply Theorem \hyperref[transformation formula for F]{\ref{transformation formula for F}} to find a similar transformation formula for $G_t(x)$.

\begin{lemma} \label{Transformation Laws}
	Define $x_1 := x^t$, $x_2 := x^{2t}$, and $x_3 := x^{4t}$. Then the following transformation formulas for $F(x_j)$ hold. \\
	
	\noindent (a) When $k_0$ is odd, for $1 \leq j \leq 3$ we have $$F(x_j) = \sqrt{2^{j-1} t_0 z} \cdot \omega_{2^{j-1} t_0 h, k_0} \exp \bigg[ \dfrac{\pi}{12k_0} \bigg( \dfrac{1}{2^{j-1} t_0 z} - 2^{j-1} t_0 z \bigg) \bigg] F(x_j')$$
	hold, where $x_j' = \exp\bigg[\dfrac{2\pi i}{k_0} \bigg( (2^{j-1} t_0)^* H + \dfrac{i}{2^{j-1} t_0 z}\bigg)\bigg]$. \\
	
	\noindent (b) Suppose $k_0 \equiv 2 \pmod{4}$. Then we have the transformation formulas
	$$F(x_1) = \sqrt{t_0 z} \cdot \omega_{t_0 h, k_0} \exp\Bigg[ \dfrac{\pi}{12k_0} \bigg( \dfrac{1}{t_0 z} - t_0 z \bigg) \Bigg] F(x_1')$$
	where $x_1' = \exp\bigg[\dfrac{2\pi i}{k_0} \bigg( t_0^* H + \dfrac{i}{t_0 z}\bigg)\bigg]$, and for $j = 2, 3$ we have
	$$F(x_j) = \sqrt{2^{j-2} t_0 z} \cdot \omega_{2^{j-2} t_0 h, k_0/2} \exp\Bigg[ \dfrac{\pi}{6k_0} \bigg( \dfrac{1}{2^{j-2} t_0 z} - 2^{j-2} t_0 z \bigg) \Bigg] F(x_j'),$$
	where $x_2' = \exp\bigg[\dfrac{2\pi i}{k_0 / 2} \bigg( t_0^* H + \dfrac{i}{t_0 z}\bigg)\bigg],$ and $x_3' = \exp\bigg[\dfrac{2\pi i}{k_0 / 2} \bigg( 2^\dagger t_0^* H + \dfrac{i}{t_0 z}\bigg)\bigg].$ \\
	
	\noindent (c) Suppose $4 | k_0$. Then we have the transformation formulas
	$$F(x_j) = \sqrt{t_0 z} \cdot \omega_{t_0 h, k_0 / 2^{j-1}} \exp\Bigg[ \dfrac{2^{3-j} \pi}{12k_0} \bigg( \dfrac{1}{t_0 z} - t_0 z \bigg) \Bigg] F(x_j'),$$
	where $x_j' = \exp\Bigg[ \dfrac{2\pi i}{k_0} \bigg( 2^{j-1} t_0^* H + \dfrac{i}{2^{1-j}t_0 z} \bigg) \Bigg]$.
\end{lemma}

\begin{proof}
	We first prove the case $j=1$ of (a) By definition, $x_1 = \exp \bigg( \dfrac{2\pi i}{k_0} (t_0 h + i t_0 z) \bigg)$. Since $\gcd(t_0 h, k_0) = 1$ and $t_0 h (t_0^* H) \equiv -1 \pmod{k_0}$, applying Theorem \hyperref[transformation formula for F]{\ref{transformation formula for F}} gives the result by the substitutions $h \mapsto t_0 h$, $k \mapsto k_0$, and $z \mapsto t_0 z$. Every other case of the result follows by rearranging terms in $x_j$ in a manner such that the terms playing the roles of $h$ and $k$ in Theorem \ref{transformation formula for F} are coprime.
\end{proof}

Using these identities for each case, we obtain the transformation law
\begin{align*}
	G_t(x) = \begin{cases}
		2^{t/2} \sqrt{z} \exp \bigg[ \dfrac{\pi}{12k} \bigg( \dfrac{4 - 3\gcd(k,t)^2}{4z} - z \bigg) \bigg] w_1(t,h,k) \dfrac{F(x') [F(x_2')]^{3t}}{[F(x_1')]^{2t}[F(x_3')]^t}& 2 \centernot | k_0, \\[+0.2in] 
		2^{t/2} \sqrt{z} \exp\bigg[ \dfrac{\pi}{12k}\bigg(\dfrac{1+3\gcd(k,t)^2}{z} - z\bigg)\bigg]w_2(t,h,k)\dfrac{F(x')\big[F(x_2')\big]^{3t}}{\big[F(x_1')\big]^{2t}\big[F(x_3')\big]^{t}} & 2 || k_0, \\[+0.2in] 
		\sqrt{z}\exp\bigg[\dfrac{\pi}{12k}\left(\dfrac{1}{z} - z\right)\bigg]w_3(t,h,k)\dfrac{F(x')\big[F(x_2')\big]^{3t}}{\big[F(x_1')\big]^{2t}\big[F(x_3')\big]^{t}}& 4 | k_0,
	\end{cases}
\end{align*}
\vspace{0.6ex}
where $w_1(t,h,k) := \omega_{h,k} \omega^{3t}_{2t_0 h, k_0} \omega^{-2t}_{t_0 h, k_0} \omega^{-t}_{4 t_0 h, k_0}$, $w_2(t,h,k) := \omega_{h,k} \omega^{3t}_{t_0 h, k_0/2} \omega^{-2t}_{t_0 h, k_0} \omega^{-t}_{2 t_0 h, k_0/2}$, and $w_3(t,h,k) := \omega_{h,k} \omega^{3t}_{t_0 h, k_0/4} \omega^{-2t}_{t_0 h, k_0} \omega^{-t}_{t_0 h, k_0/4}$. From the definition of $s(h,k)$ we can see that $s(dh, dk) = s(h,k)$ for every integer $d$, and therefore $\omega_{t_0 h, k_0/2} = \omega_{2t_0 h, k_0}$ when $2 | k_0$ and $\omega_{t_0 h, k_0/4} = \omega_{2t_0 h, k_0/2} = \omega_{4t_0 h, k_0}$ when $4 | k_0$. Therefore $w_j(t,h,k) = w(t,h,k)$ for all $j$, where
\begin{align} \label{w-def}
	w(t,h,k) := \dfrac{\omega_{h,k} \omega^{3t}_{2t_0 h, k_0} }{\omega^{2t}_{t_0 h, k_0} \omega^{t}_{4 t_0 h, k_0}}.
\end{align}
Therefore, the transformation law for $G_t(x)$ can be rewritten as
\begin{align} \label{Transformation Law for G_t}
	G_t(x) = \begin{cases}
		2^{t/2} \sqrt{z} \exp \bigg[ \dfrac{\pi}{12k} \bigg( \dfrac{4 - 3\gcd(k,t)^2}{4z} - z \bigg) \bigg] w(t,h,k) J_{t,h,k}(x') & 2 \centernot | k_0, \\ 
		2^{t/2} \sqrt{z} \exp\bigg[ \dfrac{\pi}{12k}\bigg(\dfrac{1+3\gcd(k,t)^2}{z} - z\bigg)\bigg]w(t,h,k) J_{t,h,k}(x') & 2 || k_0, \\ 
		\sqrt{z}\exp\bigg[\dfrac{\pi}{12k}\left(\dfrac{1}{z} - z\right)\bigg] w(t,h,k) J_{t,h,k}(x') & 4 | k_0,
	\end{cases}
\end{align}
where for shorthand we define $J_{t,h,k}(x') := \dfrac{F(x')\big[F(x_2')\big]^{3t}}{\big[F(x_1')\big]^{2t}\big[F(x_3')\big]^{t}}$.

\section{The Farey decomposition}

We follow closely to the notations and proofs in Chapter 5 of \cite{Apo90}. Let notation be as before, and let $N$ be any positive integer. Recall that
\begin{align*}
	G_t(x):= \sum\limits_{\lambda \in \mathcal P}x^{\vert \lambda \vert}(-1)^{\# \mathcal H_t(\lambda)} =\sum\limits_{n \geq 0}\sum\limits_{\lambda \vdash n}(-1)^{\# \mathcal H_t(\lambda)}x^n = \sum\limits_{n \geq 0}A_t(n)x^n.
\end{align*}
By Cauchy's residue theorem, we have
\begin{align*}
	A_t(n) = \frac{1}{2\pi i}\int_C \frac{G_t(x)}{x^{n + 1}}\,dx,
\end{align*}
where $C$ is any positively oriented simple closed curve in a unit disk that contains the origin in its interior. In our implementation of the circle method, we set $C = C_N$ for $C_N$ centered at zero with radius $e^{-2\pi N^{-2}}$.  Using the transformations $x = e^{2\pi i \tau}$ and $z = -ik^2\bigg(\tau -\dfrac{h}{k}\bigg)$ in succession, the circle $C_N$ is mapped onto the circle $\mathcal{K}$ with center $\frac{1}{2}$ and radius $\frac{1}{2}$. In the rest of this proof, $\mathcal{K}$ will denote this same circle. If we breakdown $C_N$ into Farey arcs, then this change of variables gives the formula
\begin{eqnarray*}
	A_t(n) &=& \displaystyle\sum\limits_{k=1}^N \Bigg[\dfrac{i}{k^2} \sum\limits_{\substack{0\leq h < k\\(h,k) =1}}e^{-\frac{2\pi i nh}{k}}\int_{z_1(h,k)}^{z_2(h,k)} G_t\left( e^{\frac{2\pi i}{k}\left( h + \frac{iz}{k} \right)} \right) e^{\frac{2\pi nz}{k^2}}\,dz\Bigg],
\end{eqnarray*}
where the integral runs along the arc of $\mathcal{K}$ between the points $z_1(h,k)$ and $z_2(h,k)$ defined by
$$z_1(h,k) = \dfrac{k^2}{k^2 + k_1^2} + i \dfrac{kk_1}{k^2 + k_1^2} \ \ \text{ and } \ \ z_2(h,k) = \dfrac{k^2}{k^2 + k_2^2} - i \dfrac{kk_2}{k^2 + k_2^2},$$
where $k_1, k, k_2$ are the denominators of consecutive terms of the Farey series of order $N$. Computing $A_t(n)$ therefore reduces to computing the integrals
\begin{align*}
	I(t,h,k,n) := \int_{z_1(h,k)}^{z_2(h,k)} G_t\left( e^{\frac{2\pi i}{k}\left( h + \frac{iz}{k} \right)} \right) e^{\frac{2\pi n z}{k^2}} dz.
\end{align*}

\section{Exact formula for $A_t(n)$}

The first step to evaluating these integrals is an application of the transformation law for $G_t(x)$. Because of the formulation of \eqref{Transformation Law for G_t}, the exact formula is naturally broken into three sums. One of these is given by
$$\sum_{\substack{k \geq 1 \\ k_0 \textnormal{ odd}}} \sum\limits_{\substack{0 \leq h < k \\ \gcd(h,k) = 1}} e^{-2\pi i n h / k} I(t,h,k,n)$$
and the other two are defined similarly with the modification that $k_0$ odd is replaced by either $2 || k_0$ or $4 | k_0$. Because of this natural breakdown, the evaluation of $I(t,h,k,n)$ also naturally breaks into three cases.

In order to estimate the integrals $I(t,h,k,n)$, we will use a series expansion for the factor $J_{t,h,k}(x') = \dfrac{F(x') [F(x_2')]^{3t}}{[F(x_1')]^{2t}[F(x_3')]^t}$ in the modular transformation law for $G_t(x)$. The variable we will use for this series expansion will depend on the value of $k_0$. In particular, define $y_j$ for $1 \leq j \leq 3$ by
$$y_1 := e^{\frac{2\pi i}{k} \left( 4^* t_0^* H + \frac{i}{4 t_0 z} \right)}, \ \ \ y_2 := e^{\frac{2\pi i}{k} \left( 2^\dagger t_0^* H + \frac{i}{2 t_0 z} \right)}, \ \ \ y_3 := e^{\frac{2\pi i}{k} \left( t_0^* H + \frac{i}{t_0 z} \right)}.$$
The utility of using $y_j$ is that it relates nicely to the variables $x'$, $x_1'$, $x_2'$, and $x_3'$ appearing in $J_{t,h,k}(x')$. From definitions, it follows that
\vspace{-5ex}
\begin{center}
	$$x_1' = y_1^{4\gcd(k,t)} = - y_2^{2\gcd(k,t)} = y_3^{\gcd(k,t)},$$
	$$x_2' = y_1^{2\gcd(k,t)} = y_2^{4\gcd(k,t)} = y_3^{2\gcd(k,t)},$$
	$$x_3' = y_1^{\gcd(k,t)} = y_2^{2\gcd(k,t)} = y_3^{4\gcd(k,t)},$$
\end{center}
and
$$x' = y_1^{4t_0} e^{\frac{-2\pi i \left( 4t_0(4t_0)^* - 1 \right) H}{k}} = y_2^{2t_0} e^{\frac{-2\pi i \left( 2 \cdot 2^\dagger t_0 t_0^* - 1 \right)H}{k}} = y_3^{t_0} e^{\frac{-2\pi i \left( t_0 t_0^* - 1 \right) H}{k}}.$$ 
Therefore, we have three series expansions for $J_{t,h,k}(x')$ given by
\begin{align} \label{c-def}
	J_{t,h,k}(x') =: \sum_{m \geq 0} c_j(t,h,k; m) y_j^m
\end{align}
for $1 \leq j \leq 3$. These series expansions, along with the transformation laws for $G_t(x)$, are used to aid in the evaluation of the integrals $I(t,h,k,n)$.

\subsection{Estimating $I(t,h,k,n)$}

The process of evaluating $I(t,h,k,n)$ breaks into three cases based on the value of $k_0$. Since the proofs in every case run along similar lines, we need only write out details in the case where $k_0$ is odd and to comment on which aspects of the proof need to be altered for the other two cases. When $k_0$ is odd, we use the series expansion for $J_{t,h,k}(x')$ in $y_1$. Applying the substitution $z \mapsto \dfrac{z}{k}$ in \eqref{Transformation Law for G_t}, we have
$$I(t,h,k,n) = \dfrac{2^{t/2} w(t,h,k)}{\sqrt{k}} \int_{z_1(h,k)}^{z_2(h,k)} \sum_{m \geq 0} e^{\frac{2\pi i (4t_0)^* H m}{k}} c_1(t,h,k;m) f_{k,t,m}(z) e^{\frac{2\pi n z}{k^2}} \, dz,$$
where
$$f_{k,t,m}(z) := \sqrt{z} \exp \bigg[ \dfrac{\pi}{12} \bigg( \dfrac{4 - 3\gcd(k,t)^2}{4z} - \dfrac{6 m}{t_0 z} - \dfrac{z}{k^2} \bigg) \bigg].$$
From the theory of Farey arcs (see Theorem 5.9 of \cite{Apo90}) we know that the path of integration has length less than $2\sqrt{2} k N^{-1}$ and that for any $z$ on the path of integration, $|z| < \sqrt{2} k N^{-1}$. Furthermore, any $z \in \mathcal{K} \backslash \{0\}$ satisfies $0 < \textnormal{Re}(z) \leq 1$ and $\textnormal{Re}(1/z) = 1$. From these facts, we can see that $m > M_{t,k} := \left\lfloor \dfrac{t_0 \left( 4 - 3\gcd(k,t)^2 \right)}{24} \right\rfloor$ if and only if
$$\left| e^{\frac{\pi}{12} \left( \frac{4 - 3\gcd(k,t)^2}{4z} - \frac{6 m}{t_0 z} - \frac{z}{k^2} \right)} \right| < 1,$$
and that therefore
$$\int_{z_1(h,k)}^{z_2(h,k)} \sum_{m > M_{t,k}} e^{\frac{2\pi i (4t_0)^* H m}{k}} c_1(t,h,k;m) f_{k,t,m}(z) e^{\frac{2\pi n z}{k^2}} \, dz = O\left( k^{3/2} N^{-3/2} \right).$$
Applying this estimate to $I(t,h,k,n)$, it follows that
\begin{align*}
	I(t,h,k,n) = \dfrac{2^{t/2} w(t,h,k)}{\sqrt{k}} \int_{z_1(h,k)}^{z_2(h,k)} \sum_{m = 0}^{M_{t,k}} e^{\frac{2\pi i (4t_0)^* H m}{k}} &c_1(t,h,k;m) f_{k,t,m}(z) e^{\frac{2\pi n z}{k^2}} \, dz \\ &+ O\left( k^{1/2} N^{-3/2} \right).
\end{align*}
Similar estimates apply in the other two cases. In particular, extend the definition of $M_{t,k}$ by
\begin{align*}
	M_{t,k} := \begin{cases}
		\left\lfloor \dfrac{t_0(4 - 3\gcd(k,t)^2)}{24} \right\rfloor & \textnormal{if } 2 \centernot | k_0, \\[+0.2in]
		\left\lfloor \dfrac{t_0(1 + 3\gcd(k,t)^2)}{12} \right\rfloor & \textnormal{if } 2 || k_0, \\[+0.2in]
		\left\lfloor \dfrac{t_0}{24} \right\rfloor & \textnormal{if } 4 | k_0,
	\end{cases}
\end{align*}
and in place of $y_1$ use $y_2$ when $2 || k_0$ or $y_3$ when $4 | k_0$. These modifications lead to the following proposition.

\begin{proposition} \label{Reduce to Finite Sum}
	Adopt all notation as above. Then if $k_0$ is odd, we have
	\begin{align*}
		I(t,h,&k,n) = \dfrac{2^{t/2} w(t,h,k)}{\sqrt{k}} \sum_{m = 0}^{M_{t,k}} e^{\frac{2\pi i (4t_0)^* H m}{k}} c_1(t,h,k;m) \\ &\cdot \int_{z_1(h,k)}^{z_2(h,k)} \sqrt{z} \exp \bigg[ \dfrac{\pi}{12} \left( \dfrac{4 - 3\gcd(k,t)^2}{4z} - \dfrac{6 m}{t_0 z} + \dfrac{(24n-1)z}{k^2} \right) \bigg] \, dz + O\left( k^{1/2} N^{-3/2} \right).
	\end{align*}
	If $2 || k_0$, then we have
	\begin{align*}
		I(t,h,&k,n) = \dfrac{2^{t/2} w(t,h,k)}{\sqrt{k}} \sum_{m = 0}^{M_{t,k}} e^{\frac{2\pi i (2^\dagger t_0^*) H m}{k}} c_2(t,h,k;m) \\ &\cdot \int_{z_1(h,k)}^{z_2(h,k)} \sqrt{z} \exp\bigg[ \dfrac{\pi}{12} \left( \dfrac{1+3\gcd(k,t)^2}{z} - \frac{12m}{t_0 z} + \dfrac{(24n-1)z}{k^2} \right) \bigg] \, dz + O\left( k^{1/2} N^{-3/2} \right).
	\end{align*}
	If $4 | k_0$, then we have
	\begin{align*}
		I(t,h,k,n) =& \dfrac{w(t,h,k)}{\sqrt{k}} \sum_{m = 0}^{M_{t,k}} e^{\frac{2\pi i (t_0^* H) m}{k}} c_3(t,h,k;m) \\ &\cdot \int_{z_1(h,k)}^{z_2(h,k)} \sqrt{z} \exp\left[ \dfrac{\pi}{12} \left( \dfrac{1}{z} - \dfrac{24 m}{t_0 z} + \dfrac{(24n-1)z}{k^2} \right) \right] \, dz + O\left( k^{1/2} N^{-3/2} \right).
	\end{align*}
\end{proposition}

In light of Proposition \ref{Reduce to Finite Sum}, the problem of evaluating $I(t,h,k,n)$ is reduced to evaluating integrals of the form
$$\int_{z_1(h,k)}^{z_2(h,k)} \sqrt{z} \exp\left[ \dfrac{\pi}{12} \left( \dfrac{A - Bm}{z} + \dfrac{(24n-1z)}{k^2} \right) \right] \, dz$$
for certain constants $A,B$. This evaluation has two main steps. Firstly, we show that extending the path of integration to the whole circle $\mathcal{K}$ introduces only a small error term. Secondly, we show how the integral along $\mathcal{K}$ is expressible by familiar functions from analysis. These steps are carried out together in the following proposition.

\begin{proposition} \label{General I-Bessel}
	Fix an integer $t > 1$, and let $A,B$ be constants independent of $z$ for which $A = O_{k}(1)$ as $N \to \infty$ and $B > 0$. Then we have
	\begin{align*}
		\dfrac{1}{2\pi i} \int_{z_1(h,k)}^{z_2(h,k)} \sqrt{z} e^{\frac{\pi}{12} \left( \frac{A - Bm}{z} + \frac{(24n-1)z}{k^2} \right)} \, dz =& \dfrac{k^{3/2} (A - Bm)^{3/4}}{(24n-1)^{3/4}} I_{\frac 32} \left( \dfrac{\pi}{6k} \sqrt{(A - Bm)(24n-1)} \right) \\ &+ O\left( k^{3/2} N^{-3/2} \right).
	\end{align*}
\end{proposition}

\begin{proof}
	For $\mathcal{K}^-$ the negative orientation of the circle $\mathcal{K}$, we can break down integrals over $\mathcal{K}^-$ by
	$$\int_{\mathcal{K}^-} = \int_{z_1(h,k)}^{z_2(h,k)} + \int_0^{z_1(h,k)} + \int_{z_2(h,k)}^0.$$
	Define the function $f(z)$ by
	$$f(z) := \sqrt{z} \exp\left[ \dfrac{\pi}{12} \left( \dfrac{A - Bm}{z} + \dfrac{(24n-1)z}{k^2} \right) \right].$$
	Then by the theory of Farey arcs, the arc on $\mathcal{K}^-$ from $0$ to $z_1(h,k)$ has length less than $\pi |z_1(h,k)| < \sqrt{2} \pi k N^{-1}$ and therefore $|z| < \sqrt{2} k N^{-1}$ on the path of integration. Recalling that $\textnormal{Re}(1/z) = 1$ and $0 < \textnormal{Re}(z) \leq 1$ on $\mathcal{K} \backslash \{0\}$,
	$$\left| \int_0^{z_1(h,k)} f(z) \, dz \right| \leq \dfrac{2^{3/4} \pi k^{3/2}}{N^{3/2}} \exp\left[ \dfrac{\pi}{12} \left( A + 24n-1 \right) \right] = O\left( k^{3/2} N^{-3/2} \right).$$
	A similar estimate holds for integrals from $z_2(h,k)$ to $0$, and therefore we have
	$$\int_{z_1(h,k)}^{z_2(h,k)} f(z) \, dz = \int_{\mathcal{K}^-} f(z) \, dz + O\left( k^{3/2} N^{-3/2} \right).$$
	It suffices now to evaluate the integral
	$$I := \int_{\mathcal{K}^-} \sqrt{z} \exp\left[ \dfrac{\pi}{12} \left( \dfrac{A - Bm}{z} + \dfrac{(24n-1)z}{k^2} \right) \right] \, dz.$$
	The substitution $w = z^{-1}$, $dw = - z^{-2} dz$ implies
	$$I = - \int_{1-i\infty}^{1+i\infty} w^{-5/2} \exp\left( \dfrac{\pi(A - Bm)}{12} w + \dfrac{\pi(24n-1)}{12k^2} w^{-1} \right) dw.$$
	Furthermore, by the substitution $s = c w$ for $c := \dfrac{\pi(A - Bm)}{12}$ we have
	$$I = - \left( \dfrac{\pi(A - Bm)}{12} \right)^{3/2} \int_{c-i\infty}^{c + i\infty} s^{-5/2} \exp\left( s + \left( \dfrac{\pi^2(A - Bm)(24n-1)}{144k^2} \right) \dfrac{1}{s} \right) ds.$$
	Since the classical modified $I$-Bessel function $I_{\frac 32}(z)$ satisfies the identity 
	$$I_{\frac 32}(z) = \dfrac{(z/2)^{3/2}}{2\pi i} \int_{c-i\infty}^{c + i\infty} s^{-5/2} \exp\left( s + \dfrac{z^2}{4s} \right) \, ds,$$
	setting $\dfrac{z}{2} = \sqrt{\dfrac{\pi^2(A - Bm)(24n-1)}{144k^2}} = \dfrac{\pi}{12k} \sqrt{(A - Bm)(24n-1)}$ yields
	\begin{align*}
		I &= \dfrac{2\pi}{i} \cdot \dfrac{k^{3/2} (A - Bm)^{3/4}}{(24n-1)^{3/4}} I_{\frac 32} \left( \dfrac{\pi}{6k} \sqrt{(A - Bm)(24n-1)} \right).
	\end{align*}
	Combining the estimation and the evaluation of $I$ completes the proof.
\end{proof}

From Proposition \ref{General I-Bessel}, we may complete the proof of the exact formula. The idea is that the error term in the evaluation of $A_t(n)$ introduced by the error in $I(t,h,k,n)$ vanishes as $N \to \infty$, and the resulting series converges.

\subsection{Completing the proof of Theorem \ref{Exact Formula}}

We have shown that
\begin{align*}
	A_t(n) = \sum\limits_{k=1}^N \dfrac{i}{k^2} \sum\limits_{\substack{0\leq h < k\\(h,k) =1}} e^{-\frac{2\pi i nh}{k}} I(t,h,k,n).
\end{align*}
By Proposition \ref{Reduce to Finite Sum} and Proposition \ref{General I-Bessel}, we obtain for every pair $h, k$ estimates for $I(t,h,k,n)$ with error term $O(k^{1/2} N^{-3/2})$. These exact formulas yield an estimate for $A_t(n)$ with error term $O(N^{-1/2})$. Therefore, as $N \to \infty$ we may replace $I(t,h,k,n)$ with these estimates and retain equality. That is,
\begin{align*}
	A_t(n) &= \sum\limits_{k=1}^\infty \dfrac{i}{k^2} \sum\limits_{\substack{0\leq h < k\\(h,k) =1}} e^{-\frac{2\pi i nh}{k}} I(t,h,k,n).
\end{align*}
This exact formula naturally splits into three sums according to the value of $k_0 \pmod{4}$. When $k_0$ is odd, the formula derived from Propositions \ref{Reduce to Finite Sum} and \ref{General I-Bessel} give the contribution
\begin{align*}
	S_1 := 2^{t/2} \sum_{\substack{k \geq 1 \\ k_0 \textnormal{ odd}}} \dfrac{2\pi}{k} \sum_{\substack{0 \leq h < k \\ \gcd(h,k) = 1}} e^{\frac{-2\pi i n h}{k}} w(t,h,k) \sum_{m = 0}^{M_{t,k}}& e^{\frac{2\pi i (4t_0)^* H m}{k}} c_1(t,h,k;m) \dfrac{(A - Bm)^{3/4}}{(24n-1)^{3/4}} \\ &\cdot I_{\frac 32} \left( \dfrac{\pi}{6k} \sqrt{(A - Bm)(24n-1)} \right),
\end{align*}
where $A = 1 - \dfrac{3}{4} \gcd(k,t)^2$, $B = \dfrac{6}{t_0}$, and $M_{t,k} = \left\lfloor \dfrac{t_0 (4 - 3\gcd(k,t)^2)}{24} \right\rfloor$. Noting that the sum is only nonempty when $k$ is odd and $\gcd(k,t) = 1$, in which case $k_0 = k$, $t_0 = t$, $M_{t,k} = \left\lfloor \dfrac{t}{24} \right\rfloor$, $A = 1/4$ and $B = 6/t$ we have
\begin{align*}
	S_1 &= \dfrac{2^{t/2}}{(24n-1)^{3/4}} \sum_{\substack{k \geq 1 \\ \gcd(k,2t) = 1}} \dfrac{\pi}{k} \sum_{\substack{0 \leq h < k \\ \gcd(h,k) = 1}} e^{\frac{-2\pi i n h}{k}} w(t,h,k) \\ & \cdot \sum_{m = 0}^{\lfloor \frac{t}{24} \rfloor} e^{\frac{2\pi i (4t)^* H m}{k}} c_1(t,h,k;m) \left( \dfrac{t - 24m}{t} \right)^{3/4} I_{\frac 32} \left( \dfrac{\pi}{12k} \sqrt{\dfrac{(t - 24m)(24n-1)}{t}} \right).
\end{align*}
The sums $S_2$, $S_3$ simplify similarly to
\begin{align*}
	S_2 &= \dfrac{2^{t/2}}{(24n-1)^{3/4}} \sum_{\substack{k \geq 1 \\ 2 || k_0}} \dfrac{2\pi}{k} \sum_{\substack{0 \leq h < k \\ \gcd(h,k) = 1}} e^{\frac{-2\pi i n h}{k}} w(t,h,k) \\ &\cdot \sum_{m = 0}^{\lfloor \frac{t_0\alpha_{t,k}}{12} \rfloor} e^{\frac{2\pi i (2^\dagger t_0^*) H m}{k}} c_2(t,h,k;m) \left( \dfrac{t_0\alpha_{t,k} - 12m}{t_0} \right)^{3/4} I_{\frac 32} \left( \dfrac{\pi}{6k} \sqrt{\dfrac{t_0\alpha_{t,k} - 12m)(24n-1)}{t_0}} \right)
\end{align*}
where $\alpha_{t,k} := 1 + 3\gcd(k,t)^2$ and
\begin{align*}
	S_3 &= \dfrac{1}{(24n-1)^{3/4}} \sum_{\substack{k \geq 1 \\ 4 | k_0}} \dfrac{2\pi}{k} \sum_{\substack{0 \leq h < k \\ \gcd(h,k) = 1}} e^{\frac{-2\pi i n h}{k}} w(t,h,k) \\ &\cdot \sum_{m = 0}^{\lfloor \frac{t_0}{24} \rfloor} e^{\frac{2\pi i (t_0^* H) m}{k}} c_3(t,h,k;m) \left(\dfrac{t_0 - 24m}{t_0} \right)^{3/4} I_{\frac 32} \left( \dfrac{\pi}{6k} \sqrt{\dfrac{(t_0 - 24m)(24n-1)}{t_0}} \right).
\end{align*}
As $A_t(n) = S_1 + S_2 + S_3$, the proof is complete.

\section{Certain Kloosterman sums}

We start by proving that the Kloosterman sum $$\displaystyle\sum_{\substack{0\leq h < k\\(h,k) =1}} \exp\bigg[\pi i \bigg(s(h,k) -\dfrac{2nh}{k}\bigg)\bigg]$$ is nonzero when $k$ is a power of $2$. Note that this Kloosterman sum can also  be rewritten as a sum of solutions modulo $24k$ to a quadratic equation as defined in the lemma below. 

\begin{lemma}
	Let $S_k(n)$ be the Kloosterman sum defined by
	\begin{equation} \label{Kloosterman Sum Definition} 
		S_k(n) := \dfrac{1}{2}\sqrt{\dfrac{k}{12}} \sum\limits_{\substack{x \pmod{24k} \\ x^2 \equiv -24n + 1 \pmod{24k}}} \chi_{12}(x) e\bigg( \dfrac{x}{12k} \bigg),
	\end{equation}
	where $\chi_{12}(x) = \legendre{12}{x}$ is the Kronecker symbol and $e(x) := e^{2\pi ix}$. If $k$ is a power of $2$, then $S_k(n) \not = 0$ for all positive integers $n$.
\end{lemma}

\begin{proof}
	Let $n \geq 1$, and let $k = 2^s$ for an integer $s \geq 0$. To show that $S_k(n) \not = 0$, we need only show that the summation given in (\ref{Kloosterman Sum Definition}) is nonzero. To evaluate this sum, consider the condition on $x$ that $x^2 \equiv -24n + 1$ (mod $24k$). Since $-24n + 1 \equiv 1$ (mod 4), $x^2 \equiv -24n+1$ (mod $2^{s+3}$) has exactly 4 incongruent solutions, and so the congruence $x^2 \equiv -24n+1$ (mod $24k$) has exactly 8 incongruent solutions. For any given solution $x$, we can see that all of $12k-x$, $12k+x$, and $24k-x$ are also solutions and are pairwise distinct.
	
	Now, let $x,y$ (mod $24k$) be solutions to $x^2 \equiv -24n+1$ (mod $24k$) such that $y$ is not congruent to any of $x$, $12k-x$, $12k+x$, or $24k-x$, so that the summation in (\ref{Kloosterman Sum Definition}) runs over the set of eight values $\{ \pm x, \pm (12k+x)\} \cup \{ \pm y, \pm (12k+y) \}$. Taking real parts in the summation in (\ref{Kloosterman Sum Definition}) yields the value $4a + 4b$, where $a = \chi_{12}(x)\cos{(\pi x / 6k)}$ and $b = \chi_{12}(y)\cos{(\pi y / 6k)}$. The equivalences known about $x$ and $y$ imply that $\chi_{12}(x), \chi_{12}(y) \not = 0$, and so the proof reduces to demonstrating that $|a| \not = |b|$. If $|a| = |b|$, then $x \equiv y$ (mod $6k$) must hold, so we may fix $y = 6k - x$. Since $x$ is odd, $y^2 = x^2 - 12kx + 36k^2 \equiv -24n+1 + 12k + 36k^2 \pmod{24k},$ and the equivalence modulo $6k$ of $x$ and $y$ implies $12k + 36k^2 \equiv 12k(1+3k) \equiv 0$ (mod $24k$). This requires that $1 + 3k$ be even, which is a contradiction since $k = 2^s$. Therefore, $|a| \not = |b|$, and it then follows that $S_k(n) \not = 0$ for all $n$.
\end{proof}

\begin{lemma}\label{Kloosterman sum nonzero}
	For $t>1$ a fixed positive integer, write $t = 2^s\ell$  with integers $s, \ell \geq 0$ such that $\ell$ is odd. Then $$\sum_{\substack{0< h < 2^{s+1} \\ h \text{ odd}}} w(t,h,2^{s+1}) e^{-\frac{\pi i n h}{2^s}} = S_{2^s}(n) \neq 0.$$
\end{lemma}

\begin{proof}
	By making use of the fact that $\omega_{dh, dk} = \omega_{h,k}$ for any integer $d$, it follows that $w(t,h,2^{s+1}) = \omega_{h,k}$, and therefore
	$$\sum_{\substack{0< h < 2^{s+1} \\ h \text{ odd}}} w(t,h,2^{s+1}) e^{-\frac{\pi i n h}{2^s}} = \sum_{\substack{0 < h < 2^{s+1} \\ h \textnormal{ odd}}} e^{\pi i \left( s(h,k) - n h / 2^s \right)} = S_{2^s}(n),$$
	which is non-vanishing by Lemma \ref{Kloosterman sum nonzero}.
\end{proof}

\section{Proof of Theorem \ref{C5 Even and Odd T Behavior}}

We are now ready to prove the main theorems. 

\begin{proposition}\label{ratios of  A_t(n): p(n)}
	Let $n$ be positive integers. Then for $t$ fixed, as $n \to \infty$ we have
	$$\dfrac{A_t(n)}{p(n)} \sim \begin{cases} (-1)^n / 2^{(t-1)/2} & \text{ if } 2 \nmid t, \\ 0 & \text{ if } 2 \mid t. \end{cases}$$
	Furthermore, $\dfrac{A_t(n)}{p(n)} \sim 0$ as $n, t \to \infty$.
\end{proposition}

\begin{proof}
	Recall that $p(n)$ satisfies $p(n) \sim \dfrac{2\pi}{(24n -1)^{3/4}}I_{\frac{3}{2}}\bigg(\dfrac{\pi\sqrt{24n -1}}{6}\bigg)$ as $n \to \infty$. Then by Corollary \hyperref[Dominating term of A_t]{\ref{Dominating term of A_t}}, as $n \rightarrow \infty$ we have  $$\dfrac{A_t(n)}{p(n)} \sim \dfrac{(1+ 3\cdot 4^s)^{3/4}}{2^{s+1+ \frac{t}{2}}}\cdot\dfrac{I_{\frac{3}{2}}\bigg(\dfrac{\pi}{6}\sqrt{\bigg(\dfrac{1}{4^{s+1}}+\dfrac{3}{4}\bigg)(24n -1)}\bigg)}{I_{\frac{3}{2}}\bigg(\dfrac{\pi\sqrt{24n-1}}{6}\bigg)}\sum_{\substack{0< h < 2^{s+1}\\h \textnormal{ odd}}} w(t,h,2^{s+1})e^{-\frac{\pi i nh}{2^s}}.$$
	When $s > 0$ and $n \rightarrow \infty$, the asymptotic behavior of $I_{3/2}$ implies that $$\dfrac{I_{\frac{3}{2}}\bigg(\dfrac{\pi}{6}\sqrt{\bigg(\dfrac{1}{4^{s+1}}+\dfrac{3}{4}\bigg)(24n -1)}\bigg)}{I_{\frac{3}{2}}\bigg(\dfrac{\pi\sqrt{24n-1}}{6}\bigg)} \sim 0.$$
	Therefore when $t$ is even, $\dfrac{A_t(n)}{p(n)} \sim 0$ as $n \rightarrow \infty$. When $s = 0$,  $\dfrac{A_t(n)}{p(n)} \sim (-1)^n 2^{\frac{-t+1}{2}}$ as $n \to \infty$.
\end{proof}

\begin{proof}[Proof of Theorem \ref{C5 Even and Odd T Behavior}]
	By Proposition \ref{ratios of  A_t(n): p(n)}, we see that
	\begin{align*}
		\delta_t^e(n) - \delta_t^o(n) \to \begin{cases}
			\dfrac{(-1)^n}{2^{(t-1)/2}} & \textnormal{if } t \textnormal{ odd}, \\
			0 & \textnormal{if } t \textnormal{ even}.
		\end{cases}
	\end{align*}
	Since $\delta_t^e(n) + \delta_t^o(n) = 1$ by definition, the result follows by solving for $\delta_t^e(n)$ and $\delta_t^e(n)$.
\end{proof}

\section{Proof of Theorem \ref{C5 Distribution property}}

By Corollary \hyperref[Dominating term of A_t]{\ref{Dominating term of A_t}}, we have $$A_t(n) \sim \displaystyle\dfrac{\pi}{2^{s+\frac{t}{2}}}\bigg(\dfrac{1+ 3\cdot 4^s}{24n -1}\bigg)^\frac{3}{4}I_{\frac{3}{2}}\bigg(\dfrac{\pi\sqrt{(1 + 3\cdot 4^s)(24n -1)}}{6\cdot 2^{s+1}}\bigg)\sum_{\substack{0< h < 2^{s+1}\\h \textnormal{ odd}}}w(t,h,2^{s+1})e^{-\frac{\pi i nh}{2^s}},$$ whose sign is determined by the summation over $h$, which on inspection is periodic in $n$ with period $2^{s+1}$.  In particular, the period is 2 when $s = 0$ which implies the $A_t(n)$ has alternating sign when $t$ is odd as $n \rightarrow \infty$.

\section{Reflections}

The surprising nature of this result justifies some reflection. Theorem \ref{C5 Even and Odd T Behavior} differs from the naive expectation of equidistribution in two ways. Not only does equidistribution frequently fail, but there are multiple limiting values when $t$ is odd. Since the distribution properties correspond to the size of $A_t(n)$ in relation to $p(n)$, the proof of Theorem \ref{Exact Formula} reveals on an analytic level the source of these discrepancies. Namely, the $I$-Bessel functions in Theorem \ref{Exact Formula} control whether equidistribution holds and when $t$ is odd the Kloosterman sums arising from $w(t,h,k)$ control the relationship between the parity of $n$ and the sign of $A_t(n)$. All of these details can be read directly off of Theorem \ref{Exact Formula}. However, the circle method does not provide insight into combinatorial explanations of this phenomena, and therefore we leave this question open.

The motivation behind this proof comes from the Nekrasov-Okounkov formula and the applications of this formula made by Han in \cite{Han10} which connect hook numbers to the expansions of various modular forms. In the context of this connection, the problem of the distribution in parity of $\# \mathcal{H}_t(\lambda)$ is translated into a question about asymptotic formulas for the coefficients of a certain modular form, or at least a $q$-series which is closely related to a modular form. This study has made use of only a microscopic portion of this world of connections, and therefore it is natural to study further problems about $t$-hooks through the lens of modular forms. In particular, in Chapter \ref{C6}, we will study the more difficult question about the distribution of $\# \mathcal{H}_t(\lambda)$ modulo odd primes.

\sglsp

\chapter{Distribution of $t$-hooks and Betti Numbers} \label{C6}
\thispagestyle{myheadings}

\dblsp
\vspace*{-.65cm}

The purpose of this chapter is to prove Theorems \ref{C6 t-hook Asymptotic}, \ref{C6 Vanishing}, \ref{C6 Betti Asymptotic} and Corollary \ref{C6 Betti Distribution}. This is joint work with Kathrin Bringmann, Joshua Males, and Ken Ono.

\section{Hook number generating functions}

Here we derive the generating functions for the modular $t$-hook functions
$p_t(a,b;n)$. To this end, we recall the following beautiful formula of Han that he derived in his work on extensions of the celebrated Nekrasov--Okounkov formula\footnote{This formula was also obtained by Westbury (see Proposition 6.1 and 6.2 of \cite{Wes06}).} (see (6.12) of \cite{NO06}) with $w \in \CC$:
$$
\sum_{\lambda \in \mathcal{P}} q^{|\lambda|} \prod_{h\in \mathcal{H}(\lambda)} \left(1-\frac{w}{h^2}\right)
=\prod_{n=1}^{\infty}\left(1-q^n\right)^{w-1}.
$$
Here $\mathcal{P}$ denotes the set of all integer partitions, including the empty partition, and $\mathcal{H}(\lambda)$ denotes the multiset of hook lengths for $\lambda.$
Han \cite{Han10} proved the following beautiful identity for the generating function for $t$-hooks in partitions
\begin{equation*}
	H_t(\xi;q):=\sum_{\lambda \in \mathcal{P}} \xi^{\# \mathcal{H}_t(\lambda)}q^{|\lambda|}.
\end{equation*}

\begin{theorem}{\text {\rm (Corollary 5.1 of \cite{Han10})}}\label{HanFunction}
	As formal power series, we have
	$$
	H_t(\xi;q)=\frac{1}{F_2(\xi;q^t)^t}\prod_{n=1}^{\infty}
	\frac{\left(1-q^{tn}\right)^t}{1-q^n}.
	$$
\end{theorem}

As a corollary, we obtain the following generating function for $p_t(a,b;n).$
\begin{corollary}\label{ptabGenFunctions}
	If $t>1$ and $0\leq a<b$, then as formal power series we have
	\begin{equation*}\label{Orthogonality}
		H_t(a,b;q):=\sum_{n=0}^{\infty}p_t(a,b;n)q^n=\frac{1}{b} \sum_{r=0}^{b-1} \zeta_b^{-ar}H_t\left(\zeta_b^r;q\right),
	\end{equation*}
	where $\zeta_b:=e^{\frac{2\pi i}b}.$
\end{corollary}
\begin{proof}
	We have that
	\begin{displaymath}
		\begin{split}
			\frac{1}{b} \sum_{r=0}^{b-1} \zeta_b^{-ar} H_t(\zeta_b^r;q)&=
			\frac{1}{b} \sum_{\lambda \in \mathcal{P}}q^{|\lambda|} \sum_{r=0}^{b-1}\zeta_b^{\left(\#\mathcal{H}_t(\lambda)-a\right)r}=H_t(a,b;q).
		\end{split}
	\end{displaymath}
	This completes the proof.
\end{proof}

The dependence of $H_t(\xi;q)$ on $F_2(\xi;q^t)$ enables us to compute asymptotic behavior of $H_t(\xi;q)$ using Theorem \ref{Theorem1} (2) and, by Corollary \ref{ptabGenFunctions}, the asymptotic behavior of $H_t(a,b;q)$.

\section{Proof of Theorem \ref{C6 Vanishing}}

Here we prove Theorem~\ref{C6 Vanishing}. We first consider the case (1), where $\ell$ is an odd prime. We consider the generating function, using Corollary \ref{ptabGenFunctions}
$$
H_2(a_1,\ell ;q)=\sum_{n=0}^{\infty}p_2(a_1,\ell;n)q^n=\frac{1}{\ell} \sum_{r_1=0}^{\ell-1} \zeta_\ell^{-a_1 r_1}H_2\left(\zeta_\ell^{r_1};q\right).
$$
Applying again orthogonality of roots of unity, keeping only those terms $a_2\pmod \ell$,  where
$a_2\in \{0, 1,\dots, \ell-1\}$, we find that
$$
\sum_{n=0}^{\infty} p_2(a_1,\ell;\ell n+a_2)q^{\ell n +a_2}=
\frac{1}{\ell^2}\sum_{r_1, r_2\pmod \ell}\zeta_{\ell}^{-a_1 r_1 -a_2 r_2} H_2\left(\zeta_{\ell}^{r_1};\zeta_{\ell}^{r_2}q\right).
$$
Making use of the definition of $H_t(\xi;q)$, 
if we define $\mathcal{B}_2(q)$ and $\mathcal{C}_2(q)$ by
\begin{equation}\label{qidentities}
	\mathcal{B}_2(q)=\sum_{n=0}^{\infty}b_2(n)q^n:=\prod_{n=1}^{\infty}\frac{1}{\left(1-q^n\right)^2} \ \ \ \
	{\text {\rm and}}\ \ \ \ 
	\mathcal{C}_2(q):=\prod_{n=1}^{\infty}\frac{\left(1-q^{2n}\right)^2}{1-q^n},
\end{equation}
then we have 
$$
\sum_{\substack{n\geq 0 \\n\equiv a_2\pmod \ell}}p_2(a_1,\ell;n)q^n=
\frac{1}{\ell^2}\sum_{r_1, r_2\pmod \ell}\zeta_{\ell}^{-a_1 r_1 -a_2 r_2}
\mathcal{B}_2\left(\zeta_{\ell}^{r_1+2r_2}q^2\right) \mathcal{C}_2\left(\zeta_{\ell}^{r_2} q\right).
$$
Thanks to the classical identity of Jacobi
\begin{equation*}
	\mathcal{C}_2(q)=\sum_{k=0}^{\infty}q^{\frac{k(k+1)}{2}},
\end{equation*}
for $N\equiv a_2\pmod{\ell}$, we find that 
\begin{align}\label{key}
	p_2(a_1,\ell;N)&=\frac{1}{\ell^2}\sum_{r_1, r_2\pmod \ell}\zeta_{\ell}^{-a_1 r_1-a_2 r_2}
	\sum_{\substack{k,m\geq 0\\
			2m+\frac{k(k+1)}{2}=N}} b_2(m)\zeta_{\ell}^{(r_1+2r_2)m+r_2\frac{k(k+1)}{2}}
	\nonumber \\
	&=\sum_{\substack{m\equiv a_1\pmod \ell\\
			2m+\frac{k(k+1)}{2}=N}} b_2(m),
\end{align}
by making the linear change of variables $r_1\mapsto r_1-2r_2$ and again using orthogonality of roots of unity.
This then requires the solvability of the congruence $a_2-2a_1\equiv \frac{k(k+1)}{2}\pmod \ell.$
Completing the square produces the quadratic residue condition which prohibits this solvability, and hence
completes the proof of (1).

The proof of (2) follows similarly, with $\ell$ replaced by $\ell^2$ for primes $\ell\equiv 2\pmod 3.$ The functions in (\ref{qidentities})
are replaced with
$$
\mathcal{B}_3(q)=\sum_{n=0}^{\infty}b_3(n)q^n:=\prod_{n=1}^{\infty}\frac{1}{\left(1-q^n\right)^3} \ \ \ \ 
{\text {\rm and}}\ \ \ \ \mathcal{C}_3(q):=\prod_{n=1}^{\infty}\frac{\left(1-q^{3n}\right)^3}{1-q^n}.
$$
It is well-known that (for example, see Section 3 of \cite{GO96} or \cite[Lemma 2.5]{HO11}),
$$
\mathcal{C}_3(q)=:\sum_{n=0}^{\infty}c_3(n)q^n=\sum_{n=0}^{\infty}\sum_{d\mid (3n+1)}\legendre{d}{3}q^n.
$$
For primes $\ell\equiv 2\pmod 3$, this implies that  $c_3(\ell^2 n+a)=0$ for every positive integer $n$, whenever $\ord_{\ell}(3a+1)=1$.
For example, this means that $c_3(4n+3)=0$ if $\ell=2$.

Let $0\leq a_1, a_2<\ell^2$. In direct analog with (\ref{key}), a calculation reveals that non-vanishing for $N\equiv a_2\pmod {\ell^2}$ relies  on sums of the form
$$
\sum_{\substack{m\equiv a_1\pmod{\ell^2}\\
		3m+k=N}} b_3(m)c_3(k).
$$
If  $\ord_{\ell}(3a+1)=1$ and
$a_2-3a_1\equiv a\pmod{\ell^2}$, then $p_3(a_1,\ell^2; \ell^2+a)=0.$
This is claim (2).

\section{Asymptotic methods}

The Hardy--Ramanujan asymptotic formula given in \eqref{Hardy-Ramanujan Asymptotic} marked the birth of the Circle Method. Its proof relied  on the modular transformation properties of {\it Dedekind's eta-function} $\eta(\tau):= q^{\frac{1}{24}} \prod_{n=1}^{\infty}(1-q^n),$ where
$q:=e^{2\pi i \tau}$ (for example, see Chapter 1 of \cite{Ono04}). Their work has been thoroughly developed in the theory of modular forms and harmonic Maass forms (for example, see Chapter~15 of \cite{BFOR17}), and
has been generalized beyond this setting in papers by  Grosswald, Meinardus, Richmond, Roth, and Szekeres \cite{Gro58, Mei54, Ric76, RS54}, to name a few.

\subsection{Statement of the results}

Generalizing the infinite product which defines $\eta,$ we consider the ubiquitous $q$-infinite products
\begin{align*}
	F_1(\xi; q) &:=\prod_{n=1}^{\infty}\left(1-\xi q^n\right), \\ F_2(\xi; q) &:=\prod_{n=1}^{\infty}\left(1-(\xi q)^n\right),\\ F_3(\xi;q) &:=\prod_{n=1}^{\infty} \left(1-\xi^{-1}(\xi q)^n\right).
\end{align*}
These infinite products are common as factors of generating functions in combinatorics, number theory, and representation theory.
We  obtain the asymptotic properties for $F_1(\xi;q), F_2(\xi;q),$ and $F_3(\xi;q),$ where $\xi$ is a root of unity, which are generally required for implementing
the Circle Method to such generating functions. This result is of independent interest.

To make this precise, we recall \textit{Lerch's transcendent}
\begin{align*}
	\Phi(z,s,a):=\sum_{n=0}^\infty \frac{z^n}{(n+a)^s}.
\end{align*}
Moreover, for coprime $h,k\in\NN$ we define
\begin{align}\label{eqn: defn omega_h,k}
	\omega_{h,k}:=\exp(\pi i \cdot s(h,k)),
\end{align}
using the \emph{Dedekind sum}
\begin{align*}
	s(h,k):=\sum_{\mu \pmod k} \left(\left(\frac{\mu}{k}\right)\right)\left(\left(\frac{h\mu}{k}\right)\right).
\end{align*}
Here we use the standard notation
\begin{align*}
	((x)):=\begin{cases} x-\lfloor x \rfloor-\frac{1}{2} & \text{if} \ x\in \mathbb R \setminus \mathbb Z, \\ 0 & \text{if} \ x \in \mathbb Z. \end{cases}
\end{align*}
For arbitrary positive integers $m$ and $n$, we define $\omega_{m,n} := \omega_{\frac{m}{\gcd(m,n)}, \frac{n}{\gcd(m,n)}}$. Note that $s(h,k)$ only depends on $h\pmod{k}$ and that $s(0,1)=0$. Moreover, we let
\begin{align} \label{LambdaEQ} 
	\lambda_{t,a,b,h,k} := \gcd(k,t)  \begin{cases} 1 & \text{if } k=1 \text{ or } \lp k > 1 \text{ and } b \centernot | \frac{k}{\gcd(k,t)}\rp, \\ b & \text{if } b | \frac{k}{\gcd(k,t)} \text{ and } \frac{ht}{\gcd(k,t)} + a\frac{k}{b \gcd(k,t)} \not \equiv 0 \pmod{b}, \\ b^2 & \text{if } b | \frac{k}{\gcd(k,t)} \text{ and } \frac{ht}{\gcd(k,t)} +a \frac{k}{b \gcd(k,t)} \equiv 0 \pmod{b}. \end{cases}
\end{align}

For $0\leq \theta < \frac{\pi}{2}$, we define the domain 
\begin{align}\label{eqn: defn D_theta}
	D_{\theta} := \left\{ z=re^{i\alpha} \colon r \geq 0 \text{ and } |\alpha| \leq \theta \right\}.
\end{align}

\begin{theorem}\label{Theorem1} 
	Assume the notation above. For $b>0$, let $\xi$ be a primitive $b$-th root of unity, then the following are true. 
	\begin{enumerate}[leftmargin=*]
		\item[\rm (1)] As $z \to 0$ in $D_\theta$ we have 
		\begin{align*}
			F_{1}\left(\xi;e^{-z}\right)  =\frac{1}{\sqrt{1-\xi}} \, e^{-\frac{\xi\Phi(\xi,2,1)}{z}}\lp 1+O\left(|z|\right) \rp.
		\end{align*}
		
		\item[\rm (2)] Suppose that $b$ is an odd prime, and let $\xi = e^{\frac{2\pi i a}{b}}$, $q = e^{\frac{2\pi i}{k}(h + iz)}$ for $0 \leq h < k$ with $\gcd(h,k) = 1$, $t \in \NN$, and $z \in \CC$ with $\mathrm{Re}(z) > 0$. 
		Then as $z \to 0$ we have 
		$$F_2\left(\xi;q^t\right) \sim \omega_{\frac{hbt+ak}{\lambda_{t,a,b,h,k}}, \frac{kb}{\lambda_{t,a,b,h,k}}}^{-1} \left(\frac{\lambda_{t,a,b,h,k}}{tbz}\right)^{\frac 12} e^{-\frac{\pi \lambda_{t,a,b,h,k}^2}{12 b^2 ktz}}.$$
		
		\item[\rm (3)] As $z\to 0$ in $D_\theta$, we have
		\begin{align*}
			F_3\left(\xi;e^{-z}\right)= \frac{\sqrt{2\pi} \left(b^2z\right)^{\frac 12-\frac 1b}}{\Gamma\left(\frac{1}{b}\right)}
			\prod_{j=1}^{b-1}\frac{1}{(1-\xi^j)^{\frac jb}}
			e^{-\frac{\pi^2}{6b^2z}}\lp 1+  O\left(|z|\right) \rp.
		\end{align*} 
	\end{enumerate}
\end{theorem}

\begin{remark}
	If $\xi=1$ and $q=e^{2\pi i \tau}$, then we have $$F_1(1;q)=F_2(1;q)=F_3(1;q)=q^{-\frac{1}{24}}\eta(\tau).$$
	Asymptotic properties in this case are well-known consequences of the modularity of $\eta(\tau).$
\end{remark}

\subsection{An integral evaluation}

We require the following integral evaluation.
\begin{lemma}\label{lem:int}
	We have for $N\in\RR^+$
	\begin{multline*}
		\int_0^\infty\left(\frac{e^{-x}}{x\left(1-e^{Nx}\right)}-\frac{1}{Nx^2}+\left(\frac 1N-\frac 12\right)\frac{e^{-x}}{x} \right)dx
		\\=\log\left(\Gamma\left(\frac 1N\right) \right) +\left(\frac 12-\frac 1N\right) \log\left(\frac 1N\right)-\frac 12\log(2\pi). 
	\end{multline*}
\end{lemma}

\begin{proof}
	Making the change of variables $x\mapsto\frac xN$, the left-hand side equals
	\begin{equation*}
		\int_0^\infty\left(\frac{e^{-\frac{x}{N}}}{x\left(1-e^{-x}\right)} - \frac{1}{x^2} +\left(\frac 1N-\frac 12\right)\frac{e^{-\frac {1}N}}{x}  \right) dx.
	\end{equation*}
	Now write
	\begin{equation*}
		\frac{1}{x \left(1-e^{-x} \right)}=\frac 1x +\frac{1}{x \left(e^x-1 \right)}.
	\end{equation*}
	Thus the integral becomes
	\begin{align*}
		\int_0^\infty\left(\frac{1}{e^x-1}+\frac 12-\frac 1x\right)&\frac{e^{-\frac{x}{N}}}{x}dx
		\\ &+ \int_0^\infty\left(\frac{e^{-\frac{x}{N}}}{x} - \frac{1}{x^2} +\left(\frac 1N-\frac 12\right)\frac{e^{-\frac {x}N}}{x}  -\frac{e^{-\frac{x}{N}}}{2x}+\frac{e^{-\frac{x}{N}}}{x^2}\right) dx.
	\end{align*}
	We evaluate the second integral as $-\frac{1}{N}$. The claim now follows, using Binet's first integral formula (see 12.31 of \cite{WW96}).
\end{proof}

\subsection{Proof of Theorem~\ref{Theorem1}}

We employ the generalized Euler--Maclaurin summation, i.e. Proposition \ref{C3 Euler-Maclaurin Sufficient Decay}, to prove Theorem~\ref{Theorem1} (1) and (3); for part (2) we use modularity.

\subsubsection{Proof of Theorem~\ref{Theorem1}~\normalfont{(1)}}

Let $|z|<1$. Taking logarithms, we have 
\begin{align*}
	G_{\xi}\left(e^{-z}\right)&:=\operatorname{Log} \left(F_{1}\left(\xi;e^{-z}\right) \right)
	=-z\sum_{j=1}^b \xi^j \sum_{m=0}^{\infty}f\left(\left(m+\frac j b\right)bz\right),
\end{align*}
where
\begin{equation*}
	f(z):=\frac{e^{-z}}{z\left(1-e^{-z}\right)}=\frac{1}{z^2}-\frac{1}{2z} +\sum_{n=0}^\infty \frac{B_{n+2}}{(n+2)!}z^n.
\end{equation*}
By Proposition \ref{C3 Euler-Maclaurin Sufficient Decay}, it follows that
\begin{align*}
	\sum_{m=0}^\infty f \left( \left( m+\frac{j}{b} \right)bz\right) =\frac{\zeta\left(2,\frac j b\right)}{b^2z^2}+\frac{I_{f,1}^*}{bz} +{\frac{1}{2bz}}\left(\Log \left( {bz}\right) +\psi \left(\frac{j}{b}\right)+\gamma \right)
	+O(1).
\end{align*}
Therefore, we find that
\begin{align*}
	G_{\xi} \left(e^{-z}\right) = -\frac{1}{b^2z} \sum_{j=1}^b \xi^j \zeta\left(2,\frac jb\right) &- \frac{I_{f,1}^*}{b} \sum_{j=1}^b \xi^{j} \\ &-\frac{1}{2b} \sum_{j=1}^b \xi^j \left(\Log\left(bz\right) +\psi\left(\frac jb\right)+\gamma\right) +O(|z|).
\end{align*}
Now note that $\sum_{j=1}^b \xi^{j}=0.$
Moreover, we require the identity \cite[p. 39]{Cam66}  (correcting a minus sign and erroneous $k$ on the right-hand side)
\begin{align}\label{digam}
	\sum_{j=1}^b \psi \left( \frac{j}{b} \right) \xi^{j}=b\operatorname{Log }\left( 1-\xi \right).
\end{align}
Combining these observations, we obtain
\begin{align*}
	G_{\xi}\left( e^{-z} \right)=-\frac{1}{b^2z} \sum_{j=1}^b \xi^{j} \zeta\left( 2,\frac{j}{b} \right)-\frac{1}{2} \operatorname{Log }(1-\xi) + O\left(|z|\right).
\end{align*}
After noting that
\begin{align*}
	\sum_{j=1}^b \xi^{j} \zeta \left(2,\frac{j}{b} \right)&=b^2 \xi\Phi(\xi,2,1),
\end{align*}
the claim follows by exponentiation. \qed

\subsubsection{Proof of Theorem~\ref{Theorem1}~\normalfont{(2)}}
Note that 
\[
F_2\left(\xi;q^t\right) = \left(\xi q^t; \xi q^t\right)_\infty,
\]
where $(a;q)_\infty := \prod_{j=1}^\infty (1 - a q^{j-1})$. The classical modular transformation law for the Dedekind $\eta$-function  (see 5.8.1 of \cite{CS17}) along with the identity $\eta(\tau) = q^{ \frac{1}{24}} (q;q)_\infty$ implies that
\begin{align}\label{Eqn: transform usual pochham}
	(q;q)_\infty = \omega_{h,k}^{-1} z^{-\frac 12} e^{\frac{\pi}{12k}\left( z - \frac{1}{z} \right)} (q_1;q_1)_\infty,
\end{align}
where $q_1 := e^{\frac{2\pi i}{k}( h' + \frac{i}{z})}$ where $0 \leq h' < k$ is defined by $h h' \equiv -1 \pmod{k}$ and $\omega_{h,k}$ is defined as in \eqref{eqn: defn omega_h,k}. In particular, this implies that
\begin{equation} \label{Eta asymptotic}
	(q;q)_\infty \sim \omega_{h,k}^{-1} z^{-\frac 12} e^{-\frac{\pi}{12kz}}
\end{equation}
as $z\rightarrow0$ with $\mathrm{Re}\lp z \rp>0$.
Now, by using the definitions of $\xi, q$ given in the statement of Theorem 2.1 (2) we have
\[
\xi q^t = e^{\frac{2\pi i}{kb}\left( hbt + ak + itbz\right)}.
\]
We claim that $\lambda_{t,a,b,h,k}$ as defined in \eqref{LambdaEQ} satisfies $\lambda_{t,a,b,h,k} = \gcd(kb, hbt + ak)$. If $k=1$, then the claim is clear, and so we assume that $k > 1$. Write $k = \gcd(k,t) k_1$ and $t = \gcd(k,t) t_1$. Then we have
\[
\gcd(kb, hbt + ak) = \gcd(k,t) \gcd(k_1 b, hbt_1 + ak_1).
\]
Noting that $\gcd(k_1,b)$ divides each of $k_1b, hbt_1$, and $ak_1$, it follows that
\[
\gcd(kb, hbt + ak) = \gcd(k,t) \gcd(k_1, b) \gcd\left( \frac{k_1 b}{\gcd(k_1, b)}, \frac{hbt_1}{\gcd(k_1, b)} + a\frac{k_1}{\gcd(k_1,b)} \right).
\]
Note that, since $b$ is prime, $\gcd(k_1, b) \in \{ 1, b \}$. If $\gcd(k_1,b) = 1$, then
\[
\gcd(k_1 b, hbt_1 + ak_1) = \gcd(k_1, hbt_1) \gcd(b, ak_1) = 1.
\]

If on the other hand $\gcd(k_1,b) = b$, then write $k_1 = b^\kappa k_2$ with $\gcd(k_2, b) = 1$. Then we have
\begin{align*}
	\gcd\left(k_1, ht_1 + a\frac{k_1}{b}\right) &= \gcd\left(b^\kappa k_2, ht_1 + a k_2 b^{\kappa-1}\right)
	= \gcd\left(b^\kappa, ht_1 + ak_2 b^{\kappa-1}\right) \gcd(k_2, ht_1)\\
	&= \gcd\left(b^\kappa, ht_1 + a k_2 b^{\kappa-1}\right).
\end{align*}
If $\kappa > 1$, then $\gcd(b^\kappa, ht_1 + ak_2 b^{\kappa-1}) = 1$ since $\gcd(b, ht_1) = 1$. If $\kappa = 1$, then we are left with $\gcd(b, ht_1 + ak_2)$. Therefore, we obtain
\[
\gcd(kb, hbt + ak) = \gcd(k,t) \begin{cases} 1 & \text{if } b \centernot | \frac{k}{\gcd(k,t)}, \\ b & \text{if } b | \frac{k}{\gcd(k,t)} \text{ and } \frac{ht}{\gcd(k,t)} + a\frac{k}{b \gcd(k,t)} \not \equiv 0 \pmod{b}, \\ b^2 & \text{if } b | \frac{k}{\gcd(k,t)} \text{ and } \frac{ht}{\gcd(k,t)} + a\frac{k}{b \gcd(k,t)} \equiv 0 \pmod{b}, \end{cases}
\]
which is equal to $\lambda_{t,a,b,h,k}$.

It follows that $\gcd(\frac{kb}{\lambda_{t,a,b,h,k}},\frac{hbt+ak}{\lambda_{t,a,b,h,k}}) = 1$. Therefore, by making the replacements $h \mapsto \frac{hbt+ak}{\lambda_{t,a,b,h,k}}$, $k \mapsto \frac{kb}{\lambda_{t,a,b,h,k}}$, and $z \mapsto \frac{tbz}{\lambda_{t,a,b,h,k}}$ in \eqref{Eta asymptotic}, the result follows.\qed

\subsubsection{Proof of Theorem \ref{Theorem1}~\normalfont{(3)}}
Again assume that $|z|<1$. Writing
\begin{equation*}
	F_3(\xi;q)=\prod_{j=1}^b\prod_{n=0}^{\infty}\left(1-\xi^{j-1}q^{bn+j}\right),
\end{equation*}
we compute
\begin{align*}
	\operatorname{Log}\left(F_3\left(\xi;e^{-z}\right)\right)=-z\sum_{1\leq j,r \leq b} \xi^{(j-1)r}\sum_{m=0}^\infty f_j\left(\left( m+\frac{r}{b}\right) bz\right),
\end{align*}
where $f_j(z):=\frac{e^{-jz}}{z(1-e^{-bz})}$.
By Proposition \ref{C3 Euler-Maclaurin Sufficient Decay}, we obtain 
\begin{equation*}\label{eqn:obtain}
	\sum_{m=0}^{\infty}f_j\left(\left(m+\frac r b\right)bz \right)
	\sim
	\frac{\zeta\left(2,\frac r b\right)}{b^3z^2}+
	\frac{I_{f_{j,1}}^*}{bz}+\frac{B_1\left(\frac{j}{b}\right)}{bz}\left(\Log\left({bz}\right)+\psi \left(\frac{r}{b}\right)+\gamma\right)+O(1)
\end{equation*}
The first term contributes $-\frac{\pi^2}{6b^2z}$.
By Lemma \ref{lem:int}, the second term contributes 
\begin{align*}
	-\frac 1b \sum_{j=1}^{b}I_{f_{j,1}}^*\sum_{r=1}^{b}\xi^{(j-1)r}& =-I_{f_{1,1}}^*
	=-\log\left(\Gamma\left(\frac{1}{b}\right)\right) - \left(\frac{1}{2}-\frac{1}{b}\right)\log\left(\frac{1}{b}\right)+\frac 12\log(2\pi) \\
	&= \log\left( \frac{b^{\frac{1}{2} -\frac{1}{b}} (2\pi)^{\frac{1}{2}} }{\Gamma\left(\frac{1}{b}\right)} \right).
\end{align*}
Next we evaluate
\begin{align*}
	-\frac{1}{b}\left(\Log\left({bz}\right)+\gamma\right)\sum_{1\leq j\leq b}B_1\left(\frac{j}{b}\right)  \sum_{1\leq r\leq b} \xi^{(j-1)r}=-B_1\left(\frac{1}{b}\right)\left(\Log\left({bz}\right)+\gamma\right).
\end{align*}
\indent Finally we are left to compute
\begin{align*}
	-\frac{1}{b}\sum_{1\leq j,r \leq b} \xi^{(j-1)r} \left(\frac{j}{b}-\frac{1}{2}\right)\psi\left(\frac{r}{b}\right)=-\frac{1}{b}\sum_{\substack{0\leq j\leq b-1 \\ 1\leq r \leq b}} \xi^{jr}\left(\frac{j}{b}+\frac{1}{b}-\frac{1}{2}\right)\psi\left(\frac{r}{b}\right).
\end{align*}
The $(\frac{1}{b}-\frac{1}{2})$-term yields $\gamma(\frac{1}{b}-\frac{1}{2})$. Thanks to \eqref{digam}, the $\frac{j}{b}$ term contributes
\begin{align*}
	-\frac{1}{b^2} \sum_{0\leq j\leq b-1} j \sum_{1\leq r \leq b} \psi \left(\frac{r}{b}\right) \xi^{jr}=-\frac{1}{b}\sum_{1\leq j\leq b-1}j\operatorname{Log }\left(1-\xi^j\right).
\end{align*}
Combining these observations yields that
\begin{align*}
	\operatorname{Log}\left(F_3\left(\xi;e^{-z}\right)\right) = \log\left( \frac{b^{\frac{1}{2} -\frac{1}{b}} (2\pi)^{\frac{1}{2}} }{\Gamma\left(\frac{1}{b}\right)} \right) &- \frac{\pi^2}{6b^2z} -B_1\left(\frac{1}{b}\right)\Log\left({bz}\right) \\ &- \sum_{1\leq j\leq b-1} \frac{j}{b} \operatorname{Log }\left(1-\xi^j\right) + O\left(|z|\right).
\end{align*}
Exponentiating gives the desired claim.  \qed

\section{Evaluation of Kloosterman sums}

The proof of Theorem \ref{C6 t-hook Asymptotic} relies on the arithmetic of the Kloosterman sums
\begin{equation*}\label{KloostermanSumDfn}
	K(a,b,t;n) := \sum_{h=1}^{b-1} \frac{\omega_{h,b}}{\omega_{th,b}^t} \zeta_b^{(at-n)h},
\end{equation*}
where $b$ is an odd prime, and $s \geq 1$, $t > 1$ are integers. We evaluate this sum if $t$ is coprime to $b$. We start by computing $\omega_{h,b} \omega_{th,b}^{-t}$.

\begin{proposition} \label{Dedekind Simplification}
	Let $b$ be an odd prime, $h$, $t$ integers coprime to $b$, and let $\omega_{h,k}$ be defined by \eqref{eqn: defn omega_h,k}. Then we have
	\begin{align*}
		\frac{\omega_{h,b}}{\omega_{th,b}^t} = \lp \frac{h}{b} \rp \lp \frac{th}{b} \rp^t e^{\pi i \frac{(1-t)(b-1)}{4}} e^{\frac{2\pi i}{b} \frac{1}{24} \left(1-t^2\right)\left(1-b^2\right)h}.
	\end{align*}
\end{proposition}

\begin{proof}
	The proof of this proposition uses the $\eta$-multiplier, which we label $\psi$. Theorem 5.8.1 of \cite{CS17} yields that
	for $\lp \begin{smallmatrix} \alpha & \beta  \\ \gamma  & \delta  \end{smallmatrix} \rp \in \text{SL}_2(\ZZ)$  with $\gamma  > 0$ odd, we have $$\psi \begin{pmatrix} \alpha & \beta  \\ \gamma  & \delta  \end{pmatrix} = \left( \frac{\delta }{\gamma } \right) e^{\frac{\pi i}{12} \left( (\alpha +\delta )\gamma  - \beta \delta \left(\gamma ^2-1\right) - 3\gamma  \right)}.$$
	We also have from formula (57b) of \cite{GR72} that for $\lp\begin{smallmatrix} \alpha & \beta  \\ \gamma  & \delta  \end{smallmatrix}\rp \in \text{SL}_2(\ZZ)$
	$$
	\psi \begin{pmatrix} \alpha & \beta  \\ \gamma  & \delta  \end{pmatrix} = e^{\pi i \left( \frac{\alpha+\delta }{12\gamma } - \frac{1}{4} \right)}\omega_{\delta,\gamma}^{-1} .
	$$
	By letting $\delta  = h$, $\gamma  = b,$ we obtain
	\begin{align*}
		\omega_{h,b} = \lp \frac{h}{b} \rp e^{\pi i \lp \frac{1}{12b}(\alpha+h - \beta hb)\left(1-b^2\right) + \frac{b-1}{4} \rp},
	\end{align*}
	where $\alpha,\beta $ satisfy $\alpha h-\beta b=1$. We therefore may conclude that
	\begin{align*}
		\frac{\omega_{h,b}}{\omega_{th,b}^t} = \lp \frac{h}{b} \rp \lp \frac{th}{b} \rp^t e^{\pi i \frac{(1-t)(b-1)}{4}} e^{\frac{\pi i}{12b} \lp (\alpha - tA )\left(1-b^2\right) + h\left(1 - \beta b - t^2\left(1-B b\right)\right)\left(1-b^2\right) \rp},
	\end{align*}
	where $\alpha h - \beta b = A  th - B  b = 1$. A straightforward calculation then gives the claim.
\end{proof}

We now turn to evaluating the Kloosterman sum $K(a,b,t;n)$. 

\begin{proposition} \label{Kloosterman Sum Evaluation}
	Suppose that $b$ is an odd prime, $a, n$ are integers, and $t > 1$ is an integer coprime to $b$. Then we have
	\begin{align*}
		K(a,b,t;n) =
		\begin{cases}
			\mathbb{I}(a,b,t,n) (-1)^{\frac{(1-t)(b-1)}{4}} \left( \frac{t}{b} \right) & \text{ if } t \text{ is odd}, \vspace{5pt} \\  (-1)^{\frac{(1-t)(b-1)}{4}}\varepsilon_b \left( \frac{\frac{1}{24}\left(1-t^2\right)\left(1-b^2\right) + at - n}{b} \right)  \sqrt{b} & \text{ if } t \text{ is even,}
		\end{cases}
	\end{align*}
	where $\mathbb{I}(a,b,t,n)$ is defined
	\begin{equation*}
		\mathbb{I}(a,b,t,n):= \begin{cases} b-1 & \text{if } \frac{1}{24} \left(1-t^2\right)\left(1-b^2\right) + at - n \equiv 0 \pmod{b}, \\ -1 & \text{otherwise}. \end{cases}\\
	\end{equation*}
\end{proposition}

\begin{proof}
	By Proposition \ref{Dedekind Simplification}, we have
	\begin{align*}
		K(a,b,t;n) &= e^{\frac{\pi i}{4}(1-t)(b-1)} \sum_{h=1}^{b-1} \left(\frac hb\right) \left(\frac{th}{b}\right)^t \zeta_b^{(at-n)h+\frac{1}{24}\left(1-t^2\right)\left(1-b^2\right)h}.
	\end{align*}
	
	The multiplicativity of the Legendre symbol implies
	\begin{align*}
		\lp \frac{h}{b} \rp \lp \frac{th}{b} \rp^t = \lp \frac{h}{b} \rp^{t+1} \lp \frac{t}{b} \rp^t =
		\begin{cases}
			\lp \frac{t}{b} \rp & \text{ if } t \text{ is odd}, \vspace{5pt} \\ \lp \frac{h}{b} \rp & \text{ if } t \text{ is even}.
		\end{cases}
	\end{align*}
	
	We proceed distinguishing on the parity of $t$.
	Suppose first that $t$ is odd. Then since $b$ is odd, $\frac14 (1-t)(b-1)$ is an integer and the claim directly follows.
	
	Suppose next that $t$ is even. Then we have
	\begin{align*}
		K(a,b,t;n) = e^{\pi i\frac{(1-t)(b-1)}{4}} \sum_{h=1}^{b-1} \left(\frac hb\right) \zeta_b^{h\left(\frac{1}{24}\left(1-t^2\right)\left(1-b^2\right)+at-n\right)}.
	\end{align*}
	Using the classical evaluation of the Gauss sum (see for example pages 12-13 of \cite{Dav80}), we obtain
	\[
	\sum_{h=1}^{b-1} \left(\frac hb\right) \zeta_b^{\left(\frac{1}{24}\left(1-t^2\right)\left(1-b^2\right)+at-n\right)h} = \left(\frac{\frac{1}{24}\left(1-t^2\right)\left(1-b^2\right)+at-n}{b}\right) \varepsilon_b \sqrt{b}. \qedhere
	\]
\end{proof}

\section{Zuckerman's exact formula}

Here we recall a result of Zuckerman \cite{Zuc39}, building on work of Rademacher \cite{Rad37}. Using the circle mthod, Zuckerman computed exact formulae for Fourier coefficients for weakly holomorphic modular forms of arbitrary non-positive weight on finite index subgroups of $\mathrm{SL}_2(\ZZ)$ in terms of the cusps of the underlying subgroup and the principal parts of the form at each cusp. 
Let $F$ be a weakly holomorphic modular form of weight $\kappa \leq 0$ with transformation law
$$F(\gamma \tau) = \chi(\gamma) (c \tau + d)^{\kappa} F(\tau),$$
for all $\gamma = \left( \begin{smallmatrix} a & b \\ c & d \end{smallmatrix} \right)$ in some finite index subgroup of $\mathrm{SL}_2(\ZZ)$. The transformation law can be viewed alternatively in terms of the cusp $\frac{h}{k} \in\QQ$. Let $h'$ be defined through the congruence $hh' \equiv -1 \pmod k$. Taking $\tau = \frac{h'}{k} + \frac{i}{kz}$ and $\gamma=\gamma_{h,k}  := \left( \begin{smallmatrix} h & \beta \\ k & -h' \end{smallmatrix} \right) \in \textnormal{SL}_2(\mathbb{Z})$, we obtain the equivalent transformation law
$$
F\lp \frac{h}{k}+\frac{iz}{k} \rp = \chi(\gamma_{h,k})(-iz)^{-\kappa}   F\lp \frac{h'}{k}+\frac{i}{kz} \rp.
$$
Let $F$ have the Fourier expansion at $i\infty$ given by
\[
F(\tau) = \sum_{n\gg-\infty} a(n)q^{n+\alpha}
\]
and Fourier expansions at each rational number $0 \leq \frac{h}{k} < 1$ given by
\begin{equation*}\label{2.5}
	F|_{\kappa}\gamma_{h,k}(\tau) = \sum_{n \gg -\infty} a_{h,k}(n) q^{\frac{n + \alpha_{h,k}}{c_{k}}}.
\end{equation*}
Furthermore, let $I_\alpha$ denote the usual $I$-Bessel function.
In this framework, the relevant theorem of Zuckerman \cite[Theorem 1]{Zuc39} may be stated as follows.
\begin{theorem}\label{Thm: Zuckerman}
	Assume the notation and hypotheses above. If $n + \alpha > 0,$ then we have
	\begin{align*}
		&a(n) =  2\pi (n+\alpha)^{\frac{\kappa-1}{2}} \sum_{k=1}^\infty \dfrac{1}{k} \sum_{\substack{0 \leq h < k \\ \gcd(h,k) = 1}}\chi(\gamma_{h,k}) e^{- \frac{2\pi i (n+\alpha) h}{k}} 
		\\ 
		&\ \times \sum_{m+\alpha_{h,k} \leq 0} a_{h,k}(m) e^{ \frac{2\pi i}{k c_{k}} (m + \alpha_{h,k}) h' } \left( \dfrac{\lvert m +\alpha_{h,k} \rvert }{c_{k}} \right)^{ \frac{1 - \kappa}{2}} I_{-\kappa+1}\left( \dfrac{4\pi}{k} \sqrt{\dfrac{(n + \alpha)\lvert m +\alpha_{h,k} \rvert}{c_{k}}} \right).
	\end{align*}
\end{theorem}

\section{Proofs of Theorem \ref{C6 t-hook Asymptotic} and Corollary \ref{C6 t-hook Distribution}}

We next provide proofs of both Theorem \ref{C6 t-hook Asymptotic} and Corollary \ref{C6 t-hook Distribution}. Our main tool is the powerful theorem of Zuckerman. For these proofs, we require the definition
\begin{align} \label{c_t def}
	c_t(a,b;n)&:= \frac 1b + \begin{cases} 0 & \text{ if } b | t, \\[+.1in] (-1)^{\frac{(1-t)(b-1)}{4}} \mathbb{I}(a,b,t,n) b^{-\frac{t+1}{2}} \left( \frac{t}{b} \right) & \text{ if } b \centernot | t \text{ and } t \text{ is odd,} \\[+.1in] i^{\frac{(1-t)(b-1)}{2}} \varepsilon_b b^{-\frac t2} \left(\frac{\frac{1}{24}\left(1-t^2\right)\left(1-b^2\right)+at-n}{b}\right)  & \text{ if } b \centernot | t \text{ and } t \text{ is even}, \end{cases}
\end{align}

\begin{proof}[Proof of Theorem \ref{C6 t-hook Asymptotic}]
	
	Using Corollary \ref{ptabGenFunctions} we have
	\begin{align}\label{eqn: split}
		H_t(a,b;q) = \frac{1}{b(q;q)_\infty} + \sum_{r=1}^{b-1}  \zeta_b^{-ar}H_t\left(\zeta_b^r;q\right).
	\end{align}
	From Theorem \ref{HanFunction} we conclude
	\begin{align*}
		H_t\left(\zeta_b^r;q \right) = \frac{\left(q^t;q^t\right)^t_\infty}{\left(\zeta_b^rq^t ; \zeta_b^rq^t\right)^t_\infty \left(q;q\right)_\infty}.
	\end{align*}
	
	To obtain the transformation formula for $H_t(\zeta_b^r;q)$ at the cusp $\frac hk$, we write 
	\begin{align*}
		q^t = e^{\frac{2\pi i t}{k} \left( h+iz \right)} = e^{\frac{2\pi i }{\frac{k}{\gcd(k,t)}} \left( h\frac{t}{\gcd(k,t)}+i\frac{t}{\gcd(k,t)}z \right)},
	\end{align*} 
	where we note that $\gcd(h\frac{t}{\gcd(k,t)}, \frac{k}{\gcd(k,t)}) = 1$. Thus we may use \eqref{Eqn: transform usual pochham} with $k \mapsto \frac{k}{\gcd(k,t)}, h \mapsto h \frac{t}{\gcd(k,t)}, z \mapsto \frac{t}{\gcd(k,t)}z$ to obtain
	\begin{multline}\label{Eqn: transformation of q^t}
		\left(q^t;q^t\right)_\infty = \omega_{h \frac{t}{\gcd(k,t)},\frac{k}{\gcd(k,t)}}^{-1} \left(\frac{t}{\gcd(k,t)}z\right)^{-\frac{1}{2}} e^{\frac{\pi \gcd(k,t)}{12k} \left( \frac{t}{\gcd(k,t)}z -\frac{\gcd(k,t)}{t z} \right)} \\
		\times \left(e^{\frac{2\pi i \gcd(k,t)}{k} \left(h_{k,t} +  i\frac{\gcd(k,t)}{tz}\right) } ; e^{\frac{2\pi i \gcd(k,t)}{k} \left(h_{k,t} +  i\frac{\gcd(k,t)}{tz}\right) }\right)_\infty,
	\end{multline}
	where $0 \leq h_{k,t} < \frac{k}{\gcd(k,t)}$ is defined by $ h \frac{t}{\gcd(k,t)} h_{k,t} \equiv -1 \pmod{\frac{k}{\gcd(k,t)}}$.
	
	Similarly, for $\left(\zeta_b^rq^t;\zeta_b^rq^t\right)_\infty$ the proof of Theorem \ref{Theorem1} (2) implies that we may use \eqref{Eqn: transform usual pochham} with  $h \mapsto \frac{hbt+rk}{\lambda_{t,r,b,h,k}}, k \mapsto \frac{kb}{\lambda_{t,r,b,h,k}}, z \mapsto \frac{tbz}{\lambda_{t,r,b,h,k}}$ and obtain
	\begin{multline}\label{Eqn: transformation of zeta q^t}
		\left(\zeta_b^rq^t;\zeta_b^rq^t\right)_\infty = \omega_{\frac{hbt+rk}{\lambda_{t,r,b,h,k}}, \frac{kb}{\lambda_{t,r,b,h,k}}}^{-1} \left(\frac{tbz}{\lambda_{t,r,b,h,k}}\right)^{-\frac{1}{2}} e^{\frac{\pi \lambda_{t,r,b,h,k}}{12kb} \left( \frac{tbz}{\lambda_{t,r,b,h,k}} - \frac{\lambda_{t,r,b,h,k}}{tbz} \right) } \\
		\times \left( e^{\frac{2\pi i \lambda_{t,r,b,h,k}}{kb} \left( h_{k,t,b,r} + i\frac{\lambda_{t,r,b,h,k}}{tbz} \right)} ; e^{\frac{2\pi i \lambda_{t,r,b,h,k}}{kb} \left( h_{k,t,b,r} + i\frac{\lambda_{t,r,b,h,k}}{tbz} \right)}  \right)_\infty,
	\end{multline}
	where $0 \leq h_{k,t,b,r} < \frac{kb}{\lambda_{t,r,b,h,k}}$ is defined by $ \frac{hbt+rk}{\lambda_{t,r,b,h,k}} h_{k,t,b,r} \equiv -1 \pmod{\frac{kb}{\lambda_{t,r,b,h,k}}}$.
	
	Combining \eqref{Eqn: transform usual pochham}, \eqref{Eqn: transformation of q^t}, and \eqref{Eqn: transformation of zeta q^t} yields
	\begin{multline}\label{eqn: transform}
		H_t\left(\zeta_b^r;q\right)
		= \Omega_{b,t}(r;h,k) \left( \frac{\gcd(k,t)b}{\lambda_{t,r,b,h,k}}\right)^{\frac{t}{2}} z^{\frac{1}{2}} e^{\frac{\pi}{12k} \left(-z + \left(1-\gcd(k,t)^2 + \frac{\lambda_{t,r,b,h,k}^2}{b^2} \right) \frac{1}{z}\right) } \\
		\times \frac{  \left(e^{\frac{2\pi i \gcd(k,t)}{k} \left(h_{k,t} +  i\frac{\gcd(k,t)}{tz}\right) } ; e^{\frac{2\pi i \gcd(k,t)}{k} \left(h_{k,t} +  i\frac{\gcd(k,t)}{tz}\right) }\right)_\infty^t  }{ \left( e^{\frac{2\pi i \lambda_{t,r,b,h,k}}{kb} \left( h_{k,t,b,r} + i\frac{\lambda_{t,r,b,h,k}}{tbz} \right)} ; e^{\frac{2\pi i \lambda_{t,r,b,h,k}}{kb} \left( h_{k,t,b,r} + i\frac{\lambda_{t,r,b,h,k}}{tbz} \right)}  \right)_\infty^t 
			\lp e^{\frac{2\pi i}{k}\lp h'+\frac iz \rp}; e^{\frac{2\pi i}{k}\lp h'+\frac{i}{z} \rp} \rp_\infty},
	\end{multline}
	where
	\begin{align*}
		\Omega_{b,t}(r;h,k) := \frac{\omega_{\frac{hbt+rk}{\lambda_{t,r,b,h,k}}, \frac{kb}{\lambda_{t,r,b,h,k}}}^{t} \omega_{h,k}}{\omega_{h \frac{t}{\gcd(k,t)},\frac{k}{\gcd(k,t)}}^{t}}.
	\end{align*}
	As usual, we define $P_t(q):=(q;q)_\infty^t=:\sum_{n=0}^\infty q_t(n)q^n$, and $P(q)^t=:\sum_{n=0}^\infty p_t(n)q^n$. Then we see that the principal part of \eqref{eqn: transform} is governed by the sum
	\begin{equation*}\label{eqn: principal part}
		\sum_{\substack{n_1,n_2,n_3 \geq 0 \\ r_{k,h,t,b}(n_1,n_2,n_3)\geq 0}} q_t(n_1) p_t(n_2)  p(n_3) \zeta_{kb}^{\gcd(k,t)b h_{k,t}  n_1 + \lambda_{t,r,b,h,k} h_{k,t,b,r} n_2 + bh'n_3}  e^{\frac{\pi}{12 kz} r_{k,h,t,b}(n_1,n_2,n_3) },
	\end{equation*}
	where
	\begin{align*}
		r_{k,h,t,b}(n_1,n_2,n_3) := 1-\gcd(k,t)^2 + \frac{\lambda_{t,r,b,h,k}^2}{b^2} - 24 \left( \frac{\gcd(k,t)^2}{t} n_1 + \frac{\lambda_{t,r,b,h,k}^2}{tb^2}n_2 + n_3 \right).
	\end{align*}
	We denote the Fourier coefficients of $H_t(\zeta_b^r;q)$ by $c_{t,b,r}(n)$. Using Theorem \ref{Thm: Zuckerman} we conclude that 
	\begin{multline}\label{Eqn: exact formula}
		c_{t,b,r}(n) = \frac{2\pi}{n^{\frac{3}{4}}} b^{\frac{t}{2}} \sum_{k = 1}^\infty \frac{\gcd(k,t)^{\frac{t}{2}}}{k} \sum_{\substack{0 \leq h < k \\ \gcd(h,k)=1}} \Omega_{b,t}(r; h,k) e^{- \frac{2\pi i n h}{k}} \lambda_{t,r,b,h,k}^{-\frac{t}{2}} \\ \times \sum_{\substack{n_1,n_2,n_3 \geq 0 \\ r_{k,h,t,b}(n_1,n_2,n_3) \geq 0}} q_t(n_1) p_t(n_2)  p(n_3) \zeta_{kb}^{\gcd(k,t)b h_{k,t}  n_1 + \lambda_{t,r,b,h,k} h_{k,t,b,r} n_2 + bh'n_3} \\ \times \left(\frac{r_{k,h,t,b}(n_1,n_2,n_3)}{24}\right)^{\frac34} I_{\frac{3}{2}} \left(\frac{\pi}{k}\sqrt{\frac{2nr_{k,h,t,b}(n_1,n_2,n_3)}{3}} \right).
	\end{multline}

	Since $x^\alpha I_{\alpha}(x)$ is monotonically increasing as $x \rightarrow \infty$ for any fixed $\alpha$, the terms which dominate asymptotically are those which have the largest possible value of $\frac1k \sqrt{r_{k,h,t,b}(n_1, n_2, n_3)}$. In particular for this we require $n_1 = n_2 = n_3 = 0$.
	Note that we have $q_t(0) = p_t(0) = p(0) = 1$.
	Since the expression in question is positive we can maximize its square, that is we maximize
	\begin{equation*}\label{max}
		\dfrac{r_{k,h,t,b}(0,0,0)}{k^2} = \dfrac{1}{k^2} \left( 1 - \gcd(k,t)^2 + \dfrac{\lambda_{t,r,b,h,k}^2}{b^2} \right).
	\end{equation*}

	We consider the three possible values of $\lambda_{t,r,b,h,k}$. If $\lambda_{t,r,b,h,k} = \gcd(k,t)$, then
	$$\dfrac{r_{k,h,t,b}(0,0,0)}{k^2} = \dfrac{1}{k^2} \left( 1 + \left( \dfrac{1}{b^2} - 1 \right) \gcd(k,t)^2 \right) \leq  \left(1+ \left(\frac{1}{9}-1\right)\right) < 1.$$
	If $\lambda_{t,r,b,h,k} = b \gcd(k,t)$, then (noting that in this case $k>1$)
	$$\dfrac{r_{k,h,t,b}(0,0,0)}{k^2} = \dfrac{1}{k^2} < 1.$$
	Finally, if $\lambda_{t,r,b,h,k} = b^2 \gcd(k,t)$, then we have
	$$\dfrac{r_{k,h,t,b}(0,0,0)}{k^2} = \frac{1}{k^2}\left(1 + \left(b^2 - 1\right) \gcd(k,t)^2\right).$$
	Since $b \mid\mid \dfrac{k}{\gcd(k,t)}$ in this case, we may write $\gcd(k,t) = b^\varrho d$ where $\gcd(b,d) = 1$, $b^\varrho \mid \mid t$, and $k = b^{\varrho + 1} d k_0$ for $\gcd(k_0, \frac{t}{\gcd(k,t)}) = \gcd(k_0,b) = 1$. Therefore, we have
	$$\dfrac{r_{k,h,t,b}(0,0,0)}{k^2} = \dfrac{1 + \left(b^2-1\right) b^{2\varrho} d^2}{b^{2\varrho + 2} d^2 k_0^2},$$
	which is maximized if $k_0 = 1$. In this case, we have $k = b \gcd(k,t)$ and therefore we may write
	$$\dfrac{r_{k,h,t,b}(0,0,0)}{k^2} = \dfrac{1 + \left(b^2 - 1\right) \gcd(k,t)^2}{b^2 \gcd(k,t)^2} = \dfrac{b^2-1}{b^2}+\dfrac{1}{b^2 \gcd(k,t)^2}.$$
	To maximize this, we need to minimize $\gcd(k,t)$, which is $\gcd(k,t) = 1$. Note that in this case
	\begin{align*}
		\dfrac{r_{k,h,t,b}(0,0,0)}{k^2} = 1 .
	\end{align*}
	
	Since $ht+r \equiv 0 \pmod b$, we have
	\begin{align*}
		\Omega_{b,t}(r;h,b) = \dfrac{\omega_{\frac{ht+r}{b}, 1}^{t} \omega_{h,b}}{\omega_{ht,b}^{t}} = \dfrac{\omega_{-r\bar{t},b}}{\omega_{-r,b}^{t}},
	\end{align*}
	where $\bar{t}$ denotes the inverse of $t \pmod b$. Then  by \eqref{Eqn: exact formula} we have
	\begin{align*}
		c_{t,b,r}(n) &\sim \dfrac{2\pi b^{\frac t2} \omega_{-r\bar{t},b} e^{\frac{2\pi i n r\bar{t}}{b}}}{(24n)^{\frac 34} \omega_{-r,b}^{t} b^{t+1}} I_{\frac 32} \left( \pi \sqrt{\frac{2n}{3}} \right)   \sim \dfrac{e^{\pi\sqrt{\frac{2n}{3}}}}{4\sqrt{3} n b^{\frac{t}{2}+1}} \dfrac{\omega_{-r\bar{t},b}}{ \omega_{-r,b}^{t}} e^{\frac{2\pi i n r\bar{t}}{b}} ,
	\end{align*}
	as $n \to \infty$, where we use that $I_\alpha(x) \sim \frac{e^x}{\sqrt{2\pi x}}$ as $x \rightarrow \infty$. 
	Using \eqref{Hardy-Ramanujan Asymptotic}, we obtain
	\begin{align*}
		\dfrac{c_{t,b,r}(n)}{p(n)} \sim \begin{cases}
			\dfrac{1}{b^{\frac{t}{2}+1}} \dfrac{\omega_{-r\bar{t},b}}{\omega_{-r,b}^{t}} e^{\frac{2\pi i n r\bar{t}}{b}} &\text{if } b \centernot | t, \\
			0 &\text{otherwise.}
		\end{cases}
	\end{align*}
	By \eqref{eqn: split}, we have
	\begin{align*}
		p_t(a,b;n) = \dfrac{1}{b}p(n) + \dfrac{1}{b} \sum_{r = 1}^{b-1} \zeta_b^{-ar} c_{t,b,r}(n),
	\end{align*}
	and so dividing through by $p(n)$ yields
	\begin{align*}
		\dfrac{p_t(a,b;n)}{p(n)} = \dfrac{1}{b} + \dfrac{1}{b} \sum_{r=1}^{b-1} \zeta_b^{-ar} \dfrac{c_{t,b,r}(n)}{p(n)} \sim \begin{cases}
			\dfrac{1}{b} + \dfrac{1}{b^{\frac{t}{2} + 2}} \sum\limits_{r=1}^{b-1}  \dfrac{\omega_{-r\bar{t},b}}{\omega_{-r,b}^{t}} \zeta_b^{\lp n \bar{t}-a \rp r} & \text{ if } b \centernot | t, \\[+0.2cm]
			\dfrac{1}{b} & \text{ otherwise}
		\end{cases}
	\end{align*}
	as $n \to \infty$. This completes the proof in the case where $b | t$. Otherwise, setting $h = -r\bar{t}$ shows
	\begin{align*}
		\dfrac{p_t(a,b;n)}{p(n)} \sim \dfrac{1}{b} + \dfrac{1}{b^{\frac{t}{2} + 2}} \sum_{h=1}^{b-1}  \dfrac{\omega_{h,b}}{\omega^t_{th,b}}\zeta_b^{(at-n)h} = \dfrac{1}{b} \left( 1 + \dfrac{K(a,b,t;n)}{b^{\frac{t}{2} + 1}} \right)
	\end{align*}
	as $n \to \infty$. The evaluation of $K(a,b,t;n)$ in Proposition \ref{Kloosterman Sum Evaluation} then completes the proof.
\end{proof}

\begin{proof}[Proof of Corollary \ref{C6 t-hook Distribution}]To derive Corollary~\ref{C6 t-hook Distribution}, it is enough to consider the leading constants in Theorem~\ref{C6 t-hook Asymptotic}. Namely, it suffices to show that for $a,b$ fixed, $c_t(a,b;n)$ depends only on $n \pmod{b}$, which is clear from the definition of \eqref{c_t def}. 
\end{proof}

\section{Examples of $t$-hook distributions}

This section includes examples of Theorem \ref{C6 t-hook Asymptotic} and Corollary \ref{C6 t-hook Distribution}. For convenience, we define the proportion functions
\begin{equation*}
	\Psi_t(a,b;n):=\frac{p_t(a,b;n)}{p(n)}.
\end{equation*}
\begin{example} In the case of $t=3$, we find that
	\begin{align*}
		H_3(\xi;q)=1+q+2q^2+3\xi q^3 &+ (2+3\xi)q^4 + (1+6\xi)q^5 +\left(2+9\xi^2\right)q^6 \\ &+\left(6\xi+9\xi^2\right)q^7
		+\left(1+3\xi+18\xi^2\right)q^8+\dots.
	\end{align*}
	and the three generating functions $H_3(a,3;q)$ begin with the terms
	\begin{displaymath}
		\begin{split}
			H_3(0,3;q)&=1+q+2q^2+2q^4+q^5+2q^6+q^8+\dots,\\
			H_3(1,3;q)&=3q^3+3q^4+6q^5+6q^7+3q^8+\dots,\\
			H_3(2,3;q)&=9q^6+9q^7+18q^8+\dots.
		\end{split}
	\end{displaymath}
	Theorem~\ref{C6 t-hook Asymptotic} implies (independently of $a$) that
	$$
	p_3(a,3;n)\sim \frac{1}{12\sqrt{3}n} \cdot e^{\pi \sqrt{\frac{2n}3}} \sim \frac{1}{3}\cdot p(n).
	$$
	The next table illustrates the conclusion of Corollary~\ref{C6 t-hook Distribution}, that the proportions $\Psi_3(a,b;n)
	\to \frac 13.$
	\medskip
	
	\begin{center}
		
		\begin{tabular}{|c|cc|cc|cc|}
			\hline \rule[-3mm]{0mm}{8mm}
			$n$       && $\Psi_3(0,3;n)$           && $\Psi_3(1,3;n)$  & $\Psi_3(2,3;n)$ & \\   \hline 
			$ 100$ &&  $\approx 0.4356$ && $\approx 0.1639$ & $\approx 0.4003$ & \\
			$\ \ \vdots \ \ $ && \vdots &&$\vdots$ &  $\vdots$ & \\
			$ 500$ &&  $\approx 0.3234$ && $\approx 0.3670$ & $\approx 0.3096$ & \\
			$ 600$ &&  $\approx 0.3318$ && $\approx 0.3114$ & $\approx 0.3567$ & \\
			$\ \ \vdots \ \ $ && \vdots &&$\vdots$ &  $\vdots$ & \\
			$ 2100$ &&  $\approx 0.3320$ && $\approx 0.3348$ & $\approx 0.3332$ & \\
			$ 2300$ &&  $\approx 0.3330$ && $\approx 0.3345$ & $\approx 0.3325$ & \\
			$ 2500$ &&  $\approx 0.3324$ && $\approx 0.3337$ & $\approx 0.3339$ & \\
			\hline
		\end{tabular}
	\end{center}
\end{example}
\medskip

\begin{example} We consider a typical case where the modular sums of $t$-hook functions are not equidistributed. 
	We consider $t=2$, where we have
	\begin{align*}
		H_2(\xi;q)=1+q+2\xi q^2+(1+2\xi)q^3+5\xi^2 q^4 +\left(2\xi+5\xi^2\right)q^5+\left(1+10\xi^3\right)q^6 \\+\left(5\xi^2+10\xi^3\right)q^7+\left(2\xi +20\xi^4\right)q^8+\dots.
	\end{align*}
	The three generating functions $H_2(a,3;q)$ begin with the terms
	\begin{displaymath}
		\begin{split}
			H_2(0,3;q)&=1+q+q^3+11q^6+10q^7+\dots,\\
			H_2(1,3;q)&=2q^2+2q^3+2q^5+22q^8+\dots,\\
			H_2(2,3;q)&=5q^4+5q^5+5q^7+\dots.
		\end{split}
	\end{displaymath}
	Theorem~\ref{C6 t-hook Asymptotic} implies that
	$$
	p_2(a,3;n)\sim \frac{A(a,n)}{12\sqrt{3}n}\cdot  e^{\pi \sqrt{\frac{2n}{3}}} \sim \frac{A(a,n)}{3} \cdot p(n),
	$$
	where $A(a,n)\in \{0, 1, 2\}$ satisfies the congruence $A(a,n)\equiv 2-a-n\pmod 3.$
	This explains the uneven distribution established by
	Corollary~\ref{C6 t-hook Distribution} in this case. 
	In particular, we have that
	$$
	\lim_{n\rightarrow \infty}\frac{p_t(a,3; 3n+2-a)}{p(n)}=0.
	$$
	Of course, this zero distribution is weaker than the vanishing obtained in Theorem~\ref{C6 Vanishing}.
	
	\noindent
	The next table illustrates the uneven asymptotics for $n\equiv 0\pmod 3.$
	\medskip
	
	\begin{center}
		
		\begin{tabular}{|c|cc|cc|cc|}
			\hline \rule[-3mm]{0mm}{8mm}
			$n$       && $\Psi_2(0,3;n)$           && $\Psi_2(1,3;n)$  & $\Psi_2(2,3;n)$ & \\   \hline 
			$ 300$ &&  $\approx 0.7347$ && $\approx 0.2653$ & $0$ & \\
			$\ \ \vdots \ \ $ && \vdots &&$\vdots$ &  $\vdots$ & \\
			$ 600$ &&  $\approx 0.6977$ && $\approx 0.3022$ & $0$ & \\
			$ 900$ &&  $\approx 0.6837$ && $\approx 0.3163$ & $0$ & \\
			$\ \ \vdots \ \ $ && \vdots &&$\vdots$ &  $\vdots$ & \\
			$ 4500$ &&  $\approx 0.6669$ && $\approx 0.3330$ & $0$ & \\
			$ 4800$ &&  $\approx 0.6669$ && $\approx 0.3330$ & $0$ & \\
			$ 5100$ &&  $\approx 0.6668$ && $\approx 0.3331$ & $0$ & \\
			\hline
		\end{tabular}
	\end{center}
\end{example}
\medskip

\begin{example} We consider another typical case where the modular sums of $t$-hook functions are not equidistributed. 
	We consider $t=4$, where we have
	\begin{align*}
		H_4(\xi;q)=1+q+2q^2+3q^3+(1+4\xi)q^4+(3+4\xi)q^5+(3+8\xi)q^6\\ +(3+12\xi)q^7+\left(4+4\xi+14\xi^2\right)q^8+\dots.
	\end{align*}
	The three generating functions $H_4(a,3;q)$ begin with the terms
	\begin{displaymath}
		\begin{split}
			H_4(0,3;q)&=1+q+2q^2+3q^3+q^4+3q^5+3q^6+3q^7+4q^8+\dots,\\
			H_4(1,3;q)&=4q^4+4q^5+8q^6+12q^7+4q^8+\dots,\\
			H_4(2,3;q)&=14q^8+\dots.
		\end{split}
	\end{displaymath}
	Theorem~\ref{C6 t-hook Asymptotic}, restricted to partitions of integers which are multiples of 12, gives 
	$$
	p_4(a,3;12n)\sim \begin{cases} \frac{4}{9}\cdot p(12n)\ \ \ \ \ &{\text {\rm if $a=0$,}}\\
		\frac{1}{3}\cdot p(12n) \ \ \ \ \ &{\text {\rm if $a=1,$}}\\
		\frac{2}{9}\cdot p(12n) \ \ \ \ \ &{\text {\rm if $a=2$.}}
	\end{cases}
	$$
	\noindent
	The next table illustrates these asymptotics.
	\medskip
	
	\begin{center}
		
		\begin{tabular}{|c|cc|cc|cc|}
			\hline \rule[-3mm]{0mm}{8mm}
			$n$       && $\Psi_4(0,3;12n)$           && $\Psi_4(1,3;12n)$  & $\Psi_4(2,3;12n)$ & \\   \hline 
			$ 10$ &&  $\approx 0.4804$ && $\approx 0.3373$ & $\approx 0.1823$ & \\
			$\ \ \vdots \ \ $ && \vdots &&$\vdots$ &  $\vdots$ & \\
			$ 50$ &&  $\approx 0.4500$ && $\approx 0.3381$ & $\approx 0.2119$ & \\
			$ 60$ &&  $\approx 0.4485$ && $\approx 0.3373$ & $\approx 0.2142$ & \\
			$\ \ \vdots \ \ $ && \vdots &&$\vdots$ &  $\vdots$ & \\
			$ 180$ &&  $\approx 0.4447$ && $\approx 0.3340$ & $\approx 0.2212$ & \\
			$ 190$ &&  $\approx 0.4447$ && $\approx 0.3339$ & $\approx 0.2214$ & \\
			$ 200$ &&  $\approx 0.4446$ && $\approx 0.3338$ & $\approx 0.2215$ & \\
			\hline
		\end{tabular}
	\end{center}
	
\end{example}

\section{Betti number generating functions}

For convenience, we let $P(X;T)$ be the usual {\it Poincar\'e polynomial}
\begin{equation*}
	P(X;T):=\sum_{j}b_j(X)T^j =\sum_{j} \dim \left( H_j(X,\QQ)\right) T^j,
\end{equation*}
which is the generating function for the Betti numbers of $X$.
For the various Hilbert schemes on $n$ points we consider, the work of G\"ottsche, Buryak, Feigin, and
Nakajima \cite{BF13, BFN15, Got94, Got02} offers the generating function
of these Poincar\'e polynomials as a formal power series in $q$. Namely, we have the following.

\begin{theorem}{\text {\rm (G\"ottsche)}}\label{HilbertGenFcn}
	We have that
	$$
	G(T;q):=\sum_{n=0}^{\infty} P\left(\left(\CC^2\right)^{[n]};T \right)q^n=\prod_{m=1}^{\infty}\frac{1}{1-T^{2m-2}q^m}=\frac{1}{ F_3(T^2; q)}.
	$$
\end{theorem}

\begin{theorem}{\text {\rm (Buryak and Feigin)}}\label{QuasiHilbertGenFcn}
	If $\alpha, \beta\in \NN$ are relatively prime, then we have that
	$$
	G_{\alpha,\beta}(T;q):=\sum_{n=0}^{\infty} P\left(\left(\left(\CC^2\right)^{[n]}\right)^{T_{\alpha,\beta}};T\right)q^n=
	\frac{1}{F_1(T^2;q^{\alpha+\beta})} \prod_{m=1}^{\infty}\frac{1-q^{(\alpha+\beta)m}}{1-q^m}.
	$$
\end{theorem}

\begin{remark}
	The Poincar\'e polynomials in these cases only have even degree terms.  The odd index
	Betti numbers are always zero. Moreover, letting $T=1$ in these generating functions give Euler's generating function for $p(n).$ Therefore, we directly see that
	$$
	p(n)=P\left(\left(\CC^2\right)^{[n]};1 \right)=P\left(\left(\left(\CC^2\right)^{[n]}\right)^{T_{\alpha,\beta}};1\right).
	$$
	Of course, the proofs of these theorems begin with partitions of size $n$.
\end{remark}

Arguing as in the proof of Corollary~\ref{ptabGenFunctions}, we obtain the following generating functions
for the modular sums of Betti numbers.

\begin{corollary}\label{HilbertModularGeneratingFunctions}
	For $0\leq a<b$, the following are true.
	
	\noindent
	\normalfont{(1)} We have that
	$$
	\sum_{n=0}^{\infty}B\left(a,b; \left(\CC^2\right)^{[n]}\right)q^n=\frac{1}{b}\sum_{r=0}^{b-1}\zeta_b^{-ar}G(\zeta_b^r;q).
	$$
	
	\noindent
	\normalfont{(2)} If $\alpha, \beta\in \NN$ are relatively prime, then we have
	$$
	\sum_{n=0}^{\infty}B\left(a,b; \left(\left(\CC^2\right)^{[n]}\right)^{T_{\alpha,\beta}}\right)q^n=\frac{1}{b}\sum_{r=0}^{b-1}\zeta_b^{-ar}G_{\alpha,\beta}(\zeta_b^r;q).
	$$
	
\end{corollary}

\section{A reformulation of Wright's circle method}

The classical circle method, as utilized by Hardy--Ramanujan and many others, derives asymptotic or exact formulas for the Fourier coefficients of $q$-series by leveraging modular properties of the generating functions. More recently, a variation of the circle method due to Wright has grown increasingly important in number theory. For the proof of Theorem \ref{C6 Betti Asymptotic} and Corollary \ref{C6 Betti Distribution}, we use Wright's variation, which obtains asymptotic formulas for generating functions carrying suitable analytic properties.

\begin{remark}
	Ngo and Rhoades \cite{NR17} proved a more restricted version\footnote{We note that hypothesis 4 in Proposition 1.8 of \cite{NR17} is stated differently than our hypothesis 2 in Proposition \ref{WrightCircleMethod2} below.} of the following proposition where the generating function $F$ splits as two functions. Our purposes do not require such a splitting, and so we state the proposition in terms of a single function $F$.
\end{remark}

\begin{proposition} \label{WrightCircleMethod2}
	Suppose that $F(q)$ is analytic for $q = e^{-z}$ where $z=x+iy \in \CC$ satisfies $x > 0$ and $|y| < \pi$, and suppose that $F(q)$ has an expansion $F(q) = \sum_{n=0}^\infty c(n) q^n$ near 1. Let $c,N,M>0$ be fixed constants. Consider the following hypotheses:
	
	\begin{enumerate}[leftmargin=*]
		\item[\rm(1)] As $z\to 0$ in the bounded cone $|y|\le Mx$ (major arc), we have
		\begin{align*}
			F(e^{-z}) = z^{B} e^{\frac{A}{z}} \left( \sum_{j=0}^{N-1} \alpha_j z^j + O_\delta\left(|z|^N\right) \right),
		\end{align*}
		where $\alpha_s \in \CC$, $A\in \RR^+$, and $B \in \RR$. 
		
		\item[\rm(2)] As $z\to0$ in the bounded cone $Mx\le|y| < \pi$ (minor arc), we have 
		\begin{align*}
			\lvert	F(e^{-z}) \rvert \ll_\delta e^{\frac{1}{\mathrm{Re}(z)}(A - \kappa)}.
		\end{align*}
		for some $\kappa\in \RR^+$.
	\end{enumerate}
	If  {\rm(1)} and {\rm(2)} hold, then as $n \to \infty$ we have for any $N\in \RR^+$ 
	\begin{align*}
		c(n) = n^{\frac{1}{4}(- 2B -3)}e^{2\sqrt{An}} \lp \sum\limits_{r=0}^{N-1} p_r n^{-\frac{r}{2}} + O\left(n^{-\frac N2}\right) \rp,
	\end{align*}
	where $p_r := \sum\limits_{j=0}^r \alpha_j c_{j,r-j}$ and $c_{j,r} := \dfrac{(-\frac{1}{4\sqrt{A}})^r \sqrt{A}^{j + B + \frac 12}}{2\sqrt{\pi}} \dfrac{\Gamma(j + B + \frac 32 + r)}{r! \Gamma(j + B + \frac 32 - r)}$. 
\end{proposition}

\begin{proof}
	By Cauchy's theorem, we have 
	$$
	c(n)=\frac{1}{2\pi i} \int_{\mathcal{C}} \frac{F(q)}{q^{n+1}} dq,
	$$
	where $\mathcal{C}$ is a circle centered at the origin inside the unit circle surrounding zero exactly once counterclockwise. We choose $|q|=e^{-\lambda}$  with $\lambda:=\sqrt{\frac{A}{n}}$. Set 
	$$
	A_j(n):= \frac{1}{2\pi i} \int_{\mathcal{C}_1} \frac{z^{B+j} e^{\frac Az}}{q^{n+1}}dq,
	$$
	where $\mathcal{C}_1$ is the major arc. We claim that
	\begin{equation}\label{main a}
		c(n)= \sum_{j=0}^{N-1} \alpha_j A_j(n) + O\lp n^{\frac12 (-B-N-1)} e^{2\sqrt{An}} \rp.
	\end{equation}
	For this write
	$$
	c(n)-\sum_{j=0}^{N-1} \alpha_j A_j(n)=\mathcal{E}_1(n) + \mathcal{E}_2(n),
	$$
	where
		$$\mathcal{E}_1(n) := \frac{1}{2\pi i} \int_{\mathcal{C}_2} \frac{F(q)}{q^{n+1}} dq,$$
		$$\mathcal{E}_2(n) := \frac{1}{2\pi i} \int_{\mathcal{C}_1} \lp F(q)z^{-B} e^{-\frac Az}  - \sum_{j=0}^{N-1} \alpha_jz^j\rp z^{B} e^{\frac Az} q^{-n-1}dq,$$
	where $\mathcal{C}_2$ is the minor arc.
	
	We next bound $\mathcal{E}_1(n)$ and $\mathcal{E}_2(n)$. For $\mathcal{E}_2(n)$ we have, by condition (1)
	$$
	\left| F\lp e^{-z}\rp z^{-B}e^{-\frac Az}-\sum_{j=0}^{N-1}\alpha_j z^j  \right| \ll_\delta |z|^N.
	$$
	Note that on $\mathcal{C}$, $x=\lambda$ and that
	$$
	\left| \exp \lp \frac Az+nz \rp\right| \leq \exp\lp 2 \sqrt{An}\rp.
	$$
	Since the length of $\mathcal{C}_1$ is $\approx \lambda$, we obtain
	$$
	\mathcal{E}_2(n) \ll \lambda |z|^{N+B} \exp\lp 2\sqrt{An} \rp.
	$$
	On $\mathcal{C}_1$, we have $y \ll \lambda$, implying $|z|\sim\frac{1}{\sqrt{n}}$. This gives $\mathcal{E}_1(n)$ satisfies the bound required in \eqref{main a}.
	
	On $\mathcal{C}_2$, we estimate
	$$
	|F(q)| \ll e^{\frac{1}{\lambda}(A - \kappa)}.
	$$ 
	Therefore, we have
	$$
	\mathcal{E}_1(n) \ll |F(q)| |q|^{-n} \ll e^{\frac{1}{\lambda}(A - \kappa) + n\lambda} \ll e^{(2 - \kappa) \sqrt{An}}.
	$$
	The required bound \eqref{main a} follows. Using Lemma 3.7 of \cite{NR17} to estimate the integrals $A_j(n)$ now gives the claim.
\end{proof}

\section{Proof of Theorem \ref{C6 Betti Asymptotic} and Corollary \ref{C6 Betti Distribution}}

We now apply the circle method to the generating functions
in Theorems~\ref{HilbertGenFcn} and \ref{QuasiHilbertGenFcn}.

\begin{proof}[Proof of Theorem \ref{C6 Betti Asymptotic}]
	Using first Corollary \ref{HilbertModularGeneratingFunctions} (1) and then Theorem \ref{HilbertGenFcn}, we obtain
	\[
	H_{a,b}(q) := \sum_{n=0}^\infty B\left(a,b;\left(\CC^2\right)^{[n]}\right)q^n = \frac1b\left(1+\delta_{2\mid b}\right) \frac{1}{(q;q)_\infty} + \frac1b \sum_{\substack{1\le r\le b-1\\r\ne\frac b2}} \zeta_b^{-ar} \frac{1}{F_3\left(\zeta_b^{2r};q\right)}.
	\]
	We want to apply Proposition \ref{WrightCircleMethod2}. For this we first show ($M>0$ arbitrary) that we have as $z\to0$ on the major arc $|y|\le Mx$
	\begin{equation}\label{major}
		H_{a,b}\left(e^{-z}\right) = \frac1b\left(1+\delta_{2\mid b}\right) \sqrt{\frac{z}{2\pi}} e^{\frac{\pi^2}{6z}} (1+O(|z|)).
	\end{equation}
	Recall that we have $P(q):=\sum_{n=0}^\infty p(n)q^n=(q;q)_\infty^{-1}$. First we note the well-known bound (for $|y|\le Mx$, as $z\to0$)
	\[
	P\left(e^{-z}\right) = \sqrt{\frac{z}{2\pi}} e^{\frac{\pi^2}{6z}} (1+O(|z|)).
	\]
	Next we consider $\frac{1}{F_3(\zeta_b^{2r};q)}$ for $\zeta_b^{2r}\ne1$ on the major arc. By Theorem \ref{Theorem1} (3)
	\[
	\frac{1}{F_3\left(\zeta_b^{2r};e^{-z}\right)} = \frac{\left(b^2z\right)^{\frac1b-\frac12}\Gamma\left(\frac1b\right)}{\sqrt{2\pi}} \prod_{j=1}^{b-1} \left(1-\zeta_b^{2rj}\right)^{\frac jb} e^{\frac{\pi^2}{6b^2z}}(1+O(|z|)) \ll |z|^{-N} e^\frac{\pi^2}{6z}
	\]
	for any $N\in\NN$. This gives \eqref{major}.
	
	Next we show that we have as $z\to0$ on the minor arc $|y|\ge Mx$
	\begin{equation}\label{minor}
		H_{a,b}\left(e^{-z}\right) \ll e^{\left(\frac{\pi^2}{6}-\kappa\right)\frac1x}.
	\end{equation}
	It is well-known (and follows by logarithmic differentiation) that for some $\mathcal{C}>0$
	\[
	\left|P\left(e^{-z}\right)\right| \le x^{\frac12} e^{\frac{\pi}{6x}-\frac{\mathcal{C}}{x}}.
	\]
	We are left to bound $\frac{1}{F_3(\zeta_b^{2r};q)}$ on the minor arc. For this we write
	\[
	\Log\left(\frac{1}{F_3\left(\zeta_b^{2r};q\right)}\right) = \sum_{m=1}^\infty \frac{q^m}{m\left(1-\zeta_b^{2rm}q^m\right)}.
	\]
	Noting that $|1-\zeta_b^{2rm}q^m|\ge1-|q|^m$, we obtain
	\[
	\left|\mathrm{Log}\left(\frac{1}{F_3\left(\zeta_b^{2r};q\right)}\right)\right| \le \left|\frac{q}{1-\zeta_b^{2r}q}\right|-\frac{|q|}{1-|q|}+\log(P|q|)
	\]
	so we are done once we show that
	\[
	\left|\frac{q}{1-\zeta_b^{2r}q}\right| - \frac{|q|}{1-|q|} < -\frac{\mathcal{C}}{x}
	\]
	for some $\mathcal{C}>0$. Note that
	\[
	\frac{1}{1-\zeta_b^{2r}q} = O_{b,r}(1),
	\]
	and thus
	\[
	\left|\frac{q}{1-\zeta_b^{2r}q}\right| - \frac{|q|}{1-|q|} = -\frac1x+O_{b,r}(1)
	\]
	giving \eqref{minor}. The claim of (1) now follows by Proposition \ref{WrightCircleMethod2}.\\
	(2) By Corollary \ref{HilbertModularGeneratingFunctions} (2) and Theorem \ref{QuasiHilbertGenFcn} we have
	\begin{align*}
		\mathcal{H}_{a,b,\alpha,\beta}(q) := \sum_{n=0}^\infty B\left(a,b;\left(\left(\CC^2\right)^{[n]}\right)^{T_{\alpha,\beta}}\right) q^n &= \frac1b (1+\delta_{2\mid b}) P(q) \\ &+ \frac1b \sum_{\substack{1\le r\le b-1\\r\ne\frac b2}} \zeta_b^{-ar} \frac{\left(q^{\alpha+\beta};q^{\alpha+\beta}\right)_\infty}{F_1\left(\zeta_b^{2r};q^{\alpha+\beta}\right)(q;q)_\infty}.
	\end{align*}
	We show the same bounds as in (1) with the only additional condition that
	\begin{align}\label{eqn: definition of M}
		M <  \frac{2\pi^2}{b^2} \min_{1\le r<\frac b2} \frac{r(b-2r)}{\left|\sum_{n=1}^\infty \frac{\sin\left(\frac{4\pi r}{b}\right)}{n^2}\right|}.
	\end{align}
	We only need to prove the bounds for
	\[
	\mathcal{H}_{\alpha,\beta}(q) := \frac{\left(q^{\alpha+\beta};q^{\alpha+\beta}\right)_\infty}{F_1\left(\zeta_b^{2r};q^{\alpha+\beta}\right)(q;q)_\infty}.
	\]
	for $\zeta_b^{2r}\ne1$. We may assume without loss of generality that $1\le2r<b$. We start by showing the major arc bound. By Theorem \ref{Theorem1} (1) and \eqref{Eta asymptotic}, we have, for $z$ on the major arc
	\[
	\mathcal{H}_{\alpha,\beta}(q) \ll \left|e^{\frac{\pi^2}{6z}-\frac{\pi^2}{6(\alpha+\beta)z}+\frac{\zeta_b^{2r}\phi\left(\zeta_b^{2r},2,1\right)}{(\alpha+\beta)z}}\right|.
	\]
	So to prove the major arc bound we need to show that for some $\varepsilon>0$
	\[
	\left(\frac{\pi^2}{6}-\varepsilon\right) \mathrm{Re}\left(\frac1z\right) - \mathrm{Re}\left(\frac{\zeta_b^{2r}\phi\left(\zeta_b^{2r},2,1\right)}{z}\right) > 0.
	\]
	We first rewrite
	\[
	\zeta_b^{2r} \phi\left(\zeta_b^{2r},2,1\right) = \sum_{n=1}^\infty \frac{\cos\left(\frac{4\pi rn}{b}\right)+i\sin\left(\frac{4\pi rn}{b}\right)}{n^2}.
	\]
	Now note the evaluation for $0\le\theta\le2\pi$ (see e.g. \cite{Zag06})
	\[
	\sum_{n=1}^\infty \frac{\cos(n\theta)}{n^2} = \frac{\pi^2}{6}-\frac{\theta(2\pi-\theta)}{4}.
	\]
	Thus we are left to show
	\[
	\frac{2\pi^2r}{b^2}(b-2r)x \ge \left|\sum_{n=1}^\infty \frac{\sin\left(\frac{4\pi rn}{b}\right)}{n^2}\right|y.
	\]
	This follows by the definition of $M$ given in \eqref{eqn: definition of M}.
\end{proof}

\begin{proof}[Proof of Corollary~\ref{C6 Betti Distribution}] This follows immediately from Theorem \ref{C6 Betti Asymptotic} and the definition of $d(a,b)$ in \eqref{C6 d(a,b) Definition}.
\end{proof}

\section{Examples of Theorem \ref{C6 Betti Asymptotic} and Corollary \ref{C6 Betti Distribution}}

Finally, we consider examples of the asymptotics and distributions in the setting of Hilbert schemes on $n$ points.

\smallskip
\begin{example} By G\"ottsche's Theorem (i.e., Theorem~\ref{HilbertGenFcn}), we have
	\begin{displaymath}
		\begin{split}
			G(T;q)&:=\sum_{n=0}^{\infty} P\left(\left(\CC^2\right)^{[n]};T \right)q^n=\prod_{m=1}^{\infty}\frac{1}{1-T^{2m-2}q^m}=\frac{1}{F_3(T^{-2}; T^2q)}\\
			&=1+q+\left(1+T^2\right)q^2+\left(1+T^2+T^4\right)q^3+\left(1+T^2+2T^4+T^6\right)q^4+\dots. \\
		\end{split}
	\end{displaymath}
	Theorem~\ref{C6 Betti Asymptotic} (1) implies that
	$$
	B\left(a,3; \left(\CC^2\right)^{[n]}\right) \sim \frac{1}{12\sqrt{3}n}\cdot  e^{\pi \sqrt{\frac{2n}{3}}},
	$$
	and so Corollary~\ref{C6 Betti Distribution} implies that $\delta(a,3;n)\to \frac 13$. The next table illustrates
	this phenomenon.
	\medskip
	
	\begin{center}
		
		\begin{tabular}{|c|cc|cc|cc|}
			\hline \rule[-3mm]{0mm}{8mm}
			$n$       && $\delta(0,3;n)$           && $\delta(1,3;n)$  & $\delta(2,3;n)$ & \\   \hline 
			$ 1$ &&  $1$ && $0$ & $0$ & \\
			$ 2$ &&  $0.5000$ && $0$ & $0.500$ &\\
			$\ \ \vdots \ \ $ && \vdots &&$\vdots$ &  $\vdots$ & \\
			$ 18$ &&  $\approx 0.3377$ && $\approx 0.3325$ & $\approx 0.3299$ & \\
			$ 19$ &&  $\approx 0.3367$ && $\approx 0.3306$ & $\approx 0.3327$ & \\
			$ 20$ &&  $\approx 0.3333$ && $\approx 0.3317$ & $\approx 0.3349$ & \\
			\hline
		\end{tabular}
	\end{center}
	\medskip

\end{example}

\begin{example} By Theorem~\ref{QuasiHilbertGenFcn}, for $\alpha=2$ and $\beta=3$ we have
	\begin{displaymath}
		\begin{split}
			G_{2,3}(T;q)&:=\sum_{n=0}^{\infty} P\left(\left(\left(\CC^2\right)^{[n]}\right)^{T_{2,3}};T\right)q^n=
			\frac{1}{F_1(T^2;q^{5})} \prod_{m=1}^{\infty}\frac{\left(1-q^{5m}\right)}{1-q^m}\\
			&=1+q+2q^2+\dots+\left(6+T^2\right)q^5+\left(10+T^2\right)q^6+\left(13+2T^2\right)q^7+\dots.
		\end{split}
	\end{displaymath}
	Theorem~\ref{C6 Betti Asymptotic} (2) implies that
	$$
	B\left(a,3;  \left(\left (\CC^2\right)^{[n]}\right)^{T_{\alpha,\beta}}\right)\sim \frac{1}{12\sqrt{3}n} \cdot e^{\pi \sqrt{\frac{2n}{3}}},
	$$
	and so Corollary~\ref{C6 Betti Distribution} yields that $\delta_{2,3}(a,3;n)\to \frac 13$. The next table illustrates
	this phenomenon.
	\medskip
	\begin{center}
		
		\begin{tabular}{|c|cc|cc|cc|}
			\hline \rule[-3mm]{0mm}{8mm}
			$n$       && $\delta_{2,3}(0,3;n)$           && $\delta_{2,3}(1,3;n)$  & $\delta_{2,3}(2,3;n)$ & \\   \hline 
			$ 1$ &&  $1$ && $0$ & $0$ & \\
			$ 2$ &&  $1$ && $0$ & $0$ &\\
			$\ \ \vdots \ \ $ && \vdots &&$\vdots$ &  $\vdots$ & \\
			$ 100$ &&  $\approx 0.3693$ && $\approx 0.2658$ & $\approx 0.3649$ & \\
			$ 200$ &&  $\approx 0.3343$ && $\approx 0.3176$ & $\approx 0.3481$ & \\
			$ 300$ &&  $\approx 0.3313$ && $\approx 0.3293$ & $\approx 0.3393$ & \\
			$ 400$ &&  $\approx 0.3318$ && $\approx 0.3324$ & $\approx 0.3358$ & \\
			$ 500$ &&  $\approx 0.3324$ && $\approx 0.3332$ & $\approx 0.3343$ & \\
			\hline
		\end{tabular}
	\end{center}
	
\end{example}

\sglsp

\chapter{Tur\'{a}n inequalities} \label{C7}
\thispagestyle{myheadings}

\dblsp
\vspace*{-.65cm}

The purpose of this chapter is to prove Theorem \ref{C7 Main Theorem} and Corollary \ref{C7 Corollary}. This is joint work with Anna Pun.

\section{Jensen polynomials and Tur\'an inequalities}

Given an arbitrary sequence $\alpha = (\alpha(0), \alpha(1), \alpha(2), \cdots)$ of real numbers, the associated \textit{Jensen polynomial $J_{\alpha}^{d,n}(X)$ of degree $d$ and shift $n$} is defined by

\begin{equation} \label{Jensen Poly}
	J_{\alpha}^{d,n}(X) := \sum\limits_{j=0}^d \binom{d}{j} \alpha(n+j)X^j.
\end{equation}

\noindent The Jensen polynomials also have a close relationship to the Riemann hypothesis. Indeed, P\'{o}lya \cite{Pol27} proved that the Riemann hypothesis is equivalent to the hyperbolicity of all of the Jensen polynomials $J^{d,n}_\gamma(X)$ associated to the Taylor coefficients $\{\gamma(j)\}_{j=0}^{\infty}$ of $\displaystyle\frac{1}{8}\Xi\bigg(\frac{i\sqrt{x}}{2}\bigg)$. Griffin, Ono, Rolen and Zagier \cite{GORZ19} proved that for each $d \geq 1$, all but finitely many $J^{d,n}_\gamma(X)$ are hyperbolic, which provides new evidence supporting the Riemann hypothesis.

There is a classical result by Hermite that generalizes the Tur\'{a}n  inequalities using Jensen polynomials. Let $$f(x) = x^n + a_{n-1}x^{n-1} + \cdots + a_1x + a_0$$ be a polynomial with real coefficients. Let $\beta_1, \beta_2, \cdots, \beta_n$ be the roots of $f$ and denote $S_0 = n$ and $$S_m = \beta_1^m + \beta_2^m + \cdots + \beta_n^m, \qquad m = 1, 2, 3, \cdots $$ their Newton sums. Let $M(f)$ be the Hankel matrix of $S_0, \cdots S_{2n-2}$, i.e. $$M(f) := \begin{pmatrix}
	S_0 &S_1&S_2&\cdots&S_{n-1}\\
	S_1 &S_2&S_3&\cdots&S_{n}\\
	S_2 &S_3&S_4&\cdots&S_{n+1}\\
	\vdots &\vdots& \vdots& \vdots&\vdots\\
	S_{n-1} &S_n&S_{n+1}&\cdots&S_{2n-2}\\
\end{pmatrix}.$$

\noindent Hermite's theorem \cite{Obr63} states that $f$ is hyperbolic if $M(f)$ is positive semi-definite. Recall that a polynomial with real coefficients is called hyperbolic if all of its roots are real. Each $S_m$ can be expressed in terms of the coefficients $a_0, \cdots, a_{n-1}$ of $f$ for $m \geq 1$, and a matrix is positive semi-definite if and only if all its principle minors are non-negative. Thus Hermite's theorem provides a set of inequality conditions on the coefficients of a polynomial $f$ to be hyperbolic: $$\Delta_1 = S_0 = n, \Delta_2 = \begin{vmatrix}S_0&S_1\\S_1&S_2\end{vmatrix} \geq 0, \cdots, \Delta_n = \begin{vmatrix}
	S_0 &S_1&S_2&\cdots&S_{n-1}\\
	S_1 &S_2&S_3&\cdots&S_{n}\\
	S_2 &S_3&S_4&\cdots&S_{n+1}\\
	\vdots &\vdots& \vdots& \vdots&\vdots\\
	S_{n-1} &S_n&S_{n+1}&\cdots&S_{2n-2}\\
\end{vmatrix}\geq 0.$$

\noindent For a given sequence $\alpha(n)$, when Hermite's theorem is applied to $J_{\alpha}^{d,n}(X)$ then the condition that all minors $\Delta_k$ of the Hankel matrix $M(J_{\alpha}^{d,n}(X))$ are non-negative gives a set of inequalities on the sequence $\alpha(n)$, and we call them the \textit{order $k$ Tur\'{a}n inequalities}. In other words, $J_{\alpha}^{d,n}(X)$ is hyperbolic if and only if the subsequence $\{\alpha(n + j)\}_{j=0}^{\infty}$ satisfies all the order $k$ Tur\'{a}n inequalities for all $1\leq k \leq d$. In particular, the result in \cite{GORZ19} shows that for any $d \geq 1$, the partition function $\{p(n)\}$ satisfies the order $d$ Tur\'{a}n inequality for sufficiently large $n$.

\subsection{Criterion of Griffin, Ono, Rolen, and Zagier}

Griffin, Ono, Rolen and Zagier \cite{GORZ19} produced the following criterion that is useful for proving that a sequence $\alpha(n)$ satisfies the Tur\'an inequalities for sufficiently large $n$.

\begin{theorem}[Theorem 3 \& Corollary 4, \cite{GORZ19}]\label{GORZ_1 Thm 3}
	Let $\{ \alpha(n) \}$, $\{ A(n) \}$, and $\{ \delta(n) \}$ be sequences of positive real numbers such that $\delta(n) \to 0$ as $n \to \infty$. Suppose further that for a fixed $d \geq 1$ and for all $0 \leq j \leq d$, we have
	\begin{equation*}
		\log\bigg( \dfrac{\alpha(n+j)}{\alpha(n)} \bigg) = A(n)j - \delta(n)^2j^2 + o\big( \delta(n)^d \big) \hspace{.3in} \text{ as } n \to \infty.
	\end{equation*}
	Then the renormalized Jensen polynomials  $\widehat{J}^{d,n}_{\alpha}(X) = \dfrac{\delta(n)^{-d}}{\alpha(n)} J^{d,n}_{\alpha}\bigg( \dfrac{\delta(n)X - 1}{\exp(A(n))}\bigg)$ satisfy $\lim\limits_{n \to \infty} \widehat{J}^{d,n}_{\alpha}(X) = H_d(X)$ uniformly for $X$ in any compact subset of $\mathbb{R}$. Furthermore, this implies that the polynomials $J^{d,n}_{\alpha}(X)$ are hyperbolic for all but finitely many values of $n$.
\end{theorem}

Because the conditions for this result are so general, the method can be utilized in a wide variety of circumstances. For instance, it is shown in Theorem 7 of \cite{GORZ19} that if $a_f(n)$ are the (real) Fourier coefficients of a modular form $f$ on $\text{SL}_2(\mathbb{Z})$ holomorphic apart from a pole at infinity, then there are sequences $A_f(n)$ and $\delta_f(n)$ such that $\alpha(n) = a_f(n)$ satisfies the required conditions. What we prove can then be regarded as a higher-level generalization of this result, since the sequences $p_k(n)$ are coefficients of weight zero weakly holomorphic modular forms on proper subgroups of $SL_2(\mathbb{Z})$. 

\section{A formula for $k$-regular partitions}

Recall that the $k$-regular partitions have generating function
\begin{align*}
	\sum_{n \geq 0} p_k(n) q^n = \prod_{n=1}^\infty \dfrac{\lp 1 - q^{kn} \rp}{\lp 1 - q^n \rp} = \dfrac{\lp q^k; q^k \rp_\infty}{\lp q; q \rp_\infty}.
\end{align*}
Because the Dedekind eta function $\eta(z) = q^{\frac{1}{24}} \lp q; q \rp_\infty$ is a modular form, we ca deduce that the generating function for $p_k(n)$ also has a modular transformation law. This fact leads via the method of Poincar\'e series, which is analogous to the Rademacher circle method of Chapter \ref{C5}, to an exact formula for $p_k(n)$.

This process was carried out by Hagis \cite{Hag71}. This result on $p_k(n)$ play a key role in the main theorem. The resulting formula expressible as a sum of modified Kloosterman sums times Bessel functions. Using facts about the asymptotics of Bessel functions and the explicit formulas, useful asymptotics for $p_k(n)$ may be derived. In particular, Hagis proves as a corollary (Corollary 4.1 in \cite{Hag71}) the asymptotic formula
\begin{equation} \label{Hagis Asymptotic}
	p_k(n) = 2\pi \sqrt{\dfrac{m_k}{k(n + k m_k)}} \cdot I_1\bigg(4\pi \sqrt{m_k(n + k m_k)} \bigg) (1 + O(\exp(-cn^{1/2}))),
\end{equation}
where $I_1$ is a modified Bessel function of

\section{Proofs of Theorem \ref{C7 Main Theorem} and Corollary \ref{C7 Corollary}}

Fix $d \geq 1$ and $k \geq 2$, and let the sequences $A_k(n)$, $\delta_k(n)$ be defined by
\begin{equation*}
	A_k(n) = 2\pi \sqrt{m_k/n} + \dfrac{3}{4}\sum\limits_{r=1}^{\lfloor 3d/4 \rfloor} \dfrac{(-1)^r}{rn^r} \ \ \text{and} \ \  \delta_k(n) = \bigg(- \sum\limits_{r=2}^\infty  \dfrac{4\pi \sqrt{m_k} \binom{1/2}{r}}{n^{r-1/2}} \bigg)^{1/2}.
\end{equation*}
Define the renormalized Jensen polynomials $\widehat{J}^{d,n}_{p_k}(X)$ by
\begin{equation} \label{J-Hat Def}
	\widehat{J}^{d,n}_{p_k}(X) := \dfrac{\delta_k(n)^{-d}}{p_k(n)} J^{d,n}_{p_k}\bigg( \dfrac{\delta_k(n)X - 1}{\exp(A_k(n))} \bigg).
\end{equation}
By application of the Jensen-P\'{o}lya method, it suffices to show that for any fixed $d$ and all $0 \leq j \leq d$,
\begin{equation} \label{Log-Quotient Eqn}
	\log\bigg( \dfrac{p_{k}(n+j)}{p_k(n)} \bigg) = A_k(n)j - \delta_k(n)^2j^2 + o\big( \delta_k(n)^d \big) \hspace{.3in} \text{ as } n \to \infty.
\end{equation}
Using (\ref{Hagis Asymptotic}), we have
\begin{equation*}
	p_k(n) = b_k (n + k m_k)^{-1/2} I_1(4\pi \sqrt{nm_k}) + O(n^{d_k} e^{-c_k \sqrt{n}}),
\end{equation*}
as $n \to \infty$, where $b_k, c_k > 0$, and $d_k$ are constants which depend at most on $k$. In light of the expansion of the Bessel functions of the first kind at infinity, this implies that $p_k(n)$ has asymptotic expansion to all orders of $1/n$ in the form
\begin{align*}
	p_k(n) \sim e^{4\pi \sqrt{nm_k}} n^{-3/4} \exp\bigg( \sum_{r=0}^\infty \frac{a_r}{n^r}\bigg),
\end{align*}
where $a_0, a_1, \cdots$ are constants depending only on $k$. Furthermore, when the exponential terms are considered asymptotically, the terms $\dfrac{a_r}{n^r}$ in the sum vanish compared with the term $4\pi \sqrt{nm_k}$ for $r \geq 1$, and so we have
\begin{align*}
	p_k(n) \sim e^{a_0 + 4\pi \sqrt{nm_k}} n^{-3/4}.
\end{align*}
It follows that for fixed $0 \leq j \leq d$, we have
\begin{eqnarray*}
	&&\log\bigg( \dfrac{p_k(n+j)}{p_k(n)} \bigg)\\&\sim& 4\pi \sqrt{m_k} \sum\limits_{r=1}^\infty \binom{1/2}{r} \dfrac{j^r}{n^{r-1/2}} - \dfrac{3}{4} \sum\limits_{r=1}^\infty \dfrac{(-1)^{r-1}j^r}{rn^r}\\
	&=& 2\pi \sqrt{m_k}\dfrac{j}{\sqrt{n}}+ 4\pi \sqrt{m_k}\sum\limits_{r=2}^\infty \binom{1/2}{r} \dfrac{j^r}{n^{r-1/2}} + \dfrac{3}{4} \sum\limits_{r=1}^\infty \dfrac{(-1)^{r}j^r}{rn^r}\\
	&=& 2\pi \sqrt{\dfrac{m_k}{n}}j  + \dfrac{3}{4} \sum\limits_{r=1}^{\lfloor 3d/4\rfloor} \dfrac{(-1)^{r}j^r}{rn^r}+4\pi \sqrt{m_k}\sum\limits_{r=2}^\infty \binom{1/2}{r} \dfrac{j^r}{n^{r-1/2}}+ \dfrac{3}{4} \sum\limits_{r=\lfloor 3d/4\rfloor + 1}^\infty \dfrac{(-1)^{r}j^r}{rn^r}\\
	&=& \Bigg(2\pi \sqrt{\dfrac{m_k}{n}}+ \dfrac{3}{4} \sum\limits_{r=1}^{\lfloor 3d/4\rfloor} \dfrac{(-1)^{r}j^{r-1}}{rn^r}\Bigg)j+\Bigg(\sum\limits_{r=2}^\infty  \dfrac{4\pi \sqrt{m_k}\binom{1/2}{r}j^{r-2}}{n^{r-1/2}}\Bigg)j^2 \\ && \hspace{7.3cm} + \hspace{0.1cm} \dfrac{3}{4} \sum\limits_{r=\lfloor 3d/4 \rfloor + 1}^\infty \dfrac{(-1)^{r}j^r}{rn^r}.\\
\end{eqnarray*}
Then the required result follows since 
\begin{align*}
	\log\bigg( \dfrac{p_{k}(n+j)}{p_k(n)} \bigg) - A_k(n)j + \delta_k(n)^2 j^2 = O(n^{- \lfloor 3d/4 \rfloor - 1}) = o(\delta_k(n)^d).
\end{align*}
\sglsp

\chapter{Variations of Lehmer's Conjecture} \label{C8}
\thispagestyle{myheadings}

\dblsp
\vspace*{-.65cm}

The purpose of this chapter is to prove the theorems of the introduction in the broader context of newforms with trivial mod 2 Galois representation. In particular, these results include Theorems \ref{PrimeDivisorsGeneral}, \ref{LehmerGeneral}, \ref{Power}, and \ref{LehmerVariantGeneral}. This section is joint work with Jennifer Balakrishnan, Ken Ono, and Wei-Lun Tsai.

\section{Lucas Sequences}

\subsection{Classical facts}

Suppose that $\alpha$ and $\beta$ are algebraic integers for which $\alpha+\beta$ and $\alpha \beta$
are relatively prime non-zero integers, where $\alpha/\beta$ is not a root of unity.
Their {\it Lucas numbers} $\{u_n(\alpha,\beta)\}=\{u_1=1, u_2=\alpha+\beta,\dots\}$ are the integers
\begin{equation}
	u_n(\alpha,\beta):=\frac{\alpha^n-\beta^n}{\alpha-\beta}.
\end{equation}
A prime  $\ell$ {\it primitive prime divisor of $u_n(\alpha,\beta)$} if $\ell \nmid (\alpha-\beta)^2 u_1(\alpha,\beta)\cdots u_{n-1}(\alpha, \beta)$ and $\ell \mid u_n(\alpha, \beta)$. We require several classical facts about Lucas numbers.

\begin{proposition}[Proposition 2.1 (ii) of \cite{BHV01}]\label{PropA}  If $d\mid n$, then $u_d(\alpha, \beta) | u_n(\alpha,\beta).$
\end{proposition}

To keep track of the first occurrence of prime divisors, we let $m_{\ell}(\alpha,\beta)$ be the smallest $n\geq 2$
for which $\ell \mid u_n(\alpha,\beta)$. We note that $m_{\ell}(\alpha,\beta)=2$ if and only if
$\alpha +\beta\equiv 0\pmod {\ell}.$

\begin{proposition}[Corollary 2.2\footnote{This corollary is stated for Lehmer numbers. The conclusions hold for Lucas numbers because $\ell \nmid (\alpha+\beta)$.} of \cite{BHV01}]\label{PropB} If $\ell\nmid \alpha \beta$ is an odd prime with
	$m_{\ell}(\alpha,\beta)>2$, then the following are true.
	\begin{enumerate}
		\item If $\ell \mid (\alpha-\beta)^2$, then $m_{\ell}(\alpha,\beta)=\ell.$
		\item If $\ell \nmid (\alpha-\beta)^2$, then $m_{\ell}(\alpha,\beta) \mid (\ell-1)$ or $m_{\ell}(\alpha,\beta)\mid (\ell+1).$
	\end{enumerate}
\end{proposition}

\begin{remark}
	If $\ell \mid \alpha \beta$, then either $\ell \mid u_n(\alpha,\beta)$ for all $n$, 
	or $\ell \nmid u_n(\alpha,\beta)$ for all $n$.
\end{remark}

\subsection{The work of Bilu-Hanrot-Voutier}

Bilu, Hanrot, and Voutier \cite{BHV01} proved the following definitive theorem. 

\begin{theorem} \label{Bilu}
	Every Lucas number $u_n(\alpha,\beta)$, with $n>30,$
	has a primitive prime divisor.
\end{theorem}

This theorem is sharp; there are sequences for which $u_{30}(\alpha,\beta)$
does not have a primitive prime divisor. 
We call a Lucas number $u_n(\alpha,\beta)$, with $n>2,$ {\it defective}\footnote{We do not consider the absence of
	a primitive prime divisor for $u_2(\alpha,\beta)=\alpha+\beta$ to be   a defect.}  if $u_{n}(\alpha,\beta)$ does 
not have a primitive prime divisor. Bilu, Hanrot and Voutier essentially complete the theory; they basically
characterized all of the defective Lucas numbers.
Their work, combined with a subsequent paper\footnote{This paper included a few cases which were omitted in \cite{BHV01}.} 
by Abouzaid \cite{Abo06},   gives the {\it complete classification} of
defective Lucas numbers.
Tables 1-4 in Section 1 of \cite{BHV01} and Theorem 4.1 of \cite{Abo06} offer this
classification. Every defective Lucas number either belongs to a  finite list of sporadic examples or
a finite list of parameterized infinite families.

We consider Lucas sequences  arising from those quadratic integral polynomials
\begin{equation}\label{Modularity}
	F(X)=X^2-AX+B=(X-\alpha)(X-\beta),
\end{equation}
where $B=\alpha \beta =p^{2k-1}$ is an odd power of a prime $p$, and $|A|=|\alpha+\beta|\leq 2\sqrt{B}=2p^{\frac{2k-1}{2}}.$ 
A straightforward analysis of these tables of defective Lucas numbers reveals a list of sporadic examples, and several potentially infinite families of examples.  A straightforward case-by-case analysis using elementary congruences, divisibilities, and the truth of Catalan's conjecture \cite{Mih04},
that $2^3$ and $3^2$ are the only consecutive perfect powers,
yields the following characterization.

\begin{theorem}\label{AwesomeList}
	Tables \ref{table1} and \ref{table2} in the Appendix list the defective $u_n(\alpha,\beta)$
	satisfying (\ref{Modularity}).
\end{theorem}

To identify the cases where $|u_n(\alpha,\beta)| = 1$ and $|u_n(\alpha,\beta)|=\ell$ is prime, we require the curves
\begin{equation}
	B_{1, k}^{r, \pm} : Y^2 = X^{2k-1} \pm 3^r,\quad \mathrm{and}\quad B_{2, k} : Y^2 = 2X^{2k-1} -1.
\end{equation}

\begin{lemma}\label{Part1DefectivePrimality}
	Suppose that $u_n(\alpha,\beta)$ is a defective Lucas number from Table \ref{table1} or Table \ref{table2}.
	\begin{enumerate}
		\item We have that $|u_n(\alpha,\beta)|=1$ if and only if 
		$$(A,B,n) \in \big\{ (\pm 1, 2, 5), (\pm 1, 2, 13), (\pm 1, 3, 5), (\pm 1, 5, 7), (\pm 2, 3, 3), (\pm 3, 2^3, 3) \big\},$$
		or $(A,B,n) = (\pm m, p, 3),$ where $p = m^2+1$ is prime with $m>1$.
		
		\item If $|u_n(\alpha,\beta)|=\ell$ is prime, 
		$(A,B,\ell, n) \in \big\{ (\pm 1, 2,7, 7), (\pm 1, 2,3, 8), (\pm 2, 11,5, 5) \big\},
		$
		or $(A,B,\ell, n) = (\pm m, p^{2k-1},3, 3),$ where $(p, \pm m)\in B^{1, \pm}_{1,k}$ and $3\nmid m$,  or $(A,B,\ell, n) = (\pm m, p^{2k-1},m, 4),$ where $(p, \pm m)\in B_{2,k}$.
	\end{enumerate}
\end{lemma}
\begin{proof} The proof of both (1) and (2) follow by a simple (and tedious) case-by-case analysis. 
\end{proof}

\section{$\Delta(z)$ and other eigenforms}

Throughout this paper we suppose that
\begin{equation}\label{qexpansion}
	f(z)=q+\sum_{n=2}^{\infty}a_f(n)q^n\in S_{2k}(\Gamma_0(N)) \cap \ZZ[[q]]
\end{equation}
is an even weight $2k$ newform.
Let $S_f$ be the finite (generally empty) set of primes $p$ for which $(A,B)=(a_f(p),p^{2k-1})$ appears in Tables \ref{table1} or \ref{table2}. 
For primes $p\not \in S_f$ and $m\geq 1$, we let
\begin{equation}
	\widehat{\sigma}(p;m):=\sigma_0(m+1)-1,
\end{equation}
while for $p\in S_f$ we define
$\widehat{\sigma}(p;m)$  in Table \ref{table3} in the Appendix.
We have the following theorem.

\begin{theorem}\label{PrimeDivisorsGeneral} Assume the notation and hypotheses above.
	If $n>1$ an integer, then
	$$
	\Omega(a_f(n))\geq \sum_{p\mid N} (k-1)\ord_p(n) +\sum_{\substack{p\nmid N \\ \ord_p(n)\geq 2}} \widehat{\sigma}(p;\ord_p(n)).
	$$
\end{theorem}

\begin{remark} 
	Theorem~\ref{PrimeDivisorsGeneral} does not take into account those primes $p\nmid N$ which
	exactly divide $n$ because it can happen that $|a_f(p)|=1$. 
	However,  if the mod 2 residual Galois representation
	is trivial, then $a_f(p)$ is even for every prime $p\nmid 2N$. In such cases, we get
	$$
	\Omega(a_f(n))\geq \sum_{p\mid N} (k-1)\ord_p(n) +\sum_{p\nmid 2N} \widehat{\sigma}(p;\ord_p(n)).
	$$
	This  applies to $\Delta(z)$, by the congruence
	$\Delta(z)\equiv \sum_{n=0}^{\infty}q^{(2n+1)^2}\pmod 2.$
	Since $(A,B)=(\tau(p),p^{11})$ does not appear in Lemma~\ref{Part1DefectivePrimality} (1), the proof of Theorem~\ref{PrimeDivisorsGeneral} gives
	Theorem~\ref{PrimeDivisors}.
\end{remark}

\subsection{Proof of Theorem~\ref{PrimeDivisorsGeneral}}

We recall some basic facts about {\it Atkin-Lehner newforms} 
(see \cite{AL70}), along with the deep theorem of Deligne  \cite{Del74, Del80} that bounds their Fourier coefficients. 

\begin{theorem}\label{Newforms} Suppose that $f(z)=q+\sum_{n=2}^{\infty}a_f(n)q^n\in S_{2k}(\Gamma_0(N))$ is a newform with integer coefficients.
	Then the following are true:
	\begin{enumerate}
		\item If $\gcd(n_1,n_2)=1,$ then $a_f(n_1 n_2)=a_f(n_1)a_f(n_2).$
		\item If $p\nmid N$ is prime and $m\geq 2$, then
		$$
		a_f(p^m)=a_f(p)a_f(p^{m-1}) -p^{2k-1}a_f(p^{m-2}).
		$$
		\item If $p\nmid N$ is prime and $\alpha_p$ and $\beta_p$ are roots of $F_p(x):=x^2-a_f(p)x+p^{2k-1},$ then
		$$
		a_f(p^m)=u_{m+1}(\alpha_p,\beta_p)=\frac{\alpha_p^{m+1}-\beta_p^{m+1}}{\alpha_p-\beta_p}.
		$$   
		Moreover, we have $|a_f(p)|\leq 2p^{\frac{2k-1}{2}}$, and $\alpha_p$ and $\beta_p$ are complex conjugates.
		
		\item If $p\mid N$ is prime, then 
		$f | U(p) := \sum_{n=1}^{\infty} a_f(np)q^n = a_f(p) f(\tau).$
		Moreover, we have 
		$$
		a_f(p^m)=\begin{cases} (\pm 1)^m p^{(k-1)m} \ \ \ \ \  &{\text {\rm if}}\ \ord_p(N)=1,\\
			0 \ \ \ \ \ &{\text {\rm if}}\ \ord_p(N)\geq 2.
		\end{cases}
		$$
	\end{enumerate}
\end{theorem}

Theorem~\ref{Newforms} leads to lower bounds for
the number of prime divisors (counted with multiplicity) of the coefficients in the sequence $\{a_f(p^2),a_f(p^3),\dots\}$, where $p$ is prime.

\begin{proposition}\label{PrimePower} Assuming the notation in Theorem~\ref{Newforms},  the following are true for $m\geq 2$.
	\begin{enumerate}
		\item If $p\mid N$ is prime, then $\ord_p(a_f(p^m))\geq (k-1)m.$
		\item If $p\nmid N$ is prime and $(A,B)=(a_f(p),p^{2k-1})$ does not appear in Tables \ref{table1} or \ref{table2}, then
		$$
		\Omega(a_f(p^m))\geq \sigma_0(m+1)-1.
		$$
		\item If $p\nmid N$ is prime and $(A,B)=(a_f(p),p^{2k-1})$ appears in Tables \ref{table1} or \ref{table2}, then Table \ref{table3} of the Appendix contains a lower bound
		for $\Omega(a_f(p^m))$.
	\end{enumerate}
\end{proposition}

\begin{proof}[Proof of Proposition~\ref{PrimePower}]
	The first claim follows from Theorem~\ref{Newforms} (4). The second claim follows from Theorem~\ref{Newforms} (3), 
	Proposition~\ref{PropA} 
	and Theorem~\ref{Bilu} in a case-by-case analysis. The point is that at least one new prime divisor is accumulated with each subsequent step in a Lucas sequence. In other words,  the relative divisibility of Lucas numbers and the presence of primitive prime divisors guarantees the lower bound. The only divisor of $m+1$ which does not contribute is $u_1=1$.
	The third claim follows similarly by taking into account the  defective Lucas numbers
	that appear in Tables \ref{table1} and \ref{table2}.
\end{proof}

\begin{proof}[Proof of Theorem~\ref{PrimeDivisorsGeneral}]
	The theorem follows from Theorem~\ref{Newforms} (1)  and Proposition~\ref{PrimePower}.
\end{proof}

\section{Statement of general results}

This section discusses the fully detailed generalizations of the main results stated in the introduction. We investigate questions about the prime divisors of Fourier coefficients and equations of the form $a_f(n) = \alpha$ for even weight newforms with integer coefficients and trivial mod 2 residual Galois representation (i.e. even Hecke eigenvalues for $T(p)$ for primes $p\nmid 2N$, where $N$ is the level).
We obtain a general theorem (see Theorem~\ref{LehmerVariantGeneral}) that theoretically locates those coefficients that are odd prime powers in absolute value for such newforms. For $\tau(n)$, this theorem gives the following criterion, which restricts arguments to explicit finite sets.

\begin{theorem}\label{LehmerVariation}
	If $\ell$ is an odd prime for which $\tau(n)=\pm\ell^m$, with $m\in \ZZ^{+},$ then $n=p^{d-1},$ where $p$ and $d\mid \ell( \ell^2-1)$
	are odd primes. Furthermore,  $\tau(n)=\pm \ell^m$ for at most finitely many $n$.
\end{theorem}

Theorem~\ref{LehmerVariation} offers a method for determining whether $|\tau(n)|=\ell^m$ has any solutions, which reduces the
problem to the determination of certain integer points on finitely many algebraic curves.  For $\ell \in \{3, 5, 7\},$ 
examples of these curves include
\begin{equation}\label{HyperEquations}
	Y^2-X^{11}=\pm 3^m,\ \ \ \ \ 
	Y^2-5X^{22}=\pm 4\cdot 5^m \ \ \ \ {\text {\rm and}}\ \ \ \ Y^{3}-5XY^2+6X^2Y-X^3=\pm 7^m.
\end{equation}
By classifying such points when $m=1$, we obtain 
the following theorem.\footnote{The {\it Journal of Number Theory} published the proceedings of the conference ``Modular forms and Drinfeld Modules'' held in 2018 in Pisa, Italy. Paper \cite{BCO22}  is an exposition of the third author's  lecture at the conference, and pertains to some of the cases  of Theorem~\ref{Lehmer135} (1). All of the other results in the present paper have not appeared elsewhere. This article is the main reference for the authors' work on variants of Lehmer's speculation.}

\begin{theorem}\label{Lehmer135}
	For every $n>1$, the following are true. 
	
	\noindent
	(1) We have that
	$$\tau(n)\not \in \{\pm 1, \pm 3, \pm 5, \pm 7, \pm 13, \pm 17, -19, \pm 23,  \pm 37, \pm 691\}.$$
	
	\noindent
	(2)  Assuming the Generalized Riemann Hypothesis, we have that
	\begin{align*}
		\tau(n) &\not\in 
		\left \{ \pm  \ell\ : \ 41\leq  \ell\leq 97  \  {\text {\rm with}}\ \legendre{\ell}{5}=-1\right\} \\ &\cup
		\left \{ -11, -29, -31, -41, -59, -61, -71, -79, -89\right\}.
	\end{align*}
	$$
	$$
\end{theorem}
There are infinite families of newforms with even level for which these methods apply. The next theorem
offers unconditional results for $3\leq \ell \leq 37,$ when $2k \in \{4, 6, 8, 10\}$ or
$\gcd(3 \cdot 5\cdot 7, 2k-1)\neq 1$. It also
gives further  results conditional on the Generalized Riemann Hypothesis (GRH).

\begin{theorem}\label{LehmerGeneral} 
	If  $f(z)=q+\sum_{n=2}^{\infty}a_f(n)q^n\in S_{2k}(\Gamma_0(2N))\cap \ZZ[[q]]$ is an even weight $2k\geq 4$ newform with trivial mod 2 residual Galois representation, then the following are true.
	\begin{enumerate}
		\item For every $n>1$ we have $a_f(n)\not \in \{\pm 1\}.$
		\item If $2k=4$, then for every $n$ we have
		$$
		a_f(n)\not \in \left\{\pm \ell \ : \ 3\leq \ell \leq 37\ {\text {\rm prime}} \right\}\setminus\left 
		\{ \pm11, -13,17,\pm19,-23,37\right\}.
		$$
		Assuming GRH, for every $n$ we have
		$$
		a_f(n) \not \in \{\pm \ell \ : \ 41 \leq \ell \leq 97 \ {\text {\rm prime}}\}\setminus\{-41,-53,-61,-67,\pm71,73,-89\}.
		$$
		\item If $2k=6$, then for every $n$ we have
		$$
		a_f(n)\not \in \left\{\pm \ell \ : \ 3\leq \ell \leq 37\  {\text {\rm prime}}\right\}\setminus\left \{ 11,13\right\}.
		$$
		Assuming GRH, for every $n$ we have
		$$
		a_f(n) \not \in \{\pm \ell \ : \ 41 \leq \ell \leq 97 \ {\text {\rm prime}}\}\setminus\{-47\}.
		$$
		\item If $2k=8$, then for every $n$ we have
		$$
		a_f(n)\not \in \left\{\pm \ell \ : \ 3\leq \ell \leq 37\  {\text {\rm prime}}\right\}.
		$$
		Assuming GRH, for every $n$ we have
		$$
		a_f(n) \not \in \{\pm \ell \ : \ 41\leq \ell \leq 97 \ {\text {\rm prime}}\}\setminus\{-71\}.
		$$
		\item If $2k=10,$ then for every $n$ we have
		$$
		a_f(n)\not \in \left\{\pm \ell \ : \ 3\leq \ell \leq 37\ {\text {\rm prime}} \right\}.
		$$
		Assuming GRH, for every $n$ we have
		$$
		a_f(n) \not \in \{\pm \ell \ : \ 41\leq \ell \leq 97\ {\text {\rm prime}}\}\setminus\{-83\}.
		$$
		\item  If $\gcd(3 \cdot 5\cdot 7\cdot 11\cdot 13, 2k-1)\neq 1$ and $2k\geq 12$, then for every $n$ we have 
		$$
		a_f(n)\not \in  \left \{ \pm \ell \ : \ 3\leq \ell <37 \ {\text {\rm prime with}}\ \legendre{\ell}{5}=-1\right\}
		\cup \{-37\}.
		$$
		Moreover, if $2k\neq 16,$ then $a_f(n)\neq 37.$
		Assuming GRH, for every $n$ we have
		$$
		a_f(n)\not \in  \left \{ \pm \ell \ : \ 41\leq \ell \leq 97 \ {\text {\rm prime with}}\ \legendre{\ell}{5}=-1\right\}.
		$$
		\item If $\gcd(3\cdot 5, 2k-1)\neq 1$ and $2k\geq 12$, then for every $n$ we have
		$$a_f(n) \not \in \left \{\pm \ell \ : \  11\leq \ell \leq 31 \  \text{ {\rm prime with }} \legendre{\ell}{5}=1\right\}.
		$$
		Assuming GRH, the range of this set can be expanded to include $\ell \leq 89.$
		\item If $7\mid (2k-1)$ and $2k\geq 12$, then for every $n$ we have
		$$a_f(n) \not \in \left \{\pm \ell \ : \  11\leq \ell \leq 31 \  \text{ {\rm prime with }} \legendre{\ell}{5}=1\right\}.
		$$
		Assuming GRH, for every $n$ we have
		$$
		a_f(n)\not \in \{\pm 41, \pm 59, \pm 61, -71, \pm 79, \pm 89\}.
		$$
		\item If $11\mid (2k-1),$ then for every $n$ we have $a_f(n)\neq -19$, and assuming GRH
		we have
		\begin{displaymath}
			a_f(n)\not \in 
			\left \{ -11, -29, -31, -41, -59, -61, -71, -79, -89\right\}.
		\end{displaymath}
		
		\item If $13\mid (2k-1),$ then for every $n$ we have $a_f(n)\neq -11$, and assuming GRH we have
		\begin{displaymath}
			a_f(n)\not \in \left \{-19, -29, -31, -41, -59, -61, -71, -79\right\}.
		\end{displaymath}
	\end{enumerate}
\end{theorem}

\begin{remark} \ \ \ \newline 
	\noindent
	(i) Theorem~\ref{LehmerGeneral} applies to all newforms \cite{OT05} with integer coefficients with level $2^aN$, where $a\geq 0$ and
	$N\in \{1, 3, 5, 15, 17\}$. Moreover,
	the result holds  for all odd levels when $a_f(2)$ is even. 
	
	\smallskip
	\noindent
	(ii) These results follow from Theorem~\ref{LehmerVariantGeneral}, which  constrains
	coefficients that are odd prime powers in absolute value. This method extends to arbitrary odd integers by Hecke multiplicativity, thereby giving an algorithm for determining whether a given odd integer is a newform coefficient. 
	
	\smallskip
	\noindent
	(iii) The proof of Theorem~\ref{LehmerGeneral} (2-6)  locates values $\pm \ell$ that are possible coefficients.
	For example, Theorem~\ref{LehmerGeneral} (2) allows weight 4 coefficients to be in the set
	$\{\pm 11, -13, 17, \pm 19, -23, 37\}.$
	The proof shows that these values can only occur as one of the following coefficients:
	\begin{displaymath}
		\begin{split}
			&a_f(3^2)=37,\  \  a_f(3^2)=-11,\  \ a_f(3^2)=-23, \ \  a_f(3^4)=19,\ \ a_f(5^2)=19, \\
			&a_f(7^2)=-19,\ \ a_f(7^4)=11,\ \   a_f(17^2)=-13, \ \ a_f(43^2)=17.
		\end{split}
	\end{displaymath}
	Similarly, Theorem~\ref{LehmerGeneral} (6) allows a coefficient  of 37 for weight $16,$
	which must be $a_f(3^2)=37.$ 
	
	\smallskip
	\noindent
	(iv) The assumption that $2k\geq 4$ guarantees that certain algebraic curves have positive genus,
	and so have finitely many integer points by Siegel's theorem.  Moreover, we do not believe that conclusions analogous to those obtained in Theorem~\ref{LehmerGeneral} hold for weight 2 newforms.
	
	\smallskip
	\noindent
	(v) Some of the results in Theorem~\ref{LehmerGeneral} rely on the GRH.  These cases pertain to situations where GRH was required to  reduce the running time of certain computational number theoretic algorithms. The unconditional bounds lead to infeasible computer calculations.
	
\end{remark}

\begin{example} By Theorem~\ref{LehmerGeneral},  the coefficients of the Hecke eigenform $E_4(z)\Delta(z)$ 
	never belong to
	$$\{-1\} \cup
	\{\pm \ell \ : \  3\leq \ell \leq 37\ {\text {\rm prime}}\}.
	$$
	Moreover, under GRH the range of the second set can be extended to the odd primes $\ell \leq 97.$
\end{example}

Theorems~\ref{Lehmer135} and \ref{LehmerGeneral} offer variants of Lehmer's speculation for individual newforms.
It is natural to consider an aspect of these questions where the newforms $f$ vary. 
Namely, can a fixed odd $\alpha$ be a Fourier coefficient of
newforms with arbitrarily large weight? We effectively show that this is generically not the case. To ease notation, if $\ell$ is an odd prime, then let
$\mathbb{S}_{\ell}$ denote the set of even weight newforms with integer coefficients,
trivial residual mod 2 Galois representation, and even level that is coprime to $\ell$.

\begin{theorem}\label{Power}
	If $\ell$ is an odd prime and $m\in \ZZ^{+},$ then there are effectively computable constants
	$M^{\pm}(\ell,m)=O_{\ell}(m)$ for which
	$\pm \ell^m$ is not a  coefficient of any
	$f\in \mathbb{S}_{\ell}$ with weight $2k>M^{\pm}(\ell,m).$
	In particular,\footnote{We offer these values to indicate that one can easily work out explicit constants.} 
	for $\ell \in \{3, 5\},$ we have
	$$
	M^{\pm}(\ell,m):= \begin{cases}
		2m+10^{23}\sqrt{m} \ \ \ \ \ &{\text {\rm if $\varepsilon=+, m$ odd, and $\ell=3$}},\\
		2m+10^{13}\sqrt{m} \ \ \ \ \ &{\text {\rm if $\varepsilon=+, m$ even, and $\ell=3$}},\\
		2m+10^{32}\sqrt{m} \ \ \ \ \ &{\text {\rm if $\varepsilon=-$ and $\ell=3$}},\\
		3m+10^{24}\sqrt{m} \ \ \ \ \ &{\text {\rm if $\varepsilon=\pm, m$ odd, and $\ell=5$}},\\
		3m+10^{13}\sqrt{m} \ \ \ \ \ &{\text {\rm if $\varepsilon=+, m$ even, and $\ell=5$}},\\
		3m+10^{30}\sqrt{m} \ \ \ \ \ &{\text {\rm if $\varepsilon=-, m$ even, and $\ell=5$}}.\\
	\end{cases}
	$$                                  
\end{theorem}

\begin{remark}\ \ \ \  \newline 
	\noindent
	(i) The condition that the level of $f$  is even  is not crucial  for the proof of Theorem~\ref{Power}.
	If the level is odd, then  the proof implies that
	$a_f(2n+1)\neq \pm \ell^m$ for all $n$ provided that $f$ has  large weight. Furthermore, if $a_f(2)$ is even, then  the stronger claim that $\pm \ell^m$
	is not a Fourier coefficient holds.
	
	\smallskip
	\noindent
	(ii) The condition that the level of $f$ is coprime to $\ell$  also is not crucial. If $\ell$ exactly divides the level, then there is at most one counterexample, and it will be
	a Fourier coefficient of the form $a_f(\ell^r)$ (see Theorem~\ref{Newforms} (4)).
	Otherwise, the stronger claim holds.
	
	\smallskip
	\noindent
	(iii) Using the methods in this paper, one can obtain a generalization of Theorem~\ref{Power} for all odd $\alpha$, as well as analogous results for odd weights and forms with real Nebentypus. 
\end{remark}

These results are related to lower
bounds for the number of prime divisors
of coefficients of newforms. We obtain a general
theorem (see Theorem~\ref{PrimeDivisorsGeneral}) which implies the following lower bound for $\Omega(\tau(n)),$ the number of prime divisors (counted with multiplicity) of $\tau(n)$. As usual,
we let $\omega(n)$ denote the number of distinct prime divisors of $n,$ and we let $\ord_p(n)$ denote the power of $p$ dividing $n.$

\begin{theorem}\label{PrimeDivisors}
	If $n>1$ is an integer, then
	$$
	\Omega(\tau(n))\geq \sum_{\substack{p\mid n\\ prime}}\left ( \sigma_0(\ord_p(n)+1)-1\right)\geq \omega(n).
	$$
\end{theorem}

\begin{remark}Theorem~\ref{PrimeDivisors} is sharp, as the
	prime in (\ref{LehmerPrime})
	satisfies $\Omega(\tau(251^2))=\sigma_0(3)-1=1.$
\end{remark}

\section{Proof of Theorem \ref{LehmerGeneral}} \label{SectionLehmer}

Regarding coefficients of newforms satisfying (\ref{qexpansion}), we classify those $n$ for which
$|a_f(n)|=\ell$ is an odd prime. For the remainder of the paper, we assume that all newforms have weight $2k\geq 4$.
We first determine when $|a_f(n)|=1$. Define the set
\begin{equation}
	\mathcal{U}_f:=\begin{cases} \{1, 4\}  \  \ \ \ \ &{\text {\rm if}}\ a_f(2)=\pm 3,\ 2k=4, {\text {\rm and}}\ N\ {\text {\rm odd}}\},\\
		\{1\} \ \ \ \ &{\text {\rm otherwise.}}
	\end{cases}
\end{equation}
\begin{proposition}\label{One}
	Suppose that the mod 2 residual Galois representation for $f(z)$ is trivial. Then we have
	$|a_f(n)|=1$ if and only if $n\in \mathcal{U}_f.$
\end{proposition}

\begin{proof}
	By multiplicativity (i.e. Theorem~\ref{Newforms} (1)), it suffices to
	determine when $|a_f(p^m)|=1$, where $p$ is prime.
	By Proposition~\ref{PrimePower} (1), we have $p\nmid N.$ By Theorem~\ref{Newforms}~(3), it suffices to determine when
	the $|u_{m+1}(\alpha_p,\beta_p)|=1,$ where $m\geq 2.$ Indeed, $a_f(p)=u_2(\alpha_p,\beta_p)$ is even for  $p\nmid 2N$.
	By Theorem~\ref{Bilu}, this reduces to
	Lemma~\ref{Part1DefectivePrimality} (1).
	The defective cases $(A,B,n)=(\pm 3, 2^3,3)$  correspond to potential weight 4 newforms, while the remaining possibilities
	are for weight 2. In the weight 4 cases we have $a_f(2)=\pm 3$, which  gives $a_f(4)=a_f(2)^2-2^3=1.$
\end{proof}

\begin{theorem}\label{LehmerVariantGeneral}
	Suppose that the mod 2 residual Galois representation for $f(z)$ is trivial. If $|a_f(n)|=\ell^m,$  with $m\in \ZZ^{+}$ and $\ell$ is an odd prime, then 
	$n= m_0p^{d-1}$, where $m_0\in \mathcal{U}_f$, $p\nmid N$ is prime, and $d \mid \ell (\ell^2-1)$
	is an odd prime.
	Moreover, $|a_f(n)|=\ell^m$ for finitely many (if any) $n$. 
\end{theorem}

\begin{proof}[Proof of Theorem~\ref{LehmerVariation} and \ref{LehmerVariantGeneral}]
	By Proposition~\ref{One} and Theorem~\ref{Newforms} (1) and (4), it suffices to determine when
	$|a_f(p^{d-1})|=|u_{d}(\alpha_p,\beta_p)|=\ell$, where $p\nmid N$ is prime. Since  $2k\geq 4,$
	$\ell$ is odd, and $A=a_f(p)$ is even, Lemma~\ref{Part1DefectivePrimality} (2) leaves the defective possibilities
	$(A,B,\ell,n)=(\pm m,p^{2k-1},3,3)$, which by Theorem~\ref{Newforms} (2), implies that $(p,a_f(p))$ is an integer point on
	$Y^2=X^{2k-1}\pm 3.$
	This means that $u_3(\alpha_p,\beta_p)=a_f(p^2)=\pm 3$, which 
	is the claimed conclusion with $d=\ell=3$.
	
	Now we consider whether a prime power can be
	a nondefective Lucas number $u_{d}(\alpha_p,\beta_p)=a_f(p^{d-1})$, for primes $p\nmid 2N$.
	Since $a_f(p)$ is even, we may assume that $\ell \nmid \alpha_p \beta_p$ and $m_{\ell}(\alpha_p,\beta_p)>2$. Moreover, Theorem~\ref{Newforms} (2) implies
	that $a_f(p^b)$ is odd if and only if $b$ is even, and so we may assume that $d$ is odd.
	Proposition~\ref{PropB} implies that $m_{\ell}(\alpha_p,\beta_p)=\ell$ or $m_{\ell}(\alpha_p,\beta_p) | (\ell-1)$ or
	$m_{\ell}(\alpha_p,\beta_p) | (\ell+1)$.
	
	Due to the generic presence of primitive prime divisors, a Lucas number that is a prime power $\ell^m$ in absolute
	value is the first multiple of $\ell$ in the sequence.
	By Theorem~\ref{Bilu}, Proposition~\ref{PropA}, and Lemma~\ref{Part1DefectivePrimality} (2), this
	holds for every sequence satisfying (\ref{Modularity}) for weights $2k\geq 4$.
	In particular, $d$ is an odd prime. 
	The finiteness of the number of $p$ for which $|a_f(p^{d-1})|=\ell$, follows from Siegel's Theorem, that positive genus curves  have at most finitely many integer points. These curves are easily assembled using Theorem~\ref{Newforms} (2) (see Lemma~\ref{DiophantineCriterion}).
\end{proof}

\section{Integral Points on some curves}\label{IntegerPoints}
To prove Theorems~\ref{Lehmer135} and \ref{LehmerGeneral}, we require knowledge of the integer points on certain curves.

\subsection{Some Thue equations}
An equation of the form
$F(X,Y)=D,
$
where $F(X,Y)\in \ZZ[X,Y]$ is homogeneous and $D$ is a non-zero integer, is known as a {\it Thue equation}.
We require such equations that arise
from the generating function
\begin{equation}\label{genfunction}
	\frac{1}{1-\sqrt{Y}T+XT^2}=\sum_{m=0}^{\infty}F_m(X,Y)\cdot T^{m}=1+\sqrt{Y}\cdot T+(Y-X)T^2+\cdots.
\end{equation}
The first few homogenous polynomials $F_{2m}(X,Y)$ are as follows:
\begin{displaymath}
	\begin{split}
		F_2(X,Y)&=Y-X,\\
		F_4(X,Y)&=Y^2-3XY+X^2\\
		F_6(X,Y)&=Y^3-5XY^2+6X^2Y-X^3.\\
		F_{10}(X,Y)&=Y^5-9XY^4+28X^2Y^3-35X^3Y^2+15X^4Y-X^5.
	\end{split}
\end{displaymath}
For every positive integer $m$, we consider the degree $m$ Thue equations of the form
\begin{equation}\label{ThueEqn}
	F_{2m}(X,Y)=\prod_{k=1}^m \left(Y-4X \cos^2\left(\frac{\pi k}{2m+1}\right)\right)=D.
\end{equation}

The next lemma gives integer points on several Thue equations that 
we shall require.

\begin{lemma}\label{Monster} The following are true.
	\begin{enumerate}
		\item
		Table \ref{thuetable} in the Appendix lists all of the integer solutions to
		$$F_{d-1}(X,Y)=\pm\ell
		$$
		for every pair of odd primes $(d,\ell)$ for which $7\leq d\mid \ell(\ell^2-1)$ and $\ell \in \{ 7\leq \ell \leq 37\}$.
		\item Conditional on GRH, Table \ref{thueGRHtable} in the Appendix lists all of the integer solutions to
		$$
		F_{d-1}(X,Y)=\pm\ell
		$$
		for every pair of odd primes $(d,\ell)$ for which $7\leq d\mid \ell(\ell^2-1)$ and $41\leq \ell \leq 97.$
		\item There are no integer solutions to
		$F_{22}(X,Y)=\pm 691.$
		\item The points $(\pm 1, \pm 4)$ are the only integer solutions to
		$F_{690}(X,Y)=\pm 691.$
	\end{enumerate}
\end{lemma}

\begin{proof} Claims (1), (2) and (3) are easily obtained using the Thue solver in \texttt{PARI/GP} \cite{pari}
	(see \cite{github} for all of the code required for this paper).
	
	The proof of (4) is more formidable, as $F_{690}(X,Y)$ has degree 345. However, for odd primes $p$, the Thue equations $F_{p-1}(X,Y)=\pm p$
	are equivalent to the well-studied equations
	\begin{equation}\label{ModifiedThue}
		\widehat{F}_p(X,Y)=\prod_{k=1}^{\frac{p-1}{2}}\left(Y-2X\cos\left(\frac{2\pi k}{p}\right)\right)=\pm p
	\end{equation}
	that were prominent in the work of Bilu, Hanrot, and Voutier on primitive prime divisors of Lucas sequences. 
	Indeed, we have  $F_{p-1}(X,Y)=\widehat{F}_p(X,Y-2X).$
	They prove the important  fact
	(see Cor. 6.6 of \cite{BHV01}) that there are no integer solutions to (\ref{ModifiedThue})
	with $|X|>e^8$  when $31\leq p\leq 787.$ By a well-known criterion (for example, see Lemma~1.1 of \cite{TW89} and Proposition 2.2.1 of \cite{BH96})), midsize solutions of $\widehat{F}_{691}(X,Y)=\pm 691$  correspond to convergents of the continued fraction expansion of some 
	$2\cos(2\pi k/691).$  A short calculation rules this out, possibly leaving some small solutions, those  with
	$|X|\leq 4$.  For these $X$, we find $(\pm 1, \pm 2)$, which implies that
	$(\pm 1, \pm 4)$ are the only integral solutions to $F_{690}(X,Y)=\pm 691.$
\end{proof}

\subsection{The elliptic and hyperelliptic curves $Y^2=X^{2d-1}\pm \ell$}

For  $d\in \{2, 3, 4 ,6, 7\}$ and odd primes $\ell  \leq 97$, we list all of the integer points on
\begin{equation}
	C_{d,\ell}^{\pm}: Y^2=X^{2d-1}\pm \ell.
\end{equation}

\begin{lemma}\label{HYPER} If $3\leq \ell \leq 97$ is prime and $d\in \{2, 3, 4, 6, 7\}$, then the following are true:
	\begin{enumerate}
		\item Table \ref{Cplustable} in the Appendix lists the integer points on $C_{d,\ell}^{+}.$
		\item Table \ref{Cminustable} in the Appendix lists the integer points on $C_{d,\ell}^{-}.$
	\end{enumerate}
\end{lemma}
\begin{proof}
	Work by Barros \cite{Bar10}, Cohn \cite{Coh93} and Bugeaud, Mignotte and Siksek \cite{BMS06b}
	establish these claims. Table \ref{Cplustable} is assembled from the Appendix of \cite{Bar10}, and Table \ref{Cminustable} is assembled from the Appendix of \cite{BMS06b}.
\end{proof}

\subsection{The hyperelliptic curves $Y^2=5X^{2d}\pm 4\ell$}

For $d\geq 2,$ we define the hyperelliptic curves
\begin{equation}
	H^{\pm}_{d,\ell}:  Y^2 =5 X^{2d}\pm 4\ell.
\end{equation}
The following satisfying lemma classifies the integer points on $H_{d,5}^{\pm}.$

\begin{lemma}\label{AnnalsCorollary}
	If $\ell=5$, then the following are true.
	\begin{enumerate}
		\item If $d=2$ and $\ell=5$, then the only integer points on $H^{+}_{2,5}$ are $(\pm 1, \pm 5)$ and $(\pm 2, \pm 10)$.
		\item If $d>2,$ then the only integer points on $H^{+}_{d,5}$ are $(\pm 1, \pm 5).$
		\item If $d\geq 2,$ then $H^{-}_{d,5}$ has no integer points.
	\end{enumerate}
\end{lemma}
\begin{proof}
	We recall the classical Lucas sequence $$\{L_n\}=\{2, 1, 3, 4, 7, 11, 18, 29, 47,76, 123, 199, 322, 521, 843,\dots\},$$ defined by
	$L_0:=2$ and $L_1:=1$ and the recurrence $L_{n+2}:=L_{n+1}+L_n$ for $n\geq 0$.
	A  theorem of Bugeaud, Mignotte, and Siksek \cite{BMS06a} asserts that $L_1=1$ and $L_3=4$
	are the only perfect power Lucas numbers.
	By the theory of Pell's equations, the positive integer $X$-coordinate solutions to
	$H^+_{1,5}$ and $H^{-}_{1,5},$
	namely $\{L_1=1, L_3= 4,L_5=11,\dots\}$ and
	$\{L_0=2, L_2=3, L_4=7,\dots\}$ respectively, split the Lucas numbers. The three claims follow immediately.
\end{proof}

For  primes $\ell \in \{691\}\cup\left \{ 11\leq \ell\leq 89 \ : \  {\text {\rm prime with }} \legendre{\ell}{5}=1
\right\}$, we have the following lemma.

\begin{lemma}\label{Hyper11_19} The following are true.
	\begin{enumerate}
		\item
		For most\footnote{We were unable to obtain results for $H_{7,71}^{+},$
			$H_{13,89}^{-},$ and any $H^{+}_{11,\ell}$ and $H^{+}_{13,\ell}.$}
		$d\in \{3, 5, 7, 11, 13\}$ and primes $\ell \in \left\{ 11\leq \ell \leq 89\ : \ \legendre{\ell}{5}=1\right\}$,
		Table \ref{Htable} in the Appendix lists (some cases conditional on GRH) the integer points on $H^{\pm}_{d,\ell}.$
		\item There are no integer points on $C^{-}_{6,691}.$
		\item There are no integer points on $H^{-}_{11,691}.$
	\end{enumerate}
\end{lemma}
\begin{proof}
	Generalized Lebesgue--Ramanujan--Nagell equations are  equations of the form
	\begin{equation}\label{GLRN}
		x^2+D=Cy^n,
	\end{equation}
	where $D$ and $C$ are non-zero integers. An integer point on (\ref{GLRN}) can be studied in the ring of integers
	of $\QQ(\sqrt{-D})$
	using the factorization
	$$(x+\sqrt{-D})(x-\sqrt{-D}) =Cy^n.
	$$
	
	This observation is a standard tool in the study of Thue equations. In particular, 
	Theorem~2.1 of  \cite{Bar10} (also see Proposition~3.1 of \cite{BMS06b}) gives a step-by-step algorithm that takes alleged solutions of (\ref{GLRN}) and produces integer points on one
	of finitely many Thue equations constructed from $C, D$ and $n$ via the algebraic number theory of $\QQ(\sqrt{-D})$.
	These equations are assembled from the knowledge of the group of units and the
	ideal class group.
	
	To prove all three parts of the lemma (apart from $H_{7,89}^{+}$), we implemented this algorithm in \texttt{SageMath} (see \cite{github} for all \texttt{SageMath} code required for this paper). 
	Some cases required GRH as a simplifying assumption. As the curves in (2) and (3) are the most complicated, we offer brief details in these two cases.
	
	To prove (2), we consider the  hyperelliptic curve $C^{-}_{6,691},$ which  corresponds to (\ref{GLRN}) for the class number 5 imaginary quadratic field
	$\QQ(\sqrt{-691})$, where $x=Y, y=X, C=1, D=691,$ and $n=11.$  In this case the algorithm gives exactly one Thue equation, which after clearing denominators can be rewritten as
	{\small
		\begin{align*}
			2\times5^{55}&=(991077174272090396)x^{11} + (119700018439220789119)x^{10}y
			\\ &- (8831599221002836172345)x^9y^2
			-(337116345512786456280840)x^8y^3 
			\\ &+ (8492967300375371034332430)x^7y^4
			+ (175189311986919278870504298)x^6y^5 \\
			&- (1881807368163995585644810248)x^5y^6
			- (22992541672786450593030038430)x^4y^7 \\
			&+ (104772541553739359102253613965)x^3y^8
			+ (697875798749922445133117312720)x^2y^9 \\
			&- (1068801486169809452619368218519)xy^{10}
			- (2292300374810647823111384294421)y^{11}.
		\end{align*}
	}

	\normalsize
	The Thue equation solver in \texttt{PARI/GP}, which implements the Bilu--Hanrot algorithm, establishes that there are no integer solutions, and so $C^{-}_{6,691}$ has no integer points.
	
	Claim (3) is about the hyperelliptic curve $H^{-}_{11,691}.$  Its integer points $(X,Y)$  satisfy
	$$
	(Y+2\sqrt{-691})(Y-2\sqrt{-691})=5X^{22}.
	$$
	Therefore, we again employ the imaginary quadratic field $\QQ(\sqrt{-691})$. In particular, we have (\ref{GLRN}), where $x=Y, y=X, C=5, D=4\cdot 691$ and $n=22$.
	The algorithm again gives one Thue equation, which after clearing denominators can be rewritten as
	
	\small 
	\begin{displaymath}
		\begin{split}
			2^2\times5^{110}&=
			-(20587212586465949627980680671826599752) x^{22} \\
			&\ \  \ \   + (1133274396835827658613802749227310922394) x^{21} y\\ 
			&\ \ \ \ +\cdots\\
			&\ \ \ \  -(79670423145107301772779399379735976309907264511718034789276856) x y^{21}\\
			&\ \ \ \  +(71809437208138431262783549625248617351731199323326115439324273)y^{22}.
		\end{split}
	\end{displaymath}
	\normalsize
	The Thue solver in \texttt{PARI/GP} establishes that there are no integer solutions, and so $H^{-}_{11,691}$ has no integer points.
\end{proof}

We use the Chabauty--Coleman method\footnote{We could have (in theory) used the Thue method as in the proof of Lemma~\ref{Hyper11_19}. We chose this method as it did not require substantial computer resources.}, which employs $p$-adic integration to determine the rational points on suitable curves of genus $g\geq 2,$  to determine  the integer points on  $C_{6,691}^{+}$, $H_{7,89}^{+}$, and
$H_{11,691}^{+}.$

\begin{lemma}\label{Plus691} The following are true.
	\begin{enumerate}
		\item There are no integer points on $C^{+}_{6,691}.$
		\item There are no integer points on  $H^{+}_{11,691}.$
		\item Assuming GRH, the only integer points on $H^{+}_{7,89}$ have $(|X|,|Y|)=(1,19).$
	\end{enumerate}
\end{lemma}
\begin{proof}
	We employ the Chabauty--Coleman method \cite{Col85} to determine the integral points on these curves. 
	
	We first prove (1). The genus 5 curve $C^{+}_{6,691}$ has Jacobian with Mordell-Weil rank 0. This can be determined using the implementation of 2-descent in \texttt{Magma} \cite{magma}. Since the rank is less than the genus,  the Chabauty--Coleman method applies, which, in this case, gives a 5-dimensional space of regular 1-forms vanishing on rational points. We take as our basis for the space of annihilating differentials the set $\{\omega_i := X^i \frac{dX}{2Y}\}_{i = 0, 1, \ldots, 4}.$ The prime $p = 3$ is a prime of good reduction for $C^{+}_{6,691}$, and taking the point at infinity $\infty$ as our basepoint, we compute the set of points $$\left\{z \in C^{+}_{6,691}(\ZZ_3): \int_{\infty}^z \omega_i = 0\;\textrm{for all}\;{i = 0, 1, \ldots, 4}\right\},$$ where the integrals are Coleman integrals computed using \texttt{SageMath} \cite{sage}. By construction,  this set contains the integral points on the working affine model of $C^{+}_{6,691}$. 
	
	The computation gives three points: two points with $X$-coordinate 0 and a third point with $Y$-coordinate 0 in the residue disk corresponding to $(2,0) \in C^{+}_{6,691}(\FF_3)$. (Indeed, the power series corresponding to the expansion of the integral of $\omega_0$ has each of these points occurring as simple zeros.)  Hence, there are no integral points on $C^{+}_{6,691}$.
	
	Turning to $H^{+}_{11,691}$, we consider the integral points on the curve $Y^2 = 5X^{11} + 4 \cdot 691$ and then pull back any points found using the map $(X,Y) \rightarrow (X^2,Y)$. Using \texttt{Magma}, we find that the rank of the Jacobian of this genus 5 curve is 0. We rescale variables to work with the monic model $Y^2 = X^{11}+4\cdot 5^{10}\cdot 691$ and we apply the Chabauty--Coleman method using $p = 3$. As before, the computation gives three points with coordinates in $\ZZ_3$: two points with $X$-coordinate 0 and a third point with $Y$-coordinate 0 in the residue disk corresponding to $(2,0)$. The power series corresponding to the expansion of the integral of $\omega_0$ has each of these points occurring as simple zeros. None of these points are rational. Therefore, $H^{+}_{11,691}$ has no integral points. This proves (2).
	
	Now we turn to (3). To compute integral points on $H^{+}_{7,89}$, we work with the genus 3 curve $Y^2 = 5X^{7} + 4 \cdot 89$ and then pull back any integral points found using the map $(X,Y) \rightarrow (X^2,Y)$. Using \texttt{Magma}, we find that the rank of the Jacobian of this genus 3 curve is 2, under the assumption of GRH\footnote{The \texttt{Magma} procedure that computes ranks requires GRH in this case to be computationally feasible.}. We work with the monic model $$H_m: Y^2 = X^{7}+4\cdot 5^{6}\cdot 89$$ and run the Chabauty--Coleman method using $p = 3$. 
	
	The points 
	$$P = [x^3 + 14x^2 - 800, 9x^2 + 200x - 4050] \qquad\textrm{and}\qquad Q = [x-5, 19 \cdot 5^3]$$
	(given in Mumford representation) are independent in the Jacobian of $H_m$. To simplify the Chabauty--Coleman computation---in particular, so that we carry out all of our computations over $\QQ_3$---we replace $P$ with $P'$, a small $\ZZ$-linear combination of $P$ and $Q$ that is linearly independent from $Q$, with the property that the first coordinate of the Mumford representation of $P'$ splits over $\QQ_3$. 
	
	We take $P' := 2P - 5Q$, with Mumford representation of $P'$ given by $[f(x),g(x)]$ where
	\tiny{
		\begin{align*}
			f(x) &= x^3 - \frac{57819608106819190393450758001494220029312032281}{243432625872206959773347921129373894485149809}x^2 +\frac{301022057022978383553067428985393708004188803800}{81144208624068986591115973709791298161716603}x -\\
			&\qquad   
			\frac{4935244227803215636634926465657011220846146763100}{243432625872206959773347921129373894485149809}, \\
			g(x)&= \frac{13467788979408324218581419111573847035681150845619031139253274307312471}{3798115572194618764136691476777323149900556269646219373513689210377}x^2 -\\  
			&\qquad            \frac{73837091689655128840131596065726589815272462202819205672839132728899500}{1266038524064872921378897158925774383300185423215406457837896403459}x +\\
			&\qquad  \frac{1249983247105360333943070938652709476597593148217064351317870016169354850}{3798115572194618764136691476777323149900556269646219373513689210377}.
		\end{align*}
	}
	\normalsize
	To compute an annihilating differential, we compute the $3\times 2$ matrix of Coleman integrals $(\int_{P'} \omega_i, \int_Q \omega_i)_{i = 0, 1, 2}$, where $\omega_i = X^i \frac{dX}{2Y}$, in \texttt{Sage}:\tiny{
		$$\left(\begin{array}{rr}
			2 \cdot 3 + 2 \cdot 3^{2} + 3^{4} + 2 \cdot 3^{6} + 3^{8} + 2 \cdot 3^{9} + O(3^{10}) & 3^{3} + 2 \cdot 3^{4} + 3^{7} + 2 \cdot 3^{8} + 3^{9} + O(3^{10}) \\
			2 \cdot 3 + 3^{2} + 3^{3} + 2 \cdot 3^{5} + 2 \cdot 3^{6} + 2 \cdot 3^{7} + O(3^{10}) & 2 \cdot 3 + 3^{2} + 3^{3} + 2 \cdot 3^{7} + 2 \cdot 3^{8} + 3^{9} + O(3^{10}) \\
			3 + 3^{2} + 2 \cdot 3^{3} + 2 \cdot 3^{4} + 2 \cdot 3^{5} + 3^{6} + 3^{7} + 2 \cdot 3^{9} + O(3^{10}) & 2 \cdot 3 + 3^{2} + 3^{3} + 2 \cdot 3^{4} + 3^{5} + 3^{7} + 2 \cdot 3^{8} + 2 \cdot 3^{9} + O(3^{10})\end{array}\right).$$}
	
	\normalsize
	
	We then compute a basis of the kernel of this matrix, which gives us our annihilating differential 
	\begin{align*}\omega &= \omega_0 + (1 + 2 \cdot 3^{2} + 2 \cdot 3^{4} + 3^{5} + 3^{6} + 2 \cdot 3^{7} + 2 \cdot 3^{8} + 2 \cdot 3^{9} +  O(3^{10})) \omega_1 \\
		&\qquad\; + (2 + 2 \cdot 3 + 3^{2} + 3^{3} + 2 \cdot 3^{4} + 3^{5} + 2 \cdot 3^{6} + 3^{9} + O(3^{10}))\omega_2.
	\end{align*}
	
	Finally, we have three residue disks to consider, corresponding to $(1,0)$ and $(2, \pm 1) \in H_{m}(\FF_3)$. We compute the set of points $z \in H_m(\ZZ_3)$ in these residue disks such that $\int_{\infty}^z \omega = 0$. This produces three points, each occurring as simple zeros of the corresponding $3$-adic power series: a Weierstrass point and the points $(5, \pm 2375).$
	The Weierstrass point is not rational, while the points $(5, \pm 2375)$ correspond to the points $( \pm 1, \pm 19)$ on $H^{+}_{7,89}$.
\end{proof}

\section{Proof of Theorem \ref{LehmerVariantGeneral}}

We combine  results from the previous section with
Theorem~\ref{LehmerVariantGeneral} to prove Theorems~\ref{Lehmer135} and \ref{LehmerGeneral}.
The following lemma, which relates Fourier coefficients to special integer points on algebraic curves, is a straightforward consequence of Theorem~\ref{Newforms} (2) and (3).

\begin{lemma}\label{DiophantineCriterion}
	Assuming the notation in Theorem~\ref{Newforms}, if $p\nmid N$ is prime, then
	we have the following:
	\begin{enumerate}
		\item If $a_f(p^2)=\alpha$, then $(p,a_f(p))$ is an integer point on
		$$
		Y^2=X^{2k-1}+\alpha.
		$$
		\item If $a_f(p^4)=\alpha$, then $(p,2a_f(p)^2-3p^{2k-1})$ is an integer point on
		$$
		Y^2=5X^{2(2k-1)}+4\alpha.
		$$
		\item For every positive integer $m$ we have that
		$F_{2m}(p^{2k-1},a_f(p)^2)=a_f(p^{2m}).$
	\end{enumerate}
\end{lemma}

\begin{proof}[Proof of Theorem~\ref{Lehmer135}]
	It is well-known that $\tau(n)$ is odd if and only if $n$ is an odd square. 
	To see this, we employ 
	the Jacobi Triple Product identity to obtain the congruence
	\begin{displaymath}
		\begin{split}
			\sum_{n=1}^{\infty}\tau(n)q^n:&=q\prod_{n=1}^{\infty}(1-q^n)^{24}
			\equiv q\prod_{n=1}^{\infty}(1-q^{8n})^3
			=\sum_{k=0}^{\infty} (-1)^k(2k+1)q^{(2k+1)^2}\pmod 2.
		\end{split}
	\end{displaymath}

	We consider the possibility that $\pm 1$ appear in sequences of the form
	\begin{equation}\label{primepowers}
		\{\tau(p),\tau(p^2), \tau(p^3),\dots\}.
	\end{equation}
	By Theorem~\ref{Newforms} (2), if $p$ is prime and  $p\mid \tau(p)$, then $p^m\mid \tau(p^m)$ for every $m\geq 1$, and
	so $|\tau(p^m)|\neq 1.$  Moreover, $|\tau(p)|\neq p$, where $p$ is an odd prime, because $\tau(p)$ is even.
	Therefore, such sequences may be completely ignored for the remainder of the proof.
	
	For primes
	$p\nmid \tau(p),$ Theorem~\ref{Newforms} (3) gives a Lucas sequence 
	with $A=\tau(p)$ and $B=p^{11}.$
	Lemma~\ref{Part1DefectivePrimality} shows that there are no defective terms with
	$u_{m+1}(\alpha_p,\beta_p)=\tau(p^m)\neq \pm 1$ or $\pm \ell$, where $\ell$ is an odd prime.
	To see this, we note that $A=\tau(p)$ is even.  Lemma~\ref{Part1DefectivePrimality} (2) does not allow
	for $A$ to be even with one exception, the possibility that  $(A,B,\ell,n )=(\pm m,p^{11}, 3,3)$, where $(p,\pm m)\in B_{1,6}^{1,\pm}.$
	However, these curves are the same as $C_{6,3}^{\pm},$ and Lemma~\ref{HYPER}  shows that there are no such points.
	Therefore, we may assume that all of 
	the values in (\ref{primepowers}) have a primitive prime divisor, and  never have  absolute value 1.
	
	We now turn to the primality of absolute values of $\tau(n)$.
	Thanks to Hecke multiplicativity (i.e. Theorem~\ref{Newforms} (1)) and the discussion above, if $\ell$ is an odd prime and $|\tau(n)|=\ell$, then $n=p^d$, where $p$ is an odd prime for which
	$p\nmid \tau(p).$ The fact that $\tau(p^d)=u_{d+1}(\alpha_p,\beta_p)$ leads to a further constraint on $d$ 
	(i.e. refining 
	the fact that $d$ is even).
	By Proposition~\ref{PropA}, which guarantees relative divisibility between Lucas numbers, and 
	Lemma~\ref{Modularity}, which guarantees the absence of defective terms in (\ref{primepowers}), it follows that 
	$d+1$ must be an odd prime, and $\tau(p^d)$ is the very first term that is divisble by $\ell$.
	
	To make use of this observation, for odd primes 
	$p$ and $\ell$ we define
	\begin{equation}
		m_{\ell}(p):=\min\{ n\geq 1\ : \ \tau(p^n)\equiv 0\!\!\!\!\pmod{\ell}\}.
	\end{equation}
	For $|\tau(p^d)|=\ell$, we have $m_{\ell}(p)=d,$ where $d+1$ is also an odd prime.
	The Ramanujan congruences \cite{BO99, Ram16, Ser68}
	\begin{displaymath}
		\tau(n)\equiv \begin{cases} &n^2\sigma_1(n)\pmod 9,\\
			&n\sigma_1(n)\pmod 5,\\
			&n\sigma_3(n)\pmod 7,\\
			&\sigma_{11}(n)\pmod{691},
		\end{cases}
	\end{displaymath}
	where $\sigma_{\nu}(n):=\sum_{1\leq d\mid n}d^{\nu}$, make it simple to compute $m_{\ell}(p)$
	for the primes $\ell \in \{3, 5, 7, 691\}.$
	
	Thanks to the mod 9 congruence, we find that
	$$
	m_3(p)=\begin{cases} 1 \ \ \ \ \ &{\text {\rm if }} p\equiv 0, 2\!\!\!\!\pmod 3,\\
		2\ \ \ \ \ &{\text {\rm if }} p\equiv 1\!\!\!\!\pmod 3.
	\end{cases}
	$$
	Therefore, $d=2$ is the only possibility.
	If $\tau(p^2)=\pm 3$, then Lemma~\ref{DiophantineCriterion} (1) implies that
	$(p,\tau(p))$ is a point on $C^{\pm}_{6,3},$ which were considered immediately above.
	Again,  Lemma~\ref{HYPER} (1) implies that there are no such integer points.
	
	Thanks to the mod 5 congruence, we find that
	$$
	m_5(p)=\begin{cases} 1 \ \ \ \ \ &{\text {\rm if }} p\equiv 0, 4\!\!\!\!\pmod{5},\\
		3 \ \ \ \ \ &{\text {\rm if }} p\equiv 2, 3\!\!\!\!\pmod{5},\\
		4 \ \ \ \ \ &{\text {\rm if }} p\equiv 1\!\!\!\!\pmod{5}.
	\end{cases}
	$$
	Therefore, $d=4$ is the only possibility.
	If $\tau(p^4)=\pm 5$, then Lemma~\ref{DiophantineCriterion} (2) implies that
	$(p, 2\tau(p)^2-3p^{11})$ is an integer point on $H^{\pm}_{11,5}.$
	Lemma~\ref{AnnalsCorollary} shows that no such points exist on these hyperelliptic curves.
	
	Thanks to the mod 7 congruence, we find that
	$$
	m_7(p)=\begin{cases}
		1 \ \ \ \ \ &{\text {\rm if }} p\equiv 0, 3, 5, 6\!\!\!\!\pmod{7},\\
		6 \ \ \ \ \ &{\text {\rm if }} p\equiv 1, 2, 4\!\!\!\!\pmod{7}.
	\end{cases}
	$$
	Hence, $d=6$ is the only possibility, and so we must rule out the possibility that $\tau(p^6)=\pm 7$.
	If there are such primes $p$, then Lemma~\ref{DiophantineCriterion} (3) implies
	that
	$F_{6}(p^{11},\tau(p)^2)=\pm 7.$
	Lemma~\ref{Monster} (1) shows that there are no such solutions to $F_6(X,Y)=\pm 7.$

	Thanks to the mod 691 congruence, we find that
	the only cases where $m_{691}(p)=d$ where $d+1$ is an odd prime are $d=2, 4, 22,$ and $690$.
	For the cases where $d=2$ and $4$ respectively, Lemma~\ref{DiophantineCriterion} (1-2) implies that
	$(p,\tau(p))$ would be an integral point on $C^{\pm}_{6,691},$ and that $(p, 2\tau(p)^2-3p^{11})$ would be an integral point on
	$H^{\pm}_{11,691}.$ Lemma~\ref{Hyper11_19} (2-3) and Lemma~\ref{Plus691} show that no such points exist.
	By Lemma~\ref{DiophantineCriterion} (3), the remaining cases (i.e. $d=22$ and $690$)  correspond to the Thue equations
	$F_{22}(p^{11},\tau(p)^2)=\pm 691$ and 
	$F_{690}(p^{11},\tau(p)^2)=\pm 691.$
	Lemma~\ref{Monster} (3) and (4) show that there are no such integer solutions.
	
	The arguments above show that $\tau(n)\not \in \{\pm1 , \pm 3, \pm 5, \pm 7, \pm 691\}.$
	The remaining cases are special cases of Theorem~\ref{LehmerGeneral} (6) and (9) and are  proved below.
\end{proof}

\begin{proof}[Proof of Theorem~\ref{LehmerGeneral}]  By hypothesis, for primes $p\nmid 2N$ we have that
	$a_f(p)$ is even. For such primes, Theorem~\ref{Newforms} (2) implies that
	$a_f(p^{m})$ is odd if and only if $m$ is even.
	Suppose that $p$ is a prime for which $p\mid a_f(p),$ which includes those primes $p\mid 2N$ by Theorem~\ref{Newforms} (4).  Theorem~\ref{Newforms} (2) and (4) imply that
	$p^m\mid a_f(p^m)$. Therefore, we do not need to consider these coefficients in the remainder of the proof.
	
	It suffices to consider the Lucas sequences corresponding to $A=a_f(p)$ and $B=p^{2k-1}$, when $p\nmid a_f(p)$.
	By applying Lemma~\ref{Part1DefectivePrimality} (2) (as above in the proof of Theorem~\ref{Lehmer135}), we may assume that
	$\{1, a_f(p), a_f(p^2),\dots\}$ is a Lucas sequence
	without any defective terms. To establish this, we must show that
	$B_{1,k}^{1,\pm},$ which are the same as $C_{k,3}^{\pm},$  have no suitable integer points.
	Since we only consider weights for which $\gcd(3\cdot 5\cdot 7\cdot 11\cdot 13, 2k-1)\neq 1$, it suffices
	to show that $C_{d,3}^{\pm}$ has no such points for $d\in \{2, 3, 4, 6, 7\}$. Lemma~\ref{HYPER} confirms this requirement for
	these ten curves.
	
	The first claim of the theorem now follows from Proposition~\ref{One}. To prove the remaining claims we apply Theorem~\ref{LehmerVariantGeneral}.
	Namely, if $|a_f(n)|=\ell,$ then $n=p^{d-1}$, where
	$d\mid \ell (\ell^2-1)$ is an odd prime. The existence of such coefficients can be ruled out
	with Lemma~\ref{DiophantineCriterion}, which reduces the proof to a case-by-case search for suitable integral points on hyperelliptic curves and solutions to Thue equations which were considered in the previous section.
	If $a_f(p^2)=\pm \ell$, then  $(p,a_f(p))\in C_{k,\ell}^{\pm}$.
	If $a_f(p^4)=\pm \ell$, then  $(p, 2_f(p)^2-3p^{2k-1})\in
	H_{2k-1,\ell}^{\pm}$. 
	Obviously, it suffices to study curves $C_{d,\ell}^{\pm}$  (resp. $H_{2d-1,\ell}^{\pm}$) with $d\mid (2k-1)$.
	Finally, if $a_f(p^{d-1})=\pm \ell$  with  $d\geq 7,$ then $(p^{2k-1}, a_f(p)^2)$ is a solution to
	$F_{d-1}(X,Y)=\pm \ell.$
	By Lemmas~\ref{Monster}, \ref{HYPER}, \ref{AnnalsCorollary}, and \ref{Hyper11_19} (i.e. inspecting the tables in the Appendix), 
	there are no such integral points (sometimes under GRH) in the cases claimed by the theorem.
\end{proof}

\section{Baker's linear forms in logarithms}

To prove Theorem~\ref{Explicit35}, we make use of the following classical result
of Baker and W\"ustholz \cite{BW93} on linear forms in logarithms.


\begin{theorem}[p. 20 of \cite{BW93}]\label{BW}
	Let $\alpha_1,\ldots,\alpha_r$ be algebraic numbers and $b_1,\ldots,b_r$ be rational integers. If $\Lambda:=b_1\log\alpha_1+\cdots+b_r\log\alpha_r$ (note. where the logarithms have their principal values such that $-\pi<\mathrm{Im}(\log \alpha)\leq\pi$) is nonzero,
	then we have
	\begin{align*}
		\log|\Lambda|>-C(r,d)\log(\mathrm{max}\left\{e,B\right\})\prod_{i=1}^{r}h'(\alpha_i),
	\end{align*}
	where $d:=[\mathbb{Q}(\alpha_1,\ldots,\alpha_r):\mathbb{Q}]$, $B:=\mathrm{max}\left\{|b_1|,\ldots,|b_r|\right\}$,
	\begin{align*}
		C(r,d):=18(r+1)!~r^{r+1}(32d)^{r+2}\log(2rd),
	\end{align*}
	and $h'(\alpha):=\mathrm{max}\left\{h(\alpha)/d,|\log\alpha|/d,1/d\right\}$, where $h(\alpha)$ is the logarithmic Weil height of $\alpha$.
\end{theorem}

This deep theorem can be applied to the  Diophantine equations in (\ref{ell3}) and (\ref{ell5}).
We shall now assume that $n$ is fixed for the remainder of this discussion.
Namely, we view potential integer points as factorizations, in the ring of integers
of  the quadratic fields $K=\QQ(\sqrt{-\varepsilon \ell^m}),$
given by
\begin{align*}
	(X+\sqrt{-\varepsilon \ell^m})(X-\sqrt{-\varepsilon \ell^m}) =Y^n \ \  \ {\text {\rm and}}\ \  \
	(X+2\sqrt{-\varepsilon \ell^m})(X-2\sqrt{-\varepsilon \ell^m}) =Y^n.
\end{align*}
Namely, if $[K:\QQ]=2$ and $h_K=1$, then we have $\beta\in\mathcal{O}_K$ such that $N_{K/\mathbb{Q}}(\beta)=Y$ and 
\begin{align*}
	(X+\sqrt{-\varepsilon \ell^m})=\beta^n ~(\mathrm{mod}~ \mathcal{O}_{K}^{\times})
	\ \ \ {\text {\rm and}}\ \ \ 
	(X+2\sqrt{-\varepsilon \ell^m})=\beta^n ~(\mathrm{mod}~ \mathcal{O}_{K}^{\times}).
\end{align*}
If $K$ does not have class number one, then we may pick $\beta\in\mathcal{O}_K$ such that $  N_{K/\mathbb{Q}}(\beta)=Y^{h_K}$ and consider $\beta^{{n}/{h_K}}$ instead.
This only applies when $\varepsilon=1, \ell=5$ and $m$ is odd, in which case  $h_{\QQ(\sqrt{-5})}=2$. 
In these cases we let $\overline{\beta}$ denote the Galois conjugate of $\beta$. Finally, if $K=\QQ$, then we may pick $\beta,\overline{\beta}\in\ZZ$  (abusing notation) such that $\beta\overline{\beta}=Y$ and $|{\beta}|\leq \sqrt{|Y|}.$
In each case, the algebraic integer $\beta$ is uniquely determined up to unit.

Given such a $\beta$, we construct a corresponding linear form in logarithms arising from $\beta/\overline{\beta}.$ 
For convenience, we denote the relevant fundamental units by
$w_{3}:=2+\sqrt{3}$ and $w_{5}:=1/2+\sqrt{5}/2$, and we denote the 6th root of unity by $w_{-3}:=1/2+\sqrt{-3}/2.$
By taking logarithms,  we obtain a triple of integers $0\leq j_4\leq 3, 0\leq j_6\leq 5,$ and $0\leq j_n<n-1,$ for which one of the
corresponding forms (depending on $\varepsilon, \ell$ and the parity of $m$), say
$\Lambda_{T^{\varepsilon}(\ell,m)}$ and  $\Lambda_{U^{\varepsilon}(m)},$ is given by
\begin{equation}
	\Lambda_{T^{\varepsilon}(\ell,m)}:= \begin{cases} j_6\log({\overline{w}_{-3}}/ w_{-3})-n\log({\overline{\beta}}/\beta)+ki\pi \ \ \ \ \ &{\text {\rm if $\varepsilon=+, m$ odd, and $\ell=3$}},\\
		j_4\log({\overline{i}}/ i)-n\log({\overline{\beta}}/\beta)+ki\pi \ \ \ \ \ &{\text {\rm if $\varepsilon=+, m$ even, and $\ell=3$}},\\
		-(n/2)\log({\overline{\beta}}/\beta)+ki\pi \ \ \ \ \ &{\text {\rm if $\varepsilon=+, m$ odd, and $\ell=5$}},\\
		j_4\log({\overline{i}}/ i)-n\log({\overline{\beta}}/\beta)+ki\pi \ \ \ \ \ &{\text {\rm if $\varepsilon=+, m$ even, and $\ell=5$}},\\
		j_n\log(\overline{w}_{3}/w_{3})-n\log({\overline{\beta}}/\beta) \ \ \ \ \ &{\text {\rm if $\varepsilon=-, m$ odd, and $\ell=3$}},\\
		-n\log(\overline{\beta}/\beta) \ \ \ \ \ &{\text {\rm if $\varepsilon=-, m$ even, and $\ell=3$}},\\
		j_n\log(\overline{w}_{5}/w_{5})-n\log({\overline{\beta}}/\beta)  \ \ \ \ \ &{\text {\rm if $\varepsilon=-, m$ odd, and $\ell=5$}},\\
		-n\log(\overline{\beta}/\beta) \ \ \ \ \ &{\text {\rm if $\varepsilon=-, m$ even, and $\ell=5$}},\\
	\end{cases}
\end{equation}     
and
\begin{equation}
	\Lambda_{U^{\varepsilon}(m)}:= \begin{cases}  -(n/2)\log({\overline{\beta}}/\beta)+ki\pi \ \ \ \ \ &{\text {\rm if $\varepsilon=+$ and $m$ odd}},\\
		j_4\log({\overline{i}}/ i)-n\log({\overline{\beta}}/\beta)+ki\pi \ \ \ \ \ &{\text {\rm if $\varepsilon=+$ and $m$ even}},\\
		j_n\log(\overline{w}_{5}/w_{5})-n\log({\overline{\beta}}/\beta) \ \ \ \ \ &{\text {\rm if $\varepsilon=-$ and $m$ odd}},\\
		-n\log(\overline{\beta}/\beta) \ \ \ \ \ &{\text {\rm if $\varepsilon=-$ and $m$ even}},
	\end{cases}
\end{equation}          
where $k\in\mathbb{Z}$ with $
|\Lambda_{T^{+}(\ell,m)}|,~|\Lambda_{U^{+}(m)}|<\pi$. 
The next lemma bounds these quantities.

\begin{lemma}\label{lambdabound} Assuming the notation and hypotheses above, the following are true.
	
	\noindent
	(1)  If $n>2\log(4\sqrt{\ell^m})/\log |Y|$ and $(X,Y)$ is an integer point on (\ref{ell3}),  with $Y\not \in \{0,\pm 1\}$, then
	\begin{align*}
		|\Lambda_{T^{\varepsilon}(\ell,m)}|\leq 2.78\cdot \frac{\sqrt{\ell^{m}}}{|Y|^{\frac{n}{2}}}.
	\end{align*}
	
	\noindent
	(2) If $n>2\log(8\sqrt{5^m})/\log |Y|$, and $(X,Y)$ is an integer point on (\ref{ell5}), with $Y\neq 0$, then 
	\begin{align*}
		|\Lambda_{U^{\varepsilon}(m)}|\leq 5.56\cdot \frac{\sqrt{5^{m}}}{|Y|^{\frac{n}{2}}}.
	\end{align*}
\end{lemma}

\begin{proof}
	By the definition of $\Lambda_{T^{\varepsilon}(\ell,m)}$, we directly find that
	\begin{align}\label{lambda1}
		|e^{\Lambda_{T^{\varepsilon}(\ell,m)}}-1|=
		\left|\frac{X+\sqrt{\pm\ell^m}}{X-\sqrt{\pm\ell^m}}-1\right|\leq\frac{2\sqrt{\ell^m}}{{|Y|^{\frac{n}{2}}}}.
	\end{align}
	For $|z|<1/2$, we note that
	$|\log(1+z)|\leq 1.39\cdot |z|.$
	Also, we note that the hypothesis on $n$ gives $|e^{\Lambda_{T^{\varepsilon}(\ell,m)}}-1|<1/2$.
	Hence, we obtain (1), the claimed inequality
	\begin{align*}
		|\Lambda_{T^{\varepsilon}(\ell,m)}|\leq 1.39\cdot |e^{\Lambda_{T^{\varepsilon}(\ell,m)}}-1|=2.78\cdot \frac{\sqrt{\ell^m}}{{|Y|^{\frac{n}{2}}}}.
	\end{align*}
	The same method gives (2), after noting that $Y=\pm 1$  has no integer point on (\ref{ell5}). 
\end{proof}

\section{More Diophantine equations}

Here we prove some Diophantine results concerning families of Lebesgue--Ramanujan--Nagell type equations which are
of independent interest.
To make them precise,
for $\ell\in \{3, 5\}, \varepsilon\in \{\pm\},$ and $m\in \ZZ^{+}$, we define
\begin{equation}
	T^{\varepsilon}(\ell,m):= \begin{cases} 2m+10^{32}\sqrt{m} \ \ \ \ \ &{\text {\rm if $\varepsilon=+$ and $\ell=3$}},\\
		2m+10^{23}\sqrt{m} \ \ \ \ \ &{\text {\rm if $\varepsilon=-, m$ odd, and $\ell=3$}},\\
		2m+10^{13}\sqrt{m} \ \ \ \ \ &{\text {\rm if $\varepsilon=-, m$ even, and $\ell=3$}},\\
		
		3m+10^{24}\sqrt{m} \ \ \ \ \ &{\text {\rm if $\varepsilon=\pm , m$ odd, and $\ell=5$}},\\
		3m+10^{30}\sqrt{m} \ \ \ \ \ &{\text {\rm if $\varepsilon=+, m$ even, and $\ell=5$}},\\
		3m+10^{13}\sqrt{m} \ \ \ \ \ &{\text {\rm if $\varepsilon=-, m$ even, and $\ell=5$}}.\\                            \end{cases}
\end{equation}           
Furthermore, we define  $U^{\varepsilon}(m)$ by
\begin{equation}
	U^{\varepsilon}(m):= \begin{cases}  3m+10 ^{24}\sqrt{m}\ \ \ \ \ &{\text {\rm if $\varepsilon=\pm $ and $m$ odd}},\\
		3m+10^{30}\sqrt{m} \ \ \ \ \ &{\text {\rm if $\varepsilon=+$ and $m$ even}},\\
		3m+10^{13}\sqrt{m} \ \ \ \ \ &{\text {\rm if $\varepsilon=-$ and $m$ even}}.
	\end{cases}
\end{equation}

\begin{theorem}\label{Explicit35}
	If $\ell \in \{3, 5\},$ $\varepsilon\in \{\pm \}$, and $m\in \ZZ^{+}$, then the following are true.
	\newline
	(1)  If $n>T^{\varepsilon}(\ell,m)=O_{\ell}(m),$ then there are no integer points\footnote{We switch $X$ and $Y$ here to be consistent with the literature on Lebesgue--Ramanujan--Nagell equations.} $(X,Y),$ with $Y\not \in \{0,\pm 1\}$, on
	\begin{align}\label{ell3}
		X^2 +\varepsilon \ell^m=Y^n.
	\end{align}
	\newline
	\noindent
	(2) If $n>U^{\varepsilon}(m)=O_{\ell}(m),$ then there are no integer points $(X,Y),$ with $Y\neq 0$, on
	\begin{align}\label{ell5}
		X^2+\varepsilon 4\cdot 5^m = Y^n.
	\end{align}
\end{theorem}

\section{Proof of Theorem \ref{Power}}

For brevity, we only consider  when $\ell=3$ and $\varepsilon=-$,  as  the same method applies to all of the cases.
Suppose that there is an integer point $(X,Y)$ on $X^2+3^m=Y^n$.
Therefore, there is an integer $0\leq j_6\leq 5$ and an algebraic integer $\beta \in \QQ(\sqrt{-3})$ for which
$N_{K/\QQ}(\beta)=Y$ and
\begin{align*}
	(X+\sqrt{-3^m})=\frac{\beta^n}{w_{-3}^{j_6}}.
\end{align*}
In particular, if $m$ is odd, then we have
\begin{align*}
	\Lambda_{T^{\varepsilon}(\ell,m)}=j_6\log({\overline{w}_{-3}}/ w_{-3})-n\log({\overline{\beta}}/\beta)+ki\pi
	=j_6\log({\overline{w}_{-3}}/ w_{-3})-n\log({\overline{\beta}}/\beta)+k\log(-1).
\end{align*}
Since $\Lambda_{T^{\varepsilon}(\ell,m)}\neq 0,$ Theorem \ref{BW} implies that
\begin{align*}
	\log|\Lambda_{T^{\varepsilon}(\ell,m)}|>-C(3,2)h'(\overline{w}_{-3}/w_{-3})h'(\overline{\beta}/\beta)h'(-1)\log(\mathrm{max}\left\{e,j_6,n,|k|\right\}.
\end{align*}
Furthermore, by a short calculation, we get
\begin{displaymath}
	\begin{split}
		&h'(\overline{w}_{-3}/w_{-3})\leq\frac{\pi}{3},\\
		&h'(\overline{\beta}/\beta)\leq\mathrm{max}\left\{\log |Y|, \pi\right\}\\
		&h'(-1)\leq\frac{\pi}{2},~\mathrm{max}\left\{e,j_6,n,|k|\right\})\leq n+5.
	\end{split}
\end{displaymath}
Therefore, Theorem~\ref{BW} implies that
\begin{align*}
	\log|\Lambda_{T^{\varepsilon}(\ell,m)}|&>-\frac{\pi^2}{6}C(3,2)\mathrm{max}\left\{\log |Y|, \pi\right\}\log(n+5).
\end{align*}
However, Lemma \ref{lambdabound} (1) gives
\begin{align*}
	\log(2.78\cdot\sqrt{3^m})-\frac{n}{2}\cdot {\color{black}\log |Y|} >\log|\Lambda_{T^{\varepsilon}(\ell,m)}|&>-\frac{\pi^3}{6}C(3,2)\log(n+5)\cdot {\color{black}\log |Y|},
\end{align*}
which in turn implies that
$$
\log(2.78\cdot\sqrt{3^m})-\frac{n}{2}\log 2  >-\frac{\pi^3}{6}C(3,2)\sqrt{n+4}.
$$
Since we have $C(3,2)=18(4)!~3^4(6
4)^5\log(12),$ a direct calculation shows that we must have
\begin{align*}
	n\leq 1.6m+(60\sqrt{m}+5.9)\cdot 10^{30},
\end{align*}
which gives a constant that is smaller than the claimed $M^{-}(3,m).$ Taking into account even $m$,
a similar calculation gives $n< 1.6m +(9.4\sqrt{m}+1.4)\cdot 10^{31}.$ The claimed $M^{-}(3,m)$ is
a ``rounded up'' version of the maximum of these two constants.

\begin{proof}[Proof of Theorem~\ref{Power}]
Suppose that $\ell^m$ is a power of an odd prime. Thanks to Theorem~\ref{LehmerVariantGeneral}, if $a_f(n)=\pm \ell^m,$
then $n=p^{d-1}$, where $p$ and $d\mid \ell(\ell^2-1)$ are odd primes. 
For each $d$, Lemma~\ref{DiophantineCriterion} gives an integer
point on an elliptic or hyperelliptic curve, or gives an integer solution to a Thue equation.

If $\ell =3$ (resp. $\ell=5$), then we find that the only possibility is $d=3$ (resp. $d=3, 5$).  This leads to the equations 
in Theorem~\ref{Explicit35}, which in turn gives the claimed bounds in these cases.
Turning to $\ell \geq 7$, we note for $d=3$ (resp. $5$) that one
can argue again as in the proof of Theorem~\ref{Explicit35} to conclude that $a_f(p^2)\neq \pm \ell^m$ 
(resp. $a_f(p^4)\neq \pm \ell^m$) for $f$ with (effectively) sufficiently large  weight $2k$.
For any $d\geq 7$, Lemma~\ref{DiophantineCriterion} (3)  gives the
integer solution $(X,Y)=(p^{2k-1}, a_f(p^2))$ to the Thue equation
$$
F_{d-1}(X,Y)=\pm \ell^m.
$$
As an implementation of Baker's theory of linear forms in logarithms, 
a well-known paper of Tzanakis and de Weger (see p. 103 of \cite{TW89})  on Thue equations gives a method for effectively determining an
upper bound\footnote{The reader should switch the roles of $X$ and $Y$ when applying the discussion in \cite{TW89}.}  for $|X|$ of any integer point satisfying $F_{d-1}(X,Y)=\pm \ell^m$, which in turn leads to an upper bound for the weight $2k$.  The linearity of these constants in $m$ aspect follows from the formal
taking of a logarithm in these Diophantine equations.
\end{proof}

\section{Appendix: Tables}

\begingroup
\setlength{\tabcolsep}{3pt} 
\renewcommand{\arraystretch}{0.7}
\begin{table}[!ht]
	\begin{center} 
		\begin{tabular}{|c|c|}
			\multicolumn{2}{c}{} \\ \hline
			$(A,B)$ & Defective $u_n(\alpha, \beta)$ \\ \hline \hline
			\multirow{2}{3.5em}{$(\pm 1,2^1)$} & $u_5 = -1$, $u_7 = 7$, $u_8 = \mp 3$, $u_{12} = \pm 45$, \\ & $u_{13} = -1$, $u_{18} = \pm 85$, $u_{30} = \mp 24475$ \\ \hline
			$(\pm 1,3^1)$ & $u_5 = 1$, $u_{12} = \pm 160$ \\ \hline
			$(\pm 1, 5^1)$ & $u_7 = 1$, $u_{12} = \mp 3024$ \\ \hline
			$(\pm 2, 3^1)$ & $u_3 = 1$, $u_{10} = \mp 22$ \\ \hline
			$(\pm 2,7^1)$ & $u_8 = \mp 40$ \\ \hline
			$(\pm 2, 11^1)$ & $u_5 = 5$ \\ \hline
			$(\pm 4, 5^1)$ & $u_6=\pm 44$\\ \hline
			$(\pm 5, 7^1)$ & $u_{10} = \mp 3725$ \\ \hline
			$(\pm 3, 2^3)$ & $u_3 = 1$ \\ \hline
			$(\pm 5, 2^3)$ & $u_6 = \pm 85$ \\ \hline
		\end{tabular}
		\medskip
		\captionof{table}{\textit{Sporadic examples of defective $u_n(\alpha, \beta)$ satisfying (\ref{Modularity})}} 
		\label{table1}
	\end{center}
\end{table}

\endgroup

\noindent
The families of defective Lucas numbers satisfying (\ref{Modularity}) are given
by the following curves.
\begin{equation}\label{Table2Curves}
	\begin{split}
		\ \ \ \ \ \ &B_{1, k}^{r, \pm} : Y^2 = X^{2k-1} \pm 3^r, \ \ \ \   B_{2,k} : Y^2 = 2X^{2k-1} - 1, \ \ \ B_{3, k}^\pm : Y^2 = 2X^{2k-1} \pm 2, \\ \ \ B_{4,k}^r &: Y^2 =  3X^{2k-1} + (-2)^{r+2}, \ \ \  B_{5,k}^{\pm} : Y^2=3X^{2k-1} \pm 3, \ \ \  B_{6,k}^{r,\pm} : Y^2 = 3X^{2k-1} \pm 3 \cdot 2^r.
	\end{split}
\end{equation}

\begingroup
\setlength{\tabcolsep}{3pt} 
\renewcommand{\arraystretch}{0.7}
\begin{table}[!ht]
	\begin{center} 
		\begin{small}
			\begin{tabular}{|c|c|c|}
				\hline
				$(A,B)$ & Defective $u_n(\alpha, \beta)$ & Constraints on parameters \\ \hline \hline
				$(\pm m, p)$ & $u_3 = -1$ & $m>1$ and $p = m^2+1$ \\ \hline
				\multirow{2}{*}{}
				$(\pm m, p^{2k-1})$ &
				$u_3 = \varepsilon 3^r$ &
				$\begin{aligned} &\ \ \ \ \ \ (p, \pm m)\in B_{1, k}^{r, \varepsilon} \text{ with } 3\nmid m,\\ 
					&(\varepsilon,r,m)\neq (1,1,2),
					\text{ and } m^2 \geq 4\varepsilon 3^{r-1} \end{aligned}$\\ \hline
				{\color{black}$(\pm m, p^{2k-1})$} & {\color{black}$u_4 = \mp m$} & {\color{black}$(p,\pm m) \in B_{2,k}$ with $m > 1$ odd} \\ \hline
				$(\pm m, p^{2k-1})$ & $u_4 = \pm 2{\color{black}\varepsilon}m$ & $\begin{aligned}(p,\pm m)\in B_{3,k}^\varepsilon &\text{ with } {\color{black}(\varepsilon, m) \not = (1,2)}\\ &\text{  and } m > 2
					\text{  even}
				\end{aligned}$ \\ \hline
				$(\pm m, p^{2k-1})$ & 
				$u_6 = {\color{black}\pm (-2)^rm(2m^2+(-2)^r)/3}$ &
				$\begin{aligned}(p, &\pm m)\in B_{4,k}^r \text{ with } \gcd(m,6) = 1, \\
					&{\color{black}(r,m) \not = (1,1)}, \text{ and } {\color{black}m^2 \geq (-2)^{r+2}}\end{aligned}$ \\ \hline
				$(\pm m, p^{2k-1})$ & ${\color{black}u_6=\pm \varepsilon m(2m^2+3\varepsilon)}$ & $(p,\pm m)\in B_{5,k}^\varepsilon$ with $3\mid m$ and $m>3$\\ \hline
				$(\pm m, p^{2k-1})$ & $u_6 = \pm 2^{r+1}{\color{black}\varepsilon} m(m^2 + 3 {\color{black}\varepsilon} \cdot 2^{r-1}) $ &$\begin{aligned} (p, \pm m)\in B_{6,k}^{r,\varepsilon}
					\text{ with } m \equiv 3 \bmod{6} \\  \text{and } m^2 \geq 3 {\color{black}\varepsilon} \cdot 2^{r+2}\end{aligned}$ \\ \hline
			\end{tabular}
		\end{small}
		\medskip
		\captionof{table}{Parameterized families of defective $u_n(\alpha, \beta)$ satisfying (\ref{Modularity})
			\label{table2}
			\textit{\newline Notation: $m, k, r\in \ZZ^{+}$, $\varepsilon = \pm 1$, $p$ is a prime number.}}
	\end{center}
\end{table}

\endgroup

\vskip.65in

\begingroup
\setlength{\tabcolsep}{5pt}
\renewcommand{\arraystretch}{0.7}
\begin{table}[!ht]
	\begin{center}
		\begin{tabular}{|c|l|} \hline
			$(a_f(p),p^{2k-1})$ & \multicolumn{1}{c|}{$\widehat{\sigma}(p,m)$} \\[1.5ex] \hline \hline
			
			\multirow{2}{*}{}
			$(\pm 3, 2^3)$ & $\begin{aligned} 
				\sigma_0(m+1) - 2 & \ \text{ when } 3|(m+1), \\
				\sigma_0(m+1) - 1 & \ \text{ otherwise.} \\
			\end{aligned}$ \\[1.5ex] \hline
			
			\multirow{2}{*}{}
			$(\pm 5, 2^3)$ & $\begin{aligned}
				\sigma_0(m+1) - 2 & \ \text{ if } 6|(m+1), \\
				\sigma_0(m+1) - 1 & \ \text{ otherwise}.
			\end{aligned}$ \\[1.5ex] \hline
			
			\multirow{2}{*}{}
			$(\pm m, p^{2k-1})$ & $\begin{aligned}
				\sigma_0(m+1) - 4 & \ \text{ if $(p, \pm m) \in S$,} \\
				\sigma_0(m+1) - 1 & \ \text{ otherwise.}
			\end{aligned}$ \\[1.5ex] \hline
		\end{tabular}
		\medskip
		\captionof{table}{\textit{Lower bounds on $\Omega(a_f(p^m))$ in defective cases for weights $2k \geq 4$.}
			\label{table3}
			\textit{\newline Notation: $S$ is the collection of all points on any of $B_{1,k}^{r,\pm}, B_{2,k}, B_{3,k}^r, B_{4,k}, B_{5,k}^r$.}}
	\end{center}
\end{table}
\endgroup


\bigskip

\begingroup
\setlength{\tabcolsep}{5pt} 
\renewcommand{\arraystretch}{0.7}
\begin{table}[!ht]
	\begin{center} 
		\begin{tabular}{|c|c|}
			\multicolumn{2}{c}{} \\ \hline
			$(d, D)$ & Integer Solutions to $F_{d-1}(X,Y)=D$  \\ \hline \hline
			$(7, \pm 7)$ & $(\pm 1, \pm 4), ( \pm 2, \pm 1), (\mp 3, \mp 5)$\\ \hline
			\multirow{2}{3.5em}{$(7,  \pm 13)$} & $( \pm 3,  \pm 10), (\pm 2,  \pm 7), ( \pm 3, \pm 4), (\pm 4, \pm 1),$ \\ & $( \pm 3,  \pm 1), (\mp 1, \pm 1), (\mp 2, \mp 5), (\mp 5,\mp 8), (\mp 7, \mp 11)$ \\ \hline
			$(7, \pm 29)$ & $\begin{aligned}(\mp 6, \mp 1), (\mp 5, &\mp 16), (\mp 4, \mp 7), (\pm 1, \pm 5), \\&(\pm 3, \pm 2), (\pm 11, \pm 17)\end{aligned}$\\ \hline
			$(11, \pm 11), (19, \pm 19),$  &\multirow{2}{3.5em}{$( \pm 1,\pm 4)$}\\ 
			$(23, \pm 23), (31, \pm 31)$ &\\ \hline
			$(11, \pm 23)$ & $( \pm 3, \pm 2), ( \pm 2, \pm 1), (\mp 2,\mp 3)$\\ \hline
			$(13, 13), (17, 17), (29, 29), (37,37)$ & $(-1, -4), (1,4)$\\ \hline
			$(13, -13),  (17, -17),$  & \multirow{2}{3.5em}{$\varnothing$}\\ 
			$(29, -29), (37,-37)$ &\\ \hline
			$(19, \pm 37)$ & $(\mp 2, \mp 5)$ \\ \hline
		\end{tabular}
		\smallskip
		\captionof{table}{\textit{Solutions for the Thue equations where $D=\pm \ell$ and $7\leq \ell \leq 37$} }
		\label{thuetable}
	\end{center}
\end{table}

\endgroup


\begingroup
\setlength{\tabcolsep}{5pt} 
\renewcommand{\arraystretch}{0.7}
\begin{table}[!ht]
	\begin{center} 
		\begin{tabular}{|c|c|}
			\multicolumn{2}{c}{} \\ \hline
			$(d, D)$ & Integer Solutions to $F_{d-1}(X,Y)=D$  \\ \hline \hline
			$(7, \pm 41) $ & $(\mp 3, \mp 7), (\mp 1, \pm 2), (\pm 4, \pm 5)  $\\ \hline
			$\begin{aligned}&(41, 41), (53, 53), (61,61), \\
				&(73, 73), (89,89), (97, 97)\end{aligned}$ & $(-1, -4), (1,4)$ \\ \hline
			$(41, -41), (23, \pm 47), (13, 53), (53,-53), (29, \pm 59),$ &  \multirow{3}{3.5em}{$\varnothing$} \\ 
			$(31, \pm 61), (61, -61), (17,-67), (37, \pm 73), (73,-73),$ & \\
			$(13,-79), (41, \pm 83), (89,-89), (97,-97)$ & \\   \hline
			$(7,  \pm 43)$ & $(\mp 3, \mp 8), (\mp 2, \pm 1), (\pm 5, \pm 7) $\\ \hline
			$(11, \pm 43)$ & $(\mp 3, \mp 5), (\pm 2, \pm 5)$\\ \hline
			$(43, \pm 43), (47, \pm 47), (59, \pm 59), (67,\pm 67),$ &  \multirow{2}{3.5em}{$( \pm 1,\pm 4)$}\\
			$(71, \pm 71), (79, \pm 79), (83, \pm 83)$ &  \\ \hline
			$(13, -53), (17,67)$ & $ (-2,-3), (2, 3)$\\ \hline
			$(11, \pm 67)$ & $(\mp 7, \mp 12), (\mp 3, \mp 11), (\mp 2, \mp 7)$ \\ \hline
			
			\multirow{2}{3.5em}{$(7,\pm 71)$} & $(\mp 16, \mp 25), (\mp 5, \mp 9), (\pm 1, \pm 6),$\\ & $(\pm 4, \pm 3), (\pm 7, \pm 23), (\pm 9, \pm 2)$\\ \hline
			$(13, 79)$ & $(-2,-5),(2,5)$\\ \hline
			
			\multirow{2}{3.5em}{$(7,\pm 83)$} & $(\mp 8, \mp 13), (\mp 7, \mp 1), (\mp 6, \mp 19),$\\
			&  $(\pm 3, \pm 11), (\pm 5, \pm 2), (\pm 13, \pm 20)$\\ \hline
			$(11, \pm 89)$ & $(\mp 1, \pm 1) $\\ \hline
			$(7,\pm 97)$ & $(\mp 4, \mp 11), (\mp 3, \pm 1), (\pm 7, \pm 10)$\\ \hline
			
		\end{tabular}
		\medskip
		\captionof{table}{\textit{Solutions (with GRH) to the Thue equations where $D=\pm \ell$ and $41\leq \ell \leq 97$} }
		\label{thueGRHtable}
	\end{center}
\end{table}
\endgroup

\bigskip

\begingroup
\setlength{\tabcolsep}{2pt} 
\renewcommand{\arraystretch}{0.5}
\begin{table}
	\begin{small}
	\begin{center}
		\begin{tabular}{|c|c|c|c|c|c|cl}
			\hline
			$\ell$ & {\text {\rm $C_{2,\ell}^{+}$}} & {\text {\rm $C_{3,\ell}^{+}$}} & {\text {\rm $C_{4,\ell}^{+}$}}  & {\text {\rm $C_{6,\ell}^{+}$}} & {\text {\rm $C_{7,\ell}^{+}$}}\\  \hline \hline
			$3$ & $(1,\pm 2)$ & $(1,\pm 2)$ & $(1,\pm 2)$  & $(1, \pm 2)$ & $(1,\pm 2)$  \\ \hline
			$5$ & $(-1,\pm 2)$ & $(-1, \pm 2)$ & $(-1,\pm 2)$  &$(-1, \pm 2)$ & $(-1.\pm 2)$ \\ \hline
			$\begin{aligned}7, &23, 29, 47,
				53,\\  &59, 61, 67, 83
			\end{aligned} $ & $\varnothing$ & $\varnothing$ & $\varnothing$  & $\varnothing$ & $\varnothing$ \\ \hline
			$11$ & $\varnothing$ & $(5,\pm 56)$ & $\varnothing$  & $\varnothing$ & $\varnothing$ \\ \hline
			$13$ & $\varnothing$ & $(3,\pm 16)$ & $\varnothing$  & $\varnothing$  & $\varnothing$ \\ \hline
			$17$ & $\begin{aligned} &(-2,\pm 3), (-1,\pm 4), (2,\pm 5), &\\ &(4,\pm 9), (8, \pm 23) 
				(43, \pm 282), \\ &(52, \pm 375), (5234, \pm 378661)\end{aligned}$ & $(-1,\pm 4)$ & $(-1,\pm 4)$  & $(-1,\pm 4)$ & $(-1,\pm 4)$ \\ \hline
			$19$ & $(5, \pm 12)$ & $\varnothing$ & $\varnothing$  & $\varnothing$ & $\varnothing$ \\ \hline
			$31$ & $(-3,\pm 2)$  & $\varnothing$ & $\varnothing$ & $\varnothing$ & $\varnothing$\\ \hline
			$37$ & $\begin{aligned} (-1,\pm 6), (3, \pm 8),\\ (243,\pm 3788) \end{aligned}$  & $(-1,\pm 6), (27, \pm 3788)$ & $(-1,\pm 6)$ & $(-1,\pm 6)$ & $(-1,\pm 6)$\\ \hline
			$41$ & $(2,\pm 7)$   & $(-2,\pm 3)$ & $(2,\pm 13)$ & $\varnothing$ & $\varnothing$\\ \hline
			$43$ & $(-3,\pm 4) $  & $\varnothing$ & $\varnothing$ & $\varnothing$ & $\varnothing$\\ \hline
			$71$ & $(5, \pm 14)$  & $\varnothing$ & $\varnothing$ & $\varnothing$ & $\varnothing$\\ \hline
			$73$ & $\begin{aligned}(-4,\pm 3), (2,\pm 9), \\ (3,\pm 10),
				(6,\pm 17),\\  (72,\pm 611), (356, \pm 6717)\end{aligned}$  & $\varnothing$ & $\varnothing$ & $\varnothing$ & $\varnothing$\\ \hline
			$79$ & $(45,\pm 302) $  & $\varnothing$ & $\varnothing$ & $\varnothing$ & $\varnothing$\\ \hline
			$89$ & $\begin{aligned}(-4, \pm 5), (-2, \pm 9), \\
				(10,\pm 33), (55,\pm 408)\end{aligned}$  & $(2,\pm 11)$ & $\varnothing$ & $\varnothing$ & $\varnothing$\\ \hline
			$97$ & $\varnothing $  & $\varnothing$ & $(2,\pm 15)$ & $\varnothing$ & $\varnothing$\\ \hline
		\end{tabular}
		\medskip
		\captionof{table}{\textit{Integer points on $C_{d,\ell}^{+}$}}
		\label{Cplustable}
	\end{center}
	\end{small}
\end{table}
\endgroup

\begingroup
\setlength{\tabcolsep}{3pt} 
\renewcommand{\arraystretch}{0.7}
\begin{table}
	\begin{center}
		\begin{tabular}{|c|c|c|c|c|c|}
			\hline
			$\ell$ & {\text {\rm $C_{2,\ell}^{-}$}} & {\text {\rm $C_{3,\ell}^{-}$}} & {\text {\rm $C_{4,\ell}^{-}$}}  & {\text {\rm $C_{6,\ell}^{-}$}} & {\text {\rm $C_{7,\ell}^{-}$}}\\  \hline \hline
			$\begin{aligned} &\ \ 3, 5, 17, 29, 37,\\  &41, 43,  59, 73, 97
			\end{aligned} $&  $\varnothing$ & $\varnothing$ & $\varnothing$  & $\varnothing$  & $\varnothing$\\ \hline
			$7$ & $(2,\pm 1), \ (32,\pm 181)$ & $(2,\pm 5), \ (8,\pm 181)$ & $(2, \pm 11)$ &  $\varnothing$ & $\varnothing$ \\ \hline
			$11$ & $(3,\pm 4), \ (15, \pm 58)$ & $\varnothing$ & $\varnothing$  & $\varnothing$ & $\varnothing$ \\ \hline
			$13$ & $(17, \pm 70)$ & $\varnothing$ & $\varnothing$  & $\varnothing$ & $\varnothing$ \\ \hline
			$19$ & $(7, \pm 18)$ & $(55, \pm 22434)$ & $\varnothing$  & $\varnothing$  & $\varnothing$ \\ \hline
			$23$& $(3,\pm 2)$ & $(2,\pm 3)$ & $\varnothing$  & $(2,\pm 45)$ & $\varnothing$\\ \hline
			$31$ & $\varnothing$  & $(2,\pm 1)$  & $\varnothing$ & $\varnothing$ & $\varnothing$ \\ \hline
			$47$ & $(6, \pm 13), (12, \pm 41), (63, \pm 500) $  & $(3, \pm 14)$ & $(2,\pm 9)$ & $\varnothing$ & $\varnothing$\\ \hline
			$53$ & $(9,\pm 26), (29, \pm 156)$  & $\varnothing$ & $\varnothing$ & $\varnothing$ & $\varnothing$\\ \hline
			$61$ & $(5,\pm 8)$   & $\varnothing$ & $\varnothing$ & $\varnothing$ & $\varnothing$\\ \hline
			$67$ & $(23, \pm 110)$  & $\varnothing$ & $\varnothing$ & $\varnothing$ & $\varnothing$\\ \hline
			$71$ & $(8,\pm 21)$  & $\varnothing$ & $(3,\pm 46)$ & $\varnothing$ & $\varnothing$\\ \hline
			$79$ & $(20, \pm 89)$  & $\varnothing$ & $(2,\pm 7)$ & $\varnothing$ & $\varnothing$\\ \hline
			$83$ & $(27, \pm 140)$  & $\varnothing$ & $\varnothing$ & $\varnothing$ & $\varnothing$\\ \hline
			$89$ & $(5,\pm 6) $  & $\varnothing$ & $\varnothing$ & $\varnothing$ & $\varnothing$\\ \hline
		\end{tabular}
		\medskip
		\captionof{table}{\textit{Integer points on $C_{d, \ell}^{-}$}}
		\label{Cminustable}
	\end{center}
\end{table}
\endgroup

\bigskip

\begingroup
\setlength{\tabcolsep}{5pt} 
\renewcommand{\arraystretch}{0.7}
\begin{table}
	\begin{center}
		\begin{tabular}{|c|c|c|c|c|c|c|c|c|c|}
			\hline
			$\ell$ & {\text {\rm $H_{3,\ell}^{-}$}} & {\text {\rm $H_{3,\ell}^{+}$}} & {\text {\rm $H_{5,\ell}^{-}$}}  & {\text {\rm $H_{5,\ell}^{+}$}}  & {\text {\rm $H_{7,\ell}^{-}$}}  & {\text {\rm $H_{7,\ell}^{+}$}}    & {\text {\rm $H_{11,\ell}^{-}$}}   & {\text {\rm $H_{13,\ell}^{-}$}}  \\  \hline \hline
			$11$ & $\varnothing$ & $(1,7), (7, 767)$ & $\varnothing$  & $(1,7)$ & $\varnothing$ & $(1,7)$ & $\varnothing_*$   &$\varnothing$  \\ \hline
			$19$ & $\varnothing$ & $(1,9), (3,61)$ & $\varnothing$  & $(1,9)$ & $\varnothing$ & $(1,9)$ & $\varnothing$ &$\varnothing_*$   \\ \hline
			$29$ & $\varnothing$ & $(1,11)$ & $\varnothing$  & $(1,11)$ & $\varnothing$ & $(1,11)$ & $\varnothing_*$ & $\varnothing_*$   \\ \hline
			$31$ & $(2,14)$ & $\varnothing$ & $\varnothing$  & $\varnothing$ & $(2,286)$ & $\varnothing$ & $\varnothing_*$   &$\varnothing_*$  \\ \hline
			$41$ & $(3,59)$ & $(1,13), (2,22)$ & $\varnothing$  & $(1,13)$ & $\varnothing$ & $(1,13)_*$ & $\varnothing_*$  & $\varnothing_*$  \\ \hline
			$59$ & $\varnothing$ & $\varnothing$ & $\varnothing$ &$\varnothing_*$ & $\varnothing$ & $\varnothing_*$ & $\varnothing_*$  &$\varnothing_*$  \\ \hline
			$61$ & $\varnothing$ & $\varnothing$ & $\varnothing$  & $\varnothing$ & $\varnothing$ & $\varnothing_*$ & $\varnothing_*$  &$\varnothing_*$  \\ \hline
			$71$ & $(2,6), (5, 279)$ & $(1,17)$ & $\varnothing$ & $(1,17)$ & $\varnothing$ & ${\text {\rm {\bf ?}}}$ & $\varnothing_*$   &$\varnothing_*$  \\ \hline
			$79$ & $(2,2), (4,142)$ & $\varnothing$ & $\varnothing$  & $\varnothing$ & $\varnothing$ & $\varnothing_*$ & $\varnothing_*$   &$\varnothing_*$  \\ \hline
			$89$ & $\varnothing$ & $(1,19), (2,26)$ & $\varnothing$  & $(1,19)_*$, $(2,74)_*$ & $\varnothing$ & $(1,19)_*$ & $\varnothing_*$  & ${\text {\rm {\bf ?}}}$  \\ \hline
		\end{tabular}
		\medskip
		\captionof{table}{\textit{$(|X|, |Y|)$ for integer points on $H_{d,\ell}^{\pm}$ with $\legendre{\ell}{5}=1$. (note. GRH assumption indicated by $_*$.)}}
		\label{Htable}
	\end{center}
\end{table}
\endgroup

\sglsp


\appendix

\printindex

\end{document}